\documentclass[opre,nonblindrev]{informs3_modified1} % current default for manuscript submission

\DoubleSpacedXI % Made default 4/4/2014 at request
\usepackage{endnotes}
\let\footnote=\endnote
\let\enotesize=\normalsize
\def\notesname{Endnotes}%
\def\makeenmark{$^{\theenmark}$}
\def\enoteformat{\rightskip0pt\leftskip0pt\parindent=1.75em
  \leavevmode\llap{\theenmark.\enskip}}

\usepackage{bm}
\usepackage{algorithm}
\usepackage{algorithmic}
\usepackage{caption}
\usepackage{subcaption}
\newcommand{\R}{\mathbb{R}}
\newcommand{\abs}[1]{\left|#1\right|}
\newcommand{\tabincell}[2]{\begin{tabular}{@{}#1@{}}#2\end{tabular}}

\usepackage{multibib}
\newcites{APX}{References}

\usepackage{natbib}
 \bibpunct[, ]{(}{)}{,}{a}{}{,}%
 \def\bibfont{\small}%
\TheoremsNumberedThrough     % Preferred (Theorem 1, Lemma 1, Theorem 2)
\ECRepeatTheorems

\EquationsNumberedThrough    % Default: (1), (2), ...
\begin{document}
%%%%%%%%%%%%%%%%

%\setcounter{page}{1}
%\DOI{10.1287/xxxx.0000.0000}%

\RUNAUTHOR{Lam and Qian}

\RUNTITLE{Bounding Optimality Gap via Bagging}

% \TITLE{Bundling Information Goods of Decreasing Value}
\TITLE{Bounding Optimality Gap in Stochastic Optimization via Bagging: Statistical Efficiency and Stability}

\ARTICLEAUTHORS{%
\AUTHOR{Henry Lam}
\AFF{Department of Industrial Engineering and Operations Research, Columbia University, New York, NY 10027, \EMAIL{khl2114@columbia.edu}} %, \URL{}}
\AUTHOR{Huajie Qian}
\AFF{Department of Industrial Engineering and Operations Research, Columbia University, New York, NY 10027, \EMAIL{h.qian@columbia.edu}}
} % end of the block

\ABSTRACT{%
We study a statistical method to estimate the optimal value, and the optimality gap of a given solution for stochastic optimization as an assessment of the solution quality. Our approach is based on bootstrap aggregating, or bagging, resampled sample average approximation (SAA). We show how this approach leads to valid statistical confidence bounds for non-smooth optimization. We also demonstrate its statistical efficiency and stability that are especially desirable in limited-data situations, and compare these properties with some existing methods. We present our theory that views SAA as a kernel in an infinite-order symmetric statistic, which can be approximated via bagging. We substantiate our theoretical findings with numerical results.
% Enter your abstract
}%

%\KEYWORDS{deterministic inventory theory; infinite linear programming duality;
%  existence of optimal policies; semi-Markov decision process; cyclic schedule}

% \HISTORY{This paper was
% first submitted on April 12, 1922 and has been with the authors for
% 83 years for 65 revisions.}

\maketitle
%%%%%%%%%%%%%%%%%%%%%%%%%%%%%%%%%%%%%%%%%%%%%%%%%%%%%%%%%%%%%%%%%%%%%%

% Samples of sectioning (and labeling) in OPRE
% NOTE: (1) \section and \subsection do NOT end with a period
%       (2) \subsubsection and lower need end punctuation
%       (3) capitalization is as shown (title style).
%
%\section{Introduction.}\label{intro} %%1.
%\subsection{Duality and the Classical EOQ Problem.}\label{class-EOQ} %% 1.1.
%\subsection{Outline.}\label{outline1} %% 1.2.
%\subsubsection{Cyclic Schedules for the General Deterministic SMDP.}
%  \label{cyclic-schedules} %% 1.2.1
%\section{Problem Description.}\label{problemdescription} %% 2.

% Text of your paper here

\section{Introduction}
\label{sec:intro}

Consider a stochastic optimization problem
\begin{equation}
Z^*=\min_{x\in\mathcal X}\{Z(x)=E_F[h(x,\xi)]\}\label{opt}
\end{equation}
where $\mathcal X$ is the decision space, $\xi\in\Xi$ is generated under some distribution $F$, and $E_F[\cdot]$ denotes its expectation. We focus on the situations where $F$ is not known, but instead a collection of i.i.d.~data for $\xi$, say $\bm\xi_{1:n}=(\xi_1,\ldots,\xi_n)$, are available. Obtaining a good solution for \eqref{opt} under this setting has been under active investigation both from the stochastic and the optimization communities. Common methods include the sample average approximation (SAA) (\cite{shapiro2009lectures,kleywegt2002sample,higle1996stochastic}), stochastic approximation (SA) or gradient descent (\cite{kushner2003stochastic,borkar2009stochastic,nemirovski2009robust}), and (distributionally) robust optimization (\cite{delage2010distributionally,bertsimas2014robust,wiesemann2014distributionally,ben2013robust}). These methods aim to find a solution that is nearly optimal, or in some way provide a safe approximation. Applications of the generic problem \eqref{opt} and its data-driven solution techniques span from operations research, such as inventory control, revenue management, portfolio selection (see, e.g., \cite{shapiro2009lectures,birge2011introduction}) to risk minimization in machine learning (e.g., \cite{friedman2001elements}).

This paper concerns the estimation of $Z^*$ using limited data. Moreover, given a solution, say $\hat x$, a closely related problem is to estimate the optimality gap
\begin{equation}
\mathcal G(\hat x)=Z(\hat x)-Z^*.\label{gap}
\end{equation}
This allows us to assess the quality of $\hat x$, in the sense that the smaller $\mathcal G(\hat x)$ is, the closer is the solution $\hat x$ to the true optimum in terms of achieved objective value. More precisely, we will focus on inferring a lower confidence bound for $Z^*$, and, correspondingly, an upper bound for $\mathcal G(\hat x)$ - noting that its first term $Z(\hat x)$ can be treated as a standard population mean of $h(\hat x,\xi)$ that is estimable using a sample independent of the given $\hat x$, or that $\mathcal G(\hat x)$ can be represented as the max of the expectation of $h(\hat x,\xi)-h(x,\xi)$ whose estimation is structurally the same as $Z^*$.

This problem is motivated by the fact that many state-of-the-art solution methods mentioned before are only amenable to crude, worst-case performance bounds. For instance, \cite{shapiro2005complexity} and \cite{kleywegt2002sample} provide large deviations bounds on the optimality gap of SAA in terms of the diameter or cardinality of the decision space and the maximal variance of the function $h$. \cite{nemirovski2009robust} and \cite{ghadimi2013stochastic} provide bounds on the expected value and deviation probabilities of the SA iterates in terms of the strong convexity parameters, space diameter and maximal variance. These bounds can be refined under additional structural information (e.g., \cite{shapiro2000rate}). While they are very useful in understanding the behaviors of the optimization procedures, using them as a precise assessment on the quality of an obtained solution may be conservative. Because of this, a stream of work studies approaches to validate solution performances by statistically bounding optimality gaps. \cite{mak1999monte,bayraksan2006assessing,love2015overlapping} and \cite{shapiro2003monte} investigate the use of SAA to estimate these bounds. \cite{lan2012validation} validate the performances of SA iterates by using convexity conditions. \cite{stockbridge2013probability} and \cite{partani2006jackknife} study approaches like the jackknife and probability metric minimization to reduce the bias in the resulting gap estimates. \cite{bayraksan2011sequential} utilize gap estimates to guide sequential sampling. \cite{duchi2016statistics,blanchet2016robust} and \cite{lam2017empirical} investigate the use of empirical and profile likelihoods to estimate optimal values. Our investigation in this paper follows the above line of work on solution validation, focusing on the situation when data are limited and hence the statistical efficiency becomes utmost important. We also point out a related series of work that validate feasibility under uncertain constraints (e.g., \cite{luedtke2008sample,pagnoncelli2009sample,wang2008sample,care2014fast,calafiore2017repetitive,lam2019validating,hong2021learning}), though their problem of interest is beyond the scope of this paper as we focus on deterministically constrained problems and objective value performances.

More precisely, we introduce a bootstrap aggregating, or commonly known as bagging (\cite{breiman1996bagging}), approach to estimate a lower confidence bound for $Z^*$. This comprises repeated resampling of data to construct SAAs, and ultimately averaging the resampled optimal SAA values. We demonstrate how this approach applies under very general conditions on the cost function $h$ and decision space $\mathcal X$, while enjoys high statistical efficiency and stability. Compared to procedures based on batching (e.g., \cite{mak1999monte}), which also have documented benefits in wide applicability and stability, the data recycling in our approach provably improves a tradeoff between the tightness of the resulting bound and the statistical accuracy encountered in batching. In cases where sufficient smoothness is present and central limit theorem (CLT) for SAA (e.g., \cite{shapiro2009lectures,bayraksan2006assessing}) can be directly applied, we also see that our approach gains stability regarding standard error estimation, thanks to the smoothing effect brought by bagging. Nonetheless, our approach generally requires higher computational load than these previous methods due to the need to solve many resampled programs, which can be viewed as the price to pay for all these statistical gains.
% While we focus primarily on statistical performances, towards the end of this paper we will discuss some computational implications. We will support these claims via theoretical comparisons of the resulting bounds and standard error estimates as well as numerical experiments.

% Intriguingly, both our bagging approach and the resampling approach

% regarding the standard error magnitude and the underlying distributional approximation. exhibited by the generally gives rise to tighter standard errors more efficient use of data that also applies to wide settings,

The theoretical justification of our bagging scheme comes from viewing SAA as a kernel in an infinite-order symmetric statistic (\cite{frees1989infinite}), and an established optimistic bound for SAA as its asymptotic limit. A symmetric statistic is a generalization of sample mean in which each summand consists of a function (i.e., kernel) acting on more than one observation (\cite{serfling2009approximation,lee1990u}). In particular, the size of the SAA program can be seen as precisely the kernel ``order" (or ``degree"), which depends on the data size and is consequently of an infinite-order nature. Our bagging scheme serves as a Monte Carlo approximation for this symmetric statistic. As a main methodological contribution, we analyze the asymptotic behaviors of the statistic and the resulting bounds as the SAA size grows, and translate them into efficient performances of our bagging scheme.  Finally, we note that the notion of infinite-order symmetric statistics has been used in analyzing ensemble machine learning predictors like random forests (\cite{wager2017estimation}); our SAA kernels are, from this view, in parallel to the base learners in the latter context.

Finally, we mention that \cite{eichhorn2007stochastic} has also studied the resampling of SAA programs to construct confidence intervals for the optimal values of stochastic programs. Our approach connects with, but also differs substantially from \cite{eichhorn2007stochastic} in several regards. In terms of scope of applicability, \cite{eichhorn2007stochastic} focuses on mixed-integer linear programs, while we consider cost functions that can be generally non-Donsker. In terms of methodology, \cite{eichhorn2007stochastic} utilizes the quantiles of the resampled distribution to generate confidence intervals, by observing the same limiting distribution between an original CLT and the bootstrap CLT. The resampling in \cite{eichhorn2007stochastic} applies when the optimal solution is unique, or otherwise requires a ``two-layer" extended bootstrap where each resample is drawn from a new sample of the true distribution (as opposed to most bootstrap methods that allow repeated resample from the same original sample, with the availability of a conditional bootstrap CLT). The latter requires substantial data size or resorting to subsampling. Our bagging approach, in contrast, is based on a direct use of Gaussian limit and standard error estimation in the CLT for the optimistic bound. Our burden lies on the bootstrap Monte Carlo size requirement to obtain consistent standard error estimate, and less on the data size requirement. Relatedly, there is an orthogonal line of works on resampling approaches to estimate solution errors for randomized algorithms such as stochastic gradient descent (\cite{fang2019scalable,fang2018online}) and Newton's methods (\cite{chen2020estimating,lopes2018error}). These works however treat the data as deterministic and focus on the quantification of algorithmic uncertainties. Last but not least, during the review process of this paper, \cite{chen2022software} implemented an open-source software for bootstrap estimation in stochastic programs, based on our proposed scheme as well as \cite{eichhorn2007stochastic}, and demonstrated numerical performances in extensive experiments.

% that leads to different problem structures.
% is less immerse. However, in the case of mixed-integer programs and abundant data size, the approach in \cite{eichhorn2007stochastic} can estimate the optimal value or gap with asymptotically vanishing bias and is conceivably superior to our method.

We summarize our contributions of this paper as follows:
\begin{enumerate}
\item Motivated from the challenges of existing techniques (Section \ref{sec:challenges}), we introduce a bagging procedure to estimate a lower confidence bound for $Z^*$, correspondingly an upper confidence bound for $\mathcal G(\hat x)$ (Section \ref{sec:bagging proc}). We present the idea of our procedure that views SAA as a kernel in a symmetric statistic, and an optimistic bound for SAA as its associated limiting quantity (Section \ref{sec:fixed}).
\item We analyze the asymptotic behaviors of our bagging estimator, which can be viewed as an infinite-order symmetric statistic, under three increasingly stringent sets of regularity conditions on the optimization problem: minimal smoothness requirements, Lipschitzness and additionally solution uniqueness. In the last case, we also demonstrate how our asymptotic is on par with the classical CLT on SAA. These results are presented in Section \ref{sec:growing}. The mathematical developments without smoothness conditions utilize a combination of an analysis-of-variance (ANOVA) decomposition of the symmetric statistic and an analysis of the high-order error of the Hajek projection (\cite{van2000asymptotic}), using a probabilistic coupling argument and the Efron-Stein inequality (Appendix \ref{proof:main}). The developments with smoothness conditions and the reconciliation of the classical CLT on SAA use the argmax theorem and a maximal deviation bound for empirical processes (Appendix \ref{sec:refined}).
\item Building on the above results, we demonstrate that the bounds generated from our bagging procedure exhibit asymptotically correct coverages. In deriving these guarantees, we also analyze and formulate sufficient conditions on the bootstrap Monte Carlo sizes. These developments are in Section \ref{sec:error}, with additional technical details in Appendices \ref{sec:debiased bagging}, \ref{proof:consistency} and \ref{proof:final}.
% , with an asymptotically negligible error that can be removed under an additional non-degeneracy condition that can be readily verified for convex problems (Theorem \ref{convex nondegeneracy}). These are established in Section \ref{sec:growing} through another probabilistic coupling argument and the argmax theorem from empirical process theory (Appendix \ref{sec:refined}).
% on the limiting behavior of the symmetric statistic. This non-degeneracy condition takes thanks to a transparent characterization of the limit variance in terms of the cost function and an associated Gaussian process. These are established in Section \ref{sec:growing} through another probabilistic coupling argument and the argmax theorem from empirical process theory (Appendix \ref{sec:refined}).
\item We compare our approach with both batching and the direct use of CLT. In particular, we show that our bagging estimator possesses a standard error no larger than both of these competing methods whenever applicable (i.e., bagging offers variance reduction). These developments are in Sections \ref{sec:properties}, with mathematical details in Appendix \ref{proof:se}. Tying to our beginning motivation, we then argue how bagging improves a tradeoff between the bound tightness and the statistical accuracy faced by batching, and elicits more stable standard error estimates than the direct use of CLT. We support these comparisons by our numerical experiments (Section \ref{sec:numerics}).
% . Compared to the direct use of CLT,  This efficiency gain can be seen by an asymptotic comparison of the standard error in our estimator and an interpretation using conditional Monte Carlo.
% \item We explain the stability in our generated bounds brought by the smoothing effect of bagging in estimating standard error. This compares favorably with the direct use of CLT in situations where the objective function is smooth. This property is supported by our numerical experiments (Section \ref{sec:numerics}).
\end{enumerate}

\section{Existing Challenges and Motivation}\label{sec:challenges}
We discuss some existing methods and their challenges, to motivate our investigation. We start the discussion with the direct use of asymptotics from sample average approximation (SAA).

\subsection{Using Asymptotics of Sample Average Approximation}\label{sec:SAA}
When the cost function $h$ in \eqref{opt} is smooth enough, it is known classically that a central limit theorem (CLT) governs the behavior of the estimated optimal value in SAA, namely
\begin{equation}
\hat Z_n=\min_{x\in\mathcal X}\frac{1}{n}\sum_{i=1}^nh(x,\xi_i).\label{SAA def}
\end{equation}
We first introduce the following Lipschitz condition:

\begin{assumption}[Lipschitz continuity in the decision]\label{Lipschitz:decision}
The cost function $h(x,\xi)$ is Lipschitz continuous with respect to $x$, in the sense that
\begin{equation*}
\lvert h(x_1,\xi)-h(x_2,\xi)\rvert\leq M(\xi)\Vert x_1-x_2\Vert
\end{equation*}
for any $x_1,x_2\in\mathcal X\subseteq \R^d$, where $\Vert \cdot \Vert$ denotes the $l_2$ norm and $M(\xi)$ satisfies $E[M^2(\xi)]<\infty$.
\end{assumption}

Denote ``$\Rightarrow$" as convergence in distribution. The following result is taken from \cite{shapiro2009lectures}:
\begin{theorem}[Extracted from Theorem 5.7 in \cite{shapiro2009lectures}]
% Suppose, a.s. in $\xi$, that $h(\cdot,\xi)$ is Lipschitz continuous in the sense
% $$\|h(x,\xi)-h(y,\xi)\|\leq L(\xi)\|x-y\|$$
% for any $x,y\in\mathcal X$, and $L(\xi)$ is such that $E[L(\xi)^2]<\infty$., and there is a unique optimal solution $x^*$ to problem \eqref{opt} Suppose further that there is a unique solution for \eqref{opt} given by $x^*$.
Suppose that Assumption \ref{Lipschitz:decision} holds, $E[h(\tilde x,\xi)^2]<\infty$ for some point $\tilde x\in\mathcal X$, and $\mathcal X\subseteq \R^d$ is compact. Given i.i.d. data $\bm\xi_{1:n}=(\xi_1,\ldots,\xi_n)$, consider the SAA problem \eqref{SAA def}.
The SAA optimal value $\hat Z_n$ satisfies
\begin{equation}
\sqrt n(\hat Z_n-Z^*)\Rightarrow \inf_{x\in \mathcal X^*}Y(x)\label{CLT SAA general}
\end{equation}
where $\mathcal X^*$ is the set of optimal solutions for \eqref{opt}, and $Y(x)$ is a centered Gaussian process on $\mathcal X^*$ that has a covariance structure defined by $Cov(h(x_1,\xi),h(x_2,\xi))$ between any $x_1,x_2\in \mathcal X^*$.
% In particular, if there is a unique solution, i.e., $S=\{x^*\}$,
% Then
% \begin{equation}
% \sqrt n(\hat Z_n-Z^*)\Rightarrow N(0,\sigma^2)\label{CLT SAA}
% \end{equation}
% where $\sigma^2=Var(h(x^*,\xi))$.
\label{thm:SAA}
\end{theorem}

% We first motivate the need for methods beyond the classical theory of SAA. z_{1-\alpha/2}

Roughly speaking, Theorem \ref{thm:SAA} stipulates that, under the depicted conditions, one can use \eqref{CLT SAA general} to obtain
\begin{equation}
\hat Z_n-\frac{\hat q}{\sqrt n}\label{CI SAA}
\end{equation}
as a valid lower confidence bound for $Z^*$ (and analogously for $\mathcal G(\hat x)$ given $\hat x$), where $\hat q$ is some suitable error term that captures the quantile of the limiting distribution in \eqref{CLT SAA general}. Indeed, in the case of estimating $\mathcal G(\hat x)$, \cite{bayraksan2006assessing} provides an elegant argument that shows that, to achieve $1-\alpha$ confidence, one can take $\hat q=z_{1-\alpha}\hat\sigma$ where $z_{1-\alpha}$ is the standard normal critical value and $\hat\sigma$ is a standard deviation estimate, regardless of whether the limit in \eqref{CLT SAA general} is a Gaussian distribution. \cite{bayraksan2006assessing} calls this the single-replication procedure. More precisely, a straightforward modification of their procedure (which focuses on bounding $\mathcal G(\hat x)$) to bounding the optimal value $Z^*$ computes the $\hat\sigma^2$ by $(1/(n-1))\sum_{i=1}^n(h(\hat x_n^*,\xi_i)-\bar h(\hat x_n^*))^2$, where $\hat x_n^*$ is the solution from \eqref{SAA def} and $\bar h(\hat x_n^*)=(1/n)\sum_{i=1}^nh(\hat x_n^*,\xi_i)$.

% $\hat\sigma^2$ is obtained from
% $$\hat\sigma^2=\frac{1}{n-1}\sum_{i=1}^n(h(\hat x,\xi_i)-h(\hat x_n^*,\xi_i)-(\bar h(\hat x)-\bar h(\hat x_n^*)))^2$$
% where $\hat x_n^*$ is the solution from \eqref{SAA def}, and $\bar h(\hat x)-\bar h(\hat x_n^*)=(1/n)\sum_{i=1}^n(h(\hat x,\xi_i)-h(\hat x_n^*,\xi_i))$.

% is the sample mean of $h(\hat x,\xi_i)-h(\hat x_n^*,\xi_i)$'s.
% of $h(\hat x,\xi)-h(x,\xi)$ evaluated at the estimated.

% This theorem is more evident when there is a unique optimal solution, in which case the limit in \eqref{CLT SAA general} is a Gaussian distribution Note that when there are multiple optimal solutions for \eqref{opt}, the limit in \eqref{CLT SAA} will become the infimum of a Gaussian process (over the set of all optimal solutions).

% Theorem \ref{thm:SAA} can be generalized in multi-stage contexts which are also related to estimating risk measures; see, e.g., \cite{dentcheva2017statistical}.

% estimate of the standard

% =(1/(n-1))\sum_{i=1}^n(h(\hat x^*,\xi_i)-\hat Z_n)^2$, with $\hat x^*$ being an optimal solution of \eqref{SAA def}, is the empirical estimate of $\sigma^2$., but this can be substantially affected by various considerations such as the cardinality of the space is not applicable to guarantee meaningful coverages.

Though Theorem \ref{thm:SAA} (and other related work, e.g., \cite{dentcheva2017statistical,kleywegt2002sample}) is very useful, if the SAA solutions have a ``jumping" behavior, namely that program \eqref{opt} has several near-optimal solutions with hugely differing objective variances, then the standard deviation estimate $\hat\sigma$ needed in the bound \eqref{CI SAA} can be unreliable. This is because $\hat\sigma$ depends heavily on $\hat x_n^*$, which can fall close to any of the possible near-optimal solutions with substantial chance and make the resulting estimation noisy. This issue is illustrated in, e.g., Examples 1 and 2 in \cite{bayraksan2006assessing}.
% there are at least two reasons why one would need more general methods:
% \begin{enumerate}
% \item When the decision space contains discrete elements (e.g., combinatorial problems), Assumption \ref{Lipschitz:decision} does not hold anymore. There is no guarantee in using the bound \eqref{CI SAA}, i.e., it may still be correct but conservative, or it may simply possess incorrect coverages. We note, however, that for some class of problems (e.g., two-stage mixed-integer linear programs), extensions to Theorem \ref{thm:SAA} and approaches such as quantile-based bootstrapping (e.g., \cite{eichhorn2007stochastic}) are useful alternatives.
% . In this case the convergence speed can exhibit a different behavior from Theorem \ref{thm:SAA} (e.g., ).

% \item
% \item When there are multiple optimal solutions, then the limit in \eqref{CLT SAA} is the infimum of a Gaussian process instead of merely a Gaussian random variable, making inference on $Z^*$ more challenging.
% \end{enumerate}

%  This is reflected in Example ? in Section \ref{sec:numerics}.
We should also mention that, as an additional issue, the bias in $\hat Z_n$ relative to $Z^*$ can be quite large in any given problem, i.e., arbitrarily close to order $1/\sqrt n$ described in the CLT, even if all the conditions in Theorem \ref{thm:SAA} hold (\cite{partani2007adaptive}). Note that this bias is in the optimistic direction (i.e., the resulting bound is still correct, but conservative), and it also appears in the ``optimistic bound" approach that we discuss next. There have been techniques such as the jackknife (\cite{partani2007adaptive,partani2006jackknife}) and probability metric minimization (\cite{stockbridge2013probability}) in reducing this bias effect.

% This bias is not reflected in the CLT and can impose substantial additional errors, as discussed in  \cite{partani2007adaptive,partani2006jackknife} and revealed in Example ? in Section \ref{sec:numerics}.
% \item Ideally, we would like a method that is more general and less biased than in using Theorem \ref{thm:SAA}.

\subsection{Batching Procedures}\label{sec:batching}
% To circumvent the concerns in Points 1 and 2 above,
An alternative approach is to use the optimistic bound (\cite{mak1999monte,shapiro2003monte,glasserman2013monte})
\begin{equation}
E[\hat Z_n]\leq Z^*\label{optimistic}
\end{equation}
where $E[\cdot]$ in \eqref{optimistic} is taken with respect to the data in constructing the SAA value $\hat Z_n$. The bound \eqref{optimistic} holds for any $n\geq 1$, as a direct consequence from Jensen's inequality in exchanging the expectation and the minimization operator in the SAA.
% and holds as long as $\bm\xi_{1:n}$ are i.i.d.

% , and in particular addressing the concerns in much greater generality than Theorem \ref{thm:SAA}
The bound \eqref{optimistic} offers a simple way to construct a lower bound for $Z^*$ under great generality. Note that the left hand side of \eqref{optimistic} is a mean of SAA. Thus, if one can ``sample" a collection of SAA values, then a lower confidence bound for $Z^*$ can be constructed readily by using a standard estimate of population mean. To ``sample" SAA values, an approach suggested by \cite{mak1999monte} is to batch i.i.d. data set $\bm\xi_{1:n}$ into say $m$ batches, each batch consisting of $k$ observations, so that $mk=n$ (we ignore rounding issues). For each $j=1,\ldots,m$, solve an SAA using the $k$ observations in the $j$-th batch; call this value $\hat Z_k^j$. Then use
\begin{equation}
\tilde Z_{n,k}-z_{1-\alpha}\frac{\tilde\sigma}{\sqrt m}\label{CI batching}
\end{equation}
where $\tilde Z_{n,k}=(1/m)\sum_{j=1}^m\hat Z_k^j$ and $\tilde\sigma^2=(1/(m-1))\sum_{j=1}^m(\hat Z_k^j-\tilde Z_{n,k})^2$ are the sample mean and variance from $\hat Z_k^j,j=1,\ldots,m$, and $z_{1-\alpha}$ is the $(1-\alpha)$-level standard normal quantile.

The bound \eqref{CI batching} does not rely on any continuity of $h$, and $\tilde\sigma/\sqrt m$ is simply the sample standard error for a sample mean. This bound largely mitigates the aforementioned unstable estimation encountered in bounds that directly use the SAA asymptotic \eqref{CLT SAA general}. Nonetheless, there is an intrinsic tradeoff between the bound tightness and statistical accuracy. On one hand, $m$ must be chosen big enough (e.g., roughly $>30$) so that one can use the CLT to justify the approximation \eqref{CI batching}. On the other hand, the larger is $k$, the closer is $E[\hat Z_k^j]$ to $Z^*$ in \eqref{optimistic}, leading to a tighter lower bound for $Z^*$. This is thanks to a monotonicity property in that $E[\hat Z_n]$ is non-decreasing in $n$ (\cite{mak1999monte,norkin1998branch}). Therefore, there is a tradeoff between the statistical accuracy controlled by $m$ (in terms of the validity of the CLT) and the tightness controlled by $k$ (in terms of the position of $E[\hat Z_k^j]$ in \eqref{optimistic}). In the batching or the so-called multiple-replication approach of \cite{mak1999monte}, this tradeoff is confined to the relation $mk=n$. There have been suggestions to improve this tradeoff, e.g., by using overlapping batches (\cite{love2015overlapping,love2011overlapping}), but their validity requires uniqueness or exponential convergence of the solution (e.g., in discrete decision space).
% , as the standard error is computed from $m$ instead of just one SAA problems.

% This procedure is shown in Algorithm \ref{batching}. (as also discussed in \eqref{mak1999monte})

% \begin{algorithm}
% \caption{The Batching Procedure from Mak et al. (1999)}
% \label{batching}
% \begin{algorithmic}
% \STATE {Given $n$ observations $\xi_1,\ldots,\xi_n$, choose $m,k$ such that $mk=n$}
% \FOR{$j=1$ \TO $m$ }
% \STATE {From $\{\xi_{(j-1)k+1},\ldots,\xi_{jk}\}$, solve
% $$\hat Z_k^j=\min_{x\in\mathcal X}\frac{1}{k}\sum_{i=1}^kh(x,\xi_{(j-1)k+i})$$
% }
% % \STATE {Compute $\bar{\psi}^b_{AV}=\frac{1}{R}\sum_{r=1}^R\tilde{\psi}_r(\widehat{F}_1^b,\ldots,\widehat{F}_m^b)$}
% \ENDFOR
% \STATE {Compute $\tilde Z_k=\frac{1}{m}\sum_{j=1}^m\hat Z_k^j$ and $\tilde\sigma^2=\frac{1}{m-1}\sum_{j=1}^m(\hat Z_k^j-\tilde Z_k)^2$
% }
% \STATE {Output $\tilde Z_k-z_{1-\alpha/2}\frac{\tilde\sigma}{\sqrt m}$}
% \end{algorithmic}
% \end{algorithm}Note that the tightness of the obtained bound in this procedure depends on the choices of $m$ and $k$. or the so-called single-replication , and thus one would expect $E[\hat Z_k^j]$, and hence $\tilde Z_k$, to be higher as $k$ increases

% Moreover, the larger is $m$, typically the smaller is the magnitude of the standard error in the second term of \eqref{CI batching}
% and the magnitude of the standard error term

\subsection{Motivation and Overview of Our Approach}
Thus, in general, when the sample size $n$ is small, the batching approach appears to necessarily settle for a conservative bound in order to retain statistical accuracy. The starting motivation for the bagging procedure that we propose next is to break free this tightness-accuracy tradeoff. In particular, we offer a bound roughly in the form
\begin{equation}
Z_{n,k}^{bag}-\frac{q^{bag}}{\sqrt n}\label{bag bound}
\end{equation}
where $Z_{n,k}^{bag}$ is a point estimate obtained from bagging many resampled SAA values, and $k$ signifies the size of the resampled SAA (i.e., the ``bags"). The quantity $q^{bag}$ relies on a standard error estimate of $Z_{n,k}^{bag}$. Table \ref{tab:comparison with existing methods} highlights the differences between our bagging bound and direct-CLT and batching bounds, which we explain further below.
\begin{table}[h]
    \centering
    \begin{tabular}{|c|c|c|c|}
    \hline
         &Direct-CLT bound &Batching bound &Bagging bound \\\hline
        \tabincell{c}{SAA size \\ (tightness)} &$n$&\tabincell{c}{$k$ s.t. \\ $mk=n$}&\tabincell{l}{$k=o(n)$ in general \\ $k$ can be $\approx n$ when smooth \\ \& unique optimum} \\\hline
        \tabincell{c}{Variance \\ reduction} &No&No&Yes \\\hline
         \tabincell{c}{Stable standard \\ error estimate}&No&Yes&Yes \\\hline
        \tabincell{c}{Problem requirements \\ except moments} &\tabincell{c}{Smooth obj. or \\ discrete decision}&None&None \\\hline
        \#SAAs to solve &$1$&$m$& \tabincell{c}{$>n$ \\ ($>\sqrt{n}$ if debiased)}\\\hline
    \end{tabular}
    \caption{Differences between bagging bound and existing bounds.}
    \label{tab:comparison with existing methods}
\end{table}

% Our method operates at a similar level of generality as batching:
Compared to batching, our method shares the same advantages that the standard error term $q^{bag}$ does not succumb to the ``jumping" solution behavior, and our bound holds regardless of the continuity to the decision. Moreover, our bagging point estimate $Z_{n,k}^{bag}$ has provably no larger variance than the batching point estimate $\tilde Z_{n,k}$ and, as described above, it allows using a resampled SAA size $k$ that is larger than the batched SAA size.
% thus gaining higher statistical precision.

% Different from  compared to the batching point estimate in \eqref{CI batching},

% In fact, this term regains the same order of precision level as the bound \eqref{CI SAA} that uses SAA asymptotics directly.

% handles the two concerns Points 1 and 2 in Section \ref{sec:SAA}
% , which attains the same level of statistical accuracy as using Theorem \ref{thm:SAA}we have gained estimation stability of $q^{bag}$ and, moreover,

Compared to direct-CLT, our bound would be almost as tight by choosing the resample size $k$ in \eqref{bag bound} to be arbitrarily close to the order of $n$. Moreover, our standard error term $q^{bag}$ is more stable than the counterpart in \eqref{CI SAA} thanks to the use of many resampled SAAs rather than a single SAA. Furthermore, our approach works under conditions more general than when \eqref{CI SAA} is applicable, and if we re-impose Lipschitz continuity on the decision required for \eqref{CI SAA}, then our bagging point estimator $Z_{n,k}^{bag}$ has no larger asymptotic variance than $\hat Z_n$ in \eqref{CI SAA}, with strictly smaller asymptotic variance in the case of multiple optima. In the case of unique optimum, the resample size $k$ is allowed to be the same order as $n$ in which case our bagging bound achieves the same level of tightness as direct-CLT.
% In the case of unique optimum, the resample size $k$ is allowed to be the same order as $n$ and the variances match, hence achieving exactly the same level of tightness as direct-CLT.

 % (i.e., Assumption \ref{Lipschitz:decision})
% On the other hand, we will show that the choice of $k$ in \eqref{bag bound}, which affects the tightness, can be taken as roughly $o(n)$ in general. Compared with the direct-CLT bound \eqref{CI SAA}, our bound appears less tight. However, we consider conditions more general than when \eqref{CI SAA} is applicable. We will see that if we re-impose Lipschitz continuity on the decision (i.e., Assumption \ref{Lipschitz:decision}), then $k$ can be set arbitrarily close to the order of $n$. This means that our approach is almost as statistically efficient as the bound \eqref{CI SAA}, with the extra benefit of stability in estimating $q^{bag}$.

Despite the above advantages, our approach requires solving a number of resampled SAA programs that is of larger order than the data size $n$ (reduced to larger order than $\sqrt{n}$ if a bias correction is applied to the variance estimator), and is thus computationally more costly than batching and direct-CLT methods. The higher computation cost is the price to pay to elicit our benefits depicted above. Our approach is thus most recommended when statistical performance is of higher concern than computation efficiency, prominently in small-sample situations.

The next section will explain our procedure in more detail. A key insight is to view SAA as a symmetric kernel and the optimistic bound \eqref{optimistic} as a limiting quantity of an associated symmetric statistic, which can be estimated by bagging.

\section{Bagging Procedure to Estimate Optimal Values}\label{sec:bagging proc}
This section presents our approach. Instead of batching the data, we uniformly resample $k$ observations from $\bm\xi_{1:n}$ for many, say $B$, times. We use each resample to form an SAA problem and solve it. We then average all these resampled SAA optimal values. The resampling can be done with or without replacement (we will discuss some differences between the two). We summarize our procedure in Algorithm \ref{bagging}.

\begin{algorithm}
\caption{Bagging Procedure for Bounding Optimal Values}
\label{bagging}
\begin{algorithmic}
\STATE {Given $n$ i.i.d. observations $\bm\xi_{1:n}=(\xi_1,\ldots,\xi_n)$, select positive integers $k$ and $B$.}

\vspace{1ex}

\FOR{$b=1$ \TO $B$ }
\STATE {Randomly sample $\bm\xi_k^b=(\xi_1^b,\ldots,\xi_k^b)$ uniformly from $\bm\xi_{1:n}$ (with or without replacement), and solve
$$\hat Z_k^b=\min_{x\in\mathcal X}\frac{1}{k}\sum_{i=1}^kh(x,\xi_i^b).$$
}
% \STATE {Compute $\bar{\psi}^b_{AV}=\frac{1}{R}\sum_{r=1}^R\tilde{\psi}_r(\widehat{F}_1^b,\ldots,\widehat{F}_m^b)$}
\ENDFOR
\STATE {Compute $\tilde Z_{n,k}^{bag}=\frac{1}{B}\sum_{b=1}^B\hat Z_k^b$ and
\begin{equation}\label{sigma_IJ}
\tilde\sigma_{IJ}^2=
\begin{cases}
\sum_{i=1}^n\widehat{Cov}_*(N_i^*,\hat Z_k^*)^2,&\text{ if resampling is with replacement}\\
\big(\frac{n}{n-k}\big)^2\sum_{i=1}^n\widehat{Cov}_*(N_i^*,\hat Z_k^*)^2,&\text{ if resampling is without replacement}
\end{cases}
\end{equation}
where
\begin{equation}
\widehat{Cov}_*(N_i^*,\hat Z_k^*)=\frac{1}{B}\sum_{b=1}^B(N_i^b-\frac{k}{n})(\hat Z_k^b-\tilde Z_{n,k}^{bag})\label{IJ}
\end{equation}
and $N_i^b$ is the number of $\xi_i$ that shows up in the $b$-th resample.
}

\vspace{1ex}

\STATE {Output $\tilde Z_{n,k}^{bag}-z_{1-\alpha}\tilde\sigma_{IJ}$.}
\end{algorithmic}
\end{algorithm}
%, $\tilde N=(1/B)\sum_{b=1}^BN_i^b$ is the sample mean of $N_i^b$'s

In the output of Algorithm \ref{bagging}, the first term $\tilde Z_{n,k}^{bag}$ is the average of many bootstrap resampled SAA values, which resembles a bagging predictor by viewing each SAA as a ``base learner" (\cite{breiman1996bagging}). The quantity $\widehat{Cov}_*(N_i^*,\hat Z_k^*)$ in \eqref{IJ} is the covariance between the count of a specific observation $\xi_i$ in a bootstrap resample, denoted $N_i^*$, and the resulting resampled SAA value $\hat Z_k^*$. The quantity $\tilde\sigma_{IJ}^2=\sum_{i=1}^n\widehat{Cov}_*(N_i^*,\hat Z_k^*)^2$ is an empirical version of the so-called infinitesimal jackknife (IJ) estimator (\cite{efron2014estimation}), which has been used to estimate the standard error of bagging schemes, including in random forests or tree ensembles (\cite{wager2014confidence}). The additional constant factor $(n/(n-k))^2$ in the second line of \eqref{sigma_IJ} is a correction specific to resampling without replacement that is required for consistency in the asymptotic regime where $k$ is of the same order as $n$.

% Although the IJ variance estimator is not affected by this factor in certain asymptotic regime, e.g., $k=o(n)$, it is necessary for ensuring  we find it significantly improves the finite-sample performance of our method.

% ensures the validity of the IJ estimator under resampling without replacement in certain asymptotic regimes that we will consider.

%regimes relating $k$ and $n$ as we will see.
% that is specific to resampling without replacement, and is. We will see that this factor does not asymptoically affect the validity of the IJ estimator in the small resample size regime $k=o(n)$, we shall see that in fact it's necessary for consistency when $k$ is of the same order as $n$.

%Although the IJ variance estimator is not affected by
%The additional constant factor $(n/(n-k))^2$ in the second line of \eqref{sigma_IJ} is a finite-sample correction specific to resampling without replacement. Although the IJ variance estimator is not affected by this factor in the asymptotic regime that we will discuss, we find it significantly improves the finite-sample performance of our method From the perspective of gradient estimation via simultaneous perturbation, this factor serves to normalize the size of random perturbation, around the empirical distribution formed by $\bm\xi_{1:n}$, induced by resampling without replacement.

% , which has also appeared in the bagging literature \cite{wager2017estimation}
\section{SAA as Symmetric Kernel}\label{sec:fixed}
We explain how Algorithm \ref{bagging} arises. In short, the $\tilde Z_{n,k}^{bag}$ in Algorithm \ref{bagging} acts as a point estimator for $E[\hat Z_k]$ in the optimistic bound \eqref{optimistic}, whereas $\tilde\sigma_{IJ}^2$ captures the standard error in using this point estimator.

To be more precise, let us introduce a functional viewpoint and write
\begin{equation}
W_k(F)=E_{F^k}[H_k(\xi_1,\ldots,\xi_k)]\label{main bound}
\end{equation}
where
$$H_k(\xi_1,\ldots,\xi_k)=\min_{x\in\mathcal X}\frac{1}{k}\sum_{i=1}^kh(x,\xi_i)$$
is the SAA value, expressed more explicitly in terms of the underlying data used. Here, the expectation $E_{F^k}[\cdot]$ is generated with respect to i.i.d. variables $(\xi_1,\ldots,\xi_k)$, i.e., $F^k$ denotes the product measure of $k$ $F$'s. For convenience, we denote $E[\cdot]$ as the expectation either with respect to $F$ or the product measure of $F$'s when no confusion arises. Also, we denote $W_k=W_k(F)$.

With these notations, the optimistic bound \eqref{optimistic} can be expressed as
$$W_k(F)\leq Z^*$$
with the best bound being $W_\infty=\lim_{k\to\infty}W_k\leq Z^*$ thanks to the monotonicity property of the expected SAA value mentioned before.

% By Mak et al., \eqref{main bound} is an optimistic, i.e., lower, bound for \eqref{opt} for any $k\geq1$. Moreover, $W_k\leq W_{k+1}$, so that the bound is getting tighter as $k$ increases, with the best bound being $W_\infty=\lim_{k\to\infty}W_k\leq Z^*$.

Suppose that we have used sampling with replacement in Algorithm \ref{bagging}. Also say we use infinitely many bootstrap replications, i.e., $B=\infty$. Then, the estimator $\tilde Z_{n,k}^{bag}$ in Algorithm \ref{bagging} becomes precisely
$$\tilde Z_{n,k}^{bag}=W_k(\hat F)$$
where $\hat F$ is the empirical distribution formed by $\bm\xi_{1:n}$, i.e., $\hat F(\cdot)=(1/n)\sum_{i=1}^n\delta_{\xi_i}(\cdot)$ where $\delta_{\xi_i}(\cdot)$ is the delta measure at $\xi_i$. If $W_k(\cdot)$ is ``smooth" in some sense, then one would expect $W_k(\hat F)$ to be close to $W_k(F)$. Indeed, when $k$ is fixed, $W_k(F)$, which is expressible as the $k$-fold expectation under $F$ in \eqref{main bound}, is multi-linear, i.e.,
$$W_k(F)=E_{F^k}[H_k(\xi_1,\ldots,\xi_k)]=\int\cdots\int H_k(\xi_1,\ldots,\xi_k)\prod_{j=1}^kdF(\xi_j)$$
and is always differentiable with respect to $F$ (in the Gateaux sense) from the theory of von Mises statistical functionals (\cite{serfling2009approximation}). This ensures that $W_k(\hat F)$ is close to $W_k(F)$ probabilistically, as elicited by a CLT (Theorem \ref{thm:symmetric} below).
% that will be made precise below.

% in approximating $W_k(F)$ using $W_k(\hat F)$.

% \frac{1}{N}\sum_{i=1}^NH_k(\xi_1,\ldots,\xi_k)$$

Note that $W_k(\hat F)$ is exactly the average of $H_k(\xi_{i_1},\ldots,\xi_{i_k})$ over all possible combinations of $\{\xi_{i_1},\ldots,\xi_{i_k}\}$ drawn with replacement from $\bm\xi_{1:n}$. This is equivalent to
\begin{equation}
V_{n,k}=\frac{1}{n^k}\sum_{i_j\in\{1,\ldots,n\},j=1,\ldots,k}H_k(\xi_{i_1},\ldots,\xi_{i_k})\label{V}
\end{equation}
which is the so-called $V$-statistic. If we have used sampling without replacement in Algorithm \ref{bagging}, we arrive at the estimator (assuming again $B=\infty$)
\begin{equation}
U_{n,k}=\frac{1}{\binom{n}{k}}\sum_{(i_1,\ldots,i_k)\in\mathcal C_k}H_k(\xi_{i_1},\ldots,\xi_{i_k})\label{U}
\end{equation}
where $\mathcal C_k$ denotes the collection of all subsets of size $k$ in $\{1,\ldots,n\}$. %the $n$ observations.
The quantity \eqref{U} is known as the $U$-statistic. The $V$ and $U$ estimators in \eqref{V} and \eqref{U} both belong to the class of symmetric statistics (\cite{serfling2009approximation,van2000asymptotic,de2012decoupling}), since the estimator is unchanged against a shuffling of the ordering of the data $\bm\xi_{1:n}$. Correspondingly, the $H_k(\cdot)$ function is known as the symmetric kernel. Symmetric statistics generalize the sample mean, the latter corresponding to the case when $k=1$.

When $B<\infty$, then $V_{n,k}$ and $U_{n,k}$ above are approximated by a random sampling of the summands on the right hand side of \eqref{V} and \eqref{U}. These are known as incomplete $V$- and $U$-statistics (\cite{lee1990u,blom1976some,janson1984asymptotic}), and are precisely our $\tilde Z_{n,k}^{bag}$. As $B$ is chosen large enough, $\tilde Z_{n,k}^{bag}$ will well approximate $V_{n,k}$ and $U_{n,k}$.

% $\tilde Z$ involves
% With only the i.i.d. data $\bm\xi_{1:n}$, we want to estimate $W_k$ by resampling in some ways. Two approaches:
% \begin{enumerate}
% \item Use $U$-statistic
% $$U_{n,k}=\frac{1}{\binom{n}{k}}\sum_{(i_1,\ldots,i_k)\in\mathcal C_k}H_k(\xi_{i_1},\ldots,\xi_{i_k})$$
% where $\mathcal C_k$ denotes the collection of all $k$-combinations from $\{1,\ldots,n\}$.%the $n$ observations.

% \item Use $V$-statistic $V_{n,k}=W_k(F_n)$ where $F_n$ is the empirical distribution from $\xi_1,\ldots,\xi_n$, i.e.,
% $$V_{n,k}=\frac{1}{n^k}\sum_{i_j\in\{1,\ldots,n\},j=1,\ldots,k}H_k(\xi_{i_1},\ldots,\xi_{i_k})$$

% \end{enumerate}

%Denoting
%$$W_k(G)=E_G[H_k(\xi_1,\ldots,\xi_n)]$$
%where $E_G[\cdot]=E_{\prod{i=1}^kG}[\cdot]$ is the expectation under $F$.

% \section{Central Limit Theorems when $k$ is Fixed}

To discuss further, we make the following assumptions:
\begin{assumption}[$L_2$-boundedness]
We have
$$E\sup_{x\in\mathcal X}|h(x,\xi)|^2<\infty$$
%where $\xi,\xi'$ are i.i.d. generated from $F$.
\label{L2}
\end{assumption}

% Assumption \ref{L2} implies that $EH_k(\xi_1,\ldots,\xi_k)^2<\infty$ for any $k$, since
% \begin{equation}
% EH_k(\xi_1,\ldots,\xi_k)^2\leq\frac{1}{k^2}E\sup_{x\in\mathcal X}\left(\sum_{i=1}^kh(x,\xi_i)\right)^2\leq E\sup_{x\in\mathcal X}|h(x,\xi)|^2<\infty\label{interim21}
% \end{equation}
% by the Minkowski inequality.We make the following assumption on $g_k$:

Denote $g_k(\xi)=E[H_k(\xi_1,\ldots,\xi_k)|\xi_1=\xi]$. Also denote $Var(\cdot)=Var_F(\cdot)$ as the variance under $F$.
\begin{assumption}[Finite non-zero variance]
We have $0<Var(g_k(\xi))<\infty$.\label{nondegeneracy}
\end{assumption}

We have the following asymptotics of $U_{n,k}$ and $V_{n,k}$ :
\begin{theorem}
Suppose $k\geq1$ is fixed, and Assumptions \ref{L2} and \ref{nondegeneracy} hold. Then
\begin{equation}
\sqrt n(U_{n,k}-W_k)\Rightarrow N(0,k^2Var(g_k(\xi)))\label{CLT fixed U}
\end{equation}
and
\begin{equation}
\sqrt n(V_{n,k}-W_k)\Rightarrow N(0,k^2Var(g_k(\xi)))\label{CLT fixed V}
\end{equation}
as $n\to\infty$, where $N(0,k^2Var(g_k(\xi)))$ is a normal distribution with mean 0 and variance $k^2Var(g_k(\xi))$.\label{thm:symmetric}
\end{theorem}
\proof{Proof.}
Assumption \ref{L2} implies that $EH_k(\xi_{i_1},\ldots,\xi_{i_k})^2<\infty$ for any (possibly identical) indices $i_1,\ldots,i_k$, since
\begin{equation}
EH_k(\xi_{i_1},\ldots,\xi_{i_k})^2\leq\frac{1}{k^2}E\sup_{x\in\mathcal X}\left(\sum_{j=1}^kh(x,\xi_{i_j})\right)^2\leq E\sup_{x\in\mathcal X}|h(x,\xi)|^2<\infty\label{interim22}
\end{equation}
by the Minkowski inequality. Then, under \eqref{interim22} and Assumption \ref{nondegeneracy}, \eqref{CLT fixed U} follows from Theorem 12.3 in \cite{van2000asymptotic}, and \eqref{CLT fixed V} follows from Section 5.7.3 in \cite{serfling2009approximation}.\Halmos
\endproof

% \begin{theorem}
% Suppose $k\geq1$ is fixed, and Assumptions \ref{L2} and \ref{nondegeneracy} hold. Then
% $$\sqrt n(V_{n,k}-W_k)\Rightarrow N(0,k^2Var(g_k(\xi)))$$
% where $N(0,k^2Var(g_k(\xi)))$ is a normal distribution with mean 0 and variance $k^2Var(g_k(\xi))$.
% \end{theorem}

% \begin{proof}
% Follows from Section 5.7.3 in Serfling and \eqref{interim22}.
% \end{proof}

Theorem \ref{thm:symmetric} is a consequence of the classical CLT for symmetric statistics. The expression $kg_k(\xi)$, as a function defined on the space $\mathcal X$, is the so-called influence function of $W_k(F)$, which can be viewed as its functional derivative with respect to $F$ (\cite{hampel1974influence}). Alternately, for a $U$-statistic $U_{n,k}$, the expression is the so-called Hajek projection (\cite{van2000asymptotic}), which is the projection of the statistic onto the subspace generated by the linear combinations of $f_i(\xi_i),i=1,\ldots,n$ for any measurable function $f_i$. It turns out that these two views coincide, and the $U$- and $V$-statistics (whose approximation uses the projection viewpoint and the functional derivative viewpoint respectively) obey the same CLT as depicted in Theorem \ref{thm:symmetric}.

The output of Algorithm \ref{bagging} is now evident given Theorem \ref{thm:symmetric}. When $B=\infty$, $\tilde Z_{n,k}^{bag}$ is precisely $U_{n,k}$ under sampling without replacement or $V_{n,k}$ under sampling with replacement. The quantity $\tilde\sigma_{IJ}^2$ in Algorithm \ref{bagging}, an empirical IJ estimator, can be shown to approximate the asymptotic variance $k^2Var(g_k(\xi))/n$ as $n,B\to\infty$, by borrowing recent results in bagging (\cite{efron2014estimation,wager2017estimation}) (Theorems \ref{consistency_IJ:U} and \ref{consistency_IJ:V} below show stronger results). Then the procedural output is the standard CLT-based lower confidence bound for $W_k$.

The discussion above holds for a fixed $k$, the sample size used in the resampled SAA. It also shows that, at least asymptotically, using with or without replacement does not matter. However, using a fixed $k$ regardless of the size of $n$ is restrictive and leads to conservative bounds. The next subsection will relax this requirement and present results on a growing $k$ against $n$, which in turn allows us to get a tighter $W_k=E[\hat Z_k]$ in the optimistic bound \eqref{optimistic}.

% one should intuits afford to use a bigger $k$ without distorting the asymptotic behavior, and this in turn allows us to get a tighter $W_k=E[\hat Z_k]$ in the optimistic bound \eqref{optimistic}. The next section shows that we could indeed do so to a certain extent.
% Moreover, our analysis reveals that without replacement is statistically more advantageous in this context.

% \begin{proof}
% It follows from Theorem 12.3 in van der Vaart and \eqref{interim21}.
% \end{proof}

% Assumption \ref{L2} also implies that $EH_k(\xi_{i_1},\ldots,\xi_{i_k})^2<\infty$ for any indices $i_1,\ldots,i_k\in\{1,\ldots,k\}$, since
% \begin{equation}
% EH_k(\xi_{i_1},\ldots,\xi_{i_k})^2\leq\frac{1}{k^2}E\sup_{x\in\mathcal X}\left(\sum_{j=1}^nh(x,\xi_{i_j})\right)^2\leq E\sup_{x\in\mathcal X}|h(x,\xi)|^2<\infty\label{interim22}
% \end{equation}
% by the Minkowski inequality again.

% \begin{theorem}
% Suppose $k\geq1$ is fixed, and Assumptions \ref{L2} and \ref{nondegeneracy} hold. Then
% $$\sqrt n(V_{n,k}-W_k)\Rightarrow N(0,k^2Var(g_k(\xi)))$$
% where $N(0,k^2Var(g_k(\xi)))$ is a normal distribution with mean 0 and variance $k^2Var(g_k(\xi))$.
% \end{theorem}

% \begin{proof}
% Follows from Section 5.7.3 in Serfling and \eqref{interim22}.
% \end{proof}

\section{Asymptotic Behaviors with Growing Resample Size}\label{sec:growing}
We first make the following strengthened version of Assumption \ref{L2}:
\begin{assumption}[$L_{2+\delta}$-bounded modulus of continuity]
We have
$$E\sup_{x\in\mathcal X}|h(x,\xi)-h(x,\xi')|^{2+\delta}<\infty$$
where $\xi,\xi'$ are i.i.d. generated from $F$.
\label{L2 strengthened}
\end{assumption}

Assumption \ref{L2 strengthened} holds quite generally, for instance under any of the following sufficient conditions:
\begin{assumption}[Uniform boundedness]
$h(\cdot,\cdot)$ is uniformly bounded over $\mathcal X\times\Xi$.\label{uniform bound}
\end{assumption}

\begin{assumption}[Uniform Lipschitz condition]
$h(x,\xi)$ is Lipschitz continuous with respect to $\xi$, where the Lipschitz constant is uniformly bounded in $x\in\mathcal X$, i.e.,
$$|h(x,\xi)-h(x,\xi')|\leq L\|\xi-\xi'\|$$
where $\|\cdot\|$ is some norm in $\Xi$. Moreover, $E\|\xi\|^{2+\delta}<\infty$.\label{Lipschitz}
\end{assumption}

\begin{assumption}[Majorization]
$$|h(x,\xi)-h(x,\xi')|\leq f(\xi)+f(\xi')$$
where $Ef(\xi)^{2+\delta}<\infty$.
\label{majorization}
\end{assumption}
That Assumption \ref{uniform bound} implies Assumption \ref{L2 strengthened} is straightforward. To see how Assumption \ref{Lipschitz} implies Assumption \ref{L2 strengthened}, note that, if the former is satisfied, we have%E\sup_{x\in\mathcal X}|h(x,\xi)-Z(x)|^{2+\delta}=E\sup_{x\in\mathcal X}|E[h(x,\xi)-h(x,\xi')|\xi]|^{2+\delta}\leq
$$E\sup_{x\in\mathcal X}|h(x,\xi)-h(x,\xi')|^{2+\delta}\leq L^{2+\delta}E\|\xi-\xi'\|^{2+\delta}<\infty$$
Similarly, Assumption \ref{majorization} implies Assumption \ref{L2 strengthened} because the former leads to%E\sup_{x\in\mathcal X}|h(x,\xi)-Z(x)|^{2+\delta}=E\sup_{x\in\mathcal X}|E[h(x,\xi)-h(x,\xi')|\xi]|^{2+\delta}\leq
$$E\sup_{x\in\mathcal X}|h(x,\xi)-h(x,\xi')|^{2+\delta}\leq E(f(\xi)+f(\xi'))^{2+\delta}<\infty$$

We have the following asymptotics:
\begin{theorem}[CLT for growing resample size under sampling without replacement]
Suppose Assumptions \ref{L2} and \ref{L2 strengthened} hold. If the resample size $k=o(n)$, then
\begin{equation*}
    \sqrt{n}(U_{n,k} - W_k) = k\sqrt{Var(g_k(\xi))}\cdot \mathcal Z_{n,k} + o_p(1)%\text{\ \ as }n\to\infty
\end{equation*}
where each $\mathcal Z_{n,k}$ is of mean $0$ and variance $1$, and every subsequence of $\{\mathcal Z_{n,k}\}$ such that
% \begin{equation*}
%     \{Z_{n,k}:\} \Rightarrow N(0,1)
% \end{equation*}
$k^2Var(g_k(\xi)) \cdot n^{\frac{\delta}{2+\delta}}\to\infty$ converges in distribution to the standard normal.\label{main thm}
\end{theorem}
% \begin{theorem}[CLT for growing resample size under sampling without replacement]
% Suppose Assumptions \ref{L2} and \ref{L2 strengthened} hold. If the resample size $k=o(n)$, then
% % $$\frac{\sqrt n(U_{n,k}-W_k)}{k\sqrt{Var(g_k(\xi))}}\Rightarrow N(0,1)$$
% for any $\alpha\in (0,1)$
% \begin{equation*}
%     \liminf_{n\to\infty} P\Big(W_k \geq U_{n,k} - z_{1-\alpha}\frac{k\sqrt{Var(g_k(\xi))}}{\sqrt n} - o_p\big(\frac{1}{\sqrt{n}}\big)\Big) \geq 1-\alpha
% \end{equation*}
% where $z_{1-\alpha}$ is the $1-\alpha$ quantile of the standard normal.\label{main thm}
% \end{theorem}
% \begin{theorem}[CLT for growing resample size under sampling without replacement]
% Suppose Assumptions \ref{L2}, \ref{L2 strengthened} and \ref{nondegeneracy strengthened} hold. If the resample size $k=o(\sqrt n)$, then
% $$\frac{\sqrt n(U_{n,k}-W_k)}{k\sqrt{Var(g_k(\xi))}}\Rightarrow N(0,1)$$
% where $N(0,1)$ is the standard normal variable.\label{main thm}
% \end{theorem}
\begin{theorem}[CLT for growing resample size under sampling with replacement]
Suppose Assumptions \ref{L2} and \ref{L2 strengthened} hold. If the resample size $k=O(n^{\gamma})$ for some constant $\gamma<\frac{1}{2}$, then the same conclusion of Theorem \ref{main thm} holds for $V_{n,k}$.\label{V thm}
\end{theorem}
%\begin{theorem}[CLT for growing resample size under sampling with replacement]
% Suppose Assumptions \ref{L2} and \ref{L2 strengthened} hold. If the resample size $k=O(n^{\gamma})$ for some constant $\gamma<\frac{1}{2}$, then there exists a sequence of normal random variables $N_{n,k}$ each with mean zero and variance $k^2Var(g_k(\xi))/n$ such that
% \begin{equation*}
%     V_{n,k} \stackrel{d}{=} W_k + N_{n,k} + o_p\big(\frac{1}{\sqrt{n}}\big)
% \end{equation*}
% where $\stackrel{d}{=}$ denotes equality in distribution.\label{V thm}
% \end{theorem}
% \begin{theorem}[CLT for growing resample size under sampling with replacement]
% Suppose Assumptions \ref{L2} and \ref{L2 strengthened} hold. If the resample size $k=O(n^{\gamma})$ for some constant $\gamma<\frac{1}{2}$, then for any $\alpha\in (0,1)$
% % $$\frac{\sqrt n(V_{n,k}-W_k)}{k\sqrt{Var(g_k(\xi))}}\Rightarrow N(0,1)$$
% \begin{equation*}
%     \liminf_{n\to\infty} P\Big(W_k \geq V_{n,k} - z_{1-\alpha}\frac{k\sqrt{Var(g_k(\xi))}}{\sqrt n} - o_p\big(\frac{1}{\sqrt{n}}\big)\Big) \geq 1-\alpha
% \end{equation*}
% where $z_{1-\alpha}$ is the $1-\alpha$ quantile of the standard normal.\label{V thm}
% \end{theorem}

Theorems \ref{main thm} and \ref{V thm} are analogs of Theorem \ref{thm:symmetric} when $k$ grows with $n$. In both theorems, we see that there is a limit in how large $k$ we can take relative to $n$, which is thresholded at roughly order $n$ and $\sqrt{n}$ for $U_{n,k}$ and $V_{n,k}$ respectively.
% (The $U$-statistic case can be generalized slight to allow $k=o(\sqrt n)$ instead of $O(n^\gamma)$ for $\gamma<1/2$).
A symmetric statistic with a growing $k$ is known as an infinite-order symmetric statistic (\cite{frees1989infinite}), and has been harnessed in analyzing random forests (\cite{mentch2016quantifying,wager2014confidence,wager2017estimation}). Theorems \ref{main thm} and \ref{V thm} give the precise conditions under which the SAA kernel results in an asymptotically converging infinite-order symmetric statistic. In particular, the requirement $k=o(n)$ in Theorem \ref{main thm} implies that, asymptotically, we can use almost the full data set to construct the resampled SAA in $U_{n,k}$.

We obtain Theorem \ref{main thm} by looking at the variance of $U_{n,k}$ via an analysis-of-variance (ANOVA) decomposition (\cite{efron1981jackknife}) of the symmetric kernel $H_k$. Thanks to the uncorrelatedness among the ANOVA terms, we can control the higher-order variance of $U_{n,k}$ at $o_p(1)$ by using a bound from \cite{wager2017estimation}. However, unlike in Theorem \ref{thm:symmetric} where a clean CLT is available as the first-order effect dominates the higher-order ones, the first-order effect $k\sqrt{Var(g_k(\xi))}\mathcal Z_{n,k}$ may become degenerate (i.e., tend to 0) as $k$ grows, and the growth conditions in Theorem \ref{main thm} allows obtaining a normality asymptotic with the $o_p(1)$ term. From Theorem \ref{main thm}, the conclusion of Theorem \ref{V thm} follows by using a relation between $U$- and $V$-statistics in the form
\begin{equation}
n^k(U_{n,k}-V_{n,k})=(n^k-{}_nP_k)(U_{n,k}-R_{n,k})\label{interim24}
\end{equation}
where ${}_nP_k=n(n-1)\cdots(n-k+1)$ and $R_{n,k}$ is the average of all $H_k(\xi_{i_1},\ldots,\xi_{i_k})$ with at least two of $i_1,\ldots,i_k$ being the same (see, e.g., Section 5.7.3 in \cite{serfling2009approximation}). By carefully controlling the difference between $U_{n,k}$ and $V_{n,k}$, one can show an asymptotic for $V_{n,k}$ under a slower growth rate of $k$. This leads to a slightly less general result for $V_{n,k}$ in Theorem \ref{V thm}. The proofs of Theorems \ref{main thm} and \ref{V thm} are both in Appendix \ref{proof:main}.
% We mention that the growth rates of $k$ in both Theorems \ref{main thm} and \ref{V thm} are sufficient conditions. In practice a $k$ of the same order as $n$ (e.g., $n/2$) still leads to correct confidence bounds, as our experiments show.

% and therefore the higher-order variance is not necessarily negligible. On the other hand, when the first-order variance $k^2Var(g_k(\xi))$ shrinks sufficiently slowly as described in Theorem \ref{main thm}, the CLT for $\mathcal{Z}_{n,k}$ follows by verifying the Lyapunov condition of the first order terms of the ANOVA decomposition.

% We will also see in the next section that, under further conditions, the growth of $k$ can be allowed bigger.
% The rate could be and could be a limitation of our analysis

% These two theorems conclude that $U_{n,k}$ and $V_{n,k}$ continue to well approximate the optimistic bound $W_k$ even as $k\to\infty$, under the depicted assumptions and bounds on the growth rate.

% The point estimator $\tilde Z_{n,k}^{bag}$ in the output of Algorithm \ref{bagging} gives precisely $U_{n,k}$ and $V_{n,k}$ when $B=\infty$. The standard error of this point estimator (when $B=\infty$) is given by $k\sqrt{Var(g_k(\xi))/n}$, which is estimated by $\tilde\sigma_{IJ}$ and will be analyzed in the next subsection.

A corollary of Theorems \ref{main thm} and \ref{V thm} is an exact CLT when the limit variance is non-degenerate:
\begin{corollary}[Exact CLT for growing resample size under non-degeneracy]
If $\liminf_{k\to\infty} k^2\allowbreak Var(g_k(\xi))>0$, then, under the same assumptions and respective growth rate of resample size $k$ in Theorem \ref{main thm} or \ref{V thm}, we have
$$\frac{\sqrt{n}(U_{n,k} - W_k)}{k\sqrt{Var(g_k(\xi))}}\Rightarrow N(0,1),\text{\ and }\frac{\sqrt{n}(V_{n,k} - W_k)}{k\sqrt{Var(g_k(\xi))}}\Rightarrow N(0,1).$$\label{exact CLT under non-degeneracy}
\end{corollary}
% Taking one step further, the following shows that bagging under sampling without replacement achieves almost the same efficiency as the direct use of CLT for SAA in \eqref{CI SAA}.
Non-degeneracy of the limit variance $\liminf_{k\to\infty} k^2\allowbreak Var(g_k(\xi))$ depends on the intricate interplay between the SAA optimal solution and the cost function, and thus may not be easily verified in general. For Lipschitz problems, however, the limit variance can be compactly characterized in terms of the cost function and minimizers of an associated Gaussian process.
\begin{theorem}[Characterization of limit variance under Lipschitzness]\label{Lipschitz characterization of limit variance}
Suppose Assumptions \ref{Lipschitz:decision}, \ref{L2} and \ref{L2 strengthened} hold, and that the decision space $\mathcal X\subseteq \R^d$ is compact. Let $Y$ be a centered Gaussian process on $\mathcal X^*$, the set of optimal solutions for \eqref{opt}, with covariance defined by $\kappa(x_1,x_2):=Cov(h(x_1,\xi),h(x_2,\xi))$ for any $x_1,x_2\in\mathcal X^*$. Then there exists a random variable $x^*_Y\in\mathcal X^*$ on the same probability space as the Gaussian process $Y$ such that $Y(x^*_Y)=\inf_{x\in\mathcal X^*}Y(x)$ almost surely and that
\begin{equation}\label{expression of limit variance for bagging}
    \lim_{k\to\infty}k^2Var(g_k(\xi)) = Var(E[h(x^*_Y,\xi)\rvert \xi])
\end{equation}
where $x_Y^*$ and $\xi$ are independent. Therefore, the non-degeneracy condition $\liminf_{k\to\infty}\allowbreak k^2 Var(g_k(\xi))>0$ holds if and only if $Var(E[h(x^*_Y,\xi)\rvert \xi])>0$. Alternatively, the limit variance can also be represented in terms of the covariance kernel, $Var(E[h(x^*_Y,\xi)\rvert \xi])\allowbreak=E[\kappa(x_Y^*, {x_Y^*}')]$, where ${x_Y^*}'$ is an independent copy of $x_Y^*$.
\end{theorem}
%To fully recover the CLT in Theorem \ref{thm:SAA}, we need a slightly stronger version of Assumption \ref{L2}:
%\begin{assumption}
%We have$$E\sup_{x\in\mathcal X}|h(x,\xi)|^{2+\delta}<\infty$$for some $\delta>0$.
%%where $\xi,\xi'$ are i.i.d. generated from $F$.
%\label{L2+delta}
%\end{assumption}
%Note that this assumption, as well as Assumption \ref{L2 strengthened}, is not required in Theorem \ref{thm:SAA}. The reason why we need those extra conditions lies in the different machineries employed to establish CLTs. Theorem \ref{thm:SAA} is directly based on a functional CLT together with an application of the delta method for which smoothness property of the SAA value is of more importance, whereas our CLTs come from Hajek projections of square integrable symmetric statistics where moments conditions  However, we do obtain more general result than , with an asymptotically negligible error in the SAA optimal value,
Theorem \ref{Lipschitz characterization of limit variance} shows that the limit variance is exactly the variance of the cost after being averaged over random minimizers of the limit Gaussian process on $\mathcal{X}^*$, or equivalently, the expected covariance kernel over a pair of independent minimizers. Theorem \ref{Lipschitz characterization of limit variance} is proved by first utilizing SAA asymptotic theories and uniform integrability of the SAA kernel to shrink the decision space from $\mathcal{X}$ to the set of optima $\mathcal{X}^*$, and then relating the limit variance to the cost function and minimizers of the limit Gaussian process through a coupling argument and an application of the argmax theorem from empirical process theory.

In order to demonstrate the generality of the non-degeneracy condition as implied by Theorem \ref{Lipschitz characterization of limit variance}, we consider general convex problems on $\R^d$ and apply Theorem \ref{Lipschitz characterization of limit variance} to derive more transparent conditions for non-degeneracy.
\begin{theorem}[Non-degeneracy for convex problems]\label{convex nondegeneracy}
Assume the conditions of Theorem \ref{Lipschitz characterization of limit variance}. Let the decision space $\mathcal{X}\subseteq \R^d$ be a compact convex set, $h(x,\xi)$ be convex in $x$ for each $\xi$, and $x_Y^*$ be the minimizer of the limit Gaussian process from Theorem \ref{Lipschitz characterization of limit variance}. Then $E[x_Y^*]$ is an optimal solution by convexity, and $Var(E[h(x^*_Y,\xi)\rvert \xi])=Var(h(E[x_Y^*],\xi))$. Therefore non-degeneracy holds if and only if $Var(h(E[x_Y^*],\xi))>0$. In particular, a sufficient condition for non-degeneracy is that $Var(h(x,\xi))>0$ for every optimal solution, or even more stringently, $Var(h(x,\xi))>0$ for every feasible solution $x\in\mathcal{X}$.
\end{theorem}
% for convex problems non-degeneracy of our bagging estimator
The last conclusion of Theorem \ref{convex nondegeneracy} stipulates that for convex problems, non-degeneracy of our bagging estimator can be guaranteed simply by having a noisy objective at every feasible solution. More generally, Theorem \ref{convex nondegeneracy} concludes that non-degeneracy is guaranteed by having a noisy objective at every optimal solution, and even more generally boils down to a noisy objective at  $E[x_Y^*]$, which can be viewed as a ``bagged optimal solution". This link between the non-degeneracy of a bagged optimal value and the non-zero objective variance at a bagged optimal solution arises from the fact that a convex objective $h(x,\xi)$ must be linear in $x$ when restricted to the (convex) set of optima $\mathcal{X}^*$, and therefore the expectation operation and the application of the cost function are exchangeable, i.e., $E[h(x_Y^*,\xi)\vert \xi] = h(E[x_Y^*],\xi)$. We illustrate the application of Theorem \ref{convex nondegeneracy} to two convex programs below. For both examples, we assume the basic required conditions (i.e., Assumptions \ref{Lipschitz:decision}, \ref{L2} and \ref{L2 strengthened}) hold.
% is approximately the expected SAA solution and thus A straightforward sufficient condition is then the objective being non-degenerate at every feasible solution. This elegant
\begin{example}\label{linear convex example}
Consider $\mathcal{X}=[-1,1]^d$, and a linear cost $h(x,\xi) = (a(\xi)+c)^Tx$, where $a(\xi)$ has mean zero and covariance matrix $\Sigma$. Then $Var(h(x,\xi))=x^T\Sigma x$, and a sufficient condition for non-degeneracy is that $x^T\Sigma x>0$ for every optimal solution $x$. For example, that $\Sigma$ is non-singular and $c\neq \mathbf{0}$ is sufficient because every optimal solution must then be on the boundary (hence nonzero) and thus have a strictly positive objective variance. If $c=\mathbf{0}$, however, the problem becomes degenerate as $E[x_Y^*]=\mathbf{0}$.
\end{example}
\begin{example}\label{center convex example}
Let $\mathcal{X}\subseteq \R^d$ be an arbitrary compact convex set, and $h(x,\xi) = f(\Vert x-\xi\Vert)$ where $f:[0,\infty)\to \R$ is a convex and strictly increasing function. Note that $h(x,\xi)$ is convex in $x$ because it is the composition of a convex and increasing function and a convex function. Theorem \ref{convex nondegeneracy} then entails that a sufficient condition for non-degeneracy is that $\xi$ is not supported on any $(d-1)$-sphere, i.e., set in the form of $\{\nu_0+r\nu:\nu\in\R^d,\Vert \nu\Vert=1\}$ for $\nu_0\in\R^d$ and $r\geq 0$, since it implies for each $x$ that $Var(\Vert x-\xi\Vert)>0$ and hence $Var(f(\Vert x-\xi\Vert))>0$ by the strict monotonicity of $f$.
\end{example}
We also point out that, since the limit variance is the same as the objective variance at the bagged optimal solution as stated in Theorem \ref{convex nondegeneracy}, it follows that the bound $\hat Z_n-z_{1-\alpha}\tilde{\sigma}_{IJ}$, where the point estimate is the full SAA optimal value from \eqref{CI SAA} and the standard error term is from our bagging bound in Algorithm \ref{bagging}, is a valid confidence bound by a similar argument from \cite{bayraksan2006assessing} for justifying the single-replication procedure. To briefly explain, we have $\hat Z_n\leq \bar{h}(E[x_Y^*])$ by optimality, where $\bar{h}(E[x_Y^*])=(1/n)\sum_{i=1}^nh(E[x_Y^*],\xi_i)$, and hence $P(\hat Z_n-z_{1-\alpha}\tilde{\sigma}_{IJ}\leq Z^*)\geq P(\bar{h}(E[x_Y^*])-z_{1-\alpha}\tilde{\sigma}_{IJ}\leq Z^*)\approx P(\bar{h}(E[x_Y^*])-z_{1-\alpha}\sqrt{Var(h(E[x_Y^*],\xi))/n}\leq E[h(E[x_Y^*],\xi)])\to 1-\alpha$. This bound combines the advantages of both bagging and direct-CLT bounds: Compared to \eqref{CI SAA}, it is conjectured to have a stabler and smaller standard error term (thanks to the discussion in the next section) and compared to our bagging bound it has a tighter point estimate. However, this bound is guaranteed to be valid only for convex problems.

Next we show yet another refinement when, in addition to Lipschitzness, the optimal solution is also unique. Under this additional assumption, Theorem \ref{Lipschitz characterization of limit variance} immediately forces the limit variance to be $Var(h(x^*,\xi))$, where $x^*$ is the unique optimum. Our bagging estimator thus elicits the same CLT as Theorem \ref{thm:SAA}.
% , and thus recovers the direct-CLT bound in \eqref{CI SAA}.
To state the next result, we define:
\begin{definition}[Essential uniqueness]\label{def: essentially unique}
We say \eqref{opt} has essentially unique optimal solution if $h(x_1,\xi)=h(x_2,\xi)$ almost surely for any two optimal solutions $x_1,x_2\in\mathcal{X}$.
\end{definition}
Essential uniqueness is more general than the usual uniqueness in that it allows multiple optimal solutions as long as they perform exactly the same under any possible scenario of the uncertain quantity, and hence enhances the applicability of our next result. More importantly, it is both a sufficient and necessary condition to ensure the SAA weak limit in \eqref{CLT SAA general} is Gaussian; otherwise, the limit is the minimum of a Gaussian process that triggers a strict variance reduction property of our approach (see Theorem \ref{compare var:asymptotic} in the sequel). The next result is as follows:
\begin{theorem}[Recovery of the classical CLT for SAA under solution uniqueness]\label{recover SAA}
Suppose Assumptions \ref{Lipschitz:decision}, \ref{L2} and \ref{L2 strengthened} hold, that the decision space $\mathcal X\subseteq \R^d$ is compact, and that \eqref{opt} has essentially unique optimal solution. Let $x^*$ be an optimal solution and assume $Var(h(x^*,\xi))>0$. We have $\sqrt n(U_{n,k}-W_k)\Rightarrow N(0,Var(h(x^*,\xi)))$ for arbitrary choices of $k\leq n$. Moreover, $W_k-Z^*=o(1/\sqrt k)$ as $k\to\infty$, and if $k\geq\epsilon n$ for some constant $\epsilon>0$ we have
\begin{equation}
\sqrt n(U_{n,k}-Z^*)\Rightarrow N(0,Var(h(x^*,\xi)))\label{new non-unique}
\end{equation}
where $N(0,Var(h(x^*,\xi)))$ is normal with mean zero and variance $Var(h(x^*,\xi))$.
\end{theorem}
Note that, compared with Theorems \ref{main thm} and \ref{V thm}, the centering quantity in \eqref{new non-unique} is changed from $W_k$ to $Z^*$. The asymptotic distribution is Gaussian with variance precisely the objective variance at $x^*$. This gives rise to the same CLT as Theorem \ref{thm:SAA} in the special case where \eqref{opt} has essentially unique optimal solution, and in particular when the set of optima is a singleton $\mathcal{X}^*=\{x^*\}$. If the essential uniqueness condition does not hold, there could be a discrepancy between the optimistic bound $W_\infty$ and $Z^*$ (This can be hinted by observing the different types of limits between Theorems \ref{main thm}, \ref{V thm} and Theorem \ref{thm:SAA}, namely Gaussian versus the minimum of a Gaussian process).

We obtain Theorem \ref{recover SAA} from a more delicate control of the high-order variance components in the ANOVA decomposition and an analysis on the negligible bias of $W_k$ with respect to the true optimal value $Z^*$, both of which are related to the maximal deviation of an empirical process generated by the centered cost function indexed by the decision, i.e., $\mathcal F:=\{h(x,\cdot)-Z(x):x\in \mathcal X\}$. The Lipschitz assumption allows us to estimate this maximal deviation using empirical process theory. Appendix \ref{sec:refined} shows the proof details for Theorems \ref{Lipschitz characterization of limit variance}, \ref{convex nondegeneracy} and \ref{recover SAA}.
% We obtain Theorem \ref{recover SAA} from a different path than Theorem \ref{main thm}, in particular by looking at the variance of $U_{n,k}$ via an analysis-of-variance (ANOVA) decomposition (\cite{efron1981jackknife}) of the symmetric kernel $H_k$. Thanks to the uncorrelatedness among the ANOVA terms, we can control the variance of $U_{n,k}$ by using a bound from \cite{wager2017estimation}, which can be shown to depend on the maximal deviation of an empirical process generated by the centered cost function indexed by the decision, i.e., $\mathcal F:=\{h(x,\cdot)-Z(x):x\in \mathcal X\}$. The Lipschitz assumption allows us to estimate this maximal deviation using empirical process theory. Appendix \ref{sec:refined} shows the proof details.

\section{Error Estimates and Coverages}\label{sec:error}
With the limit theorems in Sections \ref{sec:fixed} and \ref{sec:growing}, we now derive the coverage guarantees for the output from Algorithm \ref{bagging}. In doing so, we incorporate two additional developments. One is the analysis of the IJ estimator in approximating the standard error. Second is the analysis of the Monte Carlo error in running the bootstrap with a finite number of replications. First, we have the following consistency of the IJ variance estimator, relative to the magnitude of the target standard error:
% Together with the results in Section \ref{sec:growing}, these will give us an overall coverage guarantee
%\begin{theorem}\label{consistency_IJ:U}
%Under Assumption \ref{L2}, \ref{L2 strengthened} and \ref{nondegeneracy strengthened}, if resampling is without replacement and the resample size $k$ satisfies \eqref{range:k} as well as $k=o(n)$, then the IJ variance estimator is consistent, i.e.
%$$\frac{n^2}{(n-k)^2}\sum_{i=1}^n\mathrm{Cov}_*^2(N_i^*,H_k^*)\Big/\frac{k^2}{n}Var(g_k(\xi))\stackrel{p}{\to}1.$$
%\end{theorem}
%\begin{theorem}\label{consistency_IJ:V}
%Under Assumption \ref{L2}, \ref{L2 strengthened} and \ref{nondegeneracy strengthened}, if resampling is with replacement and the resample size $k=O(n^{\gamma})$ for some $\gamma<\frac{1}{2}$, then the IJ variance estimator is consistent, i.e.
%$$\sum_{i=1}^n\mathrm{Cov}_*^2(N_i^*,H_k^*)\Big/\frac{k^2}{n}Var(g_k(\xi))\stackrel{p}{\to}1.$$
%\end{theorem}
\begin{theorem}[Consistency of IJ estimator under resampling without replacement]\label{consistency_IJ:U}
Consider resampling without replacement. In either of the following cases:
\begin{itemize}
    \item Assumptions \ref{L2} and \ref{L2 strengthened} hold and $k=o(n)$,
    \item Assumptions \ref{Lipschitz:decision}, \ref{L2} and \ref{L2 strengthened} hold, the decision space $\mathcal{X}\subseteq \R^d$ is compact, \eqref{opt} has essentially unique optimal solution and $k\leq \theta n$ for some constant $\theta<1$,
\end{itemize}
% we have that
% In any of the following three settings:
% Under either the same conditions and resample sizes of Theorem \ref{main thm}, or the same conditions of Theorem \ref{recover SAA}
% \begin{enumerate}
% \item Assumptions \ref{L2}, \ref{L2 strengthened} and \ref{nondegeneracy strengthened} hold and $k=o(\sqrt n)$
% \item Assumptions \ref{Lipschitz:decision}, \ref{L2}, \ref{L2 strengthened} and \ref{nondegeneracy relaxed} hold, the decision space $\mathcal X$ is compact and $k=o(n)$\label{second_setting}
% \item In addition to the assumptions in $\ref{second_setting}$, further assume the problem \eqref{opt} has a unique optimal solution, and use resample size $k\leq \theta n$ for some constant $\theta<1$
% \end{enumerate}
%(1) the same conditions and resample sizes of Theorem \ref{main thm} or \ref{Lipschitz_k full}, or (2) the same conditions of Theorem \ref{recover SAA} but resample size $k\leq \theta n$ for some $\theta<1$,
the IJ variance estimator is consistent up to a negligible error, i.e.
% $$\frac{n^2}{(n-k)^2}\sum_{i=1}^n\mathrm{Cov}_*^2(N_i^*,H_k^*)\Big/\frac{k^2}{n}Var(g_k(\xi))\stackrel{p}{\to}1.$$
$$\frac{n^2}{(n-k)^2}\sum_{i=1}^n\mathrm{Cov}_*^2(N_i^*,H_k^*) = \frac{k^2}{n}Var(g_k(\xi)) + o_p\big(\frac{1}{n}\big).$$
\end{theorem}
%\begin{theorem}[Relative consistency of IJ estimator under resampling without replacement]\label{consistency_IJ:U}
%Consider resampling without replacement. Under the same conditions and resample sizes of either Theorem \ref{main thm} or \ref{Lipschitz_k full}, the IJ variance estimator is relatively consistent, i.e.
%$$\frac{n^2}{(n-k)^2}\sum_{i=1}^n\mathrm{Cov}_*^2(N_i^*,H_k^*)\Big/\frac{k^2}{n}Var(g_k(\xi))\stackrel{p}{\to}1.$$
%\end{theorem}
\begin{theorem}[Consistency of IJ estimator under resampling with replacement]\label{consistency_IJ:V}
Consider resampling with replacement. If Assumptions \ref{L2} and \ref{L2 strengthened} hold and $k=O(n^{\gamma})$ for some $\gamma < \frac{1}{2}$,
% Under the same conditions and resample sizes of Theorem \ref{V thm},
% If Assumptions \ref{L2}, \ref{L2 strengthened} and \ref{nondegeneracy strengthened} hold, and $k=O(n^{\gamma})$ for some constant $\gamma<\frac{1}{2}$,
then the IJ variance estimator is consistent up to a neglibible error, i.e.
% $$\sum_{i=1}^n\mathrm{Cov}_*^2(N_i^*,H_k^*)\Big/\frac{k^2}{n}Var(g_k(\xi))\stackrel{p}{\to}1.$$
$$\sum_{i=1}^n\mathrm{Cov}_*^2(N_i^*,H_k^*)=\frac{k^2}{n}Var(g_k(\xi))+o_p\big(\frac{1}{n}\big).$$
\end{theorem}
% The three sets of assumptions and resample size in Theorem \ref{consistency_IJ:U} are precisely those of Theorem \ref{main thm} (the general case), Theorem \ref{Lipschitz_k full} (the Lipschitz case) and Theorem \ref{recover SAA} (the Lipschitz and unique solution case) respectively, except a slight tightening on the choice of $k$ in the last case that can only be arbitrarily close to but not exactly $n$. The assumptions in Theorem \ref{consistency_IJ:V} are precisely those of Theorem \ref{V thm}. Theorem \ref{consistency_IJ:U} is justified by adopting the arguments for random forests in \cite{wager2017estimation} and a weak law of large numbers, and Theorem \ref{consistency_IJ:V} follows from analyzing the difference between $U$- and $V$-statistics as in the proof of Theorem \ref{V thm}. Appendix \ref{proof:consistency} shows the details.
Theorem \ref{consistency_IJ:U} is justified by adopting the arguments for random forests in \cite{wager2017estimation} and a weak law of large numbers, and Theorem \ref{consistency_IJ:V} follows from analyzing the difference between $U$- and $V$-statistics as in the proof of Theorem \ref{V thm}. Appendix \ref{proof:consistency} shows the details.

When a large enough bootstrap size $B$ is used in Algorithm \ref{bagging}, the Monte Carlo errors in estimating the point estimator and its variance both vanish. This gives an overall coverage guarantee for the output of our bagging procedure, as in the next theorem:
\begin{theorem}[CLT for Algorithm \ref{bagging}]\label{final_guarantee}
In the case of resampling without replacement, assume either 1) Assumptions \ref{L2} and \ref{L2 strengthened} hold and $k=o(n)$, or 2) Assumptions \ref{Lipschitz:decision}, \ref{L2} and \ref{L2 strengthened} hold, the decision space $\mathcal{X}\subseteq \R^d$ is compact, \eqref{opt} has essentially unique optimal solution and $k\leq \theta n$ for some constant $\theta<1$.
% the same conditions and resample sizes of Theorem \ref{main thm} or the same conditions of Theorem \ref{recover SAA} and resample size $k\leq \theta n$ for some constant $\theta<1$.
In the case of resampling with replacement, assume Assumptions \ref{L2} and \ref{L2 strengthened} hold and $k=O(n^{\gamma})$ for some $\gamma < \frac{1}{2}$.
% assume the same conditions and resample sizes of Theorem \ref{V thm}.
If the bootstrap size $B$ in Algorithm \ref{bagging} is such that $B/n\to\infty$, then the output of Algorithm \ref{bagging} satisfies
% $$\frac{\tilde Z_{n,k}^{bag}-W_k}{\tilde{\sigma}_{IJ}}\Rightarrow N(0,1)$$
\begin{equation*}
    \tilde Z_{n,k}^{bag} - W_k = \tilde{\sigma}_{IJ}\cdot \mathcal Z_{n,k} + o_p\big(\frac{1}{\sqrt{n}}\big)
\end{equation*}
where $\mathcal Z_{n,k}$ is the same sequence of random variables from Theorem \ref{main thm}, and the $o_p$ is with respect to the data $\bm{\xi}_{1:n}$ and the sampling randomness in Algorithm \ref{bagging} jointly.
\end{theorem}

An immediate consequence of Theorem \ref{final_guarantee} is the correct coverage of the true optimal value:
\begin{corollary}[Correct coverage from Algorithm \ref{bagging}]\label{coverage}
Under the same assumptions, growth rates of the resample size $k$ and the bootstrap size $B$ in Theorem \ref{final_guarantee}, the output of Algorithm \ref{bagging} satisfies
% $$P\left(\tilde Z_{n,k}^{bag}-z_{1-\alpha}\tilde\sigma_{IJ}\leq Z^*\right)\geq P\left(\tilde Z_{n,k}^{bag}-z_{1-\alpha}\tilde\sigma_{IJ}\leq W_k\right)\to1-\alpha$$
% \begin{eqnarray*}
%     &&\liminf_{n\to\infty} P\Big(\tilde Z_{n,k}^{bag} - z_{1-\alpha}\tilde{\sigma}_{IJ} - o_p\big(\frac{1}{\sqrt{n}}\big)\leq Z^*\Big)\\
%     &\geq&\liminf_{n\to\infty} P\Big(\tilde Z_{n,k}^{bag} - z_{1-\alpha}\tilde{\sigma}_{IJ} - o_p\big(\frac{1}{\sqrt{n}}\big)\leq W_k\Big) \geq 1-\alpha
% \end{eqnarray*}
\begin{equation}
\liminf_{n\to\infty} P\Big(\tilde Z_{n,k}^{bag} - z_{1-\alpha}\tilde{\sigma}_{IJ} - o_p\big(\frac{1}{\sqrt{n}}\big)\leq W_k\leq Z^*\Big) \geq 1-\alpha\label{coverage non-degenerate}
\end{equation}
where $z_{1-\alpha}$ is the $1-\alpha$ quantile of the standard normal, and the $o_p$ term is with respect to both the data $\bm{\xi}_{1:n}$ and the sampling randomness in Algorithm \ref{bagging} jointly. In particular, if non-degeneracy $\liminf_{k\to\infty}k^2Var(g_k(\xi))>0$ holds, then
\begin{equation}
\lim_{n\to\infty} P\Big(\tilde Z_{n,k}^{bag} - z_{1-\alpha}\tilde{\sigma}_{IJ} \leq W_k\leq Z^*\Big) = 1-\alpha.\label{coverage exact}
\end{equation}
% \begin{equation*}
%      \geq 1-\alpha.
% \end{equation*}
\end{corollary}

% \begin{corollary}[Correct coverage from Algorithm \ref{bagging}]\label{coverage2}
% Under the same assumptions, growth rates of the resample size $k$ and the bootstrap size $B$ in Theorem \ref{final_guarantee}, the output of Algorithm \ref{bagging} satisfies
% $$P\left(\tilde Z_{n,k}^{bag}-z_{1-\alpha}\tilde\sigma_{IJ}\leq Z^*\right)\geq P\left(\tilde Z_{n,k}^{bag}-z_{1-\alpha}\tilde\sigma_{IJ}\leq W_k\right)\to1-\alpha$$
% \begin{eqnarray*}
%     &&\liminf_{n\to\infty} P\Big(\tilde Z_{n,k}^{bag} - z_{1-\alpha}\tilde{\sigma}_{IJ} - o_p\big(\frac{1}{\sqrt{n}}\big)\leq Z^*\Big)\\
%     &\geq&\liminf_{n\to\infty} P\Big(\tilde Z_{n,k}^{bag} - z_{1-\alpha}\tilde{\sigma}_{IJ} - o_p\big(\frac{1}{\sqrt{n}}\big)\leq W_k\Big) \geq 1-\alpha
% \end{eqnarray*}
% \begin{equation*}
% \lim_{n\to\infty} P\Big(\tilde Z_{n,k}^{bag} - z_{1-\alpha}\tilde{\sigma}_{IJ} - o_p\big(\frac{1}{\sqrt{n}}\big)\leq W_k\leq Z^*\Big) = 1-\alpha
% \end{equation*}
% where $z_{1-\alpha}$ is the $1-\alpha$ quantile of the standard normal, and the $o_p$ term is with respect to the data, the resampling randomness in bagging and also possibly some external artificial randomness. In particular, if non-degeneracy $\liminf_{k\to\infty}k^2Var(g_k(\xi))>0$ holds, then
% \begin{equation*}
% \lim_{n\to\infty} P\Big(\tilde Z_{n,k}^{bag} - z_{1-\alpha}\tilde{\sigma}_{IJ} \leq W_k\leq Z^*\Big) = 1-\alpha.
% \end{equation*}
% % \begin{equation*}
% %      \geq 1-\alpha.
% % \end{equation*}
% \end{corollary} for the inaccuracy of the standard error term $z_{1-\alpha}\tilde{\sigma}_{IJ}$

Theorem \ref{final_guarantee} and Corollary \ref{coverage} show a correct asymptotic coverage of our bagging bound for the optimistic bound $W_k$ and in turn the true optimal value $Z^*$. This guarantee holds regardless of degeneracy. To explain in more detail, the $o_p(1/\sqrt n)$ error in \eqref{coverage non-degenerate} stipulates that our generated confidence bound is accurate up to order $1/\sqrt n$, which is an accuracy level stemming from the canonical $1/\sqrt n$-rate of the CLTs. When degeneracy occurs, it is possible that $z_{1-\alpha}\tilde\sigma_{IJ}$ is of order smaller than $1/\sqrt n$. In this case, our generated confidence bound is not refined enough to deliver correct coverage, but at the same time the amount needed to adjust $\tilde Z_{n,k}^{bag}$ to generate a valid bound is super-canonically small, i.e., of smaller order than $1/\sqrt n$. In other words, $\tilde Z_{n,k}^{bag}$ alone is already very close to delivering a confidence bound. Moreover, in this degeneracy situation, little is known about the distribution of the bagging estimator or its weak limit (if there is any), e.g., it may be discontinuous and thus not every coverage level can be exactly attained, which leads to the looseness, i.e., $\geq1-\alpha$ instead of $=1-\alpha$, in \eqref{coverage non-degenerate}.

% When the latter occurs, the $o_p(1/\sqrt{n})$ error and the looseness of the coverage in Corollary \ref{coverage} are related to the possible presence of degeneracy: The former stipulates that
%The latter arises because and the exactness of the later depends on the discrepancy between $W_k$ and $Z^*$. This is thanks to the standard normal limit.  under the same setting as when the classical SAA CLT has a normal limit

On the other hand, when non-degeneracy holds, our confidence bound in \eqref{coverage exact} delivers an exact asymptotic coverage for $W_k$ and in turn a correct coverage for the true optimal value $Z^*$. The exactness of our bound for $Z^*$ depends on the discrepancy between $W_k$ and $Z^*$. For instance, Theorem \ref{recover SAA} provides conditions under which this discrepancy vanishes and which hints that our bound is close to having exact coverage for $Z^*$.

Lastly, note that $B$ needs to be taken to have order greater than $n$ to wash away the Monte Carlo error in the IJ variance estimator $\tilde{\sigma}_{IJ}^2$ under the considered conditions. Notably, the requirement for $B$ is independent of the resample size $k$ (thanks to the diminishing variance of the SAA kernel $H_k$ implied by the Efron-Stein inequality). Thus, to achieve the best result regarding the tightness of the bound, we would choose $k$ as large as allowed regardless of how we choose $B$. In fact, the required bootstrap size $B$ can be further reduced if a bias correction is applied to the IJ variance estimator, as the major source of the Monte Carlo error is the upward bias that is introduced by squaring the noisy covariance estimates when constructing $\tilde{\sigma}_{IJ}^2$ in Algorithm \ref{bagging}. Similar computation reduction has been achieved by debiasing IJ variance estimators for random forests (\cite{wager2014confidence}). We describe a debiased variant of Algorithm \ref{bagging} in Appendix \ref{sec:debiased bagging} along with an informal analysis that suggests a required bootstrap size $B$ of order $\sqrt{n}$. Our experiments show that with $B=500$ the debiased variant consistently delivers satisfactory performances in practice for data sizes as large as several thousands.

\section{Statistical Properties of Bagging Bounds and Comparisons with Batching and Single-Replication Procedures}\label{sec:properties}
% , the standard error of $U_{n,k}$ and $V_{n,k}$, namely   Next, the following bounds the magnitude of the resulting asymptotic variance of $U_{n,k}$ and $V_{n,k}$, or in other words the standard error for $\tilde Z_{n,k}^{bag}$ in Algorithm \ref{bagging} which $\tilde\sigma_{IJ}^2$ there estimates:
We analyze the properties of our confidence bounds. Here, we focus on the statistical issues, rather than computational, i.e., assume $B=\infty$. In this case, our bounds can be viewed as consisting of a point estimator $U_{n,k}$ or $V_{n,k}$ and a standard error $k\sqrt{Var(g_k(\xi))/n}$. We compare these estimators with the bound \eqref{CI SAA} given by the single-replication procedure and the bound \eqref{CI batching} given by the batching procedure. We first show that in general the standard errors of all these bounds have the same order $1/\sqrt n$.
\begin{proposition}[Magnitude of the standard error]\label{bound se}
Recall that $\tilde Z_{n,k},\hat Z_n$ are the point estimators by the batching and single-replication procedures respectively. Under Assumption \ref{L2}, $Var(U_{n,k})$, $Var(\tilde{Z}_{n,k})$ and $Var(\hat Z_n)$ are all of order $O(1/n)$ regardless of how $k$ grows with $n$. $Var(V_{n,k})$ is also $O(1/n)$ if $k$ is chosen as described in Theorem \ref{V thm}.
\end{proposition}
The proof of Proposition \ref{bound se} applies the Efron-Stein inequality to control the total variance $Var(H_k)$ of the SAA kernel and the ANOVA decomposition of $H_k$ that contains $kVar(g_k(\xi))$ as the first-order variance component. The proof details are in Appendix \ref{proof:se}. Although the standard errors share the same order of magnitude,
%  the coupling argument in bounding the variance that appears in the proof of Theorem \ref{main thm}
% Note that Proposition \ref{bound se} is quite general in that it does not impose any growth rate restriction on $k$
% We also note that, under conditions that provide a CLT for the SAA (i.e., Theorem \ref{thm:SAA}), the $\tilde\sigma$ in the batching bound \eqref{CI batching} can be of order $O(1/\sqrt k)$ as the data size per batch $k$ grows, and thus the resulting error term there can be controlled to be $O(1/\sqrt n)$ like ours (and also the direct-CLT bound \eqref{CI SAA}). Nonetheless, Proposition \ref{bound se} is free of such type of assumptions.
% As an analog of Proposition \ref{compare var} in the asymptotic regime,
the following result shows the higher statistical efficiency of our bagging procedure than batching and single-replication procedures by a constant factor:
\begin{theorem}[Asymptotic variance reduction]\label{compare var:asymptotic}
Recall that $\tilde Z_{n,k},\hat Z_n$ are the point estimators by the batching and single-replication procedures respectively. Suppose Assumptions \ref{Lipschitz:decision}, \ref{L2} and \ref{L2 strengthened} hold, and the decision space $\mathcal{X}\subseteq \R^d$ is compact. Suppose the resample size $k=o(n)$ in the case of resampling without replacement, or $k=O(n^{\gamma})$ for some $\gamma < \frac{1}{2}$ in the case of resampling with replacement.
% Suppose the conditions of Theorem \ref{Lipschitz characterization of limit variance} hold, and the resample size $k$ is chosen as described in either Theorem \ref{main thm} in the case of resampling without replacement, or Theorem \ref{V thm} in the case of resampling with replacement.
We have
\begin{equation*}
    \begin{aligned}
        nVar(U_{n,k}),nVar(V_{n,k})&\to Var(E[h(x^*_Y,\xi)\rvert \xi])\\
        nVar(\tilde Z_{n,k}),nVar(\hat Z_n)&\to Var(Y(x^*_Y))
    \end{aligned}
\end{equation*}
as $n,k\to\infty$, where $Y$ and $x^*_Y$ are the Gaussian process and its minimizer from Theorem \ref{Lipschitz characterization of limit variance}. Moreover, we have $Var(E[h(x^*_Y,\xi)\rvert \xi])\leq Var(Y(x^*_Y))$, and in particular
% If the limit variance from Theorem \ref{Lipschitz characterization of limit variance} satisfies $Var(E[h(x^*_Y,\xi)\rvert \xi])>0$.
% With the same batch size and resample size, both denoted by $k$, we define the asymptotic ratios of variance
% \begin{equation}\label{asymptotic_efficiency}
% r_U:=\limsup_{n,k\to\infty}\frac{Var(U_{n,k})}{Var(\tilde{Z}_{n,k})},\;r_V:=\limsup_{n,k\to\infty}\frac{Var(V_{n,k})}{Var(\tilde{Z}_{n,k})}.
% \end{equation}
% We have $r_U=r_V=\frac{Var(E[h(x^*_Y,\xi)\rvert \xi])}{Var(Y(x^*_Y))}$ , in particular
\begin{itemize}
\item if \eqref{opt} has essentially unique optimum, then the SAA limit \eqref{CLT SAA general} is Gaussian and $Var(E[h(x^*_Y,\xi)\rvert \xi])=Var(Y(x^*_Y))$,
\item if optima of \eqref{opt} are not essentially unique, then the SAA limit \eqref{CLT SAA general} is non-Gaussian and $Var(E[h(x^*_Y,\xi)\rvert \xi])<Var(Y(x^*_Y))$.
% unless $Var(E[h(x^*_Y,\xi)\rvert \xi])=Var(Y(x^*_Y))\allowbreak=0$.
\end{itemize}
% \begin{itemize}
% \item $Var(E[h(x^*_Y,\xi)\rvert \xi])=Var(Y(x^*_Y))$, if \eqref{opt} has essentially unique optimum,
% \item $Var(E[h(x^*_Y,\xi)\rvert \xi])<Var(Y(x^*_Y))$ unless $Var(E[h(x^*_Y,\xi)\rvert \xi])=Var(Y(x^*_Y))\allowbreak=0$, if optima are not essentially unique.
% \end{itemize}
\end{theorem}
% The second case $r_U=r_V<1$ in Theorem \ref{compare var:asymptotic} corresponds to the case of multiple optimal solutions in $\mathcal X^*$, which induces a weak scaled limit of the SAA value $H_k$ as the infimum of a Gaussian process that is in general non-Gaussian.
The following example shows that, when there are multiple optimal solutions, the limit variance ratio between the bagging estimator and batching/single-replication estimator not only is strictly less than $1$ but also can be arbitrarily close to $0$.
%When the cost function is Lipschitz continuous in the decision and the decision space is compact,
\begin{example}\label{efficiency_example}
% Consider the cost function
% \begin{equation*}
% h(x,\xi)=
% \begin{cases}
% (2-x)\xi_1+(x-1)\xi_{2}&\text{ if }1\leq x\leq 2\\
% \hspace{10ex}\vdots&\hspace{8ex}\vdots\\
% (j+1-x)\xi_j+(x-j)\xi_{j+1}&\text{ if }j< x\leq j+1\\
% \hspace{10ex}\vdots&\hspace{8ex}\vdots\\
% (d-x)\xi_{d-1}+(x-(d-1))\xi_{d}&\text{ if }d-1< x\leq d
% \end{cases}
% \end{equation*}
Consider the stochastic linear program
\begin{equation*}
\begin{aligned}
    &\min_{x\in \R^d}&&E[\xi^Tx]\\
    &\text{s.t.}&&\sum_{i=1}^dx_i=1\\
    &&&\ x_i\geq 0\text{ for }i=1,\ldots,d
\end{aligned}
\end{equation*}
with the uncertain quantity $\xi=(\xi_1,\ldots,\xi_d)$ where $\xi_j,j=1,\ldots,d$ are independent standard normal variables. The expected objective is thus constantly zero so every feasible solution is optimal. The limit Gaussian process $Y(x)=\xi^Tx$ in distribution, hence $x^*_Y$ is uniformly distributed over $\{e_1,e_2,\ldots,e_d\}$ where each $e_i\in\R^d$ is the $i$-th canonical basis vector, and $Y(x^*_Y)=\min_{j=1,\ldots,d}\xi_j$ in distribution. A direct application of Corollary 1.9 in \cite{ding2015multiple} leads to $Var(Y(x^*_Y))\geq C/\log d$ for some universal constant $C>0$, whereas $Var(E[h(x^*_Y,\xi)\rvert \xi])=Var((1/d)\sum_{j=1}^d\xi_j)=1/d$. Therefore $Var(E[h(x^*_Y,\xi)\rvert \xi]) / Var(Y(x^*_Y))\leq \log d/(Cd)$ which shrinks to zero as $d$ grows.
\end{example}

%$$Var(g_{k,c}(\xi_1,\ldots,\xi_c))\leq\frac{Mc^2}{k^2}$$
%for some constant $M>0$.

% Besides the zero-bias property, another advantage of using
Furthermore, the following shows that the point estimator under sampling without replacement always has a smaller variance than the batching estimator, for any $n$ and $k$:
% This is illustrated in the following proposition.
\begin{theorem}[Variance reduction over batching under any finite sample]\label{compare var}
Recall that $\tilde Z_{n,k}$ is the point estimator by the batching procedure. Denote $\{\xi_1,\ldots,\xi_n\}$ as the (unordered) collection of values of the data set $\xi_1,\ldots,\xi_n$. With the same batch size and resample size, both denoted by $k$, we have
\begin{equation*}
Var(\tilde Z_{n,k})=Var(U_{n,k})+E[Var(\tilde Z_{n,k}\vert \{\xi_1,\ldots,\xi_n\})]
\end{equation*}
and hence $Var(\tilde Z_{n,k})\geq Var(U_{n,k})$ for any $k\geq1$.
\end{theorem}
% \begin{proof}
\proof{Proof.}
By the law of total variance we have
\begin{equation*}
Var(\tilde Z_{n,k})=E[Var(\tilde Z_{n,k}\vert \{\xi_1,\ldots,\xi_n\})]+Var(E[\tilde Z_{n,k}\vert \{\xi_1,\ldots,\xi_n\}]).
\end{equation*}
The desired conclusion follows from noticing that $E[\tilde Z_{n,k}\vert \{\xi_1,\ldots,\xi_n\}]\allowbreak=U_{n,k}$.\Halmos
\endproof

Theorem \ref{compare var} reinforces the smaller standard error in bagging compared to batching from asymptotic to \emph{any} finite sample, provided that we use sampling without replacement. The key reasoning behind Theorem \ref{compare var} is that the batching estimate depends on the ordering of the data; if the data are reordered, then the batching estimate changes. Bagging eliminates the variability due to the ordering of the data by averaging over all the possible combinations. Alternately, one can also interpret bagging as a conditional Monte Carlo scheme applied on the batching estimator given the unordered collection of values realized by the data.
% In other words, we are able to reduce the variance of the estimator at the cost of solving more SAA problems. However, it's not clear whether similar variance reduction holds true for the estimator $V_{n,k}$ because of resampling with replacement.

Theorems \ref{compare var:asymptotic} and \ref{compare var} focus on comparison of the point estimators, and now we compare the standard error terms in the bounds. In both the batching and bagging bounds \eqref{CI batching} and \eqref{bag bound}, the standard error terms are constructed from consistent estimates of variances of the respective point estimates, therefore the smaller variance of the bagging point estimator translates to a smaller standard error term and hence an overall tighter confidence bound than in batching. The single-replication procedure \eqref{CI SAA} as well as its variants such as the independent two-replication procedure and the averaged two-replication procedure proposed in \cite{bayraksan2006assessing}, however, follows a different rationale in that the standard error term $\hat{q}/\sqrt{n}$ in \eqref{CI SAA} is judiciously constructed using the sample variance of the objective at an SAA solution potentially inconsistent with the variance of the point estimate $\hat Z_n$. We argue that standard error terms computed this way are still larger than the bagging error terms. To see this, under certain conditions one can expect that the sample variance $\hat{\sigma}^2(x):=\frac{1}{n}\sum_{i=1}^n(h(x,\xi_i)-\overline{h}(x))^2$ where $\overline{h}(x)=\frac{1}{n}\sum_{i=1}^nh(x,\xi_i)$ converges to the true variance $Var(h(x,\xi))$ uniformly for all $x\in\mathcal{X}$, and that the SAA optimal solution $\hat x_n^*\Rightarrow x_Y^*$, therefore the sample variance $\hat{\sigma}^2(\hat x_n^*)=\frac{1}{n}\sum_{i=1}^n(h(\hat{x}_n^*,\xi_i) - \bar{h}(\hat{x}_n^*))^2$ used in the single-replication bound and its variants has an expected value $E[\hat\sigma^2(\hat x_n^*)]\to Var(h(x_Y^*,\xi))$, larger than the bagging limit variance $Var(E[h(x_Y^*,\xi)\vert \xi])$ by the law of total variance. Moreover, we can expect that the bagging limit variance is strictly smaller when optimal solutions are not essentially unique.

With the above comparisons, we now reason more precisely our beginning claim that we can improve the tradeoff between bound tightness and statistical accuracy faced by batching. This is based on two perspectives: First is that bagging allows using a larger resampled  SAA size $k$ than the batched SAA size that is confined by $mk=n$, thus utilizing a tighter optimistic bound. Second, even assuming we use the same resampled SAA size as the batched SAA size, Theorems \ref{compare var:asymptotic} and \ref{compare var} conclude that the bagging standard error is no worse than batching, which also translates to a bound that can only be tighter. Note that this latter benefit is attributed to the use of many SAAs instead of fewer in batching. In fact, we also conjecture that the variance of the standard error estimator, not only the point estimator, in bagging is also no larger than that in batching, by a similar reasoning that the bagging standard error estimator is also constructed from many SAAs. Nonetheless, checking whether such a claim indeed holds would be more suited for future work.

Our another beginning motivation, compared to direct-CLT or the single-replication procedure, is that we can alleviate the instability of standard error estimator stemming from the ``jumping" behavior of nearly optimal solutions. Theorem \ref{compare var:asymptotic} and the discussions above argue that bagging possesses a smaller standard error than single-replication, especially when the optimal solutions are not essentially unique. This is conceptually related to our motivational claim. Nonetheless, our theorems do not reveal the behaviors pertinent to near optimality, and our numerical experiments next would cover this investigation.

We close this section with a discussion on the biases of $U_{n,k}$ and $V_{n,k}$. We have the following result:
\begin{theorem}[Bias]
If $E[\sup_{x\in\mathcal{X}}\lvert h(x,\xi) \rvert]<\infty$, then $U_{n,k}$ is unbiased in estimating $W_k$, i.e., $E[U_{n,k}]=W_k$, whereas $V_{n,k}$ is downward biased, i.e., $E[V_{n,k}]\leq W_k$. If Assumptions \ref{L2} and \ref{L2 strengthened} further hold, and the resample size $k=O(n^{\gamma})$ for some $\gamma < \frac{1}{2}$,
% Under the same assumptions and resample sizes as Theorems \ref{main thm} and \ref{V thm}
then we have $W_k - E[V_{n,k}]= O((k^2/n)^l+k/n)$ as $n\to\infty$, where $l$ is any fixed positive integer.\label{thm bias}
\end{theorem}

The zero-bias property of $U_{n,k}$ is trivial: Each summand in its definition is an SAA value with distinct i.i.d. data, and thus has mean exactly $W_k$. On the other hand, the summands in $V_{n,k}$ are SAA values constructed from potentially repeated observations, which induces bias relative to $W_k$. The proof of the latter requires the development of a generalized monotonicity result on the optimistic bound for showing downward biasedness and the relation \eqref{interim24} for bounding the bias, and is left to Appendix \ref{proof:se}.

From Theorem \ref{thm bias}, we see that on average $U_{n,k}$ is a tighter bound than $V_{n,k}$ due to the downward biasedness of $V_{n,k}$. When $k$ is fixed, such an advantage for $U_{n,k}$ is relatively mild, since the bias of $V_{n,k}$ in estimating the optimistic bound $W_k$ is of order $1/n$. However, as $k$ grows, this advantage becomes more significant, and the bias of $V_{n,k}$ can be arbitrarily close to $O(1)$ (when $k\approx\sqrt n$).

Theorems \ref{main thm}, \ref{compare var} and \ref{thm bias} together justify that the bagging bound $U_{n,k}$ without replacement is more advantageous in terms of both standard error and bias. However, in practice the bound $V_{n,k}$ with replacement is similarly tight as $U_{n,k}$ and at the same time possesses a more robust coverage performance as our experiments show in Section \ref{sec:numerics}, and hence is the more recommendable choice for our bagging procedure.

Lastly, we should mention that the biases depicted in Theorem \ref{thm bias} concern the estimators of $W_k$, but do not capture the discrepancy between $W_k$ and $Z^*$. The latter quantity is of separate interest. As discussed at the end of Section \ref{sec:SAA}, it can be reduced by existing methods like the jackknife or probability metric minimization (\cite{partani2006jackknife,stockbridge2013probability}).

\section{Numerical Experiments}\label{sec:numerics}
In this section we provide numerical tests to demonstrate the validity of our bagging procedures with and without replacement, called ``BagV'' and ``BagU'' respectively, and compare them to four existing methods:
\begin{itemize}
    \item BatchP: The batching procedure given in \eqref{CI batching}, also known as the multiple-replication procedure.
    \item SRP: The single-replication procedure given in \eqref{CI SAA}.
    \item A2RP: The averaged two-replication procedure from \cite{bayraksan2006assessing}. Given a data set of size $n$, A2RP equally splits the data into two portions and computes a lower confidence bound for the optimal value in the form of $(\hat Z_1+\hat Z_2)/2 - z_{1-\alpha}\sqrt{(\hat{\sigma}_1^2+\hat{\sigma}_2^2)/2}/\sqrt{n}$, where $\hat Z_i,i=1,2$ is the optimal value of the SAA formed by the $i$-th portion only, and $\hat{\sigma}_i^2,i=1,2$ is the sample variance of the objective at an SAA optimal solution computed using the $i$-th portion only.
    \item I2RP: The independent two-replication procedure from \cite{bayraksan2006assessing}. Like A2RP, I2RP also equally splits the data, but computes the lower confidence bound as $\hat Z_1 - z_{1-\alpha}\hat{\sigma}_2/\sqrt{n/2}$ with the point estimate and the standard error estimate computed using different portions of the data.
\end{itemize}
A2RP and I2RP are proposed to improve the finite-sample performance of SRP by reducing the correlation between the point estimate and the standard error estimate. Note that SRP, I2RP and A2RP are originally designed to bound optimality gaps of given solutions, but can as well be used to bound optimal values with straightforward modifications. Also, when referring to direct-CLT bounds we include all of SRP, A2RP and I2RP, as they are based on the form of the CLT-induced confidence bound \eqref{CI SAA}.

Four stochastic optimization problems are tested. The first problem is a portfolio optimization problem that minimizes the $95\%$-level CVaR risk measure of a portfolio subject to the expected return exceeding a target level, described as
%\begin{mini*}|s|
%	{c, x}{c + \frac{1}{\alpha}\mathbb{E}(-\xi^Tx - c)_+}
%	{}{}
%	\addConstraint{\mathbb{E}\xi^Tx \geq b}
%	\addConstraint{x_i \geq 0}
%	\addConstraint{\sum_{i}x_i = 1}{}
%\end{mini*}
\begin{equation}\label{min_cvar}
\begin{aligned}
&\min_{c, x}&&c + \frac{1}{0.05}E[(-\xi^Tx - c)_+]\\
&\text{ s.t.}&&\mu^Tx \geq 3\\
&&&\sum_{i=1}^5x_i = 1\\
&&&x_i \geq 0\text{ for }i=1,\ldots,5
\end{aligned}
\end{equation}
where $\xi=(\xi_1,\ldots,\xi_5)^T$ is the vector of random returns of five different assets, $x=(x_1,\ldots,x_5)^T$ are the holding proportions of the assets, and the target return level is $3$. In particular, $\xi$ follows a multivariate normal $N(\mu, \Sigma)$ where the mean $\mu = (1, 2, 3, 4,5)^T$ is known and the covariance $\Sigma$ is randomly generated. Note that the cost function here is Lipschitz continuous, and the optimal solution is unique. Therefore we expect all the methods to perform well for this problem.

% Note that, to avoid feasibility complications that divert our focus, in \eqref{min_cvar} we assume knowledge of the expected return $\mu$ so the constraint becomes $\mu^Tx\geq b$.
% As a side note, one issue in applying our method to CVaR risk minimization is that evaluation of the expected return constraint requires the (unknown) mean of $\xi$ whereas in all existing methods the decision space is supposed to be known. To make the methods applicable

To describe the second problem, suppose there are ten different items labeled as $\#1$ through $\#10$ each of which incurs a random loss $\xi_i$, and the decision-maker is required to pick at least one out of the ten items and at most two items among $\#7,\#8,\#9,\#10$ in such a way that the total expected loss is minimized. Mathematically, the problem can be formulated as the following stochastic integer program
\begin{equation}\label{IP}
\begin{aligned}
&\min_{x}&&E[\xi^Tx]\\
&\text{ s.t.}&&Ax\leq b\\
&&&x_i\in\{0,1\}\text{ for }i=1,2,\ldots,10
\end{aligned}
\end{equation}
where $\xi$ follows $N(\mu,\Sigma)$ with mean $\mu=(-1,-7/9,-5/9,\ldots,7/9,1)^T\in \R^{10}$ and covariance $\Sigma$ randomly generated, $b=(-1,2)^T$ and
% $A=(-e_1,e_2)^T$ with $e=(0,0,0,0,0,0,1,1,1,1)^T$.
\begin{equation*}
   A=\left[\begin{matrix} % or pmatrix or bmatrix or Bmatrix or ...
      -1 &-1&-1&-1&-1&-1&-1&-1&-1&-1  \\
      0&0 &0&0&0&0&1&1&1&1 \\
   \end{matrix}\right].
\end{equation*}
It is straightforward to see that picking the items with negative expected losses, i.e., $\#1$ through $\#5$, gives the minimum total loss, hence the unique optimal solution is $x^*=(1,1,1,1\allowbreak,1,0,0,0,0,0)^T$.
% Because of the integrality requirement the single-replication procedure is not theoretically justified and can exhibit incorrect coverage.
% When implementing the methods
We solve the SAA by a direct enumeration (feasible thanks to the relatively low dimensionality).

The third optimization problem is the following simple stochastic linear program
%\begin{align}\label{binary}
%h(x,\xi)=
%\begin{cases}
%0&\text{ if }x=-1\\
%-0.1+4\xi&\text{ if }x=1\\
%\end{cases}
%\end{align}
\begin{equation}\label{binary}
\begin{aligned}
&\min_{x}&&E[-0.05 x+(3-2x)\xi]\\
&\text{ s.t.}&&-1\leq x\leq 1\\
\end{aligned}
\end{equation}
where the uncertain quantity $\xi$ is a standard normal and the decision $x$ is a scalar. It is clear that the optimal solution is $x^*=1$. This problem, as well as the stochastic integer program, serves to highlight that past methods may give subpar finite-sample performances due to a delicate interplay between the objective variance and the jumping behavior of the estimated solution. It then illustrates how bagging can be a resolution in such a scenario.

The last problem we consider is another stochastic linear optimization over a probability simplex
\begin{equation}\label{simplex}
\begin{aligned}
&\min_{x}&&E[\xi^Tx]\\
&\text{ s.t.}&&\sum_{i=1}^{10}x_i=1\\
&&&x_i\geq 0\text{ for }i=1,2,\ldots,10\\
\end{aligned}
\end{equation}
where $\xi_i,i=1,2,\ldots,10$ are independent normal variables with $\xi_i \sim N(0,1)$ for $i\leq 5$ and $\xi_i \sim N(0.1,1)$ for $i \geq 6$. Every feasible solution such that $x_i=0$ for all $i\geq 6$ is optimal. This example with multiple optimal solutions serves to demonstrate the advantage of our bagging procedures in variance reduction.

% such as the single-replication procedure exhibits rather bad finite-sample performance.

\subsection{Practical Algorithmic Configurations}\label{sec:optimal configurations}
We first provide some practical guidelines on the algorithmic configurations for our bagging procedure. More precisely, we study two elements: The bootstrap size $B$, and the bias correction to the IJ variance estimator.

 % required for ensuring a satisfactory performance and how it's affected by

\begin{figure}[h]
    \centering
    \begin{subfigure}{0.49\textwidth}
         \centering
         \includegraphics[width=\textwidth]{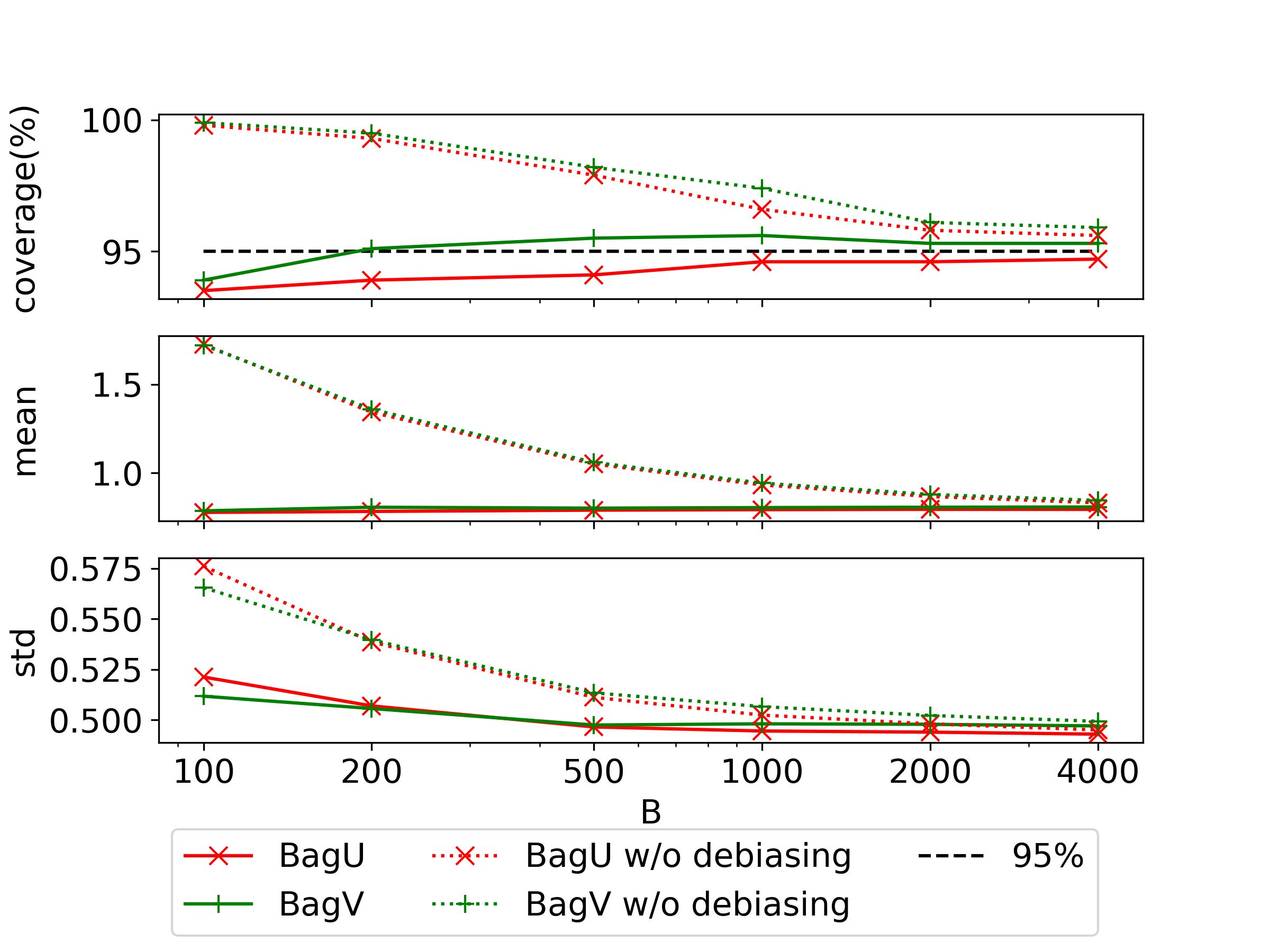}
         \caption{Integer program \eqref{IP}.}
         \label{fig:integer optval varying B}
     \end{subfigure}
     \hfill
    \begin{subfigure}{0.49\textwidth}
         \centering
         \includegraphics[width=\textwidth]{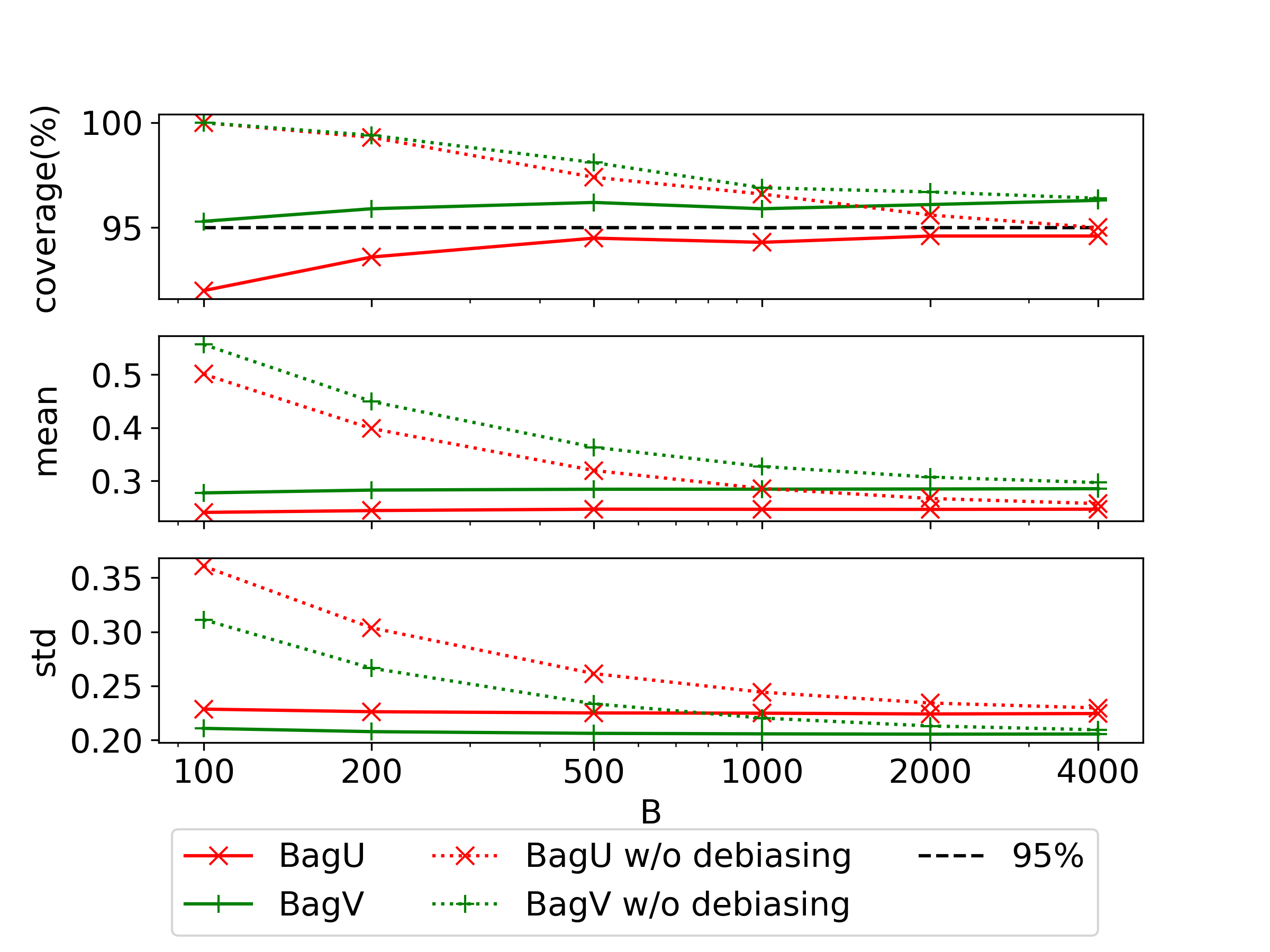}
         \caption{Simple linear program \eqref{binary}.}
         \label{fig:binary optval varying B}
     \end{subfigure}
    \caption{Comparison of Algorithm \ref{bagging} and its debiased variant Algorithm \ref{debiased bagging} under data size $n=400$ and varying bootstrap sizes.}
    \label{fig:opt val varying B}
\end{figure}

\begin{figure}[h]
    \centering
     \begin{subfigure}{0.49\textwidth}
         \centering
         \includegraphics[width=\textwidth]{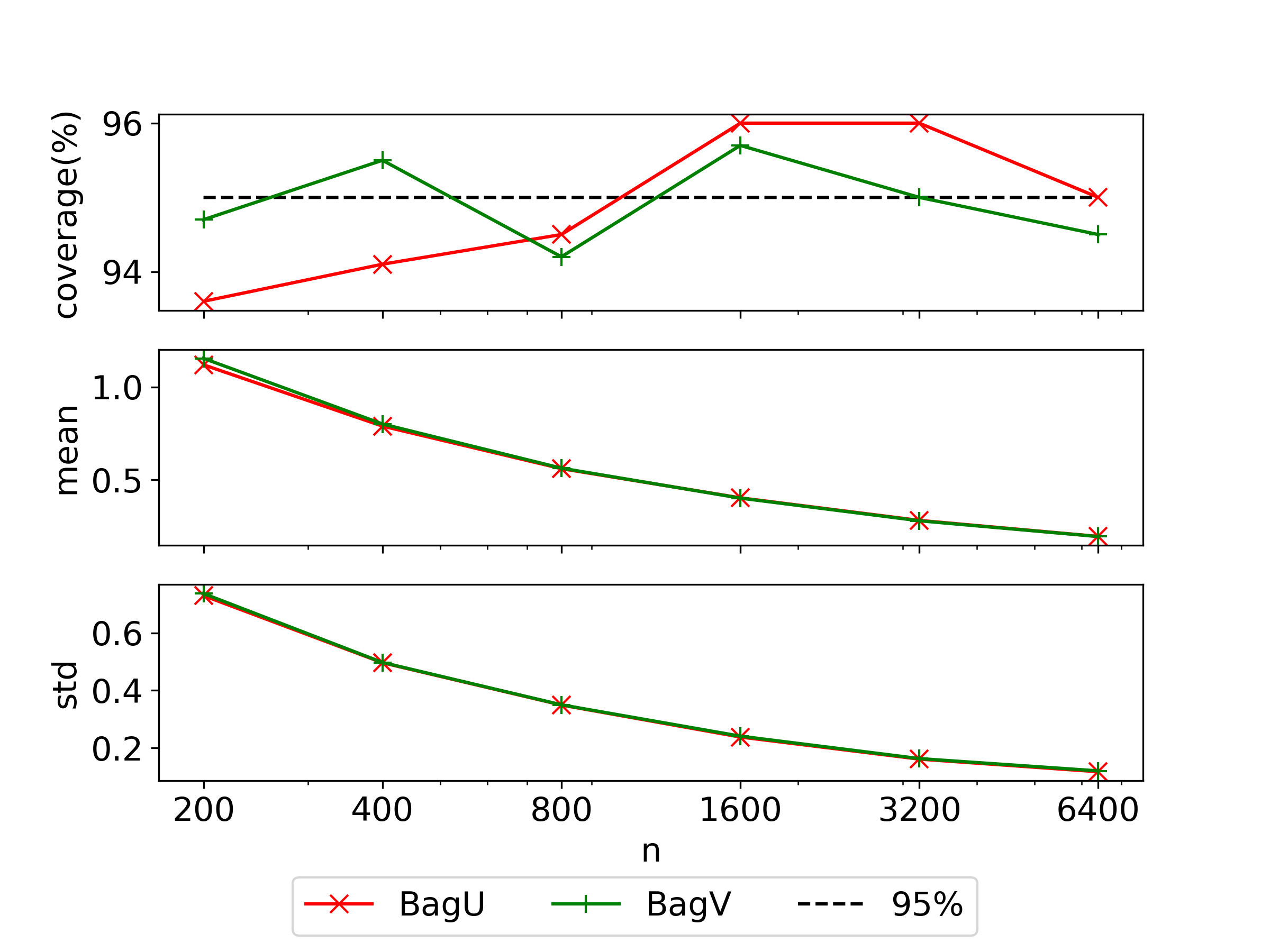}
         \caption{Integer program \eqref{IP}.}
         \label{fig:integer optval varying N}
     \end{subfigure}
     \hfill
         \begin{subfigure}{0.49\textwidth}
         \centering
         \includegraphics[width=\textwidth]{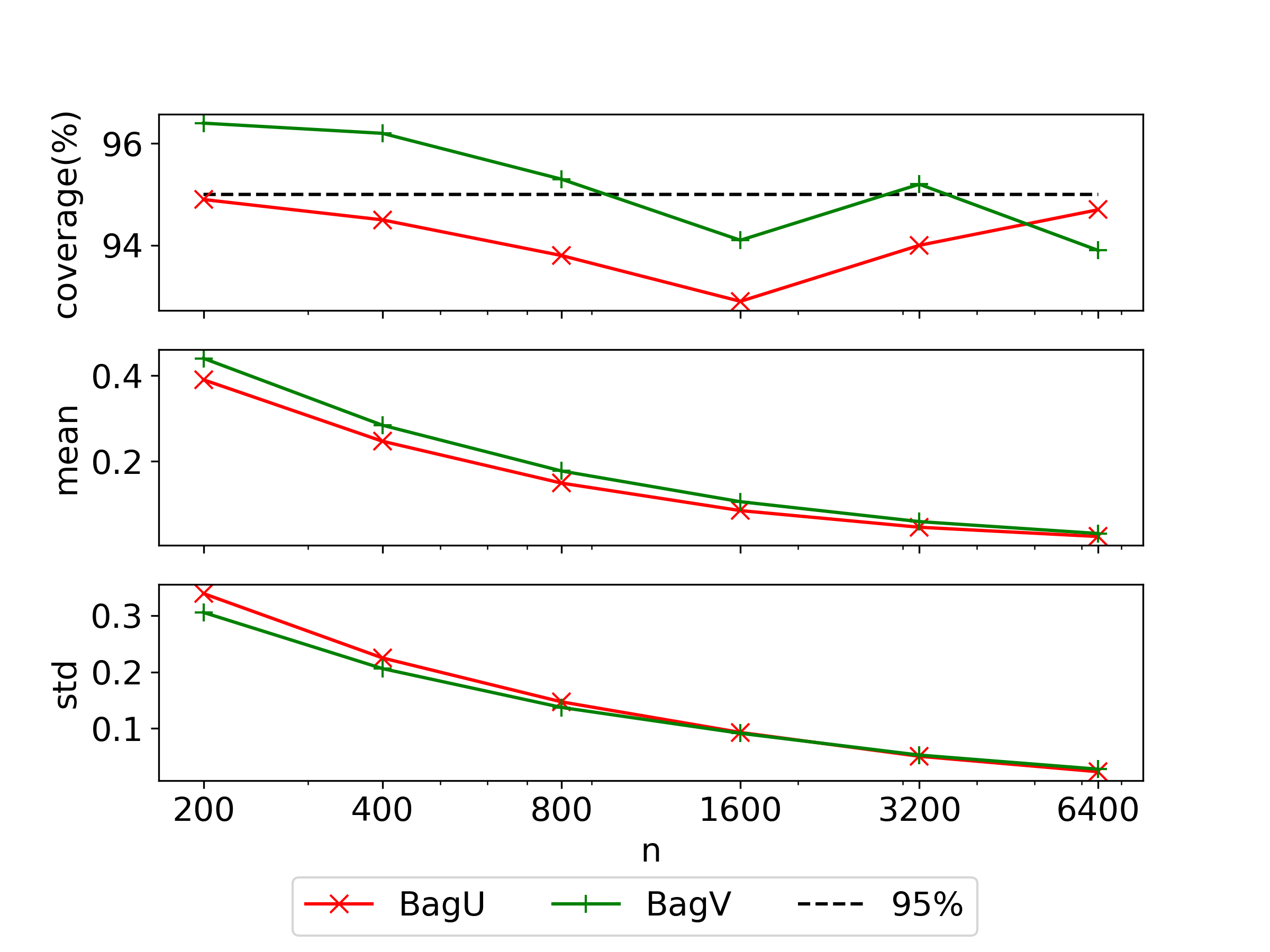}
         \caption{Simple linear program \eqref{binary}.}
         \label{fig:binary optval varying N}
     \end{subfigure}
    \caption{Bounds of optimal values with fixed $B=500$ and growing data sizes.}
    \label{fig:opt val varying N}
\end{figure}

We simulate an i.i.d.~data set $\xi_1,\ldots,\xi_n$ of size $n$, and use Algorithm \ref{bagging} and its debiased variant Algorithm \ref{debiased bagging} to compute $95\%$-level lower confidence bounds for the optimal value $Z^*=\min_{x\in\mathcal X}Z(x)$. To highlight the algorithmic difference, Algorithm \ref{debiased bagging} further subtracts from the IJ variance estimator $\tilde{\sigma}_{IJ}^2$ a correction term $k/B^2 \cdot \sum_{b=1}^B(\hat Z_k^b-\tilde Z_{n,k}^{bag})^2$ for resampling with replacement, or $kn/(B^2(n-k))\cdot \sum_{b=1}^B(\hat Z_k^b-\tilde Z_{n,k}^{bag})^2$ for resampling without replacement, in order to remove the bias. We compare the performances of the two algorithms when the data size is fixed at $n=400$, the resample size is fixed at $k=280$, and the bootstrap size $B$ varies from a small size $100$ to $4000$, a sufficiently large size relative to $n$ as required by Theorem \ref{final_guarantee} for Algorithm \ref{bagging}.

Figure \ref{fig:opt val varying B} shows the results on problems \eqref{IP} and \eqref{binary}. Each plot in Figure \ref{fig:opt val varying B} consists of three panels, where the top panel shows the estimated coverage probabilities of the constructed bounds based on $1000$ independent runs and contains a dashed horizon line at the nominal level $95\%$ for benchmarking the coverage performance, the middle panel shows the mean of the bounds after being offset by the true optimal value (hence smaller values mean tighter bounds), and the bottom panel shows the standard deviation of the bounds across the $1000$ runs. The legends ``BagV w/o debiasing'' and ``BagU w/o debiasing'' refer to Algorithm \ref{bagging} with and without replacement respectively, whereas ``BagV'' and ``BagU'' refer to the counterparts of the debiased variant.

The results clearly show that both methods deliver bounds with correct coverage probabilities that are close to or higher than $95\%$ when the bootstrap size $B$ is relatively large (e.g., above $500$), and all the three metrics gradually converge as the Monte Carlo error diminishes when $B$ increases. However, the debiased variant, for both with and without replacement, seems to outperform Algorithm \ref{bagging} in several ways under small and moderate bootstrap sizes. Firstly, the bounds generated with bias correction are tighter and less variable as evidenced by their smaller mean and standard deviation, and at the same time have less conservative and yet still accurate coverage probabilities around $95\%$. This is consistent with the fact that the IJ variance estimator in Algorithm \ref{bagging} is upward biased and may overestimate the true variance. Secondly, the performance of the debiased version is much less sensitive to the bootstrap size in the sense that the coverage, mean, and standard deviation of the bounds does not vary as drastically as Algorithm \ref{bagging} under different bootstrap sizes. Based on the these observations, we recommend the debiased variant for general use of our bagging procedure. Comparing with and without replacement, we see that BagV has slightly higher coverages than BagU on both problems and generates slightly looser but less variable bounds for the linear program \eqref{binary}.

% that bagging without debiasing generally produces more loose bounds than that with debiasing especially when the bootstrap size $B$ is relatively small.

For the choice of the bootstrap size, Figure \ref{fig:opt val varying B} suggests that $B=500$ is large enough for the debiased variant as further increasing it does not result in higher coverage accuracy or more stable bounds. To consolidate this choice, we test our debiased bagging procedure against increasing data sizes. Specifically, we fix $B=500$, increase the data size from $200$ to $6400$, a much larger size than $500$, and use a resample size $k=0.7n$. Figure \ref{fig:opt val varying N} summarizes the results on the same two problems. Although in theory the bootstrap size $B$ is required to grow with the data size, the results show that for fixed $B=500$ the coverage probabilities are consistently close to $95\%$ across all data sizes, and that the bounds get tighter and more stable as the data size grows, which is in accordance with the tighter optimistic bound and the smaller variance of the bagging estimate. All these demonstrate that the required bootstrap size of our debiased variant depends lightly on the data size and that using $B=500$ delivers satisfactory performances for data sizes of high thousands. Throughout the rest of our experiments, the debiased variant with a fixed bootstrap size of $500$ will be used for our bagging procedures. Although not presented in the following subsections, our Algorithm \ref{bagging} is found to perform similarly as the debiased variant if a larger bootstrap size that grows proportionally with the data size is used, e.g., $B=10n$. This can also be seen from Figure \ref{fig:opt val varying B} where the results of the two procedures, whether with or without replacement, closely match under large bootstrap sizes. Lastly, the size $B=500$ admittedly could still be expensive for some problems. Further computation reduction is potentially achievable using recent cheap bootstrap approaches (\cite{lam2022,lam2022cheap}); we leave the full investigation in this direction to future work.

\subsection{Lower Bounds of Optimal Values}\label{sec:numerics optimal}
In this subsection we compare our bagging procedures with four other existing methods in computing $95\%$-level lower confidence bounds for optimal values.
% . This is in accordance with our focus on statistical efficiencies, under the presumed adequate resources in solving SAA problems.
For BatchP we use the critical value of $t$-distribution with $m-1$ degrees of freedom when the number of batches $m<30$, so as to enhance finite-sample performances as suggested in \cite{mak1999monte}, whereas in other methods the normal critical value is used in the standard error term of the bound.

%The four methods are We vary the batch size of the batching procedure and the resample size of Algorithm \ref{bagging}, both denoted by $k$, to investigate how they affect the performances.

\begin{figure}[h]
    \centering
    \begin{subfigure}{0.49\textwidth}
         \centering
         \includegraphics[width=\textwidth]{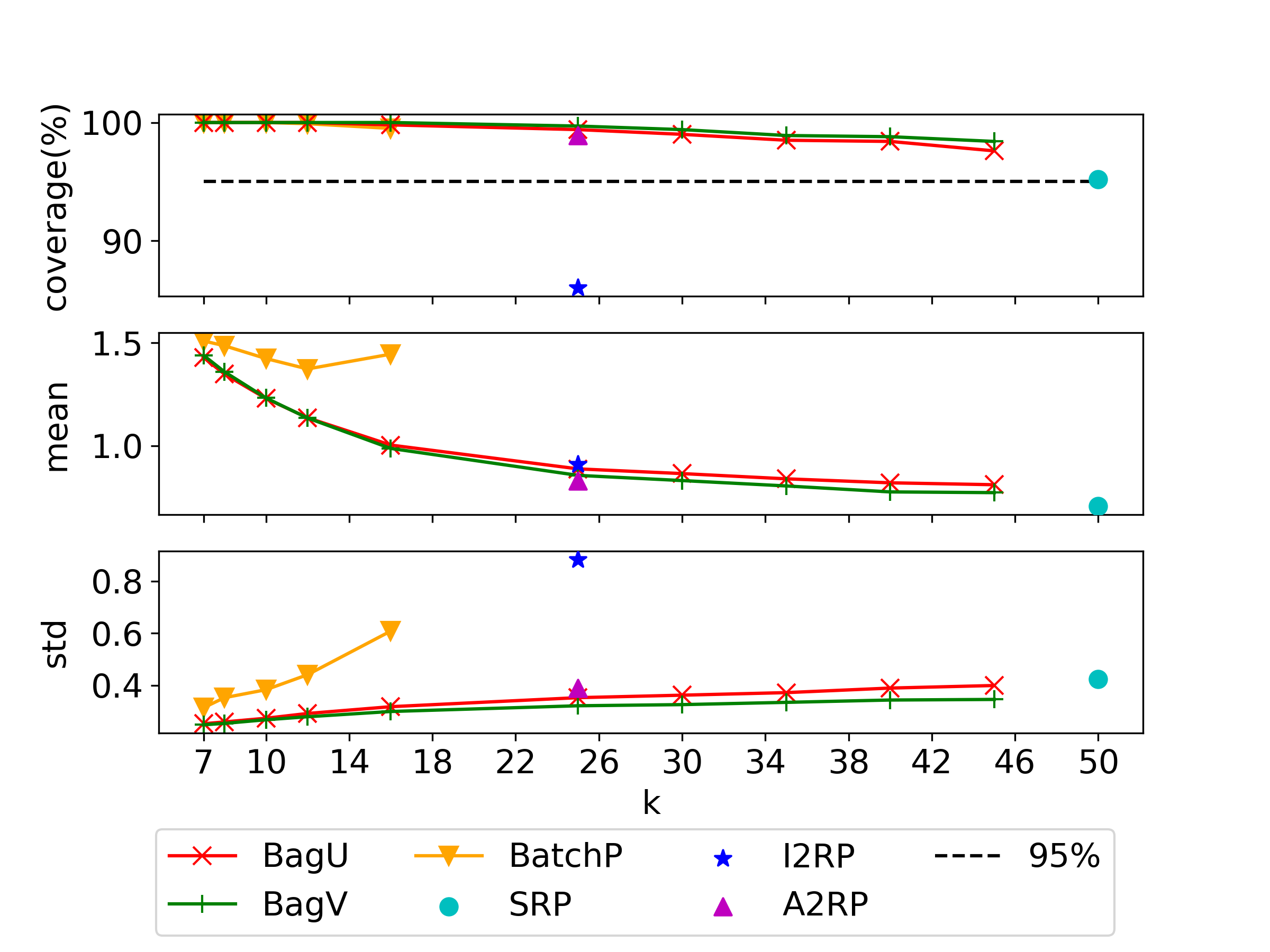}
         \caption{$n=50$}
         \label{fig:portfolio optval 50}
     \end{subfigure}
     \hfill
     \begin{subfigure}{0.49\textwidth}
         \centering
         \includegraphics[width=\textwidth]{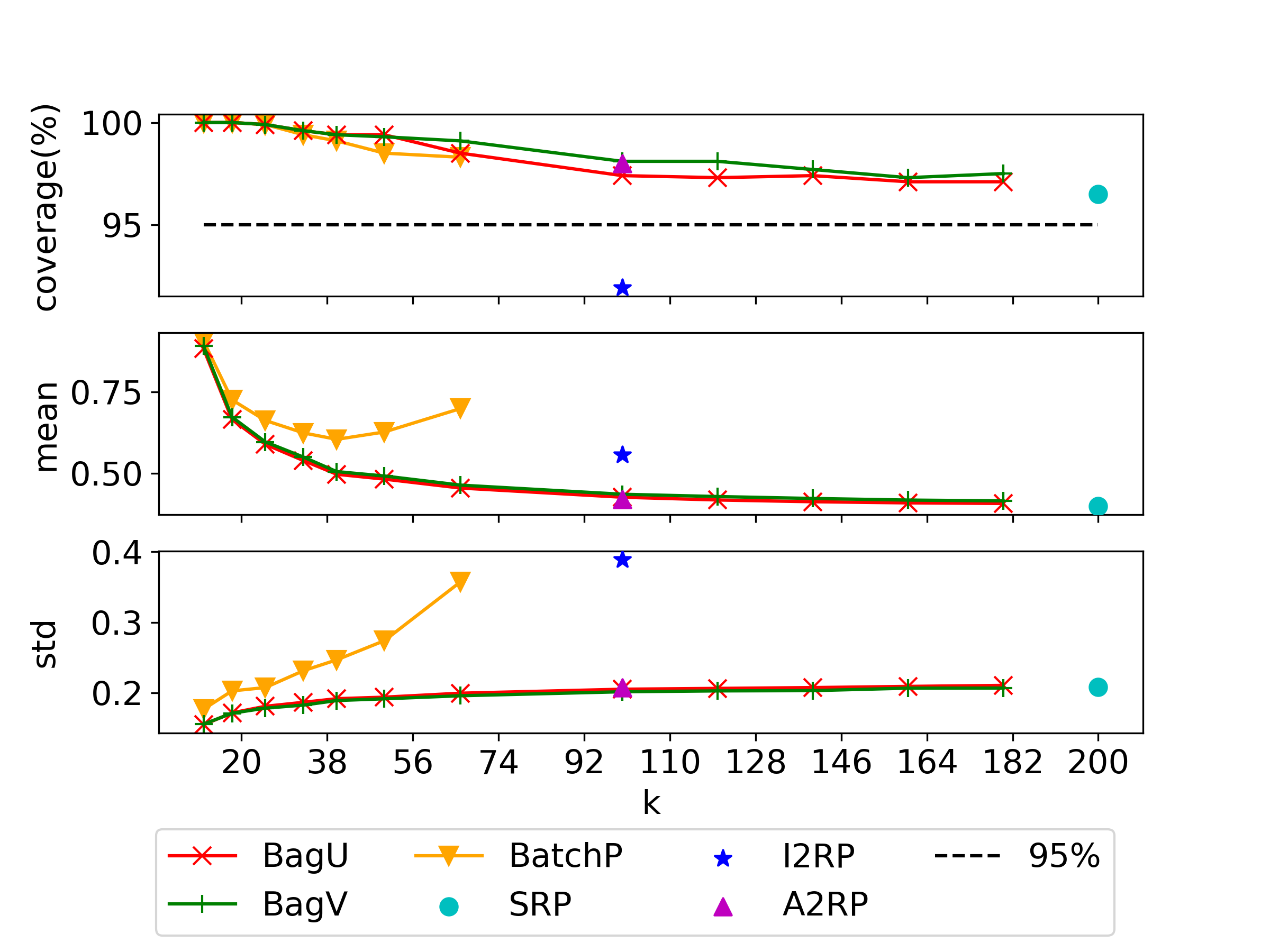}
         \caption{$n=200$}
         \label{fig:portfolio optval 200}
     \end{subfigure}
    \caption{Portfolio problem \eqref{min_cvar}. Bounds of the optimal value.}
    \label{fig:portfolio optval}
\end{figure}

\begin{figure}[h]
    \centering
    \begin{subfigure}{0.49\textwidth}
         \centering
         \includegraphics[width=\textwidth]{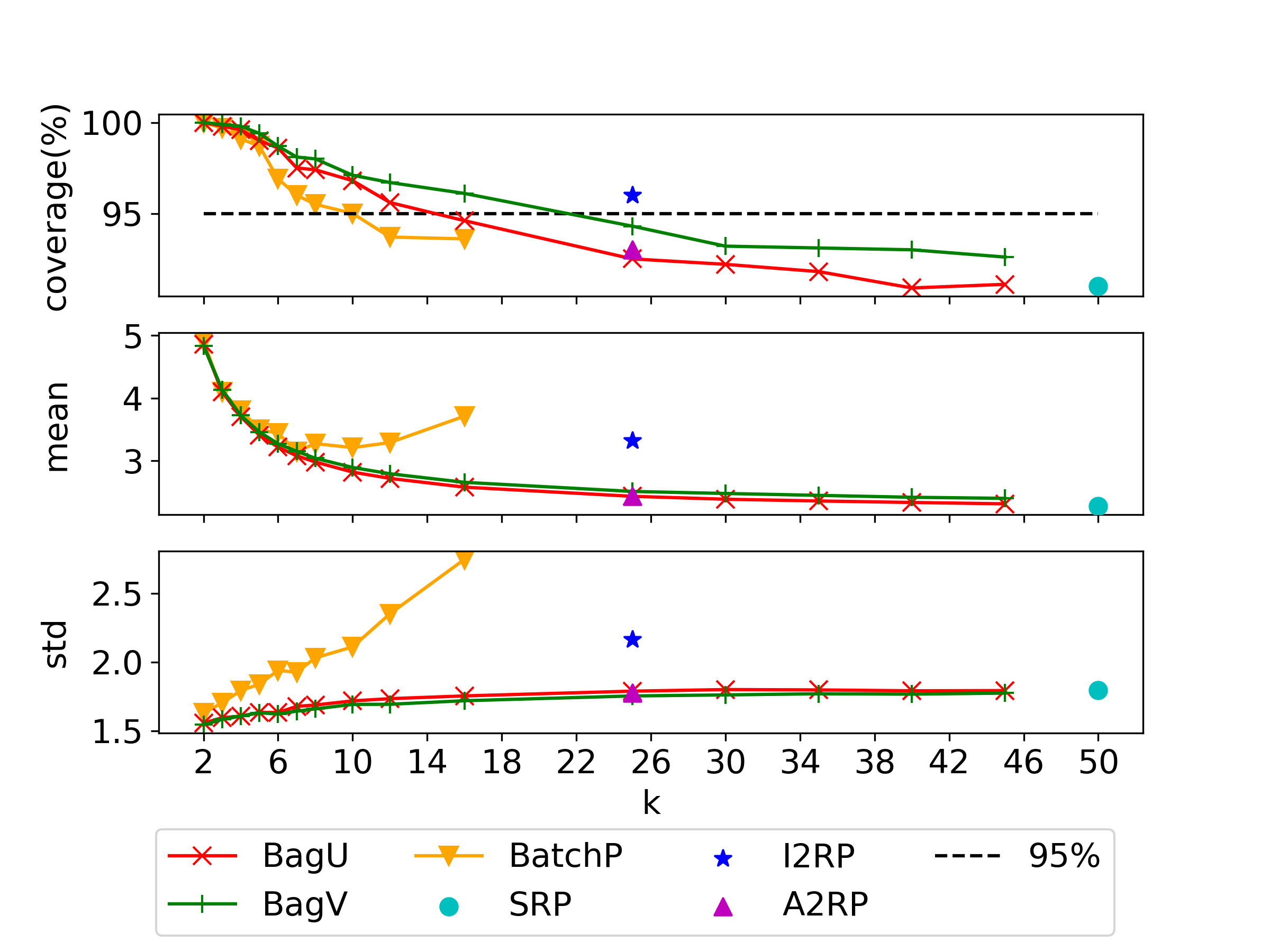}
         \caption{$n=50$}
         \label{fig:integer optval 50}
     \end{subfigure}
     \hfill
     \begin{subfigure}{0.49\textwidth}
         \centering
         \includegraphics[width=\textwidth]{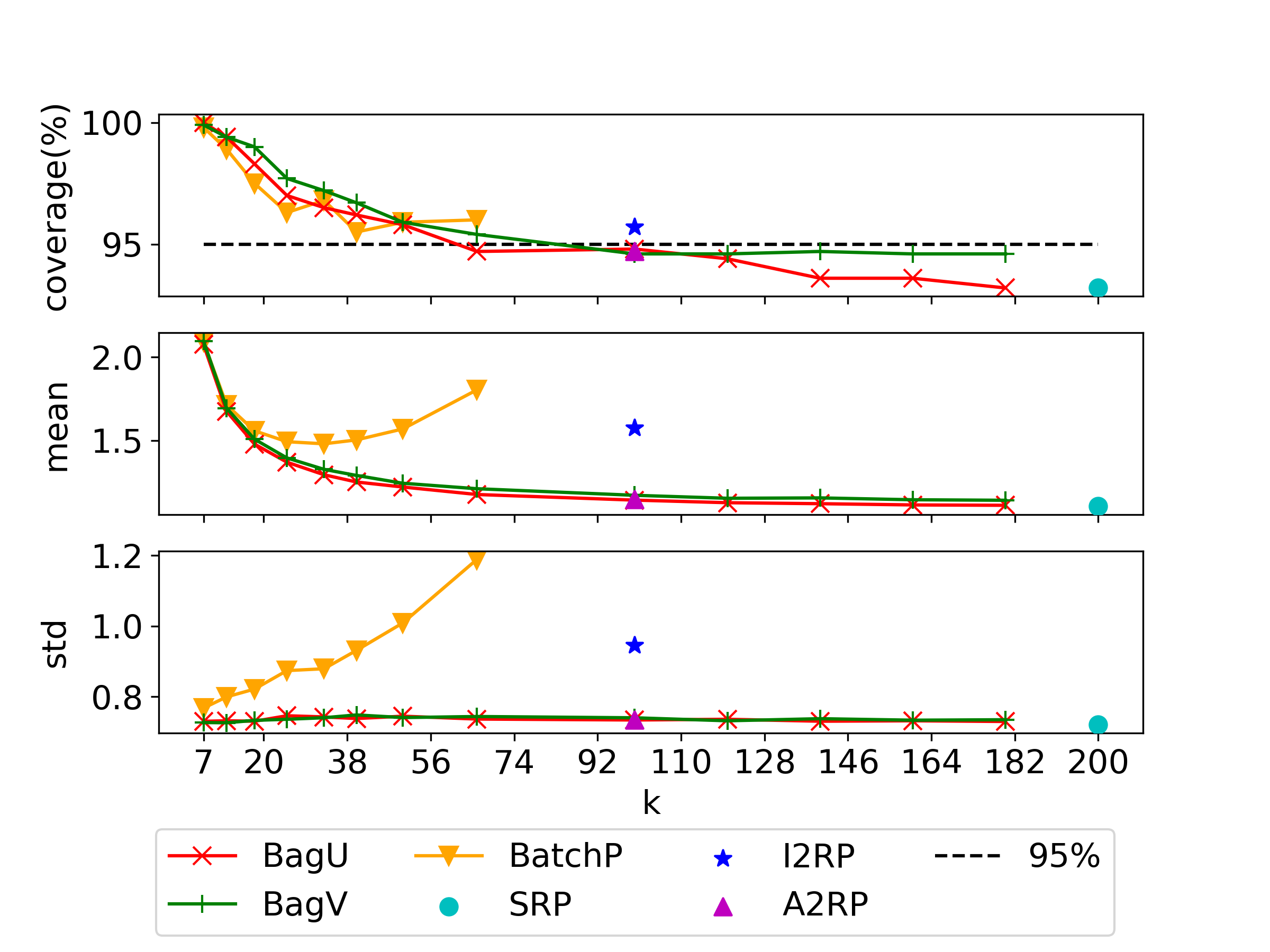}
         \caption{$n=200$}
         \label{fig:integer optval 200}
     \end{subfigure}
    \caption{Integer program \eqref{IP}. Bounds of the optimal value.}
    \label{fig:integer optval}
\end{figure}

\begin{figure}[h]
    \centering
    \begin{subfigure}{0.49\textwidth}
         \centering
         \includegraphics[width=\textwidth]{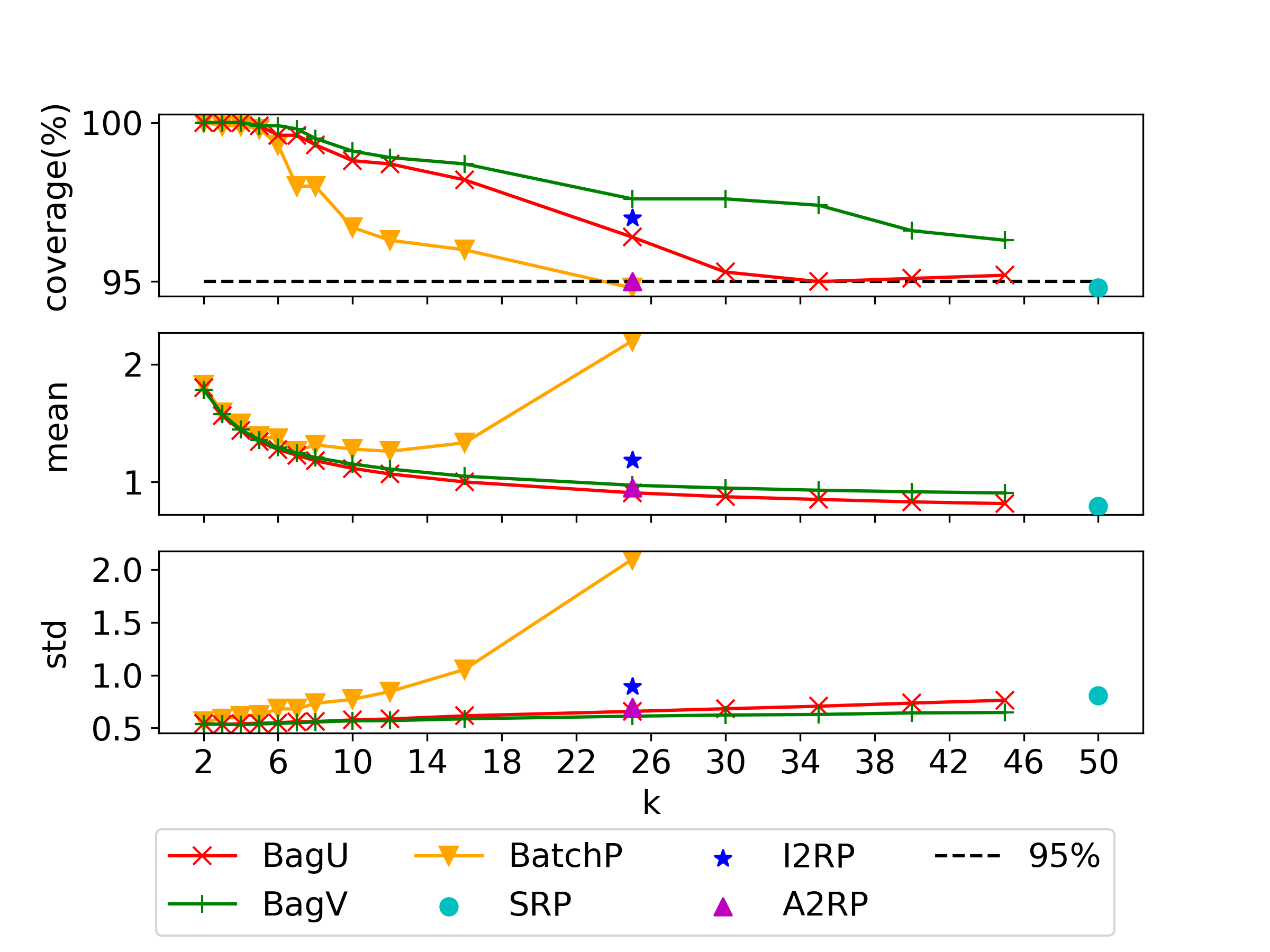}
         \caption{$n=50$}
         \label{fig:binary optval 50}
     \end{subfigure}
     \hfill
     \begin{subfigure}{0.49\textwidth}
         \centering
         \includegraphics[width=\textwidth]{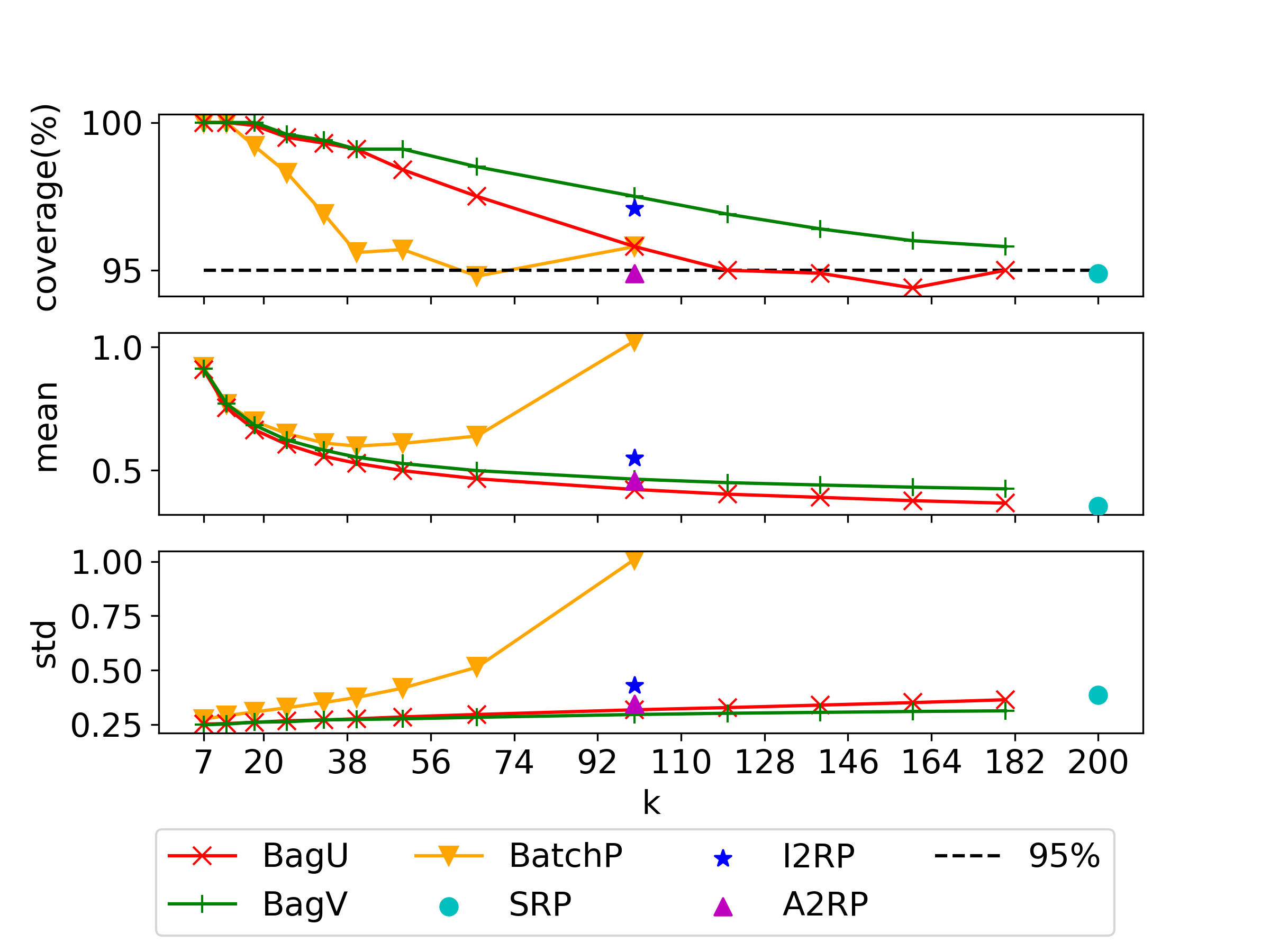}
         \caption{$n=200$}
         \label{fig:binary optval 200}
     \end{subfigure}
    \caption{Simple linear program \eqref{binary}. Bounds of the optimal value.}
    \label{fig:binary optval}
\end{figure}

%Note that for the batching procedure, the parameter $k$ refers to the batch size, and the number of batches $m$ is automatically set to $n/k$ once given $n$ and $k$.

Figures \ref{fig:portfolio optval}, \ref{fig:integer optval} and \ref{fig:binary optval} summarize the results for problems \eqref{min_cvar}, \eqref{IP} and \eqref{binary} respectively. Each figure contains two plots, one for data size $n=50$ and the other for $n=200$, and each plot shows the estimated coverage probabilities, the mean and standard deviation of the bound after being offset by the optimal value (hence smaller values mean tighter bounds) under different batch sizes for BatchP or resample sizes for BagU and BagV, both denoted by $k$. The SRP uses all the data to construct the SAA, hence its result corresponds to a point at $k=n$, whereas I2RP and A2RP use half of the data to construct SAAs and their results are plotted at $k=n/2$.

Figures \ref{fig:portfolio optval} and \ref{fig:binary optval} show that almost all the considered methods, over a wide range of resample sizes roughly from $2$ to $90\%$ of the data size for BagU and BagV, generate statistically valid bounds for problems \eqref{min_cvar} and \eqref{binary} in the sense that the coverage probabilities are equal to or above the nominal value $95\%$. The only exception is I2RP that undercovers for problem \eqref{min_cvar}. Correspondingly, the I2RP bounds also have a high variability in this example. For problem \eqref{IP}, Figure \ref{fig:integer optval} shows that BagU and BagV, as well as A2RP and SRP, undercover when the resample size is chosen large. However, BagV and A2RP undercover more mildly than BagU and SRP, e.g., the coverage is approximately $93\%$ for A2RP and BagV with $k$ close to $n$ and $91\%$ for BagU and SRP in Figure \ref{fig:integer optval 50} when $n=50$. This undercoverage is potentially due to the discrete nature of the integer program and gets improved as the data size grows from $50$ to $200$.

% However, with the same $k$, BagU and BagV have slightly more accurate coverage probabilities than BatchP,
% and similar coverages as A2RP and I2RP at $k=n/2$ and as SRP for $k$ close to $n$.
Besides coverage, the considered methods differ in tightness and stability. We observe that BagU and BagV consistently output tighter (measured by the mean of the bound offset by the optimal value) and more stable (measured by the standard deviation) bounds than BatchP when the batch size in BatchP and the resample size in bagging are set the same. This difference in tightness and stability becomes more noticeable as $k$ increases. Compared to direct-CLT bounds, our bagging bounds also appear tighter and more stable than I2RP (either when the resample size $k=n/2$ or $k$ is close to $n$) on all three problems, whereas A2RP and SRP bounds are similarly tight and stable as our bagging bounds in all the cases.

% These suggest that our bagging bound generates relatively tight bounds for the optimal value and in the meantime retains good statistical accuracy, by using a resample size $k$ that is close to the data size.

% Relatedly, the results of BagU and BagV across different values of $k$ verify the monotonicity between the resample size and tightness of the optimistic bound \eqref{optimistic}, i.e., as the resample size $k$ grows the bagging bounds get increasingly tighter and reaches the same level as SRP for $k$ close to $n$.
% These are in accordance with the benefit of variance reduction brought by bagging as illustrated by Theorems \ref{compare var:asymptotic} and \ref{compare var}
% In each of Tables \ref{cvar_50}-\ref{binary_lower}, under the same batch size or resample size $k$, the bounds given by bagging (with or without replacement) are always larger in terms of the mean, and meanwhile less variable as measured by the standard deviation, than those by batching.

% can be attributed manifests the statistical inefficiency of the batching procedure. one has to trade tightness of the bound for accuracy due to the intrinsic tradeoff, to which
Figures \ref{fig:portfolio optval}-\ref{fig:binary optval} also show the tradeoff between tightness and statistical accuracy in BatchP and how it is improved by bagging. The monotonicity relation between the batch size and the optimistic bound entails that the bound should become tighter as $k$ increases. However, the bound by BatchP gets tighter at first under relatively small batch sizes but then becomes looser instead as the size further increases. This non-monotonic behavior appears since, as the batch size gets large, too few batches are available for the procedure to maintain the desired coverage accuracy. To mitigate this issue, we resort to using $t$ critical value in place of the normal one which loosens the bound in exchange for better coverages. Such a tradeoff is significantly improved in our bagging procedures due to the many resampled SAAs, as evidenced by the monotonically improving tightness and accurate coverages of the bounds across a wide range of resample sizes.

To summarize, for bounding optimal values our bagging bounds are generally as competitive as A2RP and SRP and outperform BatchP and I2RP in terms of coverage, tightness and stability, whereas between the two bagging bounds BagV exhibits a more reliable coverage performance than BagU while performing similarly in tightness and stability. Our next experiment on bounding optimality gaps will further reveal the performance differences among bagging bounds, A2RP and SRP.

\subsection{Upper Bounds of Optimality Gaps}\label{sec:numerics optimality gap}
Now we test our methods in bounding optimality gaps of solutions. We first solve the SAA formed by $n_1$ data points $\xi_1,\ldots,\xi_{n_1}$ to obtain a solution $\hat x$, then generate $n_2$ independent data points $\xi_{n_1+1},...,\xi_{n_1+n_2}$. These (and possibly the first $n_1$ data points as well) are then used to compute an upper confidence bound for the optimality gap $\mathcal G(\hat x)=Z(\hat x)-Z^*$. For convenience we denote $n=n_1+n_2$ as the total sample size.

We consider two approaches to bounding the gap, one reusing the first $n_1$ data points, and the other not. The first approach is to use the Bonferroni Correction (BC). Specifically, we use the second group of $n_2$ data to compute $U=\bar h+z_{0.975}\hat\sigma/\sqrt{n_2}$ as a $97.5\%$ upper confidence bound of $Z(\hat x)$, where $\bar h,\hat\sigma^2$ are the sample mean and variance of $h(\hat x,\xi_{n_1+1}),\ldots,h(\hat x,\xi_{n})$, and compute a $97.5\%$ lower confidence bound $L$ of the true optimal value $Z^*$ using all the $n$ data as in the previous section. By BC we know
\begin{equation*}
P(U-L\geq Z(\hat x)-Z^*)\geq P(U\geq Z(\hat x))+P(L\leq Z^*)-1
\end{equation*}
so that if $P(U\geq Z(\hat x))$ and $P(L\leq Z^*)$ are both asymptotically at least $97.5\%$, then $U-L$ is an asymptotically valid $95\%$ confidence bound for the gap $\mathcal G(\hat x)$.

The second approach is a Common Random Numbers (CRN) variance-reduction technique proposed by \cite{mak1999monte}. This approach generates upper bounds of the gap via computing lower bounds for the optimal value of the modified objective $E[h(x,\xi)-h(\hat x,\xi)]$ where $\hat x$ is viewed as fixed. Specifically, given $\hat x$ we use the second group of $n_2$ data to compute a $95\%$ lower confidence bound for this new objective, and then negate the lower bound to obtain a valid upper bound for $\mathcal G(\hat x)$.

\begin{figure}[h]
    \centering
    \begin{subfigure}{0.49\textwidth}
         \centering
         \includegraphics[width=\textwidth]{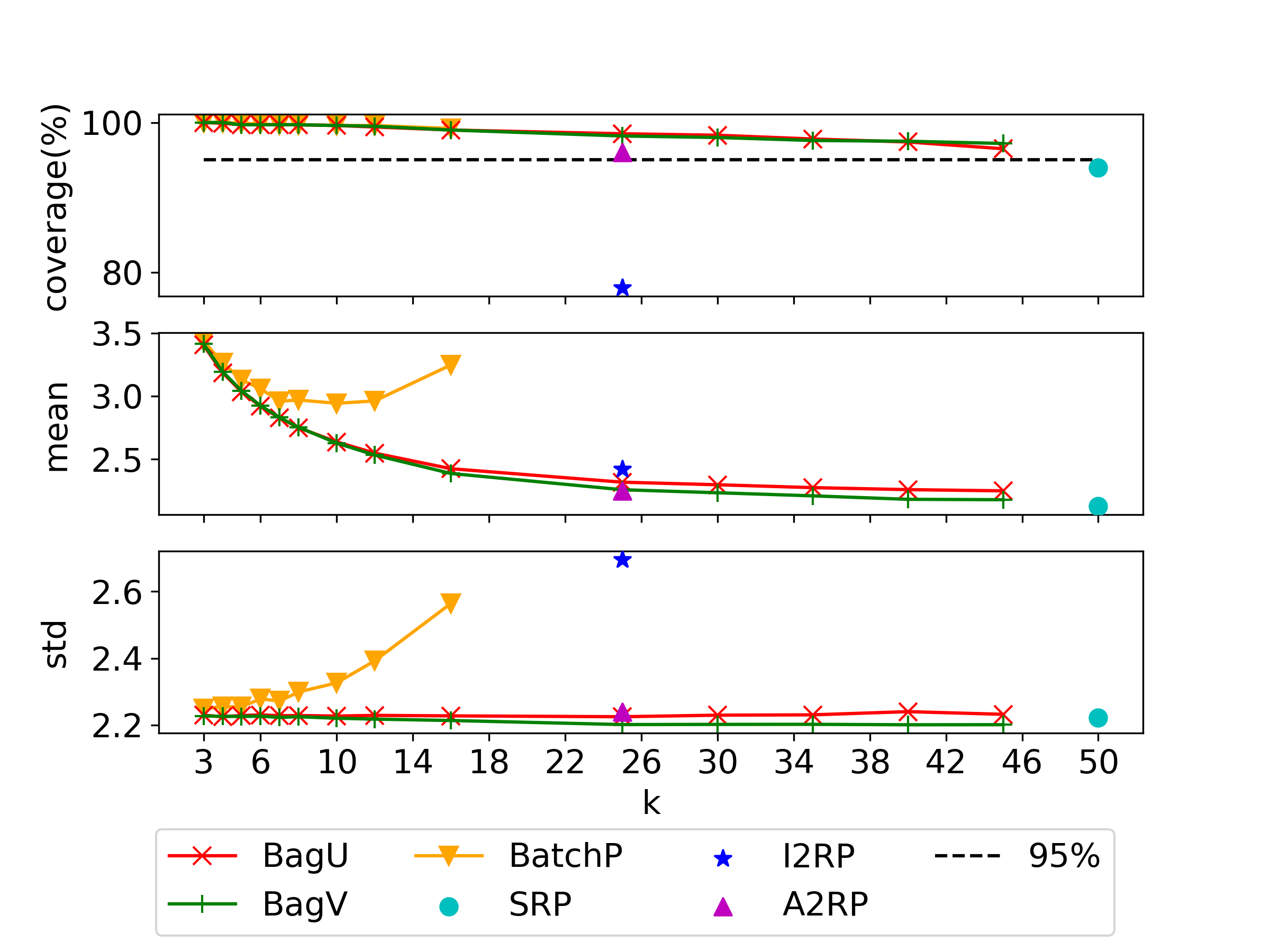}
         \caption{BC, $n=50, n_1=30, n_2=20$}
         \label{fig:portfolio gap BC 50}
     \end{subfigure}
     \hfill
    \begin{subfigure}{0.49\textwidth}
         \centering
         \includegraphics[width=\textwidth]{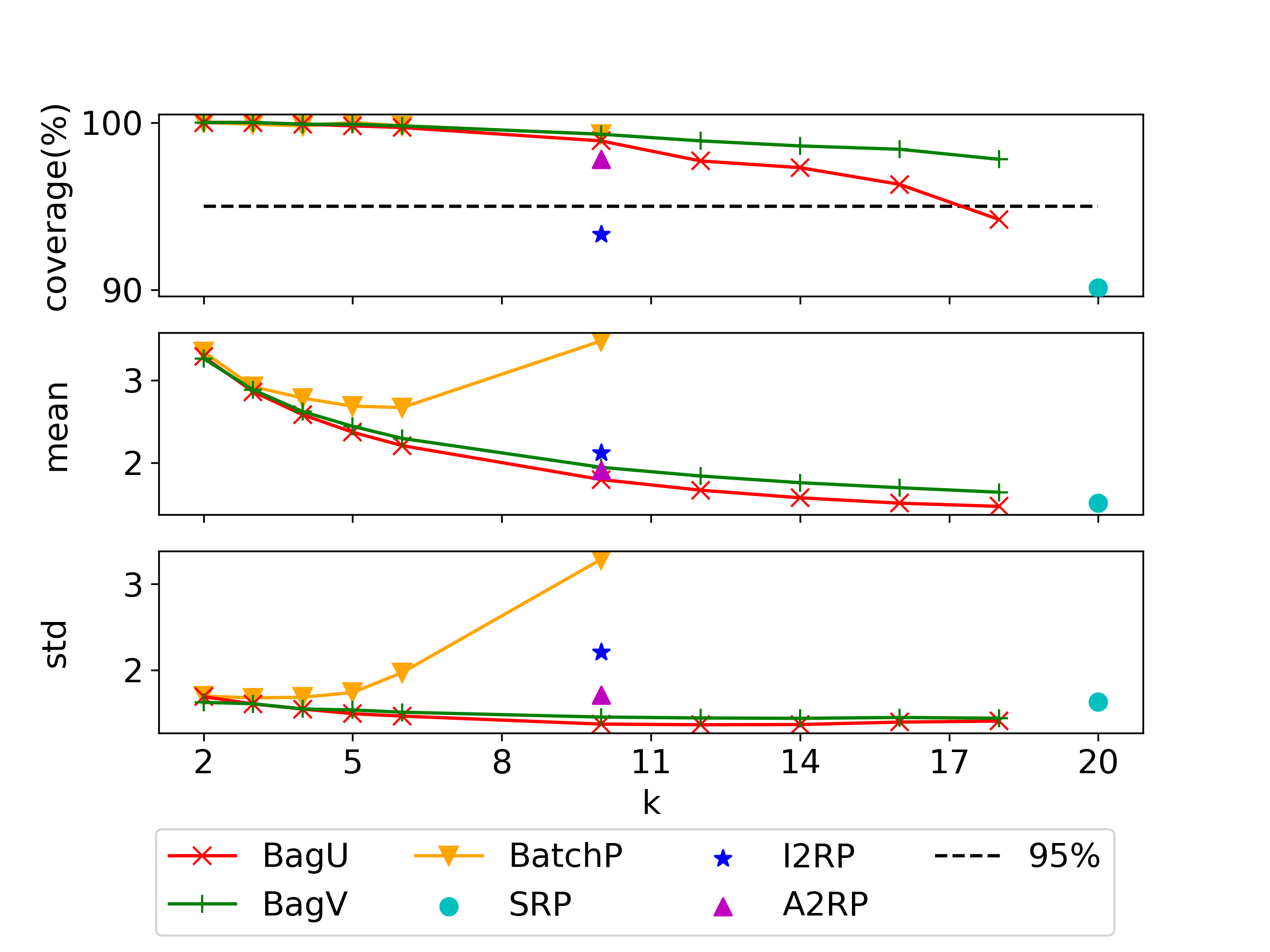}
         \caption{CRN, $n=50, n_1=30, n_2=20$}
         \label{fig:portfolio gap CRN 50}
    \end{subfigure}
    \caption{Portfolio problem \eqref{min_cvar}. Bounds of optimality gaps.}
    \label{fig:portfolio gap n=50}
\end{figure}

\begin{figure}[h]
    \centering
     \begin{subfigure}{0.49\textwidth}
         \centering
         \includegraphics[width=\textwidth]{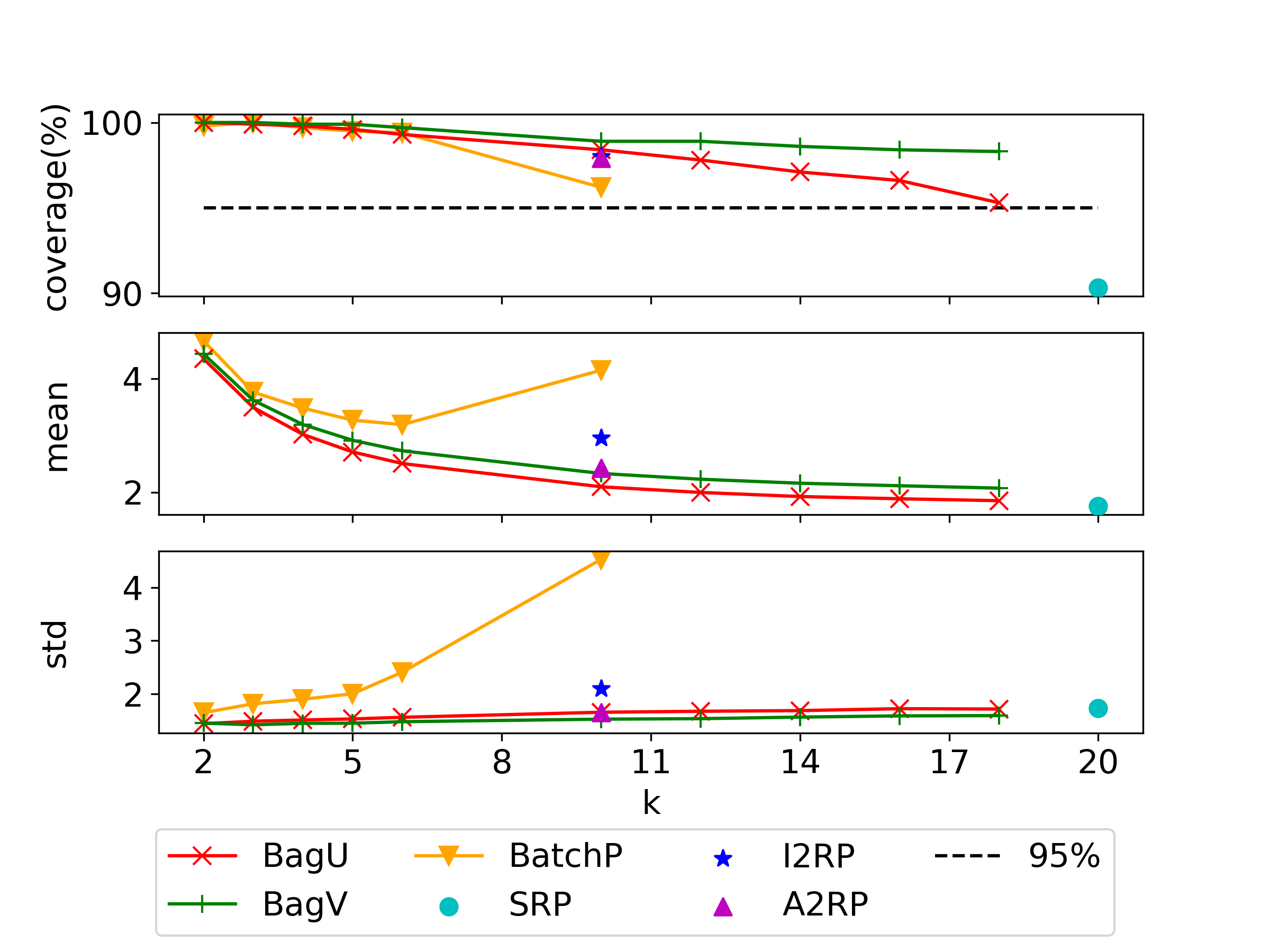}
         \caption{CRN, $n=50, n_1=30, n_2=20$}
         \label{fig:integer gap CRN 50}
     \end{subfigure}
     \hfill
    \begin{subfigure}{0.49\textwidth}
         \centering
         \includegraphics[width=\textwidth]{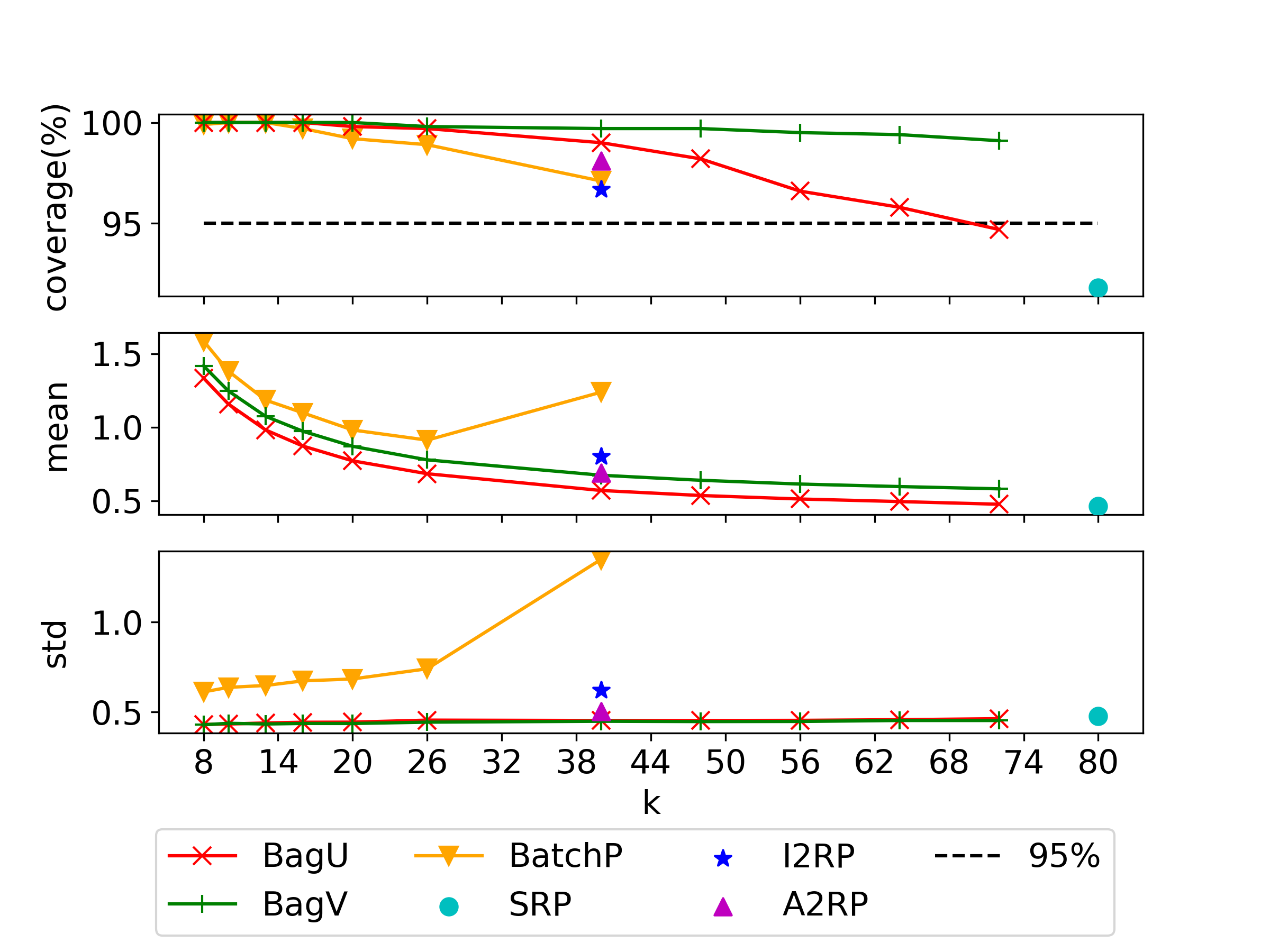}
         \caption{CRN, $n=200, n_1=120, n_2=80$}
         \label{fig:integer gap CRN 200}
     \end{subfigure}
    \caption{Integer problem \eqref{IP}. Bounds of optimality gaps.}
    \label{fig:integer gap CRN}
\end{figure}

\begin{figure}[h]
    \centering
    \begin{subfigure}{0.49\textwidth}
         \centering
         \includegraphics[width=\textwidth]{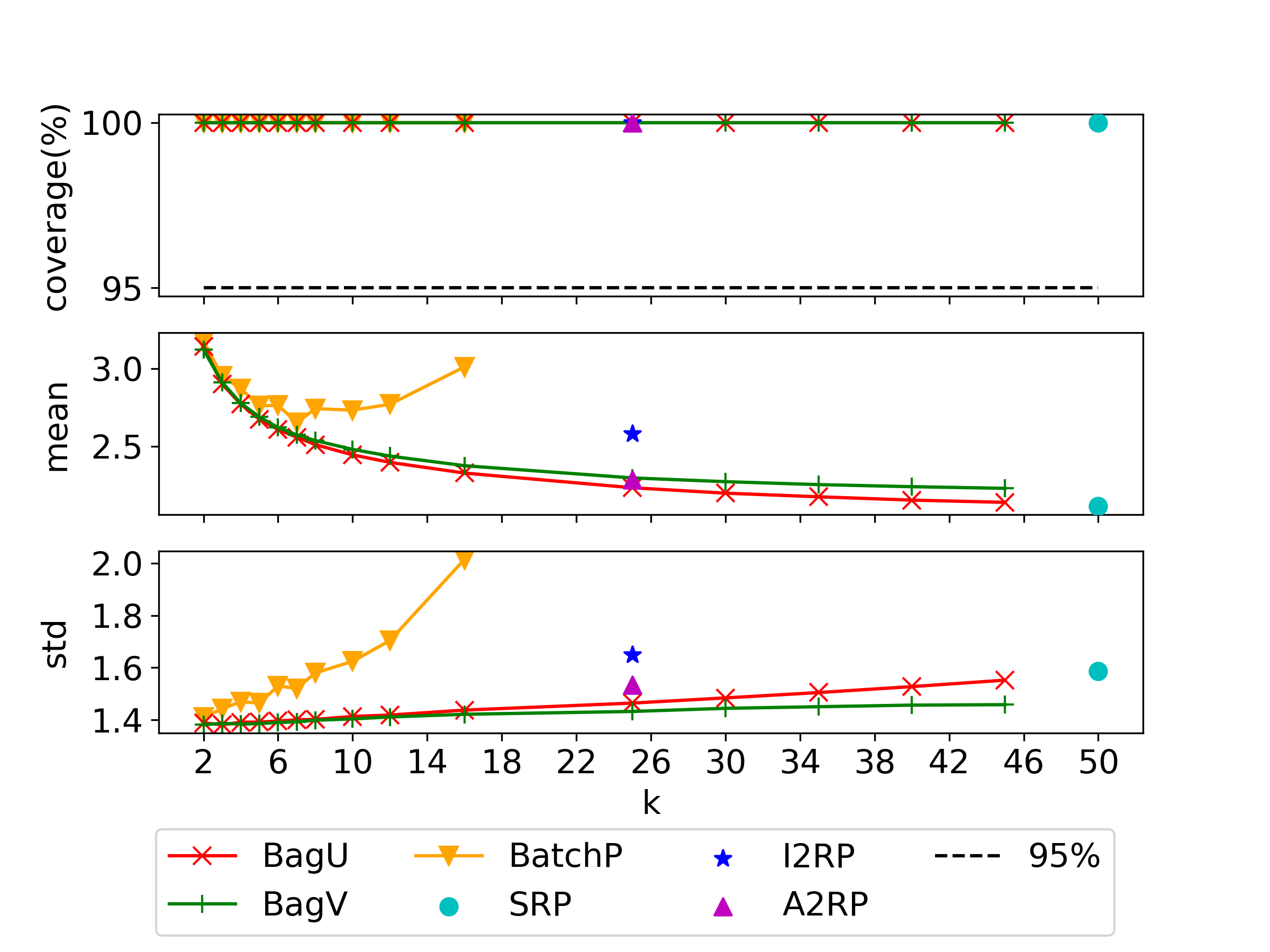}
         \caption{BC, $n=50, n_1=30, n_2=20$}
         \label{fig:binary gap BC 50}
     \end{subfigure}
     \hfill
     \begin{subfigure}{0.49\textwidth}
         \centering
         \includegraphics[width=\textwidth]{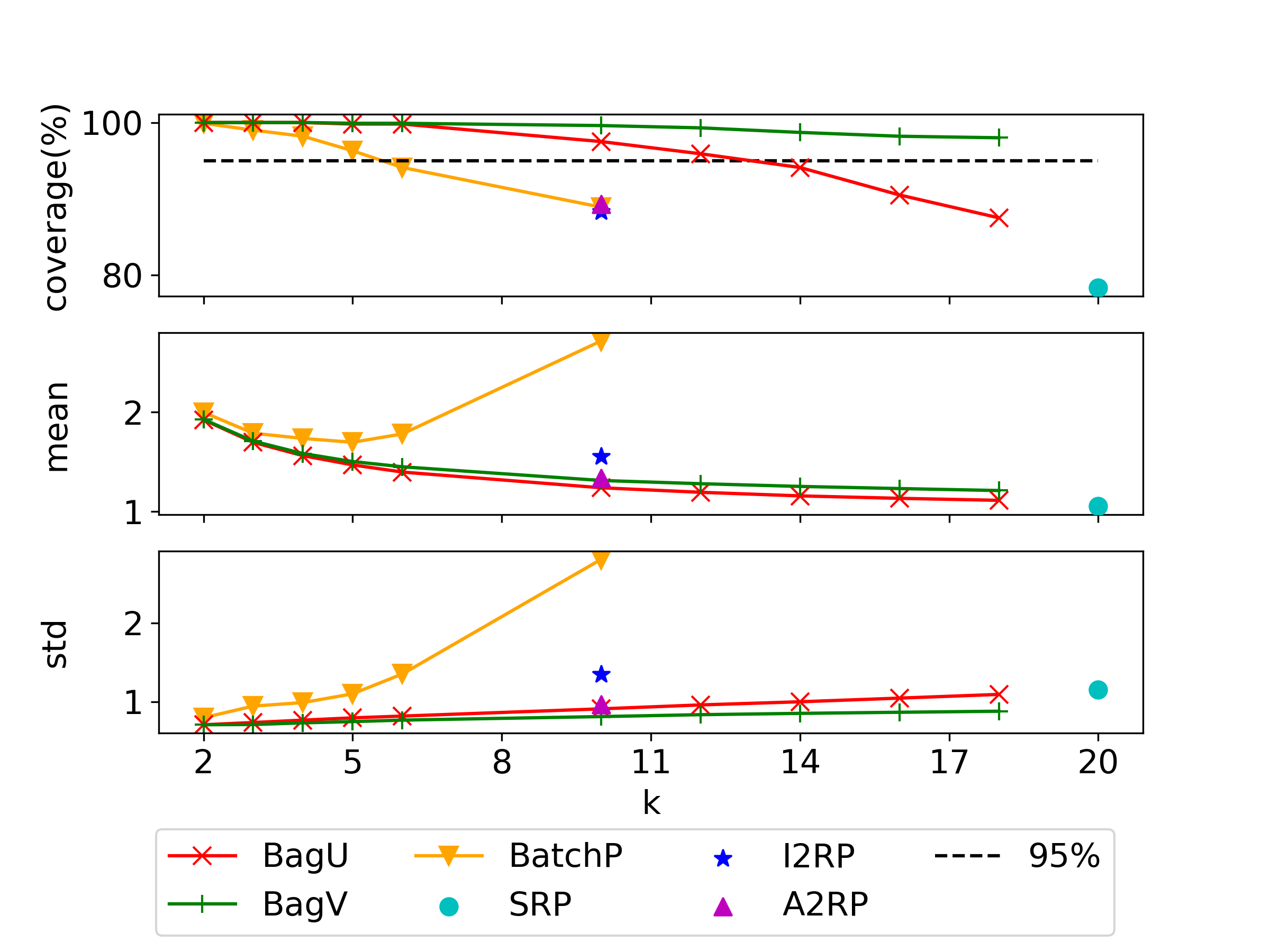}
         \caption{CRN, $n=50, n_1=30, n_2=20$}
         \label{fig:binary gap CRN 50}
     \end{subfigure}
    \caption{Simple linear problem \eqref{binary}. Bounds of optimality gaps.}
    \label{fig:binary gap n=50}
\end{figure}

Figures \ref{fig:portfolio gap n=50}, \ref{fig:integer gap CRN} and \ref{fig:binary gap n=50} summarize the results for problems \eqref{min_cvar}, \eqref{IP} and \eqref{binary} respectively in a similar fashion as in Section \ref{sec:numerics optimal}. Each plot in these figures shows again the estimated coverage probabilities, the mean and the standard deviation of the upper bounds across $1000$ independent runs. In each experiment, whether BC or CRN, the data set is split into $60\%$ and $40\%$, i.e., $n_1=0.6n$ and $n_2=0.4n$ for each total size $n$.

%  therefore from the practical point of view one would invest most data in obtaining as good a solution as possible, and leave just enough data for assessing the quality of the obtained solution.
We see a few similar observations as in Section \ref{sec:numerics optimal}. The two bagging procedures generate statistically valid upper bounds in almost all the cases. The only exception is Figure \eqref{fig:binary gap CRN 50} where all methods but BagV undercovers for problem \eqref{binary} under relatively large resample sizes. BatchP also possesses the desired coverage probability in most cases, but the generated bounds are looser and less stable than those by other methods. In particular, when the CRN approach is used (i.e., Figures \ref{fig:portfolio gap CRN 50}, \ref{fig:integer gap CRN} and \ref{fig:binary gap CRN 50}) the tightest bound by BatchP across different batch sizes can be twice of that by BagU and BagV. For direct-CLT bounds, A2RP continues to exhibit competitive performances on all the three problems except that in Figure \ref{fig:binary gap CRN 50} it undercovers on problem \eqref{binary} whereas BagV still maintains an accurate coverage with similar tightness and stability as A2RP when $k$ is chosen close to $n$. In the same example BagU also has an accurate coverage for resample sizes around $n/2$ but starts to undercover like A2RP when $k$ further approaches $n$. Compared to I2RP, our bagging bounds continue to generate tighter and stabler bounds in all cases and sometimes have more accurate coverages (Figure \ref{fig:portfolio gap n=50}).

Some new observations are as follows. First, we see that SRP suffers from severe undercoverage issues (e.g., under $80\%$ in Figure \ref{fig:binary gap CRN 50}) on all the three problems when the CRN approach is adopted. We comment that this undercoverage of SRP with CRN is not a coincidence as the objective variance used in SRP tends to frequently underestimate the true variance if the candidate solution $\hat x$ to bound the gap for is obtained from an SAA. To explain, when $\hat x$ is also generated from an SAA then it is expected to have a similar distribution as the the solution $\hat x_{n_2}$ obtained from the SAA formed by the second portion of $n_2$ data in the CRN approach, and therefore if $\hat x$ lies in a high-density area of this distribution, which is likely to happen, then $\hat x_{n_2}$ can very well be closer to $\hat x$ (or even exactly $\hat x$ in the case of discrete decision space) than to the true optimal solution $x^*$, in which case the objective variance $Var(h(\hat x_{n_2}^*,\xi)-h(\hat x,\xi))$ used in SRP becomes significantly smaller than the true variance $Var(h(x^*,\xi) - h(\hat x,\xi))$ on a relative scale. This has also been observed and discussed in length in \cite{bayraksan2006assessing} where a strategy that uses a suboptimal SAA solution in place of the exact solution $\hat x_{n_2}$ is proposed to reduce the chance of the two solutions being close. A2RP and I2RP can also mitigate this issue by using two instead of one SAA solution, whereas our bagging procedures push this further by estimating the variance using many resampled SAA solutions. As evidenced in Figures \ref{fig:portfolio gap n=50}, \ref{fig:integer gap CRN} and \ref{fig:binary gap n=50}, A2RP and our bagging bounds have significantly more accurate coverages than SRP, and in particular BagV is the only method that has a correct coverage in Figure \ref{fig:binary gap CRN 50} and also appears tighter and stabler than A2RP when the resample size is set close to the full data size.

Second, we observe that the coverage of BagV is less sensitive to the resample size than BagU, and that when SRP undercovers (e.g., Figures \ref{fig:portfolio gap CRN 50}, \ref{fig:integer gap CRN} and \ref{fig:binary gap CRN 50}) the coverage of BagU starts to resemble that of SRP while BagV does not as the resample size approaches the full data size. This reveals that in practice a larger resample size can be used with BagV than with BagU in maintaining an accurate coverage. The resemblance between BagU and SRP under large $k$ and the robust coverage of BagV can be explained based on the amount of variability brought by different resampling methods. To explain, for resampling with replacement the variability (standard deviation) of the resampled SAA objective decays at a canonical $1/\sqrt{k}$ rate as the resample size $k$ grows towards $n$ since the sampling is i.i.d. and uniform over the data, whereas for resampling without replacement the variability can be calculated to decay at the rate $\sqrt{n-k}/\sqrt{kn}$ which behaves like $1/\sqrt{k}$ for moderate $k$ but decays more quickly as $1/k$ for $k$ close to $n$. In other words, when $k$ is close to $n$ the resampled SAAs and their optimal solutions in BagU are significantly more similar to the original full SAA than those in BagV, and therefore BagU resembles SRP while BagV still has stable and accurate coverages. Although our theory does not directly capture these phenomena, the smaller resampling variability of BagU can be hinted by its additional factor $n^2/(n-k)^2$ in the IJ variance estimator in Algorithm \ref{bagging} that compensates for the variability of BagU under large resample sizes in order to match the correct magnitude of the variance. In light of this key difference between BagU and BagV, we recommend that the resample size should be no larger than $0.7n$ for BagU in limited-data situations.

% modified objective used in the CRN approach usually makes the objective variance highly variable near the optimum and the standard error estimate in SRP is most affected by the jumping SAA solution

% To explain, if the solution $\hat x$ in the modified objective $h(x,\xi) - h(\hat x,\xi)$ is of decent quality and close to the true optimum, then the SAA solution from the second portion of $n_2$ data can very well be closer to $\hat x$ than to the true optimum $x^*$, in which case the objective variance used in SRP becomes almost zero and hence severely underestimates the target variance $Var(h(x^*,\xi) - h(\hat x,\xi))$ on a relative scale.

% In problem \eqref{IP} this can be attributed to the integrality requirement on the decision.
% We find that with high probability the candidate solution $\hat x$ is $-1$ (with optimality gap $0.1$), and, given $\hat x=-1$, solving the SAA associated with the new cost function $h(x,\xi)-h(-1,\xi)$ gives the solution $-1$ again with high probability. However this way the estimated variance $\hat\sigma^2$ will be zero (because the new cost function is constantly $0$ at $x=-1$) which causes the under-cover issue.
% Similar observations have been discussed in Section 6 of \cite{bayraksan2006assessing}.

Third, in general the CRN approach generates tighter and stabler confidence bounds than the BC approach thanks to variance reduction. In particular, Figures \ref{fig:portfolio gap n=50} and \ref{fig:binary gap n=50} shows that with the same split of the data, the bounds by CRN can be up to twice tighter than those by BC as measured by the mean of the generated bounds, and the standard deviation of the bounds can be reduced by up to $30\%$ as can be seen from Figure \ref{fig:binary gap n=50}. We also observe that the BC approach tends to overcover the optimality gap, potentially because of the looseness of the union bound.

% By comparing Table \ref{IP_BC} with Table \ref{IP_CRN} or Table \ref{binary_BC} with Table \ref{binary_CRN}, we see that this benefit of CRN becomes more significant when one invests more data in obtaining $\hat x$, i.e.~when $n_1$ is chosen larger. This is because, the closer the estimated solution $\hat x$ gets to the true optimum $x^*$, the smaller is the variance of the gap function $h(x,\xi)-h(\hat x,\xi)$ at the optimum (i.e.,~$x^*$) due to the continuity of its variance (as a function of $x$), which in turn leads to a smaller standard error.

\subsection{Variance Reduction for Problems with Multiple Optima}\label{sec:numerics variance reduction}
\begin{figure}[h]
    \centering
    \begin{subfigure}{0.49\textwidth}
         \centering
         \includegraphics[width=\textwidth]{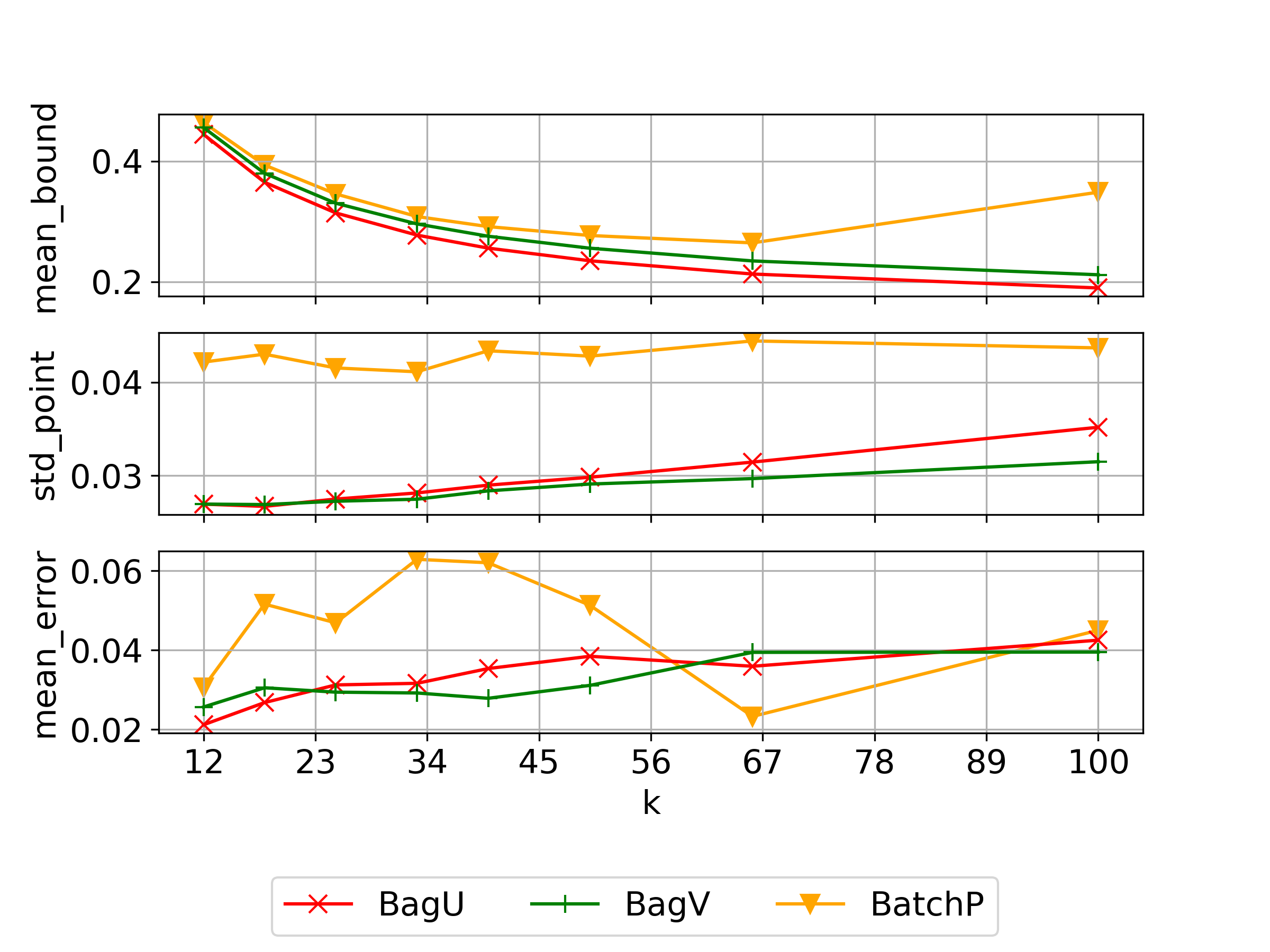}
         \caption{Bagging vs. BatchP, $n=200$}
         \label{fig:simplex with batching 200}
    \end{subfigure}
    \hfill
    \begin{subfigure}{0.49\textwidth}
         \centering
         \includegraphics[width=\textwidth]{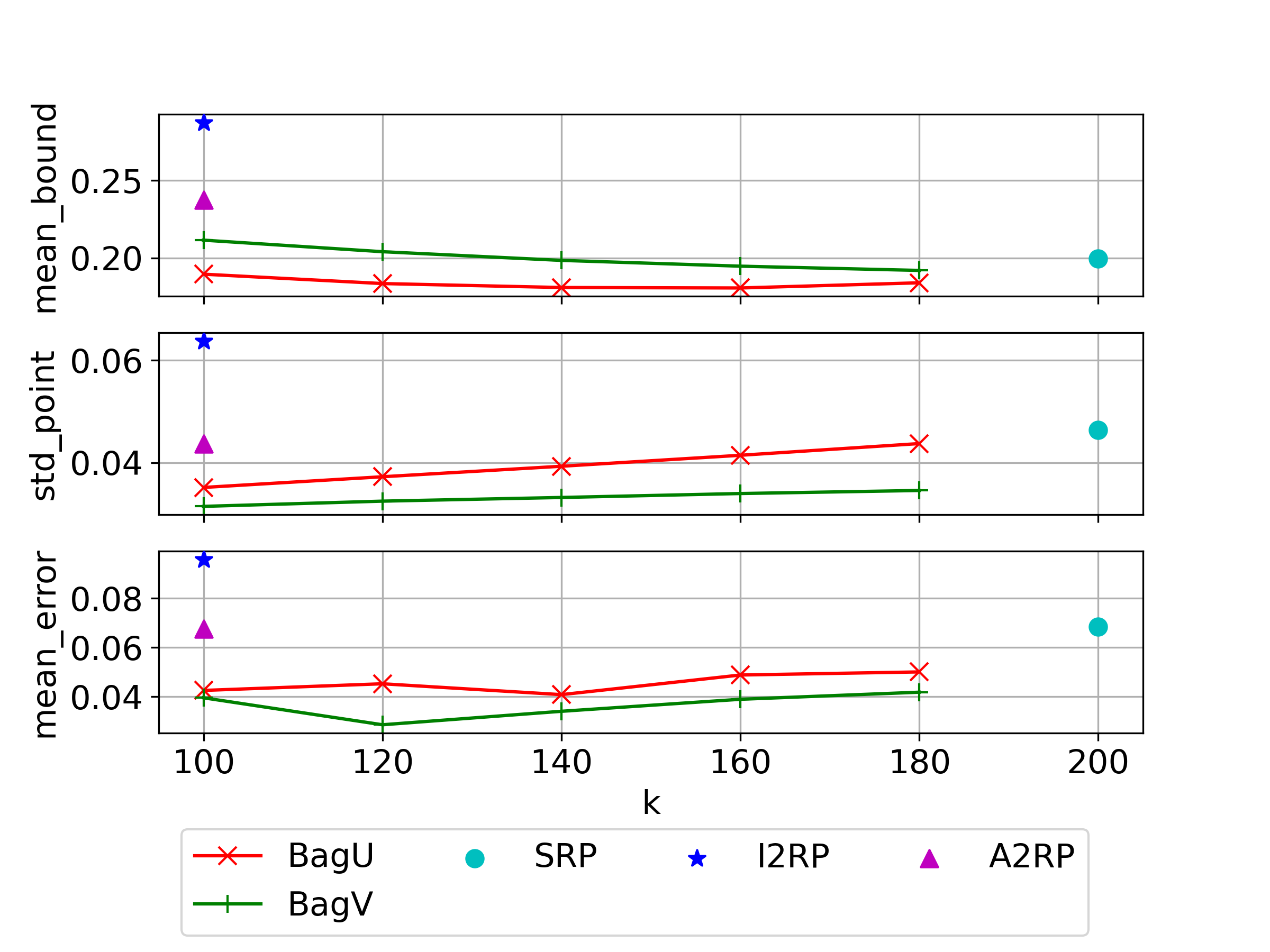}
         \caption{Bagging vs. SRP/I2RP/A2RP, $n=200$}
         \label{fig:simplex with SRP 200}
     \end{subfigure}
    \caption{Variance comparison on linear program \eqref{simplex}.}
    \label{fig:simplex opt val n=200}
\end{figure}

Our theory suggests that when there are multiple optimal solutions the bagging estimate has a smaller variance than both BatchP and direct-CLT estimates, and in this subsection we compare the variances of these estimates using problem \eqref{simplex}. We set the data size $n=200$, and compute $95\%$ lower confidence bounds for the optimal value using different resample sizes for bagging or batch sizes for BatchP. Results are summarized in Figure \ref{fig:simplex opt val n=200}, where the the top panel of each plot shows the mean of the bound (offset by the true optimal value) under different batch sizes or resample sizes, the middle panel plots the standard deviation of the point estimate of each bound, and the bottom panel shows the average standard error estimate in each bound. The coverage probabilities of all the methods are above $95\%$ and hence omitted from the plots.

The standard deviation of our bagging point estimates are consistently smaller than that of the BatchP point estimate ($0.03$ versus $0.04$) regardless of the choice for the batch size or the resample size $k$, as shown in Figure \ref{fig:simplex with batching 200}. Compared to SRP, A2RP and I2RP, Figure \ref{fig:simplex with SRP 200} also shows consistently smaller standard deviations from our bagging point estimates (especially BagV). The larger standard deviation of the I2RP estimate is expected since its point estimate uses only half of the data. All these are aligned with the variance reduction benefit of bagging as described in Theorem \ref{compare var:asymptotic}. Besides variance of the point estimates, the mean standard error term in our bagging bounds are also smaller than that of SRP, I2RP, and A2RP ($0.04$ versus $0.06$) and BatchP at most choices of $k$. This again verifies the reduced variance of our bagging point estimates. Comparing BagU and BagV, we see that the BagV bound has a slightly smaller variance in the point estimate and also a smaller standard error term than BagU, but the overall bound is slightly looser. This relative looseness of BagV is consistent with our theory that the BagU point estimate is unbiased with respect to the optimistic bound of the optimal value whereas BagV is downward biased.

\subsection{Summary and Recommendation}\label{sec:experiments summary}
We summarize our experimental findings. Our bagging bounds, especially BagV, in general appear as competitive as existing methods in terms of coverage attainment, tightness and stability, while outperform them in different specific cases. More precisely, compared to BatchP, our bagging bounds are significantly tighter and stabler in almost all cases thanks to the improved tradeoff between tightness and statistical accuracy. Compared to direct-CLT bounds, BagU and BagV have more accurate coverages than I2RP (Figures \ref{fig:portfolio optval}, \ref{fig:portfolio gap n=50}, \ref{fig:binary gap CRN 50}) and SRP (Figures \ref{fig:integer optval}, \ref{fig:integer gap CRN} and \ref{fig:binary gap CRN 50}) especially when bounding optimality gaps, and generate tighter and stabler bounds than I2RP in all the cases. A2RP shows similarly competitive performances as bagging methods in almost all the cases, but our BagV with a resample size close to the data size appears to have more accurate coverages on problem \eqref{binary} (Figure \ref{fig:binary gap CRN 50}), generates stabler bounds in some cases (Figures \ref{fig:portfolio gap CRN 50} and \ref{fig:binary gap BC 50}), and in general is slightly tighter (e.g., Figure \ref{fig:portfolio optval 50}, \ref{fig:portfolio gap n=50}, \ref{fig:integer gap CRN} and \ref{fig:binary gap CRN 50}). In the case of multiple optima, our bagging bounds are tighter than all existing bounds (Figure \ref{fig:simplex opt val n=200}) when the resample size is chosen close to the data size.

For algorithmic configurations of our bagging procedures, we recommend the debiased variant (Algorithm \ref{debiased bagging}) with a fixed bootstrap size $B=500$ for higher accuracy in the IJ variance estimate. The resample size $k$ can be chosen close to the data size for BagV to generate tight bounds, but no larger than $70\%$ of the data size for BagU under limited data to prevent its behavior mimicking SRP and maintain an accurate coverage.

Regarding the choice between BagU and BagV, we note that BagU generates slightly tighter bounds than BagU in most cases except Figures \ref{fig:portfolio optval 50} and \ref{fig:portfolio gap BC 50} due to the downward bias of $V_{n,k}$. However, bounds by BagV have slightly smaller standard deviations than those by BagU, and BagV's coverage performance is less correlated with that of SRP under large resample sizes due to the larger variability in the resampled SAA than BagU as explained in Section \ref{sec:numerics optimality gap}, making BagV less prone to coverage issues. In fact, the only case that BagV has minor undercoverage is Figure \ref{fig:integer optval 50} where BagU undercovers more severely. From these observations, we recommend BagV based on its superior stability and safer coverage attainment, despite its slight looseness compared to BagU. However, if bound tightness is of importance, then BagU could be preferred. On a final note, results presented in this section are a representative part of our experiments, and additional results can be found in Appendix \ref{sec:additional experiments}.
% Ultimately, the choice depends on the user's preference between bound tightness and coverage attainment: If obtaining a tight bound is the primary goal, then BagU is the preferred choice; Otherwise, if there is a stringent requirement on the coverage of the bound, BagV is the safer one.

% Regarding the choice between BagU and BagV in practice, note that in all the results except Figures \ref{fig:portfolio optval 50} and \ref{fig:portfolio gap BC 50} the bounds by BagV are always slightly looser but have higher coverages than those by BagU.

% \section{Conclusion}
% We have studied a bagging approach to estimate bounds for the optimal value, and consequently the optimality gap for a given solution in stochastic optimization. We demonstrate how our approach works under minimal regularity conditions, including for non-smooth problems, and exhibits competitive statistical efficiency and stability. Compared to batching, our approach generates a new tradeoff between bound tightness and statistical accuracy that is especially beneficial in small-sample situations. Compared to approaches based on direct SAA asymptotics, our approach requires less smoothness conditions on the objectives and gives more stable estimates thanks to the smoothing effect of bagging. These benefits, however, are offset by the price of more computation in repeatedly solving SAA programs. We have developed the theoretical properties of our approach by viewing SAA as a kernel in infinite-order symmetric statistics, and have illustrated our findings with numerical results.

%%
% \theendnotes

% Acknowledgments here
\ACKNOWLEDGMENT{We gratefully acknowledge support from the National Science Foundation under grants CMMI-1542020, CMMI-1523453 and CAREER CMMI-1653339/1834710.
We thank David Morton and David Woodruff for the greatly helpful suggestions and communications on this work.
A preliminary conference version of this work, \cite{lq18}, has appeared in the Proceedings of the Winter Simulation Conference 2018.}

% References here (outcomment the appropriate case)

% CASE 1: BiBTeX used to constantly update the references
%   (while the paper is being written).
\bibliographystyle{informs2014} % outcomment this and next line in Case 1
\bibliography{references} % if more than one, comma separated

%% Here starts the e-companion (EC)
%%%%%%%%%%%%%%%%%%%%%%%%%%%%%%%%%%%%%%%%%%%%%%%%%%%%%%%%%%
\ECSwitch

%%% Main head for the e-companion
\ECHead{The Debiased Bagging Procedure, Proofs of Statements and Additional Experimental Results}
\section{A Debiased Variant of Algorithm \ref{bagging}}\label{sec:debiased bagging}
The major source of Monte Carlo error in $\tilde{\sigma}_{IJ}^2$ from Algorithm \ref{bagging} is the bias $E_*[\sum_{i=1}^n(\widehat{Cov}_i^2 - Cov_i^2)]$ (to be shown in the proof of Theorem \ref{final_guarantee} in Section \ref{proof:final}), and here we consider applying a bias correction to reduce the Monte Carlo error as in \citeAPX{wager2014confidence}. Specifically, for resampling without replacement, when $k$ and $n$ are large $N_i^*$ and $\hat{Z}_k^*$ are approximately independent \citepAPX{wager2014confidence}, therefore we have for each $i$
\begin{equation*}
    Var_*(\widehat{Cov}_*(N_i^*,\hat{Z}_k^*))\approx \frac{1}{B}Var_*(N_i^*)Var_*(\hat{Z}_k^*)= \frac{k}{Bn}(1-\frac{k}{n})Var_*(\hat{Z}_k^*)
\end{equation*}
which suggests that a good correction for the overall bias is
\begin{equation*}
    \frac{k}{B^2}(1-\frac{k}{n})\sum_{b=1}^B(\hat Z_k^b-\tilde Z_{n,k}^{bag})^2.
\end{equation*}
Similarly, in the case of resampling with replacement
\begin{equation*}
    Var_*(\widehat{Cov}_*(N_i^*,\hat{Z}_k^*))\approx \frac{k}{Bn}Var_*(\hat{Z}_k^*)
\end{equation*}
suggesting the bias correction
\begin{equation*}
    \frac{k}{B^2}\sum_{b=1}^B(\hat Z_k^b-\tilde Z_{n,k}^{bag})^2.
\end{equation*}
The full details of our debiased bagging procedure are provided in Algorithm \ref{debiased bagging}.
\begin{algorithm}
\caption{Debiased Bagging Procedure for Bounding Optimal Values}
\label{debiased bagging}
\begin{algorithmic}
\STATE {Given $n$ i.i.d. observations $\bm\xi_{1:n}=(\xi_1,\ldots,\xi_n)$, select positive integers $k$ and $B$.}

\vspace{1ex}

\STATE {Perform the same steps as in Algorithm \ref{bagging} except that the IJ variance estimates are now computed as
\begin{equation*}%\label{sigma_IJ}
\tilde\sigma_{IJ}^2=
\sum_{i=1}^n\widehat{Cov}_*(N_i^*,\hat Z_k^*)^2 - \frac{k}{B^2}\sum_{b=1}^B(\hat Z_k^b-\tilde Z_{n,k}^{bag})^2
\end{equation*}
if resampling is with replacement, or
\begin{equation*}%\label{sigma_IJ}
\tilde\sigma_{IJ}^2=\big(\frac{n}{n-k}\big)^2\big[\sum_{i=1}^n\widehat{Cov}_*(N_i^*,\hat Z_k^*)^2 - \frac{k}{B^2}(1-\frac{k}{n})\sum_{b=1}^B(\hat Z_k^b-\tilde Z_{n,k}^{bag})^2\big]
\end{equation*}
if resampling is without replacement.}

\vspace{1ex}

\STATE {Output $\tilde Z_{n,k}^{bag}-z_{1-\alpha}\tilde\sigma_{IJ}$.}
\end{algorithmic}
\end{algorithm}

We study the bootstrap size $B$ required by Algorithm \ref{debiased bagging} by following an informal Monte Carlo error analysis for IJ variance estimator from \citeAPX{wager2014confidence}. Consider the case of resampling with replacement, for which equation (13) in \citeAPX{wager2014confidence} provides approximate first and second order moments for the IJ variance estimator without debiasing
\begin{eqnarray}
    \notag &&E_*\big[\sum_{i=1}^n\widehat{Cov}_*(N_i^*,\hat{Z}_k^*)^2\big] - \sum_{i=1}^nCov_*(N_i^*,\hat{Z}_k^*)^2\approx \frac{k}{B}\hat{\sigma}^2_*(\hat{Z}_k^*)\\
    &&Var_*\big(\sum_{i=1}^n\widehat{Cov}_*(N_i^*,\hat{Z}_k^*)^2\big)\approx 2\frac{k^2\hat{\sigma}^4_*(\hat{Z}_k^*)}{nB^2} + 4\frac{k\hat{\sigma}^2_*(\hat{Z}_k^*)\sum_{i=1}^nCov_*(N_i^*,\hat{Z}_k^*)^2}{nB}\label{approximate variance of IJ estimator}
\end{eqnarray}
where $\hat{\sigma}^2_*(\hat{Z}_k^*)=1/B\cdot \sum_{b=1}^B(\hat Z_k^b-\tilde Z_{n,k}^{bag})^2$. Therefore, after applying the bias correction, the dominating Monte Carlo error is the variance \eqref{approximate variance of IJ estimator}. Since it's shown in the proof of Theorem \ref{final_guarantee} that $Var_*(\hat{Z}_k^*)=O_p(1/k)$, and the IJ variance estimator $\sum_{i=1}^nCov_*(N_i^*,\hat{Z}_k^*)^2=O_p(1/n)$, we see that our debiased IJ variance estimator $\tilde{\sigma}_{IJ}^2$ from Algorithm \ref{debiased bagging} has an approximate mean squared error of order
\begin{equation*}
    \frac{1}{nB^2} + \frac{1}{n^2B}.
\end{equation*}
Using a $B$ such that $B/\sqrt{n}\to\infty$, we can control the mean squared error of the debiased $\tilde{\sigma}_{IJ}^2$ at the order of $o_p(1/n^2)$, leading to a negligible standard error of order $o_p(1/n)$.

\section{Preliminaries for the Proofs}\label{proof:preliminaries}
We provide in this section several key technical tools and results to be used in the proof of the main Theorems \ref{main thm} and \ref{V thm}. We first present a useful technical tool, the ANOVA decomposition of a symmetric statistic, and several related concepts and preparatory results for the main proof.
\begin{lemma}[ANOVA decomposition, adapted from \citeAPX{efron1981jackknife}]\label{anova}
For any symmetric function $T:\Xi^k\to\R$ and i.i.d.~random elements $\xi_1,\ldots,\xi_k\in\Xi$ such that $Var[T(\xi_1,\ldots,\xi_k)]<\infty$, there exist functions $T_1,\ldots,T_k$ such that
\begin{eqnarray*}
T(\xi_1,\ldots,\xi_k)=E[T]+\sum_{i=1}^kT_1(\xi_i)+\sum_{i_1<i_2}T_2(\xi_{i_1},\xi_{i_2})+\cdots+T_k(\xi_1,\ldots,\xi_k).
\end{eqnarray*}
Moreover, all the $2^k-1$ random variables on the right hand side have mean zero and are mutually uncorrelated.
\end{lemma}
% For any random variable in the form $T=T(\xi_1,\ldots,\xi_n)$, we also use the notation $\mathring T$ to denote the Hajek projection, namely, the projection of $T$ onto the space spanned by
% $\sum_{i=1}^nf_i(\xi_i)$ where $f_i$'s are any measurable functions. By \citeAPX{van2000asymptotic}, we know that, if $\xi_1,\ldots,\xi_n$ are i.i.d. and $T$ has finite second moment,
% $$\mathring T=\sum_{i=1}^nE[T|\xi_i]-(n-1)ET$$
Note that $T_1(x)$ must be $E[T(\xi_1,\ldots,\xi_k)\vert \xi_1=x]-E[T]$ by the zero mean and uncorrelatedness, and the total variance of a symmetric statistic $T(\xi_1,\ldots,\xi_k)$ can be decomposed as $Var(T)=\sum_{s=1}^k\binom{k}{s}V_s$, where $V_s:=Var(T_s(\xi_1,\ldots,\xi_s))$. The Hajek projection of $T$ is defined as
\begin{equation*}
    \mathring{T}:=E[T]+\sum_{i=1}^kT_1(\xi_i)=E[T]+\sum_{i=1}^k\big(E[T(\xi_1,\ldots,\xi_k)\vert \xi_i]-E[T]\big)
\end{equation*}
i.e., the first order effect in the ANOVA decomposition. Recall that $H_k(\xi_1,\ldots,\xi_k)$ is the SAA kernel and $g_k(\xi):=E[H_k(\xi_1,\ldots,\xi_k)\vert \xi_1=\xi]$, therefore the Hajek projections of the symmetric kernel $H_k$ and the symmetric statistic $U_{n,k}$ are
\begin{align*}
\mathring{H}_{k}&=W_k+\sum_{i=1}^k(g_k(\xi_i)-W_k)\\
\mathring{U}_{n,k}&=W_k+\frac{k}{n}\sum_{i=1}^n(g_k(\xi_i)-W_k)
\end{align*}
% As discussed in Section \ref{sec:growing}, we will use the ANOVA decomposition (\citeAPX{efron1981jackknife}) of the symmetric kernel $H_k$ to allow for a larger resample size $k$ in obtaining Theorem \ref{Lipschitz_k full}.
respectively and the remainder $H_k - \mathring{H}_{k}$ or $U_{n,k} - \mathring{U}_{n,k}$ contains all the high-order errors of the corresponding symmetric statistic. A key result we use is the following variance bound from \citeAPX{wager2017estimation} in analyzing random forests:
\begin{lemma}[Adapted from Lemma 3.3 of \citeAPX{wager2017estimation}]\label{anova:variance}
Under Assumption \ref{L2}, for any $k\leq n$ it holds
\begin{equation*}
E[(U_{n,k}-\mathring{U}_{n,k})^2]\leq \frac{k^2}{n^2}E[(H_k-\mathring{H}_k)^2].
\end{equation*}
\end{lemma}
\proof{Proof.}
\citeAPX{wager2017estimation} proves this bound in the context of random forests where $H_k$ is a regression tree and $U_{n,k}$ is the random forest obtained from aggregating the resampled trees (without replacement). Although the context they focus on is different from ours, their proof works for general symmetric kernels and U-statistics including the SAA values considered in this paper. Note that in Lemma 3.3 of \citeAPX{wager2017estimation} the right hand side is the total variance $Var(H_k)$ instead of $E[(H_k-\mathring{H}_k)^2]$, however this comes from upper bounding $E[(H_k-\mathring{H}_k)^2]$ by $Var(H_k)$ in their proof so the bound remains valid with $Var(H_k)$ replaced by $E[(H_k-\mathring{H}_k)^2]$.\Halmos
\endproof

A direct consequence of Lemma \ref{anova:variance} is the following result regarding the order of magnitude of the high-order errors under an additional assumption on the resample size $k$:
\begin{lemma}\label{new_clt:U}
Under Assumptions \ref{L2}, if the resample size $k$ is chosen such that
\begin{equation}\label{range:k}
k^2E[(H_k-\mathring{H}_k)^2]=o(n)
\end{equation}
then $E[(U_{n,k}-\mathring{U}_{n,k})^2]=o\big(\frac{1}{n}\big)$ and hence $U_{n,k}-\mathring{U}_{n,k}=o_p\big(\frac{1}{\sqrt{n}}\big)$.
\end{lemma}
% This allows us to derive a CLT under an additional assumption on $k$:
% \begin{theorem}\label{new_clt:U}
% Under Assumptions \ref{L2}, \ref{L2 strengthened} and \ref{nondegeneracy strengthened}, if the resample size $k$ is chosen such that
% \begin{equation}\label{range:k}
% k^2E(H_k-\mathring{H}_k)^2=o(n)
% \end{equation}
% then
% $$\frac{\sqrt n(U_{n,k}-W_k)}{k\sqrt{Var(g_k(\xi))}}\Rightarrow N(0,1)$$
% where $N(0,1)$ is the standard normal.
% \end{theorem}
% \begin{proof}
% According to the proof of Theorem \ref{main thm}, we only need to verify that $E(U_{n,k}-\mathring{U}_{n,k})^2/Var(U_{n,k})\to 0$, or equivalently $E(U_{n,k}-\mathring{U}_{n,k})^2/Var(\mathring{U}_{n,k})\to 0$. Under the choice of $k$ we have $E(U_{n,k}-\mathring{U}_{n,k})^2=o(1/n)$ due to Proposition \ref{anova:variance}, whereas $Var(\mathring{U}_{n,k})=k^2Var(g_k(\xi))/n\geq \epsilon/n$ for $k$ large enough. This completes the proof.
% \end{proof}

% We need the following result from \citeAPX{van2000asymptotic}:
% \begin{theorem}[Theorem 11.2 in \citeAPX{van2000asymptotic}]
% Let $\mathcal L_n$ be a linear space of random variables with finite second moment that contains the constants. Let $T_n$ be a random variable with projection $S_n$ onto $\mathcal L_n$. If
% $$\frac{Var(T_n)}{Var(S_n)}\to1\text{\ \ as\ }n\to\infty$$
% then
% $$\frac{T_n-ET_n}{sd(T_n)}-\frac{S_n-ES_n}{sd(S_n)}\stackrel{p}{\to}0\text{\ \ as\ }n\to\infty$$
% where $sd(\cdot)$ denotes the standard deviation.\label{projection}
% \end{theorem}

Our plan for proving Theorem \ref{main thm} is to show that in the decomposition $U_{n,k} = \mathring{U}_{n,k} + U_{n,k}-\mathring{U}_{n,k}$ the high-order error $U_{n,k}-\mathring{U}_{n,k}$ is controlled by Lemma \ref{new_clt:U} and the Hajek projection $\mathring{U}_{n,k}$ gives rise to the main term $W_k+k\sqrt{Var(g_k(\xi))/n}\cdot \mathcal Z_{n,k}$. To proceed, we define
$$g_{k,c}(\tilde\xi_1,\ldots,\tilde\xi_c)=E[H_k(\xi_1,\ldots,\xi_k)|\xi_1=\tilde\xi_1,\ldots,\xi_c=\tilde\xi_c]$$
as the conditional expectation of $H_k$ given the first $c$ variables. In particular, by our definition before, $g_k(\xi)=g_{k,1}(\xi)$ and $H_k(\xi_1,\ldots,\xi_k)=g_{k,k}(\xi_1,\ldots,\xi_k)$.

In order to use Lemma \ref{new_clt:U} to control the high-order error $U_{n,k}-\mathring{U}_{n,k}$, we need bounds for the high-order error of the SAA kernel $H_k-\mathring{H}_k$. To this end we derive two useful results. One is on bounding the difference of $H_k$ and $g_{k,c}$, and the other is a variance bound for the SAA kernel $H_k$. The bounds for $H_k$ and $g_{k,c}$ are as follows:% on the implication of Assumption \ref{L2 strengthened}:
% \begin{lemma}
% %Assumption \ref{L2 strengthened} implies
% Suppose Assumption \ref{L2} holds.
% For $\xi_1,\ldots,\xi_c,\xi_1',\ldots,\xi_c'\stackrel{i.i.d.}{\sim}F$, we have
% $$|g_{k,c}(\xi_1',\ldots,\xi_c')-E[g_{k,c}(\xi_1,\ldots,\xi_c)]|\leq\frac{1}{k}\sum_{i=1}^cE\left[\sup_{x\in\mathcal X}|h(x,\xi_i')-h(x,\xi_i)|\Bigg|\xi_i'\right]$$
% %$$Var(g_{k,c}(\xi_1,\ldots,\xi_c))\leq\frac{Mc^2}{k^2}$$
% %for some constant $M>0$.
% \label{bound lemma}
% \end{lemma}
\begin{lemma}
%Assumption \ref{L2 strengthened} implies
For every $\xi_1,\ldots,\xi_c,\xi_1',\ldots,\xi_c'\in\Xi$ with $c\leq k$, we have
\begin{eqnarray*}
    &&\sup_{\xi_{c+1},\ldots,\xi_k\in\Xi}\lvert H_k(\xi_1',\ldots,\xi_c',\xi_{c+1},\ldots,\xi_k) - H_k(\xi_1,\ldots,\xi_c,\xi_{c+1},\ldots,\xi_k)\rvert\\
    &\leq& \frac{1}{k}\sum_{i=1}^c\sup_{x\in\mathcal X}\lvert h(x,\xi_i')-h(x,\xi_i)\rvert
\end{eqnarray*}
and therefore if $\xi_1,\ldots,\xi_c,\xi_1',\ldots,\xi_c'\stackrel{i.i.d.}{\sim}F$ we have
\begin{equation*}
    \lvert g_{k,c}(\xi_1',\ldots,\xi_c')-E[g_{k,c}(\xi_1,\ldots,\xi_c)]\rvert\leq\frac{1}{k}\sum_{i=1}^cE\left[\sup_{x\in\mathcal X}\lvert h(x,\xi_i')-h(x,\xi_i)\rvert\Bigg|\xi_i'\right].
\end{equation*}
% $$Var(g_{k,c}(\xi_1,\ldots,\xi_c))\leq\frac{Mc^2}{k^2}$$
%for some constant $M>0$.
\label{bound lemma}
\end{lemma}
\proof{Proof.}
For any $\xi_{c+1},\ldots,\xi_k\in\Xi$, we consider bounding the absolute difference $\lvert H_k(\xi_1',\ldots,\allowbreak\xi_c',\xi_{c+1},\ldots,\xi_k) - H_k(\xi_1,\ldots,\xi_c,\xi_{c+1},\ldots,\xi_k)\rvert$. If $H_k(\xi_1',\ldots,\xi_c',\xi_{c+1},\ldots,\xi_k) \geq H_k(\xi_1,\ldots,\xi_c,\allowbreak\xi_{c+1},\ldots,\xi_k)$ then
\begin{eqnarray*}
    &&\lvert H_k(\xi_1',\ldots,\xi_c',\xi_{c+1},\ldots,\xi_k) - H_k(\xi_1,\ldots,\xi_c,\xi_{c+1},\ldots,\xi_k)\rvert\\
    &=&H_k(\xi_1',\ldots,\xi_c',\xi_{c+1},\ldots,\xi_k) - H_k(\xi_1,\ldots,\xi_c,\xi_{c+1},\ldots,\xi_k)\\
    &=&\min_{x\in\mathcal X}\left\{\frac{1}{k}\sum_{i=1}^ch(x,\xi_i')+\frac{1}{k}\sum_{i=c+1}^kh(x,\xi_i)\right\} - \min_{x\in\mathcal X}\left\{\frac{1}{k}\sum_{i=1}^ch(x,\xi_i)+\frac{1}{k}\sum_{i=c+1}^kh(x,\xi_i)\right\}\\
    &\leq& \min_{x\in\mathcal X}\left\{\frac{1}{k}\sum_{i=1}^ch(x,\xi_i')+\frac{1}{k}\sum_{i=c+1}^kh(x,\xi_i)\right\} -\left(\frac{1}{k}\sum_{i=1}^ch(x_{\epsilon},\xi_i)+\frac{1}{k}\sum_{i=c+1}^kh(x_{\epsilon},\xi_i)\right) + \epsilon\\
    &&\text{\ \ where $x_{\epsilon}$ is an $\epsilon$-optimal solution of $\min_{x\in\mathcal X}\frac{1}{k}\sum_{i=1}^kh(x,\xi_i)$}\\
    &\leq& \left(\frac{1}{k}\sum_{i=1}^ch(x_{\epsilon},\xi_i')+\frac{1}{k}\sum_{i=c+1}^kh(x_{\epsilon},\xi_i)\right) -\left(\frac{1}{k}\sum_{i=1}^ch(x_{\epsilon},\xi_i)+\frac{1}{k}\sum_{i=c+1}^kh(x_{\epsilon},\xi_i)\right) + \epsilon\\
    &=&\frac{1}{k}\sum_{i=1}^c(h(x_{\epsilon},\xi_i') - h(x_{\epsilon},\xi_i)) + \epsilon\\
    &\leq& \frac{1}{k}\sum_{i=1}^c\sup_{x\in\mathcal X}\lvert h(x,\xi_i')-h(x,\xi_i)\rvert + \epsilon.
\end{eqnarray*}
Since the $\epsilon>0$ is arbitrary, it must hold that $\lvert H_k(\xi_1',\ldots,\xi_c',\xi_{c+1},\ldots,\xi_k) - H_k(\xi_1,\ldots,\xi_c,\xi_{c+1},\allowbreak\ldots,\xi_k)\rvert\leq \frac{1}{k}\sum_{i=1}^c\sup_{x\in\mathcal X}\lvert h(x,\xi_i')-h(x,\xi_i)\rvert$. Similarly, if $H_k(\xi_1',\ldots,\xi_c',\xi_{c+1},\ldots,\xi_k) \leq H_k(\xi_1,\allowbreak\ldots,\xi_c,\xi_{c+1},\ldots,\xi_k)$ the same bound can be shown to be valid by symmetry. This proves our first bound on the difference.

The second bound then follows by applying Jensen's inequality. Specifically, by the definition of $g_{k,c}$ we have
\begin{eqnarray*}
&&\lvert g_{k,c}(\xi_1',\ldots,\xi_c')-E[g_{k,c}(\xi_1,\ldots,\xi_c)]\rvert\\
&=&\lvert E[H_k(\xi_1',\ldots,\xi_c',\xi_{c+1},\ldots,\xi_k)-H_k(\xi_1,\ldots,\xi_c,\xi_{c+1},\ldots,\xi_k)\vert \xi_1',\ldots,\xi_c']\rvert\\
&&\text{\ \ where each $\xi_i,\xi_i'\stackrel{i.i.d.}{\sim}F$}\\
&\leq&E[\lvert H_k(\xi_1',\ldots,\xi_c',\xi_{c+1},\ldots,\xi_k)-H_k(\xi_1,\ldots,\xi_c,\xi_{c+1},\ldots,\xi_k)\rvert\big\vert \xi_1',\ldots,\xi_c']\text{\ by Jensen's inequality}\\
&\leq& E[\frac{1}{k}\sum_{i=1}^c\sup_{x\in\mathcal X}\lvert h(x,\xi_i')-h(x,\xi_i)\rvert\big\vert \xi_1',\ldots,\xi_c']\text{\ by the first bound}\\
&=&\frac{1}{k}\sum_{i=1}^cE\left[\sup_{x\in\mathcal X}\lvert h(x,\xi_i')-h(x,\xi_i)\rvert\Bigg|\xi_i'\right].
\end{eqnarray*}
This completes the proof.\Halmos
\endproof
To state the other result, we need the Efron-Stein inequality:
\begin{lemma}[Efron-Stein inequality, Theorem 3.1 in \citeAPX{boucheron2013concentration}]\label{efron-stein inequality}
Let $X_1,\ldots,X_k$ be $k$ independent random elements and $f(X_1,\ldots,X_k)$ a function such that $E[(f(X_1,\ldots,X_k))^2]<\infty$, then
\begin{eqnarray*}
    &&Var(f(X_1,\ldots,X_k))\\
    &\leq& \frac{1}{2}\sum_{i=1}^kE[(f(X_1,\ldots,X_{i-1},X_i,X_{i+1},\ldots,X_k) - f(X_1,\ldots,X_{i-1},X'_i,X_{i+1},\ldots,X_k))^2]
\end{eqnarray*}
where $X'_i$ is an independent copy of $X_i$ for each $i=1,\ldots,k$.
\end{lemma}
Setting $X_i=\xi_i$ and $f$ to be the SAA kernel $H_k$ in Lemma \ref{efron-stein inequality} gives rise to the following variance bound for the SAA optimal value:
\begin{proposition}[Variance bound for SAA kernel]\label{var bound for SAA kernel}
Under Assumptions \ref{L2} and \ref{L2 strengthened}, for $\xi_1,\ldots,\xi_k\stackrel{i.i.d.}{\sim}F$ we have
\begin{equation*}
    Var(H_k(\xi_1,\ldots,\xi_k))=O\big(\frac{1}{k}\big).
\end{equation*}
\end{proposition}
\proof{Proof.}
Assumption \ref{L2} implies that $E[(H_k(\xi_1,\ldots,\xi_k))^2]<\infty$ as argued in the proof of Theorem \ref{thm:symmetric}, therefore by Lemma \ref{efron-stein inequality} we can write
\begin{eqnarray*}
Var(H_k(\xi_1,\ldots,\xi_k))&\leq& \frac{1}{2}\sum_{i=1}^kE[(H_k(\xi_1,\ldots,\xi_{i-1},\xi'_i,\xi_{i+1},\ldots,\xi_k) - H_k(\xi_1,\ldots,\xi_{i-1},\xi_i,\xi_{i+1},\ldots,\xi_k))^2]\\
&\leq& \frac{1}{2}\sum_{i=1}^kE[\frac{1}{k^2}\sup_{x\in\mathcal{X}}\lvert h(x,\xi'_i) - h(x,\xi_i)\rvert^2]\text{\ \ by Lemma \ref{bound lemma}}\\
&\leq& \frac{1}{2k}E[\sup_{x\in\mathcal{X}}\lvert h(x,\xi') - h(x,\xi)\rvert^2]\text{\ \ where }\xi',\xi\stackrel{i.i.d.}{\sim}F.
\end{eqnarray*}
Since Assumption \ref{L2 strengthened} implies that $E[\sup_{x\in\mathcal{X}}\lvert h(x,\xi') - h(x,\xi)\rvert^2]<\infty$, the $O(1/k)$ bound immediately follows.\Halmos
\endproof
\section{Proof of Theorems \ref{main thm} and \ref{V thm}}\label{proof:main}
We first provide the most general form of the central limit theorem for $U_{n,k}$ which will be cited at several places throughout the paper.
\begin{theorem}[General central limit theorem for $U_{n,k}$]\label{meta clt}
If Assumptions \ref{L2} and \ref{L2 strengthened} hold and $k$ is chosen such that \eqref{range:k} holds, then
\begin{equation*}
    \sqrt{n}(U_{n,k} - W_k) = k\sqrt{Var(g_k(\xi))}\cdot \mathcal Z_{n,k} + o_p(1)
\end{equation*}
where each $\mathcal Z_{n,k}$ is of mean $0$ and variance $1$, and every subsequence of $\{\mathcal Z_{n,k}\}$ such that
$k^2Var(g_k(\xi)) \cdot n^{\frac{\delta}{2+\delta}}\to\infty$ converges in distribution to the standard normal.
\end{theorem}
\proof{Proof of Theorem \ref{meta clt}.}
Lemma \ref{new_clt:U} implies that $U_{n,k} - \mathring{U}_{n,k}=o_p(1/\sqrt{n})$. Therefore we can write
\begin{eqnarray*}
\sqrt{n}(U_{n,k} - W_k)&=&\sqrt{n}(\mathring{U}_{n,k} - W_k) +\sqrt{n}(U_{n,k} - \mathring{U}_{n,k})\\
&=&\sqrt{n}\cdot \frac{k}{n}\sum_{i=1}^n(g_k(\xi_i)-W_k) + o_p(1).
% &=&k\sqrt{Var(g_k(\xi))}\cdot \mathcal{Z}_{n,k}\\
% &=&k\sqrt{Var(g_k(\xi))}\cdot \mathcal{Z}_{n,k}
\end{eqnarray*}
Define
\begin{equation*}
    \mathcal{Z}_{n,k}:=
    \begin{cases}
    \frac{\sum_{i=1}^n(g_k(\xi_i)-W_k)}{\sqrt{nVar(g_k(\xi))}}&\text{if }Var(g_k(\xi))>0\\
    \text{an arbitrary standard normal variable}&\text{if }Var(g_k(\xi))=0
    \end{cases}
\end{equation*}
then we have
\begin{equation*}
    \sqrt{n}(U_{n,k} - W_k)=k\sqrt{Var(g_k(\xi))}\mathcal{Z}_{n,k}+o_p(1)
\end{equation*}
for every $n$ and $k$. Note that by definition $\mathcal{Z}_{n,k}$ always has mean zero and variance one in either case. It remains to show the central limit convergence for qualified subsequences of $\{\mathcal{Z}_{n,k}\}$.

To show the central limit convergence, we verify the Lyapunov condition. By Lemma \ref{bound lemma}, denoting $\xi,\xi'\stackrel{i.i.d.}{\sim}F$, we have
\begin{eqnarray*}
E[|g_k(\xi)-W_k|^{2+\delta}]&\leq& E\left(\frac{1}{k}E\left[\sup_{x\in\mathcal X}|h(x,\xi')-h(x,\xi)|\Bigg|\xi\right]\right)^{2+\delta}\\
&\leq&\frac{1}{k^{2+\delta}}E\left[\sup_{x\in\mathcal X}|h(x,\xi')-h(x,\xi)|^{2+\delta}\right]\text{\ \ by Jensen's inequality}\\
&\leq&\frac{\tilde M}{k^{2+\delta}}
\end{eqnarray*}
for some $\tilde M<\infty$ by Assumption \ref{L2 strengthened}. Moreover, by the condition that $k^2Var(g_k(\xi))\cdot n^{\delta/(2+\delta)}\to\infty$ we have $Var(g_k(\xi))>0$ and hence $\mathcal{Z}_{n,k}=\sum_{i=1}^n(g_k(\xi_i)-W_k)/\sqrt{nVar(g_k(\xi))}$ for sufficiently large $n$. The Lyapunov condition can be verified as
\begin{equation*}
\frac{nE[|g_k(\xi)-W_k|^{2+\delta}]}{(nVar(g_k(\xi)))^{1+\delta/2}}\leq\frac{n\tilde M/k^{2+\delta}}{(nVar(g_k(\xi)))^{1+\delta/2}}=\frac{\tilde M}{(n^{\delta/(2+\delta)}\cdot k^2Var(g_k(\xi)))^{1+\delta/2}}\to 0
\end{equation*}
% $$\frac{nE|g_k(\xi)-W_k|^{2+\delta}}{(nVar(g_k(\xi)))^{1+\delta/2}}\leq\frac{n\tilde M/k^{2+\delta}}{(n\epsilon/k^2)^{1+\delta/2}}=\frac{\tilde M}{n^{\delta/2}\epsilon^{1+\delta/2}}\to0$$
as $n\to\infty$. The Lyapunov condition then implies the central limit theorem for every subsequence of $\{\mathcal{Z}_{n,k}\}$ such that $k^2Var(g_k(\xi))\cdot n^{\delta/(2+\delta)}\to\infty$.\Halmos
\endproof

We now prove Theorem \ref{main thm}:
\proof{Proof of Theorem \ref{main thm}.}
We only need to verify \eqref{range:k}. Since Assumptions \ref{L2} and \ref{L2 strengthened} hold, Proposition \ref{var bound for SAA kernel} states that $Var(H_k)=O(1/k)$, therefore with $k=o(n)$ we have $k^2E[(H_k-\mathring{H}_k)^2]\leq k^2Var(H_k)=k^2\cdot O(1/k)=O(k)=o(n)$ and hence the desired result follows from Theorem \ref{meta clt}.\Halmos
\endproof
% \proof{Proof of Theorem \ref{main thm}.}

%
\proof{Proof of Theorem \ref{V thm}.}
Let $c(n,k,s)$ count the number of mappings $\phi:\{1,2,\ldots,k\}\to\{1,2,\ldots,n\}$ such that $\lvert\phi(\{1,2,\ldots,k\})\rvert=s$, or equivalently, count the number of $\xi_{i_1},\ldots,\xi_{i_k}$ such that $i_1,\ldots,i_k$ covers $s$ distinct indices, and let $A_{n,s}$ be the average of all $H_k(\xi_{i_1},\ldots,\xi_{i_k})$ with $s$ distinct indices. In particular, $A_{n,k}=U_{n,k}$. The V-statistic can be expressed for a fixed $l\geq 0$ as
\begin{equation*}
n^kV_{n,k}=\sum_{s=k-l}^kc(n,k,s)A_{n,s}+\big(n^k-\sum_{s=k-l}^kc(n,k,s)\big)R_{n,l}
\end{equation*}
where $R_{n,l}$ is the average of all $H_k(\xi_{i_1},\ldots,\xi_{i_k})$ with at most $k-l-1$ distinct indices. We have
\begin{eqnarray}
\nonumber n^k(U_{n,k}-V_{n,k})&=&n^kU_{n,k}-\sum_{s=k-l}^kc(n,k,s)(U_{n,k}+A_{n,s}-U_{n,k})-\big(n^k-\sum_{s=k-l}^kc(n,k,s)\big)R_{n,l}\\
\nonumber&=&\big(n^k-\sum_{s=k-l}^kc(n,k,s)\big)(U_{n,k}-R_{n,l})-\sum_{s=k-l}^{k-1}c(n,k,s)(A_{n,s}-U_{n,k})\\
&=&\big(\sum_{s=1}^{k-l-1}c(n,k,s)\big)(U_{n,k}-R_{n,l})-\sum_{s=k-l}^{k-1}c(n,k,s)(A_{n,s}-U_{n,k}).\label{UV_decompose}
\end{eqnarray}
We want to show that
\begin{equation}\label{diff of U and V statistics}
E[(U_{n,k}-V_{n,k})^2]=o\big(\frac{1}{n}\big)
\end{equation}
so that $\sqrt{n}(U_{n,k} - V_{n,k})=o_p(1)$ and we can conclude $\sqrt{n}(V_{n,k}-W_k) = k\sqrt{Var(g_k(\xi))}\cdot \mathcal{Z}_{n,k}+o_p(1)$ based on Theorem \ref{main thm}. It suffices to show that the two terms in \eqref{UV_decompose} both have a second moment of order $o(n^{2k-1})$. To this end, we let
\begin{equation}\label{max_order}
l=\big\lfloor\frac{1}{2(1-2\gamma)}\big\rfloor
\end{equation}
the reason for which shall be clear later.

To bound the first term in \eqref{UV_decompose}, note that $c(n,k,s)$ can be written as
\begin{equation*}
c(n,k,s)=S(k,s)n(n-1)\cdots(n-s+1)
\end{equation*}
where $S(k,s)$ is the Stirling number of the second kind with parameters $k,s$, which is the number of partitions of a set of size $k$ into $s$ non-empty subsets. It's shown in \citeAPX{rennie1969stirling} that for $k\geq 2$ and $1\leq s\leq k-1$
\begin{equation}\label{bound_stirling}
S(k,s)\leq \frac{1}{2}\binom{k}{s}s^{k-s}.
\end{equation}
Hence
\begin{equation*}
\sum_{s=1}^{k-l-1}c(n,k,s)\leq\frac{1}{2} \sum_{s=1}^{k-l-1}\binom{k}{s}s^{k-s}n^s.
\end{equation*}
Note that the ratio between two neighboring $\binom{k}{s}s^{k-s}n^s$ is
\begin{equation*}
\binom{k}{s-1}(s-1)^{k-s+1}n^{s-1}\Big/\binom{k}{s}s^{k-s}n^s=\frac{(s-1)^{k-s+1}}{(k-s+1)s^{k-s-1}n}\leq \frac{s^2}{n}\leq \frac{k^2}{n}=o(1),
\end{equation*}
therefore
\begin{eqnarray*}
\sum_{s=1}^{k-l-1}c(n,k,s)&\leq&\frac{1}{2}\big(1+\sum_{s=1}^{k-l-2}\big(\frac{k^2}{n}\big)^s\big)\binom{k}{l+1}(k-l-1)^{l+1}n^{k-l-1}\\
   &\leq &\frac{1}{2(1-k^2/n)}\binom{k}{l+1}(k-l-1)^{l+1}n^{k-l-1}=O(k^{2l+2}n^{k-l-1})=O\big(\big(\frac{k^2}{n}\big)^{l+1}n^k\big).
\end{eqnarray*}
%\begin{equation*}
%\sum_{s=1}^{k-l-1}c(n,k,s)\leq\frac{1}{2} \left({\begin{matrix} % or pmatrix or bmatrix or Bmatrix or ...
%      k \\
%      l+1 \\
%   \end{matrix}}\right)(k-l-1)^{l+2}n^{k-l-1}=O(k^{2l+3}n^{k-l-1})=O(n^{k+(2\gamma-1)l+3\gamma-1}).
%\end{equation*}
For the particular choice of $l$ shown in \eqref{max_order}, the above bound is $o(n^{k-1/2})$. Under Assumption \ref{L2}, $U_{n,k}$ and $R_{n,l}$ satisfy $E[U_{n,k}^2](\text{ or }E[R_{n,l}^2])\leq E[\sup_{x\in\mathcal{X}}\lvert h(x,\xi) \rvert^2]<\infty$ by Minkowski inequality, therefore the first term in \eqref{UV_decompose} has second moments of order $o(n^{2k-1})$.

For the second term in \eqref{UV_decompose}, it suffices to show that for each $k-l\leq s\leq k-1$ it holds $c(n,k,s)^2E[(A_{n,s}-U_{n,k})^2]=o(n^{2k-1})$ since there are only $l$ of them. Since $l$ is now viewed as a constant, from the upper bound \eqref{bound_stirling} for $s\geq k-l$ it follows that $S(k,s)=O(k^{2(k-s)})$, resulting in $c(n,k,s)=O(k^{2(k-s)}n^{s})$. If we can argue that $E[(A_{n,s}-U_{n,k})^2]=O(k^{-2})$, then the second moment of each summand can be bounded as
\begin{equation*}
O(k^{4(k-s)-2}n^{2s})=O(n^{4\gamma(k-s)-2\gamma+2s})=O(n^{2k+2\gamma-2})
\end{equation*}
where the last equality holds because $\gamma<1/2$ hence $4\gamma(k-s)-2\gamma+2s$ increases in $s$. This implies a second moment of order $o(n^{2k-1})$ for each summand because $\gamma<1/2$. Now we show $E[(A_{n,s}-U_{n,k})^2]=O(k^{-2})$ by a coupling argument. The value of $A_{n,s}$ can be computed from the same resamples $\xi_{i_1},\ldots,\xi_{i_k}$ (with $k$ distinct data points) used to compute $U_{n,k}$, by first removing $k-s$ of the $k$ distinct resampled data points and then drawing from the remaining $s$ data points to fill in the $k-s$ positions. To be specific, we use $I_k=(I(1),\ldots,I(k))$ to represent a sequence of length $k$ where $I(j)\in\{1,\ldots,n\}$ for each $j\leq k$, define $\abs{I_k}$ to be the number of distinct indices in $I_k$. For convenience we denote by $I_k(j_1:j_2)=(I_k(j_1),\ldots,I_k(j_2))$ the sub-sequence for $1\leq j_1\leq j_2\leq k$ and $\bm\xi_{I_k}=(\xi_{I_k(1)},\ldots,\xi_{I_k(k)})$. Then for each without-replacement resample of size $k$ represented by $I_k$ with $\abs{I_k}=k$, we consider removing the last $k-s$ data points of the resample and replacing each of them with one of the first $s$ data points to obtain a new resample $I'_k$, so that $A_{n,s}$ can be computed as
\begin{equation*}
A_{n,s}=\frac{(n-k)!}{n!}\sum_{\abs{I_k}=k}\frac{1}{s^{k-s}}\sum_{\abs{I'_k}=s,I'_k(1:s)=I_k(1:s)}H_k(\bm\xi_{I'_k}).
\end{equation*}
This leads to
\begin{eqnarray}
\notag\lvert A_{n,s}-U_{n,k}\rvert&\leq &\frac{(n-k)!}{n!}\sum_{\abs{I_k}=k}\frac{1}{s^{k-s}}\sum_{\abs{I'_k}=s,I'_k(1:s)=I_k(1:s)}\abs{H_k(\bm\xi_{I'_k})-H_k(\bm\xi_{I_k})}\\
% &\leq &\frac{(n-k)!}{n!}\sum_{\abs{I_k}=k}\frac{1}{s^{k-s}}\sum_{\abs{I'_k}=s,I'_k(1:s)=I_k(1:s)}\sup_{x\in \mathcal X}\abs{\sum_{j=s+1}^k\frac{1}{k}(h(x,\xi_{I'_k(j)})-h(x,\xi_{I_k(j)}))}\\
\notag&\leq &\frac{(n-k)!}{n!}\sum_{\abs{I_k}=k}\frac{1}{s^{k-s}}\sum_{\abs{I'_k}=s,I'_k(1:s)=I_k(1:s)}\frac{1}{k}\sum_{j=s+1}^k\sup_{x\in \mathcal X}\abs{h(x,\xi_{I'_k(j)})-h(x,\xi_{I_k(j)})}\\
\notag&&\text{\ \ \ by Lemma \ref{bound lemma}}\\
\notag&\leq &\frac{1}{k}\sum_{j=s+1}^k\frac{(n-k)!}{n!s^{k-s}}\sum_{\abs{I_k}=k}\sum_{\abs{I'_k}=s,I'_k(1:s)=I_k(1:s)}\sup_{x\in \mathcal X}\abs{h(x,\xi_{I'_k(j)})-h(x,\xi_{I_k(j)})}\\
&=&\frac{k-s}{k}\frac{2}{n(n-1)}\sum_{1\leq i_1<i_2\leq n}\sup_{x\in\mathcal X}\lvert h(x,\xi_{i_1})-h(x,\xi_{i_2})\rvert\label{count distinct indices}\\
\notag&\leq &\frac{k-s}{k}\frac{2}{n(n-1)}\sum_{1\leq i_1<i_2\leq n}\big(\sup_{x\in\mathcal X}\lvert h(x,\xi_{i_1})\rvert + \sup_{x\in\mathcal{X}}\lvert h(x,\xi_{i_2})\rvert\big)\\
\notag&=&\frac{k-s}{k}\frac{2}{n}\sum_{i=1}^n\sup_{x\in\mathcal X}\lvert h(x,\xi_i)\rvert
\end{eqnarray}
where the equality \eqref{count distinct indices} holds because $I'_k(j)$ and $I_k(j)$ are distinct indices and the gross sum over $I_k,I'_k$ puts equal weights on each pair $(i_1,i_2)$. Due to Assumption \ref{L2}, we have
\begin{align*}
E[(A_{n,s}-U_{n,k})^2]\leq 4\Big(\frac{k-s}{k}\Big)^2E[\sup_{x\in\mathcal X}\lvert h(x,\xi)\rvert^2]=O\big(\frac{l^2}{k^2}\big)=O\big(\frac{1}{k^2}\big).
\end{align*}
by Minkowski inequality. This completes the proof.\Halmos
\endproof

\section{Proof of Theorems \ref{Lipschitz characterization of limit variance}, \ref{convex nondegeneracy} and \ref{recover SAA}}\label{sec:refined}
The proof of Theorems \ref{Lipschitz characterization of limit variance} and \ref{recover SAA} heavily rely on theories of empirical processes which we first present in Subsection \ref{subsec:empirical process theory}. The proof of Theorem \ref{Lipschitz characterization of limit variance} also involves non-trivial measurability issues of the optimum functional of the limit Gaussian process and/or the SAA which we deal with in Subsection \ref{subsec:measurability}. The main proofs are deferred to Subsection \ref{subsec:main proof for Lipschitz case}.
\subsection{Empirical Process Theory and Preparatory Results}\label{subsec:empirical process theory}
We introduce concepts in empirical processes and some notations. Denote by
\begin{equation}\label{centered function class}
\mathcal F:=\{h(x,\cdot)-Z(x):x\in \mathcal X\}
\end{equation}
the family of centered cost functions indexed by the decision $x\in\mathcal X$. Note that for centered functions the Lipschitz condition holds with a slightly larger constant than $M(\xi)$
\begin{equation*}
\abs{h(x_1,\xi)-Z(x_1)-(h(x_2,\xi)-Z(x_2))}\leq (M(\xi)+EM(\xi))\Vert x_1-x_2\Vert.
\end{equation*}
We define $l^{\infty}(\mathcal{X})$ as the metric space of all bounded function from $\mathcal{X}$ to $\R$, with the supremum distance $\Vert f-g\Vert_{\infty}:=\sup_{x\in\mathcal{X}}\lvert f(x)-g(x) \rvert$ for $f,g\in l^{\infty}(\mathcal{X})$. The stochastic process
\begin{equation}\label{empirical process}
    \mathbb{G}_k(\cdot):=\sqrt{k}\big(\frac{1}{k}\sum_{i=1}^kh(\cdot,\xi_i)-Z(\cdot)\big)\in l^{\infty}(\mathcal{X})
\end{equation}
indexed by the decision $x\in\mathcal{X}$, where $\xi_1,\ldots,\xi_k$ are independent and follow the distribution $F$, is called the empirical process. The function class $\mathcal{F}$ is called $F$-Donsker if
\begin{equation*}
    \mathbb{G}_k\Rightarrow \mathbb{G}\text{\ as }k\to\infty
\end{equation*}
where $\mathbb{G}\in l^{\infty}(\mathcal{X})$ is a centered Guassian process on $\mathcal{X}$ with covariance structure defined by $Cov(\mathbb{G}(x_1), \mathbb{G}(x_2))=Cov(h(x_1,\xi), h(x_2,\xi))$ for any $x_1,x_2\in\mathcal{X}$. Moreover, the Gaussian process $\mathbb{G}$ almost surely has uniformly continuous sample paths with respect to the intrinsic semimetric
\begin{equation*}
\rho(x_1,x_2):=\sqrt{Var(h(x_1,\xi) - h(x_2,\xi))}.
\end{equation*}
Note that, with Lipschitz continuity as stated in Assumption \ref{Lipschitz:decision}, the paths of $\mathbb{G}$ are also uniformly continuous with respect to the Euclidean distance. When the cost function $h$ is Lipschitz and the decision space $\mathcal{X}$ is compact, the function class $\mathcal{F}$ is well konwn to be $F$-Donsker, as the following proposition states:
\begin{proposition}[From page 17 of \citeAPX{kosorok2008introduction}]\label{donskerness}
If Assumption \ref{Lipschitz:decision} holds and the decision space $\mathcal{X}\subseteq\R^d$ is compact, then the function class defined in \eqref{centered function class} is $F$-Donsker.
\end{proposition}
Another convergence theorem we need is the argmax theorem that allows the translation of weak convergence from the empirical process to minimizers of its sample paths:
\begin{lemma}[Argmax theorem from Theorem 3.2.2 of \citeAPX{van1996weak}]\label{argmax theorem}
Let $\mathbb{M}_k,\mathbb{M}$ be stochastic processes indexed by a metric space $E$ such that $\mathbb{M}_k\Rightarrow \mathbb{M}$ in $l^{\infty}(K)$ for every compact $K\subseteq E$. Suppose that almost surely the sample path $\mathbb{M}(e),e\in E$ is lower semicontinuous and possesses a unique minimizer at a (random) point $e^*$, which as a random variable in $E$ is tight. If a sequence of random variables $e^*_k\in E$ is uniformly tight and satisfies $\mathbb{M}_k(e^*_k)\leq \inf_{e\in E}\mathbb{M}_k(e)+o_p(1)$, then $e^*_k\Rightarrow e^*$ in $E$.
\end{lemma}

Apart from the convergence of the empirical process $\mathbb{G}_k$, we will also need its moment bounds which we derive next through a series of results. For a vector $x\in R^d$, let $\Vert x\Vert$ be its $l_2$ norm, and for a random variable $X$ we define $\Vert X\Vert_p:=(E\abs{X}^p)^{1/p}$ for $p\geq 1$. We equip the function space $\mathcal F$ defined above with the norm $\Vert\cdot\Vert_2$. We denote by $N(\epsilon,\mathcal X,\Vert\cdot\Vert)$ the covering number, with ball size $\epsilon$, of the decision space, and by $N_{[\,]}(\epsilon,\mathcal F,\Vert\cdot\Vert_2)$ the bracketing number, with bracket size $\epsilon$, of the function space $\mathcal F$.

We need a few results adapted from \citeAPX{van1996weak}. The first result connects the complexity of the function space $\mathcal F$ to that of the decision space $\mathcal X$:
\begin{lemma}[Adapted from Theorem 2.7.11 of \citeAPX{van1996weak}]\label{bracket_by_cover}
Suppose Assumption \ref{Lipschitz:decision} holds and the decision space $\mathcal X\subseteq \R^d$ is compact, then for any $\epsilon>0$
\begin{equation*}
N_{[\,]}(4\epsilon\Vert M(\xi)\Vert_2,\mathcal F,\Vert\cdot\Vert_2)\leq N(\epsilon,\mathcal X,\Vert\cdot\Vert).
\end{equation*}
\end{lemma}
The second result gives an upper bound of the covering number of the decision space $\mathcal X$, hence an upper bound of the bracketing number of $\mathcal F$ because of the first result.
\begin{lemma}\label{cover_bd}
Let $D_{\mathcal X}$ be the diameter of the decision space $\mathcal X\subseteq \R^d$ with respect to the $L_2$ norm $\Vert\cdot\Vert$, then $N(\epsilon,\mathcal X,\Vert\cdot\Vert)\leq \big(3D_{\mathcal X}/\epsilon\big)^d$ for all $\epsilon\leq D_{\mathcal X}$.
\end{lemma}
\proof{Proof.}
Problem 6 in Section 2.1 of \citeAPX{van1996weak} states that the $\epsilon$-packing number of a Euclidean ball of radius $R$ in $\R^d$ is bounded above by $(3R/\epsilon\big)^d$, and the lemma follows from the fact that the covering number is always no more than the packing number and that $\mathcal X$ can be contained in a Euclidean ball of radius $D_{\mathcal X}$.\Halmos
\endproof

The third result relates the first order moment of the maximum deviation to the bracketing number of $\mathcal{F}$.
\begin{lemma}[Adapted from Theorem 2.14.2 of \citeAPX{van1996weak}]\label{deviation:first_order}
Let $\tilde{h}(\xi)=\sup_{x\in\mathcal X}\abs{h(x,\xi)-Z(x)}$. We have for all $k$
\begin{equation*}
\sqrt{k}E\Big[\sup_{x\in \mathcal X}\big\lvert\frac{1}{k}\sum_{i=1}^kh(x,\xi_i)-Z(x)\big\rvert\Big]\leq C\Vert\tilde{h}(\xi)\Vert_2\int_{0}^1\sqrt{1+\log N_{[\,]}(\epsilon \Vert\tilde{h}(\xi)\Vert_2,\mathcal F,\Vert\cdot\Vert_2)}d\epsilon
\end{equation*}
where $C$ is a universal constant.
\end{lemma}
We also need the following result that translates an upper bound of the first order moment to one for higher order moments:
\begin{lemma}[Adapted from Theorem 2.14.5 of \citeAPX{van1996weak}]\label{deviation:high_order}
For any $p\geq 2$ it holds
\begin{equation*}
\sqrt k\Big(E\Big[\sup_{x\in \mathcal X}\big\lvert\frac{1}{k}\sum_{i=1}^kh(x,\xi_i)-Z(x)\big\rvert^p\Big]\Big)^{\frac{1}{p}}\leq C\Big(\sqrt{k}E\Big[\sup_{x\in \mathcal X}\big\lvert\frac{1}{k}\sum_{i=1}^kh(x,\xi_i)-Z(x)\big\rvert\Big]+k^{\frac{1}{p}-\frac{1}{2}}\Vert\tilde{h}(\xi)\Vert_p\Big)
\end{equation*}
where $C$ is a constant depending only on $p$, and $\tilde h$ is the same as in Lemma \ref{deviation:first_order}.
\end{lemma}

Now we can derive moment bounds for the maximum deviation of the empirical process generated by the cost function. Specifically, we show that they can be controlled at the canonical rate $1/\sqrt k$ in the case of Lipschitz continuous cost function. We have:
\begin{proposition}\label{Lipschitz_k}
Suppose Assumptions \ref{Lipschitz:decision}, \ref{L2} and \ref{L2 strengthened} hold, and that the decision space $\mathcal X\subseteq\R^d$ is compact, then we have
\begin{equation*}
\sqrt{k}\Big(E\Big[\sup_{x\in \mathcal X}\big\lvert\frac{1}{k}\sum_{i=1}^kh(x,\xi_i)-Z(x)\big\rvert^{2+\delta}\Big]\Big)^{\frac{1}{2+\delta}}=O(1)\ \text{ as }k\to\infty
\end{equation*}
where $\delta$ is the same constant from Assumption \ref{L2 strengthened}.
\end{proposition}
\proof{Proof.}
First we conclude the following upper bound of the expected maximum deviation
\begin{eqnarray}
\notag&&\sqrt{k}E\Big[\sup_{x\in \mathcal X}\big\lvert\frac{1}{k}\sum_{i=1}^kh(x,\xi_i)-Z(x)\big\rvert\Big]\notag\\
\notag&\leq &C\Vert\tilde{h}(\xi)\Vert_2\int_{0}^1\sqrt{1+\log N\big(\frac{\epsilon \Vert \tilde h(\xi)\Vert_2}{4\Vert M(\xi)\Vert_2},\mathcal X,\Vert\cdot\Vert\big)}d\epsilon\text{\ \ by Lemmas \ref{deviation:first_order} and \ref{bracket_by_cover}}\notag\\
\notag &\leq &C\Vert\tilde{h}(\xi)\Vert_2\Big(1+\int_{0}^1\sqrt{\log N\big(\frac{\epsilon \Vert \tilde h(\xi)\Vert_2}{4\Vert M(\xi)\Vert_2},\mathcal X,\Vert\cdot\Vert\big)}d\epsilon\Big)\text{\ \ since\ }\sqrt{a+b}\leq \sqrt a+\sqrt b\notag\\
&\leq &C\Vert\tilde{h}(\xi)\Vert_2\Big(1+\int_{0}^{\frac{4D_{\mathcal X}\Vert M(\xi)\Vert_2}{\Vert \tilde h(\xi)\Vert_2}\wedge 1}\sqrt{d\log \frac{12D_{\mathcal X}\Vert M(\xi)\Vert_2}{\epsilon \Vert \tilde h(\xi)\Vert_2}}d\epsilon\Big)\\
\notag&&\text{\ \ \ by Lemma \ref{cover_bd} and }N(\epsilon,\mathcal X,\Vert\cdot\Vert)=1\text{ for }\epsilon\geq D_{\mathcal X}\notag\\
\notag&= &C\Vert\tilde{h}(\xi)\Vert_2+12CD_{\mathcal X}\Vert M(\xi)\Vert_2\int_{0}^{\frac{1}{3}\wedge \frac{\Vert \tilde h(\xi)\Vert_2}{12D_{\mathcal X}\Vert M(\xi)\Vert_2}}\sqrt{d\log \frac{1}{\epsilon}}d\epsilon\notag\\
&\leq &C'\Big(\Vert\tilde{h}(\xi)\Vert_2+\sqrt{d\log\big(3\vee\frac{12D_{\mathcal X}\Vert M(\xi)\Vert_2}{\Vert \tilde h(\xi)\Vert_2} \big)}(4D_{\mathcal X}\Vert M(\xi)\Vert_2\wedge \Vert\tilde{h}(\xi)\Vert_{2})\Big)<\infty\label{first_order_bd}
\end{eqnarray}
where $C'$ is another universal constant, and $\Vert\tilde{h}(\xi)\Vert_2<\infty$ because of Assumption \ref{L2}. Then we apply Lemma \ref{deviation:high_order} with $p=2+\delta$ to get
\begin{eqnarray*}
\sqrt{k}\Big(E\Big[\sup_{x\in \mathcal X}\big\lvert\frac{1}{k}\sum_{i=1}^kh(x,\xi_i)-Z(x)\big\rvert^{2+\delta}\Big]\Big)^{\frac{1}{2+\delta}}&\leq& C(\sqrt{k}E\Big[\sup_{x\in \mathcal X}\big\lvert\frac{1}{k}\sum_{i=1}^kh(x,\xi_i)-Z(x)\big\rvert\Big]+\Vert\tilde{h}(\xi)\Vert_{2+\delta})\\
&<&\infty
\end{eqnarray*}
where $\Vert\tilde{h}(\xi)\Vert_{2+\delta} < \infty$ because
\begin{eqnarray*}
\Vert\tilde{h}(\xi)\Vert_{2+\delta}^{2+\delta}&=&E_{\xi}\big[\sup_{x\in\mathcal X}\abs{h(x,\xi)-Z(x)}^{2+\delta}\big]\\
&\leq&E_{\xi}\big[\sup_{x\in\mathcal X}E_{\xi'}\abs{h(x,\xi)-h(x,\xi')}^{2+\delta}\big]\text{\ \ by Jensen's inequality}\\
&\leq&E_{\xi}E_{\xi'}\big[\sup_{x\in\mathcal X}\abs{h(x,\xi)-h(x,\xi')}^{2+\delta}\big]<\infty\text{\ \  by Assumption \ref{L2 strengthened}}.
\end{eqnarray*}
This concludes Proposition \ref{Lipschitz_k}.\Halmos
\endproof

\subsection{Measurability of the Optimum Functional}\label{subsec:measurability}
The statement of Theorem \ref{Lipschitz characterization of limit variance} involves the optimum functional $x^*_Y$ that is a minimizer of the limit Gaussian process, and our proof in Subsection \ref{subsec:main proof for Lipschitz case} also involves minimizers of the SAA. In this subsection we briefly certify their measurability for completeness using measurable selection theorems.

To introduce measurable selection, let $(\Omega, \mathcal{P}, P)$ be a probability space, $T$ be a compact metric space endowed with the Borel $\sigma$-algebra, and $\mathcal{G}$ be a set-valued function on $\Omega$ that maps each $\omega\in\Omega$ to a subset of $T$, then a measurable selection for $\mathcal{G}$ is a random variable (i.e., measurable function) $g:\Omega\to T$ such that $g(\omega)\in \mathcal{G}(\omega)$. We provide the following measurable selection theorem:
\begin{lemma}[Measurable selection]\label{measurable selection}
Let $(\Omega, \mathcal{P}, P)$ be a probability space, $T$ be a compact metric space endowed with the Borel $\sigma$-algebra. Let a function $v(t,\omega):T\times \Omega\to \R$ be such that $v(t,\cdot)$ is a random variable (i.e., measurable function) on $\Omega$ for each $t\in T$ and that the function $v(\cdot,\omega):T\to \R$ is continuous for each $\omega\in\Omega$. Then there exists a measurable function $g:\Omega \to T$ such that $v(g(\omega),\omega)=\min_{t\in T}v(t,\omega)$.
\end{lemma}
\proof{Proof.}
The proof is based on a classical measurable selection theorem, Theorem 5.3.1 from \citeAPX{srivastava2008course}. Once we show that $v(t,\omega)$ is measurable with respect to the product measure on the product space $T\times \Omega$, the lemma immediately follows from Theorem 5.3.1 in \citeAPX{srivastava2008course}. We thus prove product measurability. Since the space $T$ is compact, for any $\delta>0$ there exists a finite partition $\{T_{\delta}^1,T_{\delta}^2,\ldots \allowbreak, T_{\delta}^{M_{\delta}}\}$ of the space $T$ such that (1) $\sup_{t,t'\in T_{\delta}^m}d(t,t') < \delta$ for all $m=1,\ldots,M_{\delta}$, where $d(t,t')$ denotes the distance between $t$ and $t'$; (2) all $T_{\delta}^m$'s are measurable sets and disjoint. We choose $t_m\in T_{\delta}^m$ for each $m$, and then approximate $v$ via
$$\hat{v}_{\delta}(t,\omega)=\sum_{m=1}^{M_{\delta}}\mathbf{1}\{t\in T_{\delta}^m\}v(t_m,\omega).$$
Since a continuous function on a compact space is uniformly continuous, it is to see that $\lim_{\delta\to 0}\hat{v}_{\delta}(t,\omega)=v(t,\omega)$ for each $t\in T$ and $\omega\in\Omega$ on one hand. On the other hand, each $\hat{v}_{\delta}$ is measurable with respect to the product measure on $T\times \Omega$. Therefore, as the limit of $\hat{v}_{\delta}$, $v$ is also measurable with respect to the product measure on $T\times \Omega$.\Halmos
\endproof
In our setting, the metric space $T$ will be the compact decision space $\mathcal{X}$ or its quotient space (as defined later in Subsection \ref{subsec:main proof for Lipschitz case}), and the function $v$ from Lemma \ref{measurable selection} can be the SAA objective or the Gaussian process $Y$ on $\mathcal{X}^*$. Since both the SAA objective and the sample path of the Gaussian process are continuous with respect to the decision $x\in\mathcal{X}^*$, Lemma \ref{measurable selection} immediately ensures the existence of a measurable optimum of both the SAA and the Gaussian process.

\subsection{Main Proofs for Theorems \ref{Lipschitz characterization of limit variance}, \ref{convex nondegeneracy} and \ref{recover SAA}}\label{subsec:main proof for Lipschitz case}
\proof{Proof of Theorem \ref{Lipschitz characterization of limit variance}.}
The proof consists of three steps: We first restrict the optimization domain of the SAA from the whole decision space $\mathcal{X}$ to the set of optima $\mathcal{X}^*$ and show that the error incurred in the SAA and the limit variance is asymptotically negligible, then simplify the limit variance (with the full SAA replaced by the restricted SAA) that involves a growing size of data as the variance of a conditional expectation of the cost function through a probabilistic coupling argument, and finally use the argmax theorem and uniform integrability to conclude the desired convergence of the limit variance.

Before getting to the three steps, we define the so-call quotient space of the set of optima $\mathcal{X}^*$, denoted by $\overline{\mathcal{X}}^*$, that will be used in place of $\mathcal{X}^*$ to make the intrinsic semimetric $\rho(x_1,x_2)$ a metric in our setting. Formally, consider the equivalence relation $\sim$ on the set $\mathcal{X}^*$ defined by almost sure equality, i.e., $x_1\sim x_2$ if and only if $h(x_1,\xi)=h(x_2,\xi)$ almost surely. The set $\mathcal{X}^*$ can then be divided into disjoint equivalence classes such that for any $x_1,x_2\in\mathcal{X}^*$ we have $x_1\sim x_2$ if and only if they belong to the same equivalence class. The quotient space $\overline{\mathcal{X}}^*$ is then defined as the set of all equivalence classes of $\mathcal{X}^*$ with the metric
\begin{equation*}
    \overline{\rho}(\overline{x}_1, \overline{x}_2):=\rho(x_1,x_2)\text{ where }x_1\in\overline{x}_1\text{ and }x_2\in\overline{x}_2
\end{equation*}
for any $\overline{x}_1,\overline{x}_2\in \overline{\mathcal{X}}^*$. Note that the value of $\overline{\rho}(\overline{x}_1, \overline{x}_2)$ does not depend on the choice of $x_1$ and $x_2$ as long as they belong to $\overline{x}_1$ and $\overline{x}_2$ respectively, and that $\overline{\rho}(\overline{x}_1, \overline{x}_2)=0$ if and only if $\overline{x}_1$ and $\overline{x}_2$ are the same equivalence class. Therefore, $\overline{\mathcal{X}}^*$ is a metric space. Further more, it can be shown to be a compact space:
\begin{lemma}
The quotient space $\overline{\mathcal{X}}^*$ defined above as a metric space is compact.
\end{lemma}
\proof{Proof.}
Since $\overline{\mathcal{X}}^*$ is a metric space, it suffices to show that it is sequentially compact, i.e., every sequence in $\overline{\mathcal{X}}^*$ has a convergent subsequence with a limit in $\overline{\mathcal{X}}^*$. Let $\overline{x}_n$ be a sequence in $\overline{\mathcal{X}}^*$ and $x_n$ be a corresponding sequence in $\mathcal{X}^*$ such that $x_n\in\overline{x}_n$ for all $n$. Since $\mathcal{X}^*\subseteq \mathcal{X}$ is closed due to the Lipschitz continuity of $E[h(\cdot, \xi)]$ and $\mathcal{X}$ is compact, $\mathcal{X}^*$ is also compact, therefore there exists a subsequence $x_{n_i}$ of $x_n$ converging to some limit $x_{\infty}\in\mathcal{X}^*$. Let $\overline{x}_{\infty}\in\overline{\mathcal{X}}^*$ be the equivalence class containing $x_{\infty}$, then the subsequence $\overline{x}_{n_i}$ of $\overline{x}_n$ satisfies
% \begin{equation*}
$$\overline{\rho}(\overline{x}_{n_i}, \overline{x}_{\infty})=\rho(x_{n_i},x_{\infty})\to 0$$
% \end{equation*}
therefore converges to $\overline{x}_{\infty}$. This concludes the compactness of $\overline{\mathcal{X}}^*$.\Halmos
\endproof

We are now ready to present the main proof:
\vspace{5pt}

\textbf{Step One: Shrink the decision space of SAA from $\mathcal{X}$ to $\mathcal{X}^*$.} We define an approximation of the SAA optimal value
\begin{equation*}
    H_k^*(\xi_1,\ldots,\xi_k) = \min_{x\in\mathcal{X}^*}\frac{1}{k}\sum_{i=1}^kh(x,\xi_i)
\end{equation*}
and the corresponding $g_k^*(\xi)=E[H_k^*(\xi_1,\ldots,\xi_k)\vert \xi_1=\xi]$, where the original decision space $\mathcal{X}$ is replaced by the set of optima $\mathcal{X}^*$. We want to show that
\begin{equation}\label{neglibile error of limit variance}
    k\sqrt{Var(g_k(\xi))} - k\sqrt{Var(g_k^*(\xi))}=o(1) \text{ as }k\to\infty
\end{equation}
so that we can work with $k^2Var(g_k^*(\xi))$ instead without affecting the limit. According to the SAA asymptotics from \citeAPX{shapiro2009lectures} (equation (5.24) in Theorem 5.7) we have $\sqrt{k}(H_k-H^*_k)=o_p(1)$. We show that $k(H_k-H^*_k)^2$ is uniformly integrable, so that
\begin{equation}\label{diminishing gap between full and restricted SAA}
    kE[(H_k-H^*_k)^2]\to 0.
\end{equation}
To show uniform integrability, we state a simple lemma:
\begin{lemma}\label{continuity of SAA value}
We have $\max\{\abs{H_k-Z^*}, \abs{H_k^*-Z^*}\}\leq \sup_{x\in \mathcal X}\big\lvert\frac{1}{k}\sum_{i=1}^kh(x,\xi_i)-Z(x)\big\rvert$.
\end{lemma}
\proof{Proof.}
Let $x^*$ be an optimal solution of the original optimization \eqref{opt}, and $x^*_k$ be an optimal solution of the SAA formed by $\xi_1,\ldots,\xi_k$ on the original decision space $\mathcal{X}$. If $H_k\leq Z^*$, since $Z(x^*_k)\geq Z^*$, we have $\lvert H_k-Z^*\rvert \leq \lvert H_k-Z(x^*_k)\rvert\leq \sup_{x\in \mathcal X}\big\lvert\frac{1}{k}\sum_{i=1}^kh(x,\xi_i)-Z(x)\big\rvert$. Otherwise, if $H_k> Z^*$, then obviously $Z^*<H_k\leq \frac{1}{k}\sum_{i=1}^kh(x^*,\xi_i)$, hence again $\lvert H_k-Z^*\rvert \leq \lvert\frac{1}{k}\sum_{i=1}^kh(x^*,\xi_i)-Z(x^*)\rvert\leq \sup_{x\in \mathcal X}\big\lvert\frac{1}{k}\sum_{i=1}^kh(x,\xi_i)-Z(x)\big\rvert$. Since $x^*\in\mathcal{X}^*$, the inequality for $H_k^*$ follows from the exactly same argument with $\mathcal{X}$ replaced by $\mathcal{X}^*$.\Halmos
\endproof
Therefore Lemma \ref{continuity of SAA value} immediately forces that $\lvert H_k - H_k^*\rvert \leq 2\sup_{x\in\mathcal{X}}\lvert \frac{1}{k}\sum_{i=1}^kh(x,\xi_i) - Z(x) \rvert$, and hence $E[(\sqrt{k}\lvert H_k - H_k^* \rvert)^{2+\delta}]\leq E[2^{2+\delta}k^{1+\frac{\delta}{2}}\sup_{x\in\mathcal{X}}\lvert \frac{1}{k}\sum_{i=1}^kh(x,\xi_i) - Z(x) \rvert^{2+\delta}]=O(1)$, where the $O(1)$ bound comes from Proposition \ref{Lipschitz_k}, certifying the uniform integrability of $k(H_k-H^*_k)^2$. Now that we have $kE[(H_k-H^*_k)^2]\to 0$, we can write
\begin{eqnarray*}
\big\lvert \sqrt{Var(g_k(\xi))} - \sqrt{Var(g_k^*(\xi))} \big\rvert &\leq &\sqrt{Var(g_k(\xi) - g_k^*(\xi))} \text{\ \ by Cauchy-Schwarz inequality}\\
&\leq &\sqrt{Var(E[H_k(\xi_1,\ldots,\xi_k)-H_k^*(\xi_1,\ldots,\xi_k)\vert \xi_1=\xi])}\\
&= &\sqrt{\frac{1}{k}Var(\mathring{H_k-H_k^*})}\\
&&\text{\ \ where }\mathring{H_k-H_k^*}\text{ is the Hajek projection of }H_k-H_k^*\\
&\leq &\sqrt{\frac{1}{k}Var(H_k-H_k^*)}\\
&\leq &\sqrt{\frac{1}{k}E[(H_k-H_k^*)^2]}\\
&=&o\big(\frac{1}{k}\big).
\end{eqnarray*}
This proves \eqref{neglibile error of limit variance}.
\vspace{5pt}

\textbf{Step Two: Conditional expectations of the cost function $h$ as lower and upper bounds for $g_k^*(\xi)$.} We do so using a coupling argument that is similar to the one used in the proof of Lemma \ref{bound lemma}. To proceed, we have
\begin{eqnarray}
\notag k(g_k^*(\xi'_1) - E[g_k^*(\xi'_1)])&=&kE[H_k^*(\xi'_1,\ldots,\xi_k)\vert \xi'_1] - kE[H_k^*(\xi_1,\ldots,\xi_k)]\\
&&\text{\ \ where $\xi_1,\xi_1',\xi_2,\ldots,\xi_k$ are all independent}\notag\\
&=&E\left[\min_{x\in\mathcal X^*}\left\{h(x,\xi_1')+\sum_{i\neq1}h(x,\xi_i)\right\}\Bigg|\xi_1'\right]-E\left[\min_{x\in\mathcal X^*}\sum_{i=1}^kh(x,\xi_i)\right]\notag\\
&=&E\left[\min_{x\in\mathcal X^*}\left\{h(x,\xi_1')+\sum_{i\neq1}h(x,\xi_i)\right\}-\min_{x\in\mathcal X^*}\sum_{i=1}^kh(x,\xi_i)\Bigg|\xi_1'\right]\notag\\
\notag&\geq&E\left[h(x^*(\bm\xi'),\xi_1')+\sum_{i\neq1}h(x^*(\bm\xi'),\xi_i)-\sum_{i=1}^kh(x^*(\bm\xi'),\xi_i)\Bigg|\xi_1'\right]\\
\notag&&\text{\ \ where $x^*(\bm\xi')$ is a measurable optimum of the first SAA on }\mathcal{X}^*\\
\notag&=&E\left[h(x^*(\bm\xi'),\xi_1')-h(x^*(\bm\xi'),\xi_1)\big|\xi_1'\right]\\
\notag&=&E\left[h(x^*(\bm\xi'),\xi_1')-Z(x^*(\bm\xi'))\big|\xi_1'\right]\\
&=&E\left[h(x^*(\bm\xi'),\xi_1')-Z^*\big|\xi_1'\right]\text{\ \ since }x^*(\bm\xi')\in\mathcal{X}^*.
\end{eqnarray}
This gives a lower bound of $kg_k^*(\xi'_1)$. To obtain a similar upper bound, we use an optimal solution of the second SAA to write
\begin{eqnarray}
\notag k(g_k^*(\xi'_1)-E[g_k^*(\xi'_1)])&\leq&E\left[h(x^*(\bm\xi),\xi_1')+\sum_{i\neq1}h(x^*(\bm\xi),\xi_i)-\sum_{i=1}^kh(x^*(\bm\xi),\xi_i)\Bigg|\xi_1'\right]\\
\notag&&\text{\ \ where $x^*(\bm\xi)$ is a measurable optimum of the second SAA on }\mathcal{X}^*\\
\notag&=&E\left[h(x^*(\bm\xi),\xi_1')-h(x^*(\bm\xi),\xi_1)\big|\xi_1'\right].
\end{eqnarray}
Denoting by $\underline{g}_k(\xi'_1)$ and $\overline{g}_k(\xi'_1)$ the lower and upper bound functions derived above respectively, our plan is to show that
\begin{equation}\label{limit variance convergence}
    E[(\underline{g}_k(\xi'_1))^2]\to Var(E[h(x^*_Y,\xi)\vert \xi])\text{ as }k\to\infty
\end{equation}
and
\begin{equation}\label{diminishing gap of limit variance}
    E[(\overline{g}_k(\xi'_1) - \underline{g}_k(\xi'_1))^2]\to 0\text{ as }k\to\infty
\end{equation}
in Step three below, and then the conclusion of the theorem immediately follows because
\begin{eqnarray*}
k\sqrt{Var(g_k^*(\xi))} &=& \sqrt{E[(k(g_k^*(\xi'_1)-E[g_k^*(\xi'_1)]))^2]}\\
&\leq &\sqrt{E[(\underline{g}_k(\xi'_1))^2]} + \sqrt{E[(k(g_k^*(\xi'_1)-E[g_k^*(\xi'_1)]) - \underline{g}_k(\xi'_1))^2]}\text{\ \ by Minkowski inequality}\\
&\leq &\sqrt{E[(\underline{g}_k(\xi'_1))^2]} + \sqrt{E[(\overline{g}_k(\xi'_1) - \underline{g}_k(\xi'_1))^2]}\\
&\to& \sqrt{Var(E[h(x^*_Y,\xi)\vert \xi])}
\end{eqnarray*}
and similarly
\begin{eqnarray*}
k\sqrt{Var(g_k^*(\xi))}&\geq &\sqrt{E[(\underline{g}_k(\xi'_1))^2]} - \sqrt{E[(k(g_k^*(\xi'_1)-E[g_k^*(\xi'_1)]) - \underline{g}_k(\xi'_1))^2]}\\
&\geq &\sqrt{E[(\underline{g}_k(\xi'_1))^2]} - \sqrt{E[(\overline{g}_k(\xi'_1) - \underline{g}_k(\xi'_1))^2]}\\
&\to& \sqrt{Var(E[h(x^*_Y,\xi)\vert \xi])}
\end{eqnarray*}
hence $\lim_{k\to\infty}k\sqrt{Var(g_k(\xi))}=\lim_{k\to\infty}k\sqrt{Var(g_k^*(\xi))}= \sqrt{Var(E[h(x^*_Y,\xi)\vert \xi])}$. The equality between $Var(E[h(x^*_Y,\xi)\vert \xi])$ and $E[\kappa(x_Y^*,{x_Y^*}')]$ comes from rewriting the variance as
\begin{eqnarray*}
Var(E[h(x^*_Y,\xi)\vert \xi])&=&E[(E[h(x^*_Y,\xi)-Z^*\vert \xi])^2]\\
&=&E[E[(h(x^*_Y,\xi)-Z^*)(h({x_Y^*}',\xi)-Z^*)\vert \xi]]\text{\ \ by independence of $x_Y^*,{x_Y^*}'$}\\
&=&E[E[(h(x^*_Y,\xi)-Z^*)(h({x_Y^*}',\xi)-Z^*)\vert x_Y^*,{x_Y^*}']]\\
&=&E[\kappa(x_Y^*,{x_Y^*}')]
\end{eqnarray*}
where $Z^*=E[h(x_Y^*,\xi)]$ is the optimal value.
\vspace{5pt}

\textbf{Step Three: Prove \eqref{limit variance convergence} and \eqref{diminishing gap of limit variance}. }We need to work with the quotient space $\overline{\mathcal{X}}^*$ instead, and for a given $x\in\mathcal{X}^*$ we denote by $\overline{x}\in\overline{\mathcal{X}}^*$ the equivalence class that contains $x$. For convenience, we shall abuse the notation a bit by writing $h(\overline{x},\xi):=h(x,\xi)$. The lower and upper bounding functions can then be written in the form
\begin{equation*}
\begin{aligned}
    \underline{g}_k(\xi'_1) &= E\left[h(\overline{x}^*(\bm\xi'),\xi_1')-Z^*\big|\xi_1'\right]\\
    \overline{g}_k(\xi'_1) &= E\left[h(\overline{x}^*(\bm\xi),\xi_1')-h(\overline{x}^*(\bm\xi),\xi_1)\big|\xi_1'\right]
\end{aligned}
\end{equation*}
where $\overline{x}^*(\bm\xi'),\overline{x}^*(\bm\xi)\in\overline{\mathcal{X}}^*$. We first apply the argmax theorem to conclude the weak convergence of $\overline{x}^*(\bm\xi'),\overline{x}^*(\bm\xi)$. We need a result on the uniqueness of Gaussian process optimum:
\begin{lemma}[Uniqueness of Guassian process optimum, Lemma 2.6 in \citeAPX{kim1990cube}]\label{uniqueness of GP optimum}
Let $G$ be a Guassian process indexed by a compact metric space $T$ such that $G$ has continuous sample paths and $Var(G(t_1)-G(t_2))>0$ for every $t_1,t_2\in T$ and $t_1\neq t_2$. Then, with probability one, a sample path of $G$ achieves minimum at a unique point in $T$.
\end{lemma}
Recall that, when confined to $\mathcal{X}^*$, the empirical process $\mathbb{G}_k$ from \eqref{empirical process} is a scaled SAA objective, the limit Gaussian process $\mathbb{G}$ is the Gaussian process $Y$. The way the quotient space $\overline{\mathcal{X}}^*$ is constructed ensures that, for every $\overline{x}_1,\overline{x}_2\in \overline{\mathcal{X}}^*$ such that $\overline{x}_1\neq \overline{x}_2$, we have $Var(h(\overline{x}_1,\xi) - h(\overline{x}_2,\xi))>0$, therefore Lemma \eqref{uniqueness of GP optimum} implies that $Y$ has a unique (measurable) minimizer $\overline{x}_Y^*$ on $\overline{\mathcal{X}}^*$. This verifies the uniqueness condition of the minimizer of the limit process that is required in Lemma \ref{argmax theorem}. For a fixed $\xi'_1$, the optimum $\overline{x}^*(\bm{\xi}')$ is a solution for the scaled SAA objective for $\overline{x}\in\overline{\mathcal{X}}^*$ defined as
\begin{eqnarray*}
    \mathbb{G}_k(\xi'_1)&:=&\sqrt{k}\Big(\frac{1}{k}\big(h(\overline{x},\xi'_1)+\sum_{i\neq 1}h(\overline{x},\xi_i)\big) - Z(\overline{x})\Big)\\
    &=&\frac{1}{\sqrt{k}}h(\overline{x},\xi'_1) + \frac{\sqrt{k}}{\sqrt{k-1}}\cdot \sqrt{k-1}\Big(\frac{1}{k-1}\sum_{i\neq 1}h(\overline{x},\xi_i) - Z(\overline{x})\Big).
\end{eqnarray*}
Note that since all other $\xi_i,i\neq 1$ are independent of $\xi'_1$, the term $\sqrt{k-1}\big(1/(k-1)\cdot \sum_{i\neq 1}h(\overline{x},\xi_i) - Z(\overline{x})\big)$ is an empirical process with the same weak limit $Y$. The term $1/\sqrt{k}\cdot h(\overline{x},\xi'_1)\Rightarrow 0$ and $\sqrt{k}/\sqrt{k-1}\to 1$, therefore by Slutsky's theorem $\mathbb{G}_k(\xi'_1)\Rightarrow Y$ on $\overline{\mathcal{X}}^*$ for every fixed $\xi'_1$. By Lemma \ref{argmax theorem} we have $\overline{x}^*(\bm{\xi}')\Rightarrow \overline{x}_Y^*$ for every fixed $\xi'_1$. Since the cost function is Lipschitz continuous and $\mathcal{X}^*$ is compact, the cost function $h(\cdot, \xi'_1)$ is also continuous on $\overline{\mathcal{X}}^*$ with respect to the metric $\overline{\rho}$. Therefore by the continuous mapping theorem we have \begin{equation*}
    h(\overline{x}^*(\bm\xi'),\xi_1')\Rightarrow h(\overline{x}_Y^*, \xi_1')\text{ as $k\to\infty$ for every fixed }\xi_1'.
\end{equation*}
Similarly we can show that
\begin{equation*}
    h(\overline{x}^*(\bm\xi),\xi_1') - h(\overline{x}^*(\bm\xi),\xi_1)\Rightarrow h(\overline{x}_Y^*, \xi_1')-h(\overline{x}_Y^*, \xi_1)\text{ as $k\to\infty$ for every fixed }\xi_1',\xi_1.
\end{equation*}
Since $h(\overline{x}^*(\bm\xi'),\xi_1')$ and $h(\overline{x}^*(\bm\xi),\xi_1') - h(\overline{x}^*(\bm\xi),\xi_1)$ for fixed $\xi_1,\xi_1'$ are bounded due to the continuity of $h$ on $\overline{\mathcal{X}}^*$ and compactness, uniform integrability holds and we immediately have that
% \begin{eqnarray*}
%     \underline{g}_k(\xi'_1) &\to& E\left[h(\overline{x}_Y^*,\xi_1')-Z^*\big|\xi_1'\right]\\
%     \overline{g}_k(\xi'_1) &\to& E\left[h(\overline{x}_Y^*,\xi_1')-h(\overline{x}_Y^*,\xi_1)\big|\xi_1'\right]=E\left[h(\overline{x}_Y^*,\xi_1')-Z^*\big|\xi_1'\right]
% \end{eqnarray*}
\begin{equation*}
\begin{aligned}
    \underline{g}_k(\xi'_1) &\to E\left[h(\overline{x}_Y^*,\xi_1')-Z^*\big|\xi_1'\right]\\
    \overline{g}_k(\xi'_1) &\to E\left[h(\overline{x}_Y^*,\xi_1')-h(\overline{x}_Y^*,\xi_1)\big|\xi_1'\right]=E\left[h(\overline{x}_Y^*,\xi_1')-Z^*\big|\xi_1'\right]
\end{aligned}
\end{equation*}
almost surely as $k\to\infty$, where $\overline{x}_Y^*$ is independent of $\xi_1,\xi_1'$. We now want to prove uniform integrability of $\underline{g}_k^2(\xi'_1)$ and $(\overline{g}_k(\xi'_1) -\underline{g}_k(\xi'_1))^2$ to conclude \eqref{limit variance convergence} and \eqref{diminishing gap of limit variance}. To this end we write
\begin{eqnarray*}
E\big[\lvert \underline{g}_k(\xi_1')\rvert^{2+\delta}\big] &=& E\big[\lvert E\big[h(\overline{x}^*(\bm{\xi}'),\xi_1')-Z^*\vert \xi_1'\big] \rvert^{2+\delta}\big]\\
&\leq & E\big[E\big[\lvert h(\overline{x}^*(\bm{\xi}'),\xi_1')-Z^*\rvert^{2+\delta} \vert \xi_1'\big]\big]\text{\ \ by Jensen's inequality}\\
&\leq &E\big[E\big[\lvert h(\overline{x}^*(\bm{\xi}'),\xi_1')-h(\overline{x}^*(\bm{\xi}'),\xi_1)\rvert^{2+\delta} \vert \xi_1'\big]\big]\text{\ \ again by Jensen's inequality}\\
&\leq &E\big[E\big[\sup_{\overline{x}\in\overline{\mathcal{X}}^*}\lvert  h(\overline{x},\xi_1')-h(\overline{x},\xi_1)\rvert^{2+\delta} \vert \xi_1'\big]\big]\\
&\leq &E\big[\sup_{x\in\mathcal{X}}\lvert  h(x,\xi_1')-h(x,\xi_1)\rvert^{2+\delta}\big]\leq \infty\text{\ \ by Assumption \ref{L2 strengthened}}.
\end{eqnarray*}
Hence $\underline{g}_k^2(\xi'_1)$ is uniformly integrable. The same argument also proves uniform integrability of the upper bounding function $\overline{g}_k^2(\xi_1')$, and hence that of $(\overline{g}_k(\xi_1')-\underline{g}_k(\xi_1'))^2$. We can therefore conclude \eqref{limit variance convergence} and \eqref{diminishing gap of limit variance} by noting that $h(x_Y^*,\xi)=h(\overline{x}_Y^*,\xi)$. This completes the proof.\Halmos
\endproof

\proof{Proof of Theorem \ref{recover SAA}.}
We have shown in the proof of Theorem \ref{Lipschitz characterization of limit variance} that $E[(H_k - H_k^*)^2] = o(1/k)$ (see equation \ref{diminishing gap between full and restricted SAA}), and we have $Var(H_k^*)=(1/k)\cdot Var(h(x^*,\xi))$ when the optimum is essentially unique. This allows us to bound the high-order variance of the SAA kernel as
\begin{eqnarray*}
E(H_k-\mathring{H}_k)^2&=&Var(H_k)-Var(\mathring{H}_k)\\
&\leq& Var(H_k - H_k^*) + 2\sqrt{Var(H_k - H_k^*)Var(H_k^*)} + Var(H_k^*) - Var(\mathring{H}_k)\\
&&\text{\ \ \ by Cauchy-Schwarz inequality}\\
&\leq& E[(H_k - H_k^*)^2] + 2\sqrt{E[(H_k - H_k^*)^2]Var(H_k^*)} + Var(H_k^*) - Var(\mathring{H}_k)\\
&=& o\big(\frac{1}{k}\big) + 2\sqrt{o\big(\frac{1}{k}\big)O\big(\frac{1}{k}\big)} + \frac{1}{k}\cdot Var(h(x^*,\xi)) - kVar(g_k(\xi))\\
&=&o\big(\frac{1}{k}\big) + \frac{Var(h(x^*,\xi)) - k^2Var(g_k(\xi))}{k}\\
&=&o\big(\frac{1}{k}\big)\text{\ \ by Theorem \ref{Lipschitz characterization of limit variance}}.
\end{eqnarray*}
Therefore \eqref{range:k} holds with $k\leq n$ and hence the central limit theorem $\sqrt{n}(U_{n,k} - W_k)\Rightarrow N(0, Var(h(x^*,\xi)))$ follows from Theorem \ref{meta clt} and that $k^2Var(g_k(\xi))\to Var(h(x^*,\xi))$ as ensured by Theorem \ref{Lipschitz characterization of limit variance}.

The bias $W_k-Z^*$ can be bounded as
\begin{eqnarray*}
\lvert W_k - Z^*\rvert &=& \lvert E[H_k - H_k^*]\rvert\text{\ \ since }H_k^*=\frac{1}{k}\sum_{i=1}^k h(x^*,\xi_i)\\
&\leq &E[\lvert H_k - H_k^*\rvert]\text{\ \ by Jensen's inequality}\\
&\leq &\sqrt{E[(H_k - H_k^*)^2]}\\
&=&o\big(\frac{1}{\sqrt{k}}\big).
\end{eqnarray*}
The central limit theorem $\sqrt{n}(U_{n,k}-Z^*)\Rightarrow N(0, Var(h(x^*,\xi)))$ with $k\geq \epsilon n$ then follows from the bias being $W_k-Z^*=o(1/k)=o(1/\sqrt{\epsilon n})=o(1/\sqrt{n})$.\Halmos
\endproof

\proof{Proof of Theorem \ref{convex nondegeneracy}.}
Under Assumption \ref{Lipschitz:decision} the expected cost function $Z(x)$ is continuous on $\mathcal{X}$, which together with the condition that $\mathcal{X}$ is convex and compact implies that the set of optima $\mathcal{X}^*$ is convex and compact.

We argue that, for almost every $\xi$, $h(x,\xi)$ must be affine in $x$ on $\mathcal{X}^*$. Since $h(x,\xi)$ is convex in $x$, we have $h(\lambda x_1 + (1-\lambda) x_2, \xi) \leq \lambda h(x_1,\xi) + (1-\lambda) h(x_2,\xi)$ for every $\lambda\in[0, 1]$ and every $x_1,x_2\in\mathcal{X}^*$. The convexity of $\mathcal{X}^*$ implies that $\lambda x_1 + (1-\lambda) x_2 \in \mathcal{X}^*$ and hence $E[h(\lambda x_1 + (1-\lambda) x_2, \xi)] = \lambda E[h(x_1,\xi)] + (1-\lambda) E[h(x_2,\xi)]=Z^*$, which implies that for each fixed $(x_1,x_2,\lambda)\in \mathcal{X}^*\times \mathcal{X}^* \times [0,1]$ we have $h(\lambda x_1 + (1-\lambda) x_2, \xi) = \lambda h(x_1,\xi) + (1-\lambda) h(x_2,\xi)$ for almost every $\xi$. Since $\mathcal{X}^*$ as a compact set in $\R^d$ is separable, i.e., $\mathcal{X}^*$ has a countable dense (with respect to the standard Euclidean distance) subset, and so does the closed interval $[0,1]$, therefore the product space $\mathcal{X}^*\times \mathcal{X}^* \times [0,1]$ is also separable with a countable dense subset $\mathcal{S}\subseteq \mathcal{X}^*\times \mathcal{X}^* \times [0,1]$. By the countability of $\mathcal{S}$ we have for almost every $\xi$ that $h(\lambda x_1 + (1-\lambda) x_2, \xi) = \lambda h(x_1,\xi) + (1-\lambda) h(x_2,\xi)$ for every $(x_1,x_2,\lambda)\in \mathcal{S}$. Since both sides of the equality are continuous with respect to $x_1,x_2,\lambda$, equality on a dense subset implies global equality, i.e., for almost every $\xi$ we have $h(\lambda x_1 + (1-\lambda) x_2, \xi) = \lambda h(x_1,\xi) + (1-\lambda) h(x_2,\xi)$ for every $(x_1,x_2,\lambda)\in \mathcal{X}^*\times \mathcal{X}^* \times [0,1]$. This proves that $h(x,\xi)$ is affine on $\mathcal{X}^*$.

By affineness, there exists an $a(\xi)\in \R^d$ and a $b(\xi)\in \R$ such that $h(x,\xi)=a(\xi)^Tx + b(\xi)$ for $x\in\mathcal{X}^*$, and in particular, $h(x_Y^*,\xi)=a(\xi)^Tx_Y^* + b(\xi)$. Therefore the limit variance from Theorem \ref{Lipschitz characterization of limit variance} can be expressed as
\begin{eqnarray*}
Var(E[h(x_Y^*,\xi)\vert \xi])&=&Var(E[a(\xi)^Tx_Y^* + b(\xi)\vert \xi])\\
&=&Var(a(\xi)^TE[x_Y^*] + b(\xi))\\
&=&Var(h(E[x_Y^*], \xi)).
\end{eqnarray*}
This proves the theorem.\Halmos
\endproof

\section{Proof of Theorems \ref{consistency_IJ:U} and \ref{consistency_IJ:V}}\label{proof:consistency}
\proof{Proof of Theorem \ref{consistency_IJ:U}.}
\citeAPX{wager2017estimation} provides a proof in the context of random forests. Since their proof can be adapted to our optimization context, we shall directly borrow some intermediate results there which hold for general symmetric kernels and U-statistics, and only focus on parts that rely on the particular SAA kernel considered here. Readers are referred to the proof of Theorem 9 in \citeAPX{wager2017estimation} for explanations of the borrowed results.

%  Theorems \ref{main thm} and \ref{recover SAA} can be viewed as special cases of Theorem \ref{meta clt} where $E(H_k-\mathring{H}_k)^2$ is $O(1)$, $O(1/k)$ and $o(1/k)$ respectively.
% setting under Theorem \ref{new_clt:U} and, if not implied by \eqref{range:k}, the additional requirement $k\leq \theta n$ for some $\theta<1$

Note that, in both Theorems \ref{main thm} and \ref{recover SAA}, the resample size $k$ is chosen such that \eqref{range:k} holds, therefore it suffices to prove the theorem under the relaxed condition that \eqref{range:k} holds and that $k\leq \theta n$ for some $\theta<1$. The IJ variance estimator now can be expressed as
\begin{eqnarray}
\nonumber\frac{n^2}{(n-k)^2}\sum_{i=1}^n\mathrm{Cov}_*^2(N_i^*,H_k^*)&=&\frac{n^2}{(n-k)^2}\sum_{i=1}^n(E_*[H_k^*\sum_{j=1}^k\mathbf{1}(\xi_{i_j}=\xi_i)]-E_*[N_i^*]E_*[H_k^*])^2\\
\nonumber&=&\frac{n^2}{(n-k)^2}\sum_{i=1}^n(kE_*[H_k^*\mathbf{1}(\xi_{i_1}=\xi_i)]-\frac{k}{n}U_{n,k})^2\\
&=&\frac{n^2}{(n-k)^2}\frac{k^2}{n^2}\sum_{i=1}^n(E_*[H_k^*|\xi_{i_1}=\xi_i]-U_{n,k})^2\label{IJ:U}\\
\nonumber&=&\frac{k^2}{(n-k)^2}\sum_{i=1}^n(A_i+R_i)^2
\end{eqnarray}
where $\xi_{i_1},\ldots,\xi_{i_k}$ are resampled from $\xi_1,\ldots,\xi_n$ without replacement, and
\begin{align*}
A_i&=E_*[\mathring{H}_k^*|\xi_{i_1}=\xi_i]-E_*[\mathring{H}_k^*]\\
R_i&=E_*[H_k^*-\mathring{H}_k^*|\xi_{i_1}=\xi_i]-E_*[H_k^*-\mathring{H}_k^*].
\end{align*}
We aim to show that
\begin{equation}\label{A+R}
\frac{k^2}{(n-k)^2}\sum_{i=1}^nA_i^2 = \frac{k^2Var(g_k(\xi))}{n}+o_p\big(\frac{1}{n}\big),\;\;\frac{k^2}{(n-k)^2}\sum_{i=1}^nR_i^2=o_p\big(\frac{1}{n}\big).
\end{equation}
Since $k^2Var(g_k(\xi))=kVar(\mathring{H}_k)\leq kVar(H_k)=O(1)$ by Proposition \ref{var bound for SAA kernel}, we see that $k^2/(n-k)^2\cdot\sum_{i=1}^nA_i^2=O_p(1/n)$ and hence the cross term $k^2/(n-k)^2\cdot \sum_{i=1}^n2A_iR_i=o_p(1/n)$. Therefore the desired conclusion immediately follows once \eqref{A+R} is proved.

First we deal with $R_i$'s. Lemma 13 in \citeAPX{wager2017estimation} shows that
\begin{align*}
ER_i^2=\sum_{s=2}^{k}(a_s+b_s)V_s^H
\end{align*}
where
\begin{equation*}
\begin{aligned}
a_s&=\binom{n-1}{s-1}\Big(\binom{k-1}{s-1}\Big/\binom{n-1}{s-1}-\binom{k}{s}\Big/\binom{n}{s}\Big)^2\\
b_s&=\binom{n-1}{s}\Big(\binom{k-1}{s}\Big/\binom{n-1}{s}-\binom{k}{s}\Big/\binom{n}{s}\Big)^2
\end{aligned}
\end{equation*}
with $b_k=0$, and $V_s^H$ is the variance of the $s$-th order function in the ANOVA decomposition of $H_k$ (see the discussion after Lemma \ref{anova}). Note that $Var(H_k)=\sum_{s=1}^k\binom{k}{s}V_s^H$ and $Var(\mathring{H}_k)=kV_1^H$. Some basic algebra shows that
\begin{equation*}
\frac{a_{s+1}/\binom{k}{s+1}}{a_s/\binom{k}{s}}=\frac{(s+1)(k-s)}{s(n-s)},\;\frac{b_{s+1}/\binom{k}{s+1}}{b_s/\binom{k}{s}}=\frac{(s+1)^2(k-s)}{s^2(n-s-1)}.
\end{equation*}
Therefore, if $k\leq \theta n$ for $\theta<1$, the above two ratios are both less than one when $s\geq s^*:=\max\{2,\lceil\sqrt \theta/(1-\sqrt\theta)\rceil\}$, meaning that the maximum of $a_s/\binom{k}{s}$ or $b_s/\binom{k}{s}$ over $s$ is attained at some $s\leq s^*$. Moreover, by upper bounding $(k-s)/(n-s-1)< 1$ we have for all $s\leq s^*$ that $b_s/\binom{k}{s}/\big(b_2/\binom{k}{2}\big)\leq s^2/4\leq s^{*2}/4$ and that $a_s/\binom{k}{s}/\big(a_2/\binom{k}{2}\big)\leq s/2\leq s^*/2\leq s^{*2}/4$. Hence
\begin{align*}
ER_i^2&\leq\frac{s^{*2}}{4}\frac{a_2+b_2}{\binom{k}{2}}\sum_{s=2}^{s^*}\binom{k}{s}V_s^H+\sum_{s=s^*+1}^{k}\frac{a_s+b_s}{\binom{k}{s}}\binom{k}{s}V_s^H\\
&\leq \frac{s^{*2}}{4}\frac{a_2+b_2}{\binom{k}{2}}\sum_{s=2}^{k}\binom{k}{s}V_s^H\leq C(\theta)\frac{(n-k)^2}{n^3}E(H_k-\mathring{H}_k)^2
\end{align*}
where $C(\theta)$ is a constant that only depends on $\theta$. This bound implies
%$ER_i^2\leq CE(H_k-\mathring{H}_k)^2/n$ for some universal constant $C$. Note that here the expectation is taken with respect to the original data $\xi_1,\ldots,\xi_n$ rather than the resampled data. In that lemma they only conclude the weaker result $ER_i^2\leq CVar(H_k)/n$, however, it is obvious from their proof that $Var(H_k)$ can be replaced by $E(H_k-\mathring{H}_k)^2$ since $R_i$ contains only high order ANOVA terms of $H_k$. This bound implies
\begin{equation}\label{IJ:remainder}
E\big[\frac{k^2}{(n-k)^2}\sum_{i=1}^nR_i^2\big]=O\big(\frac{k^2}{n^2}E(H_k-\mathring{H}_k)^2\big)=o\big(\frac{1}{n}\big)
\end{equation}
where the second equality follows from the condition \ref{range:k}.

Now we analyze the $A_i$'s. Lemma 12 in \citeAPX{wager2017estimation} shows that
\begin{equation*}
A_i=\big(1-\frac{k}{n}\big)(g_k(\xi_i)-W_k)+\big(\frac{k-1}{n-1}-\frac{k}{n}\big)\sum_{j\neq i}(g_k(\xi_j)-W_k)
\end{equation*}
therefore one can write
\begin{equation*}
\frac{(n-1)^2k^2}{n^2(n-k)^2}\sum_{i=1}^nA_i^2=\frac{k^2}{n}\Big(\frac{1}{n}\sum_{i=1}^n(g_k(\xi_i)-W_k)^2-(\bar g_k-W_k)^2\Big),\text{ where }\bar g_k=\frac{1}{n}\sum_{i=1}^ng_k(\xi_i).
\end{equation*}
Since $E[k^2(\bar g_k-W_k)^2/n]=k^2Var(g_k(\xi))/n^2=O(1/n^2)=o(1/n)$ it suffices to prove
\begin{equation}\label{g_k:LLN}
\frac{1}{n}\sum_{i=1}^n(g_k(\xi_i)-W_k)^2=Var(g_k(\xi))+o_p\big(\frac{1}{k^2}\big)
\end{equation}
in order to justify the first equality in \eqref{A+R}. To proceed, we write
\begin{eqnarray*}
    \frac{1}{n}\sum_{i=1}^n(g_k(\xi_i)-W_k)^2 &=& \frac{1}{n}\sum_{i=1}^n(g_k(\xi_i)-W_k)^2\cdot \mathbf{1}\{k^2Var(g_k(\xi))>\frac{1}{n^{\delta/(4+2\delta)}}\} +\\
    &&\hspace{2em}\frac{1}{n}\sum_{i=1}^n(g_k(\xi_i)-W_k)^2\cdot \mathbf{1}\{k^2Var(g_k(\xi))\leq \frac{1}{n^{\delta/(4+2\delta)}}\}
\end{eqnarray*}
where $\delta$ is the constant from Assumption \ref{L2 strengthened}, we aim to show \eqref{g_k:LLN} in either case. When $k^2Var(g_k(\xi))>1/n^{\delta/(4+2\delta)}$, we can calculate the expected value
\begin{equation*}
    E\Big[\frac{1}{n}\sum_{i=1}^n(g_k(\xi_i)-W_k)^2\Big] = Var(g_k(\xi))=O\big(\frac{1}{k^2n^{\delta/(4+2\delta)}}\big)
\end{equation*}
and therefore
\begin{eqnarray*}
\frac{1}{n}\sum_{i=1}^n(g_k(\xi_i)-W_k)^2-Var(g_k(\xi))&=&O_p\big(\frac{1}{k^2n^{\delta/(4+2\delta)}}\big)-Var(g_k(\xi))\\
&=&O_p\big(\frac{1}{k^2n^{\delta/(4+2\delta)}}\big)-O\big(\frac{1}{k^2n^{\delta/(4+2\delta)}}\big)\\
&=&o_p\big(\frac{1}{k^2}\big).
\end{eqnarray*}
To prove \eqref{g_k:LLN} when $k^2Var(g_k(\xi))>1/n^{\delta/(4+2\delta)}$, we need the following weak law of large numbers:
\begin{lemma}[Theorem 2.2.9 from \citeAPX{durrett2010probability}]\label{WLLN}
For each $n$ let $Y_{n,i},1\leq i\leq n$ be independent. Let $b_n>0$ with $b_n\to\infty$, and let $\bar{Y}_{n,i}=Y_{n,i}\mathbf{1}(\abs{Y_{n,i}}\leq b_n)$. Suppose that, as $n\to \infty$, $\sum_{i=1}^nP(\abs{Y_{n,i}}>b_n)\to 0$ and $b_n^{-2}\sum_{i=1}^nE\bar{Y}_{n,i}^2\to 0$, then
\begin{equation*}
\frac{\sum_{i=1}^nY_{n,i}-\sum_{i=1}^nE\bar{Y}_{n,i}}{b_n}\stackrel{p}{\to}0.
\end{equation*}
\end{lemma}
We apply the weak law to $Y_{n,i}=(g_k(\xi_i)-W_k)^2/Var(g_k(\xi))$ with $b_n=n$. We verify the two conditions $\sum_{i=1}^nP(\abs{Y_{n,i}}>b_n)\to 0$ and $b_n^{-2}\sum_{i=1}^nE\bar{Y}_{n,i}^2\to 0$. The first condition can be verified as
\begin{eqnarray*}
nP\big(\frac{(g_k(\xi_i)-W_k)^2}{Var(g_k(\xi))}>n\big)&=&nP(\abs{g_k(\xi_i)-W_k}^{2+\delta}>(nVar(g_k(\xi)))^{1+\frac{\delta}{2}})\\
&\leq&\frac{n}{(nVar(g_k(\xi)))^{1+\frac{\delta}{2}}}E\abs{g_k(\xi_i)-W_k}^{2+\delta}\text{\ by Markov inequality}\\
&\leq&\frac{n}{(nVar(g_k(\xi)))^{1+\frac{\delta}{2}}}\frac{\tilde M}{k^{2+\delta}}\text{\ by the proof of Theorem \ref{meta clt}}\\
&=&\frac{\tilde M}{n^{\frac{\delta}{2}}(k^2Var(g_k(\xi)))^{1+\frac{\delta}{2}}}=O(n^{-\frac{\delta}{4}})\to 0
\end{eqnarray*}
and the second condition is verified as
\begin{eqnarray*}
\frac{1}{n}E\left[\frac{(g_k(\xi_i)-W_k)^4}{(Var(g_k(\xi)))^2}\mathbf{1}\big(\frac{(g_k(\xi_i)-W_k)^2}{Var(g_k(\xi))}\leq n\big)\right]&\leq&\frac{1}{n}E\left[\frac{\abs{g_k(\xi_i)-W_k}^{2+\delta}}{(Var(g_k(\xi)))^{1+\frac{\delta}{2}}}n^{1-\frac{\delta}{2}}\mathbf{1}\big(\frac{(g_k(\xi_i)-W_k)^2}{Var(g_k(\xi))}\leq n\big)\right]\\
&\leq&\frac{1}{n^{\frac{\delta}{2}}}E\left[\frac{\abs{g_k(\xi_i)-W_k}^{2+\delta}}{(Var(g_k(\xi)))^{1+\frac{\delta}{2}}}\right]\\
&\leq &\frac{\tilde M}{n^{\frac{\delta}{2}}(k^2Var(g_k(\xi)))^{1+\frac{\delta}{2}}}=O(n^{-\frac{\delta}{4}})\to 0.
\end{eqnarray*}
The weak law of large number thus applies, and it remains to show that each $E\bar{Y}_{n,i}\to0$. This is proved as
\begin{eqnarray*}
&&\abs{1-E\left[\frac{(g_k(\xi_i)-W_k)^2}{Var(g_k(\xi))}\mathbf{1}\big(\frac{(g_k(\xi_i)-W_k)^2}{Var(g_k(\xi))}\leq n\big)\right]}\\
&=&\abs{E\left[\frac{(g_k(\xi_i)-W_k)^2}{Var(g_k(\xi))}\mathbf{1}\big(\frac{(g_k(\xi_i)-W_k)^2}{Var(g_k(\xi))}> n\big)\right]}\\
&\leq&\left(E\left[\frac{\abs{g_k(\xi_i)-W_k}^{2+\delta}}{(Var(g_k(\xi)))^{1+\frac{\delta}{2}}}\right]\right)^{\frac{2}{2+\delta}}\left(P\big(\frac{(g_k(\xi_i)-W_k)^2}{Var(g_k(\xi))}> n\big)\right)^{\frac{\delta}{2+\delta}}\text{\ \ by Holder's inequality}\\
&\leq &\left(\frac{\tilde M}{(k^2Var(g_k(\xi)))^{1+\frac{\delta}{2}}}\right)^{\frac{2}{2+\delta}}\left(\frac{1}{n}\right)^{\frac{\delta}{2+\delta}}=O\big(n^{-\frac{\delta}{4+2\delta}}\big)\to 0\text{\ \ by Markov inequality}.
\end{eqnarray*}
With all these conditions verified, we can conclude
\begin{eqnarray*}
    \frac{1}{n}\sum_{i=1}^n(g_k(\xi_i)-W_k)^2&=&Var(g_k(\xi))(1+o_p(1))\\
    &=&Var(g_k(\xi))+o_p(Var(g_k(\xi)))\\
    &=&Var(g_k(\xi))+o_p\big(\frac{1}{k^2}\big)
\end{eqnarray*}
from Lemma \ref{WLLN} in the case that $k^2Var(g_k(\xi))>1/n^{\delta/(4+2\delta)}$. Combining the two cases $k^2Var(g_k(\xi))\leq 1/n^{\delta/(4+2\delta)}$ and $k^2Var(g_k(\xi))>1/n^{\delta/(4+2\delta)}$ proves \eqref{g_k:LLN}, and hence completes the proof.\Halmos
\endproof

\proof{Proof of Theorem \ref{consistency_IJ:V}.}
Given Theorem \ref{consistency_IJ:U}, it suffices to show that the IJ variance estimator under resampling with replacement differs by only $o_p(1/n)$ from the one without replacement. Since quantities under both resampling with and without replacement will be involved in this proof, we attach $*$ to quantities under resampling without replacement, and $\tilde *$ to those with replacement. Note that $k=O(n^{\gamma})$ for some $\gamma<1/2$ which implies $n^2/(n-k)^2\to 1$, so the without-replacement IJ variance estimate without the factor $n^2/(n-k)^2$, i.e.~$\sum_{i=1}^n\mathrm{Cov}_*^2(N_i^*,H_k^*)$, is also consistent. We have
\begin{equation}\label{IJ:V}
\sum_{i=1}^n\mathrm{Cov}_{\tilde*}^2(N_i^{\tilde*},H_k^{\tilde*})=\frac{k^2}{n^2}\sum_{i=1}^n(E_{\tilde*}[H_k^{\tilde*}|\xi_{i_1}=\xi_i]-V_{n,k})^2
\end{equation}
where $\xi_{i_1},\ldots,\xi_{i_k}$ are resampled from $\xi_1,\ldots,\xi_n$ with replacement. By comparing \eqref{IJ:U} (without $n^2/(n-k)^2$) and \eqref{IJ:V} and using Cauchy Schwartz inequality
\begin{eqnarray*}
&&\abs{\sum_{i=1}^n\mathrm{Cov}_{\tilde*}^2(N_i^{\tilde*},H_k^{\tilde*})-\sum_{i=1}^n\mathrm{Cov}_*^2(N_i^*,H_k^*)}\\
&\leq& \frac{k^2}{n^2}\sum_{i=1}^n(v_i-u_i)^2+2\sqrt{\sum_{i=1}^n\mathrm{Cov}_{*}^2(N_i^{*},H_k^{*})\cdot\frac{k^2}{n^2}\sum_{i=1}^n(v_i-u_i)^2}
\end{eqnarray*}
where $v_i=E_{\tilde*}[H_k^{\tilde*}|\xi_{i_1}=\xi_i]-V_{n,k}$ and $u_i=E_{*}[H_k^*|\xi_{i_1}=\xi_i]-U_{n,k}$. If we show that $E(V_{n,k}-U_{n,k})^2=o(1/n)$ and $E(E_{\tilde*}[H_k^{\tilde*}|\xi_{i_1}=\xi_i]-E_*[H_k^*|\xi_{i_1}=\xi_i])^2=o(1/n)$, then $E[\sum_{i=1}^n(v_i-u_i)^2]=o(1)$ and under the condition $k=O(n^{\gamma})$ with $\gamma<1/2$ we have
\begin{equation*}
\sum_{i=1}^n\mathrm{Cov}_{\tilde*}^2(N_i^{\tilde*},H_k^{\tilde*})-\sum_{i=1}^n\mathrm{Cov}_*^2(N_i^*,H_k^*)=\frac{k^2}{n^2}o_p(1)+\sqrt{o_p\big(\frac{1}{n}\cdot\frac{k^2}{n^2}\big)}=o_p\big(\frac{1}{n}\big)
\end{equation*}
which concludes the theorem.

The first error $E(V_{n,k}-U_{n,k})^2=o(1/n)$ has been proved in the proof of Theorem \ref{V thm} (equation \eqref{diff of U and V statistics}). The second error $E(E_{\tilde*}[H_k^{\tilde*}|\xi_{i_1}=\xi_i]-E_*[H_k^*|\xi_{i_1}=\xi_i])^2=o(1/n)$ needs some further analysis. We study $E(E_{\tilde*}[H_k^{\tilde*}|\xi_{i_1}=\xi_1]-E_*[H_k^*|\xi_{i_1}=\xi_1])^2$ without loss of generality. Given that the first resampled data point $\xi_{i_1}$ is $\xi_1$, for any fixed integer $l\geq 0$ we obtain the following decomposition of $E_{\tilde*}[H_k^{\tilde*}|\xi_{i_1}=\xi_1]$ similar to that in the proof of Theorem \ref{V thm}
\begin{equation*}
n^{k-1}E_{\tilde*}[H_k^{\tilde*}|\xi_{i_1}=\xi_1]=\sum_{s=k-1-l}^{k-1}c(n-1,k-1,s)A_{s}+(n^{k-1}
-\sum_{s=k-1-l}^{k-1}c(n-1,k-1,s))R_l
\end{equation*}
where $A_s$ is the average of all $H_k(\xi_1,\xi_{i_2},\ldots,\xi_{i_k})$'s where $\xi_{i_2},\ldots,\xi_{i_k}$ contain exactly $s$ distinct data and none of them is $\xi_1$, and $R_l$ is the average of all other $H_k(\xi_1,\xi_{i_2},\ldots,\xi_{i_k})$'s. Note that, in particular, $A_{k-1}=E_{*}[H_k^*|\xi_{i_1}=\xi_1]$. We have the following analog of \eqref{UV_decompose}
\begin{eqnarray*}
&&n^{k-1}(E_{*}[H_k^*|\xi_{i_1}=\xi_1]-E_{\tilde*}[H_k^{\tilde*}|\xi_{i_1}=\xi_1])\\
&=&(n^{k-1}-\sum_{s=k-1-l}^{k-1}c(n-1,k-1,s))(A_{k-1}-R_l)-\sum_{s=k-1-l}^{k-2}c(n-1,k-1,s)(A_{s}-A_{k-1}).
\end{eqnarray*}
Note that the coefficient of the first term does not match the form of \eqref{UV_decompose}, but we have
\begin{equation*}
n^{k-1}-\sum_{s=k-1-l}^{k-1}c(n-1,k-1,s)=n^{k-1}-(n-1)^{k-1}+\sum_{s=1}^{k-l-2}c(n-1,k-1,s).
\end{equation*}
Like in the proof of Theorem \ref{V thm}
\begin{align*}
&\sum_{s=1}^{k-l-2}c(n-1,k-1,s)=O\big(\big(\frac{k^2}{n}\big)^{l+1}(n-1)^{k-1}\big),\;E(A_{k-1}-R_l)^2=O(1)\\
&c(n-1,k-1,s)=O(k^{2(k-1-s)}n^s)\text{ and }E(A_{s}-A_{k-1})^2=O\big(\frac{1}{k^2}\big)\text{ for }s\geq k-1-l.
\end{align*}
Moreover by Bernoulli's inequality $(1+x)^r\geq 1+rx$ for any integer $r\geq 0$ and real $x\geq -1$
\begin{equation*}
n^{k-1}-(n-1)^{k-1}=n^{k-1}(1-(1-\frac{1}{n})^{k-1})\leq n^{k-2}(k-1).
\end{equation*}
With all these bounds and Minkowski inequality we get
\begin{eqnarray*}
&&E(E_{*}[H_k^*|\xi_{i_1}=\xi_1]-E_{\tilde*}[H_k^{\tilde*}|\xi_{i_1}=\xi_1])^2\\
&=&O\left(\big(\frac{k}{n}+\big(\frac{k^2}{n}\big)^{l+1}\big)^2E(A_{k-1}-R_l)^2+\sum_{s=k-1-l}^{k-2}\big(\frac{k^2}{n}\big)^{2(k-1-s)}E(A_{s}-A_{k-1})^2\right)\\
&=&O\left(\big(\frac{k}{n}+\big(\frac{k^2}{n}\big)^{l+1}\big)^2+\frac{k^2}{n^2}\right)=o\big(\frac{1}{n}\big)
\end{eqnarray*}
when $l$ is chosen according to \eqref{max_order}.\Halmos
\endproof

\section{Proof of Theorem \ref{final_guarantee} and Corollary \ref{coverage}}\label{proof:final}
\proof{Proof of Theorem \ref{final_guarantee}.}
We denote by $\sigma^2_{IJ}$ the infinitesimal jackknife (IJ) variance estimate $n^2/(n-k)^2\sum_{i=1}^n\mathrm{Cov}_*^2(N_i^*,H_k^*)$ in the case of resampling without replacement, or $\sum_{i=1}^n\mathrm{Cov}_*^2(N_i^*,H_k^*)$ in the case of resampling with replacement. We prove the following three statements:
\begin{align}
    &\tilde Z_{n,k}^{bag}-U_{n,k}=o_p\big(\frac{1}{\sqrt{n}}\big),\tilde Z_{n,k}^{bag}-V_{n,k}=o_p\big(\frac{1}{\sqrt{n}}\big)\label{point estimate error}\\
    &\tilde{\sigma}^2_{IJ}=\sigma^2_{IJ}+o_p\big(\frac{1}{n}\big)\label{var estimate error}\\
    &\sqrt{\tilde{\sigma}^2_{IJ}+o_p\big(\frac{1}{n}\big)}=\tilde{\sigma}_{IJ} + o_p\big(\frac{1}{\sqrt{n}}\big).\label{sqrt var error}
\end{align}
Once we have these three results, the desired conclusion follows from the representation from Theorem \ref{main thm} as
\begin{eqnarray*}
\tilde Z_{n,k}^{bag}-W_k&=& U_{n,k}(\text{or }V_{n,k})-W_k+o_p\big(\frac{1}{\sqrt{n}}\big)\\
&=&\frac{k\sqrt{Var(g_k(\xi))}}{\sqrt{n}}\mathcal{Z}_{n,k}+o_p\big(\frac{1}{\sqrt{n}}\big)\text{\ \ by Theorem \ref{main thm} or \ref{V thm}}\\
&=&\sqrt{\sigma^2_{IJ}+o_p\big(\frac{1}{n}\big)}\mathcal{Z}_{n,k}+o_p\big(\frac{1}{\sqrt{n}}\big)\text{\ \ by Theorem \ref{consistency_IJ:U} or \ref{consistency_IJ:V}}\\
&=&\sqrt{\tilde \sigma^2_{IJ}+o_p\big(\frac{1}{n}\big)}\mathcal{Z}_{n,k}+o_p\big(\frac{1}{\sqrt{n}}\big)\text{\ \ by \eqref{var estimate error}}\\
&=&\big(\tilde \sigma_{IJ}+o_p\big(\frac{1}{\sqrt{n}}\big)\big)\mathcal{Z}_{n,k}+o_p\big(\frac{1}{\sqrt{n}}\big)\text{\ \ by \eqref{sqrt var error}}\\
&=&\tilde \sigma_{IJ}\mathcal{Z}_{n,k}+o_p\big(\frac{1}{\sqrt{n}}\big).
\end{eqnarray*}

% One is that $\tilde Z_{n,k}^{bag}-U_{n,k}=o_p(1/\sqrt n)$ when resampling without replacement, or $\tilde Z_{n,k}^{bag}-V_{n,k}=o_p(1/\sqrt n)$ with replacement, so that by Slutsky's theorem the CLTs still hold with $U_{n,k}$ or $V_{n,k}$ replaced by their estimate $\tilde Z_{n,k}^{bag}$.

% The first task is relatively easy. Note that $\tilde Z_{n,k}^{bag}$ is unbiased (for estimating $U_{n,k}$ and $V_{n,k}$ respectively) in either case, and
% \begin{equation}\label{bound:H_k^2}
% Var_*(\tilde Z_{n,k}^{bag})=\frac{1}{B}Var_*(H_k^{*})\leq \frac{1}{B}E_*H_k^{*2}\leq \frac{1}{Bn}\sum_{i=1}^n\sup_{x\in\mathcal X}\abs{h(x,\xi_i)}^2
% \end{equation}
% where the last inequality follows from the argument used in \eqref{interim22}. Due to Assumption \ref{L2} and the strong law of large numbers $\sum_{i=1}^n\sup_{x\in\mathcal X}\abs{h(x,\xi_i)}^2/n\stackrel{p}{\to}E\sup_{x\in\mathcal X}\abs{h(x,\xi)}^2<\infty$, hence $Var_*(\tilde Z_{n,k}^{bag})=O_p(1/B)$. If $B/(kn)\to\infty$ we have
% \begin{equation*}
% E_*(\tilde Z_{n,k}^{bag}-U_{n,k})^2=o_p\big(\frac{1}{kn}\big)=o_p\big(\frac{1}{n}\big),\;E_*(\tilde Z_{n,k}^{bag}-V_{n,k})^2=o_p\big(\frac{1}{kn}\big)=o_p\big(\frac{1}{n}\big)
% \end{equation*}
% For a non-negative random variable, if its conditional expectation is of order $o_p(1)$, then itself is also $o_p(1)$. Therefore $(\tilde Z_{n,k}^{bag}-U_{n,k})^2=o_p(1/n)$ and $(\tilde Z_{n,k}^{bag}-V_{n,k})^2=o_p(1/n)$.

We first prove \eqref{var estimate error}. We need to show
\begin{equation}\label{IJ var error in terms of cov}
    \abs{\sum_{i=1}^n\widehat{Cov}^2_*(N_i^*,\hat Z_k^*)-\sum_{i=1}^nCov_*^2(N_i^*,H_k^*)}=o_p(1/n).
\end{equation}
Note that showing this error also proves \eqref{var estimate error} for resampling without replacement, because the condition $k\leq \theta n$ for some $\theta<1$ implies $1\leq n^2/(n-k)^2\leq1/(1-\theta)^2$, hence the error remains $o_p(1/n)$ after multiplying the factor $n^2/(n-k)^2$ on both sides.

We first deal with resampling without replacement. By Cauchy Schwartz inequality the Monte Carlo error can be bounded as
\begin{equation*}
\abs{\sum_{i=1}^n\widehat{Cov}^2_*(N_i^*,\hat Z_k^*)-\sum_{i=1}^nCov_*^2(N_i^*,H_k^*)}\leq \sum_{i=1}^n(\widehat{Cov}_i-Cov_i)^2+2\sqrt{\sum_{i=1}^nCov_i^2\sum_{i=1}^n(\widehat{Cov}_i-Cov_i)^2}
\end{equation*}
where $Cov_i=Cov_*(N_i^*,H_k^*)$ and $\widehat{Cov}_i=\widehat{Cov}^2_*(N_i^*,\hat Z_k^*)$ for short. Since $\sum_{i=1}^nCov_i^2$ is the desired IJ variance of order $O_p(1/n)$, we only need to show $\sum_{i=1}^n(\widehat{Cov}_i-Cov_i)^2=o_p(1/n)$. By computing variances of the sample covariances one can get
\begin{eqnarray}
\nonumber&&E_*\big[\sum_{i=1}^n(\widehat{Cov}_i-Cov_i)^2\big]\\
\nonumber&\leq& \sum_{i=1}^n\left(\frac{1}{B}E_*[(H_k^*-E_*H_k^*)^2(N_i^*-\frac{k}{n})^2]+\frac{1}{B^2}Var_*(H_k^*)Var_*(N_i^*)+\frac{2}{B}Cov_i^2\right)\\
&\leq&\frac{1}{B}E_*[(H_k^*-E_*H_k^*)^2\sum_{i=1}^n(N_i^*-\frac{k}{n})^2]+\frac{1}{B^2}Var_*(H_k^*)\sum_{i=1}^nVar_*(N_i^*)+\frac{2}{B}\sum_{i=1}^nCov_i^2.\label{cov_bd}
\end{eqnarray}
Note that $\sum_{i=1}^nCov_i^2=O_p(1/n)$, and $\sum_{i=1}^n(N_i^*-\frac{k}{n})^2=k(n-k)/n,Var_*(N_i^*)=k(n-k)/n^2$ since $N_i^*=0$ or $1$ and $\sum_{i=1}^nN_i^*=k$. To bound $Var_*(H_k^*)$, we consider bounding its expected value
\begin{eqnarray*}
E[Var_*(H_k^*)] &=& E[E_*[H_k^{*2}]] - E[U_{n,k}^2]\\
&=&E\Big[\frac{1}{n(n-1)\cdots (n-k+1)}\sum_{i_1< i_2<\cdots< i_k}H_k^2(\xi_{i_1},\xi_{i_2},\ldots,\xi_{i_k})\Big] - E[U_{n,k}^2] \\
&=&\frac{1}{n(n-1)\cdots (n-k+1)}\sum_{i_1< i_2<\cdots< i_k}E[H_k^2(\xi_{i_1},\xi_{i_2},\ldots,\xi_{i_k})] - E[U_{n,k}^2]\\
&=&W_k^2 + Var(H_k) - (W_k^2+Var(U_{n,k}))\\
&\leq &Var(H_k)=O\big(\frac{1}{k}\big)\text{\ \ by Proposition \ref{var bound for SAA kernel}}.
\end{eqnarray*}
Therefore $Var_*(H_k^*)=O_p(1/k)$. With all these bounds, we have from \eqref{cov_bd} that
\begin{equation*}
E_*\big[\sum_{i=1}^n(\widehat{Cov}_i-Cov_i)^2\big]=O_p\big(\frac{1}{B}+\frac{1}{B^2}+\frac{1}{Bn}\big)=O_p\big(\frac{1}{B}\big).
\end{equation*}
If $B/n\to \infty$, then $E_*\big[\sum_{i=1}^n(\widehat{Cov}_i-Cov_i)^2\big]=o_p(1/n)$, which implies $\sum_{i=1}^n(\widehat{Cov}_i-Cov_i)^2=o_p(1/n)$, i.e., \eqref{IJ var error in terms of cov}.

In the case of resampling with replacement, we have the same bound as \eqref{cov_bd}, where $Var_{*}(N_i^{*})=k(n-1)/n^2$ and $\sum_{i=1}^nCov_i^2=O_p(1/n)$. To bound $Var_*(H_k^{*})$, note that resampling with replacement is essentially i.i.d. sampling from the uniform distribution over $\{\xi_1,\ldots,\xi_n\}$, therefore proceeding as in the proof of Proposition \ref{var bound for SAA kernel} gives
\begin{equation*}
    Var_*(H_k^{*})\leq \frac{1}{2kn^2}\sum_{i_1\neq i_2}\sup_{x\in\mathcal{X}}\lvert h(x,\xi_{i_1}) - h(x,\xi_{i_2}) \rvert^2
\end{equation*}
and hence $E[Var_*(H_k^{*})]\leq 1/(2kn^2)\cdot n(n-1)E\big[\sup_{x\in\mathcal{X}}\lvert h(x,\xi) - h(x,\xi') \rvert^2\big]=O(1/k)$ by Assumption \ref{L2 strengthened}. This shows that $Var_*(H_k^{*})=O_p(1/k)$. We also need to bound $E_*[(H_k^*-E_*H_k^*)^2\sum_{i=1}^n(N_i^*-\frac{k}{n})^2]$, which requires a more careful analysis than for replacement without replacement. For each $i=1,\ldots,n$ and $j=1,\ldots, k$, define $\eta_{i,j}=1$ if the $\xi_i$ is the $j$-th resampled data and $0$ otherwise, and then we can write $N_i^*=\sum_{j=1}^k\eta_{i,j}$ and
\begin{eqnarray}
\notag E_*[(H_k^*-E_*H_k^*)^2\sum_{i=1}^n(N_i^*-\frac{k}{n})^2]&=&E_*[(H_k^*-E_*H_k^*)^2\big(\frac{k(n-k)}{n}+\sum_{i=1}^n\sum_{j_1\neq j_2}\eta_{i,j_1}\eta_{i,j_2}\big)]\\
\notag &=&\frac{k(n-k)}{n}\cdot Var_*(H_k^*) + \sum_{i=1}^n\sum_{j_1\neq j_2}E_*[(H_k^*-E_*H_k^*)^2\eta_{i,j_1}\eta_{i,j_2}]\\
\notag &=&O_p(1) + \sum_{i=1}^n\sum_{j_1\neq j_2}E_*[(H_k^*-E_*H_k^*)^2\eta_{i,j_1}\eta_{i,j_2}]\\
&=&O_p(1) + k(k-1)\sum_{i=1}^nE_*[(H_k^*-E_*H_k^*)^2\eta_{i,1}\eta_{i,2}]\label{bound term star}
\end{eqnarray}
where the last equality holds because of the symmetry of $H_k^*$.
% We analyze the second term in \eqref{bound term star} using Lemma \ref{bound lemma} with the distribution of $\xi$ being the uniform distribution over $\{\xi_1,\ldots,\xi_n\}$. For any $\tilde \xi_1,\tilde \xi_2$ let $g_{k,2}^*(\tilde \xi_1,\tilde \xi_2)=E_*[H_k(\xi_{i_1},\ldots,\xi_{i_k})\vert \xi_{i_1}=\tilde \xi_1,\xi_{i_2}=\tilde \xi_2]$, where each $i_j$ is uniform in $\{1,2,\ldots,n\}$ for $j\leq k$. We can write
% \begin{eqnarray*}
% E_*[(H_k^*-E_*H_k^*)^2\eta_{i,1}\eta_{i,2}]&=&E_*[(H_k^*-E_*H_k^*)^2\cdot \mathbf{1}\{i_1=i_2=i\}]\\
% &=& \frac{1}{n^2}((g_{k,2}^*(\xi_i,\xi_i)-E_*H_k^*)^2+Var_*(H_k^*\vert i_1=i_2=i))\\
% &\leq& \frac{4}{n^2k^2}\Big(\frac{1}{n}\sum_{i'=1}^n\sup_{x\in\mathcal{X}}\lvert h(x,\xi_i)-h(x,\xi_i')\rvert\Big)^2+\frac{Var_*(H_k^*\vert i_1=i_2=i))}{n^2}\\
% &&\text{\ \ by Lemma \ref{bound lemma}}\\
% \end{eqnarray*}
$H_k^*=H_k(\xi_{i_1},\xi_{i_2},\xi_{i_3},\ldots,\xi_{i_k})$ and ${H_k^*}'=H_k(\xi_{i'_1},\xi_{i'_2},\xi_{i_3},\ldots,\xi_{i_k})$ are two resampled SAAs where only the first two data points can be different. We then have the bound
\begin{equation*}
\lvert H_k^*-{H_k^*}'\rvert\leq \frac{1}{k}\cdot \big(\sup_{x\in\mathcal{X}}\lvert h(x,\xi_{i_1})-h(x,\xi_{i'_1}) \rvert+\sup_{x\in\mathcal{X}}\lvert h(x,\xi_{i_2})-h(x,\xi_{i'_2}) \rvert\big)
\end{equation*}
from Lemma \ref{bound lemma} and hence by Minkowski inequality we can write
\begin{eqnarray*}
\sqrt{E_*[(H_k^*-{H_k^*}')^2\eta_{i,1}\eta_{i,2}]}&=&\sqrt{E_*[\big(\lvert H_k^*-{H_k^*}'\rvert\cdot \mathbf{1}\{i_1=i_2=i\}\big)^2]}\\
&\leq& \frac{1}{k}\sqrt{E_*[\big(\sup_{x\in\mathcal{X}}\lvert h(x,\xi_{i_1})-h(x,\xi_{i'_1}) \rvert\cdot \mathbf{1}\{i_1=i_2=i\}\big)^2]}+\\
&&\hspace{2em}\frac{1}{k}\sqrt{E_*[\big(\sup_{x\in\mathcal{X}}\lvert h(x,\xi_{i_2})-h(x,\xi_{i'_2}) \rvert\cdot \mathbf{1}\{i_1=i_2=i\}\big)^2]}\\
&=&\frac{2}{k}\sqrt{E_*[\big(\sup_{x\in\mathcal{X}}\lvert h(x,\xi_{i_1})-h(x,\xi_{i'_1}) \rvert\cdot \mathbf{1}\{i_1=i_2=i\}\big)^2]}\\
&&\text{\ \ by the equivalence of $i_1,i_2$}\\
&=&\frac{2}{kn}\sqrt{E_*[\big(\sup_{x\in\mathcal{X}}\lvert h(x,\xi_i)-h(x,\xi_{i'_1}) \rvert\big)^2]}\\
&=&\frac{2}{kn}\sqrt{\frac{1}{n}\sum_{i'=1}^n\big(\sup_{x\in\mathcal{X}}\lvert h(x,\xi_i)-h(x,\xi_{i'}) \rvert\big)^2}
\end{eqnarray*}
Applying Minkowski inequality again we can write
\begin{eqnarray*}
\sqrt{E_*[(H_k^*-E_*H_k^*)^2\eta_{i,1}\eta_{i,2}]}&\leq&\sqrt{E_*[({H_k^*}'-E_*H_k^*)^2\eta_{i,1}\eta_{i,2}]}+\sqrt{E_*[(H_k^*-{H_k^*}')^2\eta_{i,1}\eta_{i,2}]}\\
&\leq&\sqrt{E_*[({H_k^*}'-E_*H_k^*)^2]E_*[\eta_{i,1}\eta_{i,2}]} + \frac{2}{kn}\sqrt{\frac{1}{n}\sum_{i'=1}^n\big(\sup_{x\in\mathcal{X}}\lvert h(x,\xi_i)-h(x,\xi_{i'}) \rvert\big)^2}\\
&\leq&\frac{1}{n}\sqrt{Var_*(H_k^*)} + \frac{2}{kn}\sqrt{\frac{1}{n}\sum_{i'=1}^n\big(\sup_{x\in\mathcal{X}}\lvert h(x,\xi_i)-h(x,\xi_{i'}) \rvert\big)^2}
\end{eqnarray*}
Substituting this bound into \eqref{bound term star} we get
\begin{eqnarray*}
&&\sum_{i=1}^nE_*[(H_k^*-E_*H_k^*)^2\eta_{i,1}\eta_{i,2}]\\
&\leq&\sum_{i=1}^n\Big(\frac{1}{n}\sqrt{Var_*(H_k^*)} + \frac{2}{kn}\sqrt{\frac{1}{n}\sum_{i'=1}^n\big(\sup_{x\in\mathcal{X}}\lvert h(x,\xi_i)-h(x,\xi_{i'}) \rvert\big)^2}\Big)^2\\
&\leq&\sum_{i=1}^n\Big(\frac{2}{n^2}Var_*(H_k^*) + \frac{8}{k^2n^3}\sum_{i'=1}^n\big(\sup_{x\in\mathcal{X}}\lvert h(x,\xi_i)-h(x,\xi_{i'}) \rvert\big)^2\Big)\text{\ \ by Young's inequality}\\
&=&\frac{2}{n}Var_*(H_k^*) + \frac{8}{k^2n^3}\sum_{i,i'=1}^n\big(\sup_{x\in\mathcal{X}}\lvert h(x,\xi_i)-h(x,\xi_{i'}) \rvert\big)^2\\
&=&O_p\big(\frac{1}{nk}\big)+O_p\big(\frac{1}{nk^2}\big)=O_p\big(\frac{1}{nk}\big).
\end{eqnarray*}
Since $k\leq n$, this shows that \eqref{bound term star} is overall $O_p(1)$ and hence $E_*[(H_k^*-E_*H_k^*)^2\sum_{i=1}^n(N_i^*-\frac{k}{n})^2]=O_p(1)$. With all these bounds, we have from \eqref{cov_bd} for resampling with replacement that
\begin{equation*}
E_*\big[\sum_{i=1}^n(\widehat{Cov}_i-Cov_i)^2\big]=O_p\big(\frac{1}{B}+\frac{1}{B^2}+\frac{1}{Bn}\big)=O_p\big(\frac{1}{B}\big).
\end{equation*}
If $B/n\to \infty$, then $E_*\big[\sum_{i=1}^n(\widehat{Cov}_i-Cov_i)^2\big]=o_p(1/n)$, which implies $\sum_{i=1}^n(\widehat{Cov}_i-Cov_i)^2=o_p(1/n)$, i.e., \eqref{IJ var error in terms of cov}.
% \begin{eqnarray*}
% \lvert H_k^*-{H_k^*}'\rvert \eta_{i,1}\eta_{i,2}]&=&E_*[(H_k^*-{H_k^*}')^2\cdot \mathbf{1}\{i_1=i_2=i\}]\\
% &\leq&\frac{1}{k^2n^2}E_*\big[\big(\sup_{x\in\mathcal{X}}\lvert h(x,\xi_i)-h(x,\xi_{i'_1}) \rvert+\sup_{x\in\mathcal{X}}\lvert h(x,\xi_i)-h(x,\xi_{i'_2}) \rvert\big)^2\big]\\
% &&\text{\ \ by Lemma \ref{}}
% \end{eqnarray*}

We then prove \eqref{point estimate error}. Note that, on one hand, $\tilde Z_{n,k}^{bag}$ is unbiased (for estimating $U_{n,k}$ and $V_{n,k}$ respectively) for both resampling with and without replacement. On the other hand, when proving \eqref{var estimate error} we have already shown that $Var_*(H_k^*)=O_p(1/k)$ for both cases. Therefore $E_*[(\tilde Z_{n,k}^{bag}-U_{n,k})^2]=O_p(1/(Bk))=o_p(1/n)$ and $E_*[(\tilde Z_{n,k}^{bag}-V_{n,k})^2]=O_p(1/(Bk))=o_p(1/n)$ for each case. For a non-negative random variable, if its conditional expectation is of order $o_p(1)$, then itself is also $o_p(1)$, hence $(\tilde Z_{n,k}^{bag}-U_{n,k})^2=o_p(1/n)$ and $(\tilde Z_{n,k}^{bag}-V_{n,k})^2=o_p(1/n)$ and \eqref{point estimate error} is proved.

Lastly, we prove \eqref{sqrt var error}. Consider the expression $\sqrt{a+b}-\sqrt{a}$. For any $\epsilon > 0$, if $a>\epsilon/2$ and $\lvert b\rvert \leq \epsilon/4$ we have $\lvert \sqrt{a+b}-\sqrt{a}\rvert\leq \frac{\lvert b\rvert}{\sqrt{\epsilon}}$, and if $0\leq a\leq \epsilon/2$ we have $\lvert \sqrt{a+b}-\sqrt{a}\rvert \leq \sqrt{\epsilon/2+b} + \sqrt{\epsilon/2}$. Thus we can bound the probability
\begin{eqnarray*}
    &&P(\lvert \sqrt{n\tilde{\sigma}^2_{IJ}+o_p(1)} - \sqrt{n}\tilde{\sigma}^2_{IJ}\rvert>2\sqrt{\epsilon})\\
    &\leq& P(n\tilde{\sigma}^2_{IJ}>\epsilon/2\text{ and }\lvert o_p(1)\rvert>\frac{\epsilon}{4})+P(n\tilde{\sigma}^2_{IJ}>\epsilon/2,\lvert o_p(1)\rvert\leq \frac{\epsilon}{4}\text{ and } \frac{\lvert o_p(1)\rvert}{\sqrt{\epsilon}}>2\sqrt{\epsilon})+\\
    &&\hspace{2em}P(n\tilde{\sigma}^2_{IJ}\leq \epsilon/2\text{ and }\sqrt{\frac{\epsilon}{2} + o_p(1)}+\sqrt{\frac{\epsilon}{2}}>2\sqrt{\epsilon})\\
    &\leq& P(\lvert o_p(1)\rvert>\frac{\epsilon}{4})+0 + P(o_p(1)>2(2-\sqrt{2})\epsilon)\to 0.
\end{eqnarray*}
This shows that $\sqrt{n\tilde{\sigma}^2_{IJ}+o_p(1)} - \sqrt{n}\tilde{\sigma}^2_{IJ}=o_p(1)$, hence \eqref{sqrt var error}. This completes the proof.\Halmos
\endproof
\proof{Proof of Corollary \ref{coverage}.}
We consider two cases, $k^2Var(g_k(\xi))\cdot n^{\frac{\delta}{4+2\delta}}> 1$ and $k^2Var(g_k(\xi))\cdot n^{\frac{\delta}{4+2\delta}}\leq 1$. If $k^2Var(g_k(\xi))\cdot n^{\frac{\delta}{4+2\delta}}> 1$, then we have $k^2Var(g_k(\xi))\cdot n^{\frac{\delta}{2+\delta}}>n^{\frac{\delta}{4+2\delta}}\to\infty$ and hence $\mathcal{Z}_{n,k}\Rightarrow N(0,1)$ by Theorem \ref{meta clt} (without replacement) or Theorem \ref{V thm} (with replacement). In this case we have
\begin{eqnarray*}
\liminf_{n\to\infty}P(\tilde \sigma_{IJ}\mathcal{Z}_{n,k}\leq z_{1-\alpha}\tilde \sigma_{IJ})&= &\liminf_{n\to\infty}P(\mathcal{Z}_{n,k}\leq z_{1-\alpha} \text{ or }\tilde \sigma_{IJ}=0)\\
&\geq &\liminf_{n\to\infty}P(\mathcal{Z}_{n,k}\leq z_{1-\alpha})=1-\alpha.
\end{eqnarray*}
If $k^2Var(g_k(\xi))\cdot n^{\frac{\delta}{4+2\delta}}\leq 1$, then we have $k^2Var(g_k(\xi))=o(1)$ and hence $\tilde \sigma^2_{IJ}=o_p(1/n)$ by Theorems \ref{consistency_IJ:U} and \ref{consistency_IJ:V} and the consistency of the approximated IJ variance estimate \eqref{var estimate error}, therefore $z_{1-\alpha}\tilde \sigma_{IJ} - \tilde \sigma_{IJ}\mathcal{Z}_{n,k} = o_p(1/\sqrt{n})$ and it holds that
\begin{equation*}
    \liminf_{n\to\infty}P(\tilde \sigma_{IJ}\mathcal{Z}_{n,k}+o_p\big(\frac{1}{\sqrt{n}}\big)\leq z_{1-\alpha}\tilde \sigma_{IJ})=1\geq 1-\alpha
\end{equation*}
Combining the two cases, we see that there exists some $o_p(1/\sqrt{n})$ random variable such that
\begin{equation*}
    \liminf_{n\to\infty}P(\tilde \sigma_{IJ}\mathcal{Z}_{n,k}+o_p\big(\frac{1}{\sqrt{n}}\big)\leq z_{1-\alpha}\tilde \sigma_{IJ})\geq 1-\alpha.
\end{equation*}
regardless of whether $k^2Var(g_k(\xi))\cdot n^{\frac{\delta}{4+2\delta}}>1$ or $k^2Var(g_k(\xi))\cdot n^{\frac{\delta}{4+2\delta}}\leq1$. From Theorem \ref{final_guarantee} we have $\tilde \sigma_{IJ}\mathcal{Z}_{n,k}=\tilde Z_{n,k}^{bag} - W_k - o_p\big(\frac{1}{\sqrt{n}}\big)$, and hence
\begin{eqnarray*}
    &&\liminf_{n\to\infty}P(\tilde Z_{n,k}^{bag} - W_k - o_p\big(\frac{1}{\sqrt{n}}\big) +o_p\big(\frac{1}{\sqrt{n}}\big)\leq z_{1-\alpha}\tilde \sigma_{IJ})\\
    &=&\liminf_{n\to\infty}P(\tilde \sigma_{IJ}\mathcal{Z}_{n,k}+o_p\big(\frac{1}{\sqrt{n}}\big)\leq z_{1-\alpha}\tilde \sigma_{IJ})\geq 1-\alpha.
\end{eqnarray*}
When the non-degeneracy condition holds, i.e., $k^2Var(g_k(\xi))>\epsilon$ for some $\epsilon>0$ as $n,k$ grow, then in the case of resampling without replacement we have
\begin{equation*}
    \frac{\sqrt{n}(U_{n,k}-W_k)}{k\sqrt{Var(g_k(\xi))}}\Rightarrow N(0,1)
\end{equation*}
by Theorem \ref{meta clt}. In the case of resampling with replacement we have
\begin{equation*}
    \frac{\sqrt{n}(V_{n,k}-W_k)}{k\sqrt{Var(g_k(\xi))}}\Rightarrow N(0,1)
\end{equation*}
by Theorem \ref{V thm}. On one hand, from Theorem \ref{consistency_IJ:U} (without replacement) or \ref{consistency_IJ:V} (with replacement) and the consistency of approximated IJ variance estimate \eqref{var estimate error} we have that $n\tilde \sigma^2_{IJ}/(k^2Var(g_k(\xi)))\to 1$ in probability. On the other hand, from \eqref{point estimate error} we have $\tilde Z_{n,k}^{bag}-V_{n,k}=o_p(1/\sqrt{n})$ and $\tilde Z_{n,k}^{bag}-U_{n,k}=o_p(1/\sqrt{n})$ respectively for with and without replacement, therefore by Slutsky's theorem
\begin{equation*}
    \frac{\tilde Z_{n,k}^{bag}-W_k}{\tilde \sigma_{IJ}}\Rightarrow N(0,1)
\end{equation*}
for both resampling with and without replacement. This leads to
\begin{equation*}
    \lim_{n\to\infty}P(\tilde Z_{n,k}^{bag} - z_{1-\alpha}\tilde \sigma_{IJ}\leq W_k)\to 1-\alpha.
\end{equation*}
This concludes the proof.\Halmos
\endproof

\section{Proof of Results in Section \ref{sec:properties}}\label{proof:se}
\proof{Proof of Proposition \ref{bound se}.}
% Using Lemma \ref{bound lemma} with $c=1$, we have
% \begin{align*}
% Var(g_k(\xi))&=E(E[H_k(\xi_1,\ldots,\xi_k)|\xi_1]-E[H_k(\xi_1,\ldots,\xi_k)])^2\\
% &\leq\frac{1}{k^2}E\left(E\left[\sup_{x\in\mathcal X}|h(x,\xi_1')-h(x,\xi_1)|\Big|\xi_1'\right]\right)^2\\
% &\leq\frac{1}{k^2}E\left(\sup_{x\in\mathcal X}|h(x,\xi_1')-h(x,\xi_1)|\right)^2\text{\ \ by Jensen's inequality}\\
% &\leq\frac{1}{k^2}E\left(\sup_{x\in\mathcal X}|h(x,\xi_1')|+\sup_{x\in\mathcal X}|h(x,\xi_1)|\right)^2\\
% &\leq\frac{4}{k^2}E\sup_{x\in\mathcal X}|h(x,\xi)|^2=O\big(\frac{1}{k^2}\big)\text{\ \ by Minkowski inequality and Assumption \ref{L2}}.
% \end{align*}
% This concludes the bound for $k^2Var(g_k(\xi))$.
A direct consequence of Proposition \ref{var bound for SAA kernel} is that $Var(H_k)=O(1/k)$, therefore $Var(\tilde Z_{n,k})=Var(H_k)/m=O(1/k)/m=O(1/n)$ and $Var(\hat Z_n)=Var(H_n)\allowbreak=O(1/n)$. By ANOVA decomposition we have $Var(U_{n,k})= Var(\mathring U_{n,k}) + Var(U_{n,k} - \mathring U_{n,k})$, where
\begin{equation*}
Var(\mathring U_{n,k})=\frac{k^2}{n}Var(g_k(\xi))= \frac{k}{n} Var(\mathring H_k)\leq \frac{k}{n} Var(H_k)=\frac{k}{n}O\big(\frac{1}{k}\big)=O\big(\frac{1}{n}\big),
\end{equation*}
and
\begin{eqnarray*}
Var(U_{n,k} - \mathring U_{n,k})&=&E[(U_{n,k} - \mathring U_{n,k})^2]\\
&\leq &\frac{k^2}{n^2}E[(H_k - \mathring H_k)^2]\text{\ \ by Lemma \ref{anova:variance}}\\
&\leq &\frac{k^2}{n^2}Var(H_k)=\frac{k^2}{n^2}O\big(\frac{1}{k}\big)=O\big(\frac{k}{n^2}\big)=O\big(\frac{1}{n}\big),
\end{eqnarray*}
hence $Var(U_{n,k})=O(1/n)$. As for $V_{n,k}$, we have shown in the proof of Theorem \ref{V thm} that $E[(U_{n,k}-V_{n,k})^2]=o(1/n)$ under the given resample sizes, hence $Var(V_{n,k})=O(1/n)$ as well.\Halmos
\endproof

%
% \section{Proof of Theorem \ref{compare var:asymptotic} and the Claim in Example \ref{efficiency_example}}\label{proof:efficiency_compare}
\proof{Proof of Theorem \ref{compare var:asymptotic}.}
We first prove the convergence of $nVar(U_{n,k})$ and $nVar(V_{n,k})$. Theorem \ref{Lipschitz characterization of limit variance} shows that $nVar(\mathring{U}_{n,k})=k^2Var(g_k(\xi))\to Var(E[h(x^*_Y,\xi)\vert \xi])$ as $n,k\to\infty$. When $k=o(n)$, we have shown $E[(U_{n,k}-\mathring{U}_{n,k})^2]=o(1/n)$ in the proof of Theorem \ref{main thm} using Lemma \ref{new_clt:U}, therefore $\lim_{n,k\to\infty}nVar(U_{n,k})=\lim_{n,k\to\infty} nVar(\mathring{U}_{n,k})=Var(E[h(x^*_Y,\xi)\vert \xi])$. As for $V_{n,k}$, we have shown in the proof of Theorem \ref{V thm} that $E[(U_{n,k}-V_{n,k})^2]=o(1/n)$, therefore $\lim_{n,k\to\infty}nVar(V_{n,k})=\lim_{n,k\to\infty} nVar(U_{n,k})=Var(E[h(x^*_Y,\xi)\vert \xi])$.

We then prove the convergence of $nVar(\tilde Z_{n,k})$ and $nVar(\hat Z_n)$. Note that $Var(\tilde Z_{n,k})=Var(H_k)/m$ and $Var(\hat Z_n)=Var(H_n)$, therefore it suffices to study the asymptotics of $Var(H_k)$. We already have the weak limit $\sqrt{k}(H_k - Z^*)\Rightarrow \inf_{x\in\mathcal{X}^*}Y(x)=Y(x^*_Y)$ from Theorem \ref{CLT SAA general}, and want to show uniform integrability for $k(H_k-Z^*)^2$ to conclude convergence of the first and second moments. To prove uniform integrability, we use the bound from Lemma \ref{continuity of SAA value} to write
\begin{eqnarray*}
E[(\sqrt{k}\lvert H_k - Z^* \rvert)^{2+\delta}]&\leq&E\Big[k^{1+\frac{\delta}{2}}\sup_{x\in\mathcal{X}}\lvert \frac{1}{k}\sum_{i=1}^kh(x,\xi_i) - Z(x) \rvert^{2+\delta}\Big]\\
&=&O(1)
\end{eqnarray*}
where the $O(1)$ bound comes from Proposition \ref{Lipschitz_k}. Therefore $kVar(H_k)\to Var(Y(x^*_Y))$, and the desired limit variance for $nVar(\tilde Z_{n,k})$ and $nVar(\hat Z_n)$ follows.

The statement $Var(E[h(x^*_Y,\xi)\vert \xi])\leq Var(Y(x^*_Y))$ follows because Theorem \ref{compare var} on finite-sample variance reduction states that $nVar(U_{n,k})\leq nVar(\tilde Z_{n,k})$ for any $k,n$ such that $n/k$ is an integer and hence taking $n,k$ to $\infty$ gives $Var(E[h(x^*_Y,\xi)\vert \xi])\leq Var(Y(x^*_Y))$.

If the optimal solution is essentially unique, we denote by $x^*$ one of the optimal solutions, and can write $Var(E[h(x^*_Y,\xi)\vert \xi])=Var(h(x^*,\xi))$ and $Var(Y(x^*_Y))=Var(Y(x^*))=Var(h(x^*,\xi))$, therefore it holds that $Var(E[h(x^*_Y,\xi)\vert \xi]) = Var(Y(x^*_Y))$.

Otherwise if the optimal solution is not essentially unique, we must have that $\sup_{x\in\mathcal{X}^*}Var(Y(x))>0$ and hence $Var(Y(x_Y^*))>0$. Suppose $Var(E[h(x^*_Y,\xi)\vert \xi]) = Var(Y(x^*_Y)) > 0$ and we derive contradictions. Note that $\hat Z_n=U_{n,n}=H_n$ where the resample size $k$ is chosen to be $n$, and the Hajek projection $\mathring{H}_n$ has a variance $Var(\mathring{H}_n)=nVar(g_n(\xi))$ hence $nVar(\mathring{H}_n)\to Var(E[h(x^*_Y,\xi)\vert \xi])$ by Theorem \ref{Lipschitz characterization of limit variance}. This implies that
\begin{eqnarray*}
nVar(H_n-\mathring{H}_n)&=&n(Var(H_n) - Var(\mathring{H}_n))\text{\ \ by Lemma \ref{anova:variance}}\\
&=&n(Var(\hat Z_n) - Var(\mathring{H}_n))\text{\ \ since $\hat Z_n=H_n$}\\
&\to& Var(Y(x^*_Y)) - Var(E[h(x^*_Y,\xi)\vert \xi])\\
&=&0.
\end{eqnarray*}
Therefore \eqref{range:k} holds with $k=n$, and subsequently Theorem \ref{meta clt} forces that $U_{n,n}$, i.e., $\hat Z_n$, has a weak limit of Gaussian distribution. However, Theorem \ref{CLT SAA general} states that the weak limit of the sequence $\hat Z_n$ must be $\inf_{x\in\mathcal{X}^*}Y(x)$ which is non-Gaussian because the Gaussian process $Y(x),x\in\mathcal{X}^*$ can not be reduced to a Gaussian variable, thus leading to a contradiction.\Halmos
\endproof

\proof{Proof of Theorem \ref{thm bias}.}
For $U_{n,k}$, note that each summand in its definition is an SAA value with distinct i.i.d.~data, and thus has mean exactly $W_k$. To show $E[V_{n,k}] \leq W_k$, we provide a monotonicity result for SAA with arbitrary weights on data:
\begin{lemma}[Generalized monotonicity]\label{general monotonicity}
Let $\xi_1,\ldots,\xi_k$ be i.i.d., and for any $p_i\geq 0,i=1,\ldots,k$ such that $\sum_{i=1}^kp_i=1$, let
$$H^{p_1,p_2,\ldots,p_k}_k(\xi_1,\ldots,\xi_k):=\min_{x\in\mathcal{X}}\sum_{i=1}^kp_ih(x,\xi_i).$$
In particular $H^{1/k,1/k,\ldots,1/k}_k$ is the SAA kernel $H_k$ with uniform weights on $k$ distinct data. We have
\begin{equation}\label{monotonicity with weights}
    E[H^{p_1,p_2,\ldots,p_k}_k(\xi_1,\ldots,\xi_k)]\leq E[H_k(\xi_1,\ldots,\xi_k)]
\end{equation}
for any such $p_1,p_2,\ldots,p_k$.
\end{lemma}
\proof{Proof.}
Denote by $\Pi$ the set of all $k!$ permutations on $\{1,2,\ldots,k\}$, then we can rewrite the SAA with uniform weights as
\begin{equation*}
    \sum_{i=1}^k\frac{1}{k}h(x,\xi_i)=\frac{1}{k!}\sum_{\pi\in\Pi}\sum_{i=1}^kp_{\pi(i)}h(x,\xi_i).
\end{equation*}
Therefore we have
\begin{eqnarray}
\notag E[H_k(\xi_1,\ldots,\xi_k)] &=& E\big[\min_{x\in\mathcal{X}}\sum_{i=1}^k\frac{1}{k}h(x,\xi_i)\big]\\
\notag&=&E\big[\min_{x\in\mathcal{X}}\frac{1}{k!}\sum_{\pi\in\Pi}\sum_{i=1}^kp_{\pi(i)}h(x,\xi_i)\big]\\
&\geq &E\big[\frac{1}{k!}\sum_{\pi\in\Pi}\min_{x\in\mathcal{X}}\sum_{i=1}^kp_{\pi(i)}h(x,\xi_i)\big]\label{monotonicity key step}\\
\notag&=&\frac{1}{k!}\sum_{\pi\in\Pi}E\big[\min_{x\in\mathcal{X}}\sum_{i=1}^kp_{\pi(i)}h(x,\xi_i)\big]\\
\notag&=&\frac{1}{k!}\sum_{\pi\in\Pi}E[H^{p_1,p_2,\ldots,p_k}_k(\xi_1,\ldots,\xi_k)]\text{\ \ since $\xi_i$'s are i.i.d.}\\
\notag&=&E[H^{p_1,p_2,\ldots,p_k}_k(\xi_1,\ldots,\xi_k)].
\end{eqnarray}
This concludes the lemma.\Halmos
\endproof
The monotonicity result from \citeAPX{mak1999monte} and \citeAPX{norkin1998branch}, stated as $E[H_{k-1}(\xi_1,\ldots,\xi_{k-1})]\leq E[H_k(\xi_1,\ldots,\xi_k)]$, is a special case of Lemma \ref{general monotonicity} where $p_i=1/(k-1)$ for $i\leq k-1$ and $p_k=0$. As a side note, such type of monotonicity result can in fact be even more general than Lemma \ref{monotonicity with weights}: The key step \eqref{monotonicity key step} of the proof is the exchange of a minimization and an expectation with respect to the uniform distribution over $\Pi$, and \eqref{monotonicity key step} remains valid if this expectation is with respect to any other distribution over the permutation set $\Pi$, i.e., $\min_{x\in\mathcal{X}}\sum_{\pi\in\Pi}f_{\pi}\sum_{i=1}^kp_{\pi(i)}h(x,\xi_i)\geq \sum_{\pi\in\Pi}f_{\pi}\min_{x\in\mathcal{X}}\sum_{i=1}^kp_{\pi(i)}h(x,\xi_i)$ for any $f_{\pi}$ such that $f_{\pi}\geq 0$ and $\sum_{\pi\in \Pi}f_{\pi}=1$. Therefore \eqref{monotonicity with weights} continues to hold with the right-hand side replaced by $E[H^{q_1,q_2,\ldots,q_k}_k(\xi_1,\ldots,\xi_k)]$ as long as the vector $(q_1,q_2,\ldots,q_k)$ lies in the convex hull of $\{(p_{\pi(1)}, p_{\pi(2)},\ldots,p_{\pi(k)}):\pi\in\Pi\}$. Using duality theory of linear programs and the rearrangement inequality, it can be shown that $\mathrm{ConvexHull}(\{(p_{\pi(1)}, p_{\pi(2)},\ldots,p_{\pi(k)}):\pi\in\Pi\})=\{(q_1,q_2,\ldots,q_k):\sum_{j=i}^kp_{(j)}\geq \sum_{j=i}^kq_{(j)}\ \forall \ 0\leq i\leq k,\sum_{i=1}^kq_i=1,q_i\geq 0\ \forall \ i\}$, where $q_{(j)}$ and $p_{(j)}$ are the $j$-th smallest component of $(q_1,\ldots,q_k)$ and $(p_1,\ldots,p_k)$ respectively.

% $E[H^{\sum_{\pi\in \Pi}q(\pi)p_{\pi(1)},\sum_{\pi\in \Pi}q(\pi)p_{\pi(2)},\ldots,\sum_{\pi\in \Pi}q(\pi)p_{\pi(k)}}_k(\xi_1,\ldots,\xi_k)]$.

To show the downward biasedness of $V_{n,k}$, note that each summand $H_k(\xi_{i_1},\ldots,\xi_{i_k})$ in \eqref{V} can be cast in the form $H^{p_1,p_2,\ldots,p_k}_k(\xi_1,\ldots,\xi_k)$ with $p_i=\sum_{j=1}^k\mathbf{1}\{i_j=i\}/k$ for each $i$. Therefore Lemma \ref{general monotonicity} immediately implies that $E[H_k(\xi_{i_1},\ldots,\xi_{i_k})]\leq E[H_k(\xi_1,\ldots,\xi_k)]=W_k$. Since this holds for each summand in \eqref{V}, we have $E[V_{n,k}]\leq W_k$.

To bound the bias, recall the relation \eqref{UV_decompose}
\begin{equation*}
n^k(U_{n,k}-V_{n,k})=\big(\sum_{s=1}^{k-l-1}c(n,k,s)\big)(U_{n,k}-R_{n,l})-\sum_{s=k-l}^{k-1}c(n,k,s)(A_{n,s}-U_{n,k})
\end{equation*}
for arbitrary but fixed integer $l\geq 0$. Note that $U_{n,k}$ is unbiased for $W_k$, and that $E[R_{n,l}]=O(1)$ since Assumption \ref{L2} implies for any indices $i_1,\ldots,i_k\in\{1,\ldots,n\}$ that $\abs{E[H_k(\xi_{i_1},\ldots,\xi_{i_k})]}\leq E[\sup_{x\in\mathcal X}|h(x,\xi)|]$. In the proof of Theorem \ref{V thm} we have shown that $E[\abs{A_{n,s}-U_{n,k}}]=O(1/k)$, $\sum_{s=1}^{k-l-1}c(n,k,s)=O((k^2/n)^{l+1}n^k)$ whenever $k=o(\sqrt n)$, and that $c(n,k,s)=O(k^{2(k-s)}n^{s})$ for $s\geq k-l$. Therefore
\begin{equation*}
n^k\abs{E[V_{n,k}]-W_k}\leq O\big(\big(\frac{k^2}{n}\big)^{l+1}n^k\big)+O(1/k)\sum_{s=k-l}^{k-1}O(k^{2(k-s)}n^{s}).
\end{equation*}
Since $k^2/n=o(1)$, it holds $\sum_{s=k-l}^{k-1}O(k^{2(k-s)}n^{s})=O(k^2n^{k-s-1}n^s)=O(k^{2}n^{k-1})$, which leads to $W_k-E[V_{n,k}]=O((k^2/n)^{l+1}+k/n)$ for any fixed $l\geq 0$.\Halmos
\endproof

\section{Additional Experiments}\label{sec:additional experiments}
This section contains extra experimental results that complement those in Section \ref{sec:numerics}. We consider one more problem that solves for the $(1-\alpha_1)$-level conditional value at risk (CVaR) of a standard normal variable $\xi$
\begin{equation}\label{cvar}
\min_{x\in \R}x+\frac{1}{\alpha_1} E[(\xi-x)_+]
\end{equation}
where $(\cdot)_+:=\max\{\cdot,0\}$ denotes the positive part. We set $\alpha_1=0.1$, namely, we are solving for the $90\%$-level CVaR of the standard normal. Figure \ref{fig:gap CRN varying B} summarizes results for bagging in computing bounds of optimality gaps for problems \eqref{IP} and \eqref{binary} under a fixed data size and growing bootstrap sizes. Figures \ref{fig:cvar varying N} and \ref{fig:gap CRN varying N} show results on bagging for problems \eqref{cvar}, \eqref{IP} and \eqref{binary} under fixed bootstrap size $b=500$ and growing data sizes. Figure \ref{fig:cvar optval} presents results of various methods on bounds of the optimal value for problem \eqref{cvar}. Figures \ref{fig:cvar gap}, \ref{fig:portfolio gap n=200}, \ref{fig:integer gap BC} and \ref{fig:binary gap n=200} contain additional results of various methods on bounds of optimality gaps. Figure \ref{fig:simplex opt val n=50} shows results of various methods on problem \eqref{simplex} with the data size $n=50$.

\begin{figure}[h]
    \centering
     \begin{subfigure}{0.49\textwidth}
         \centering
         \includegraphics[width=\textwidth]{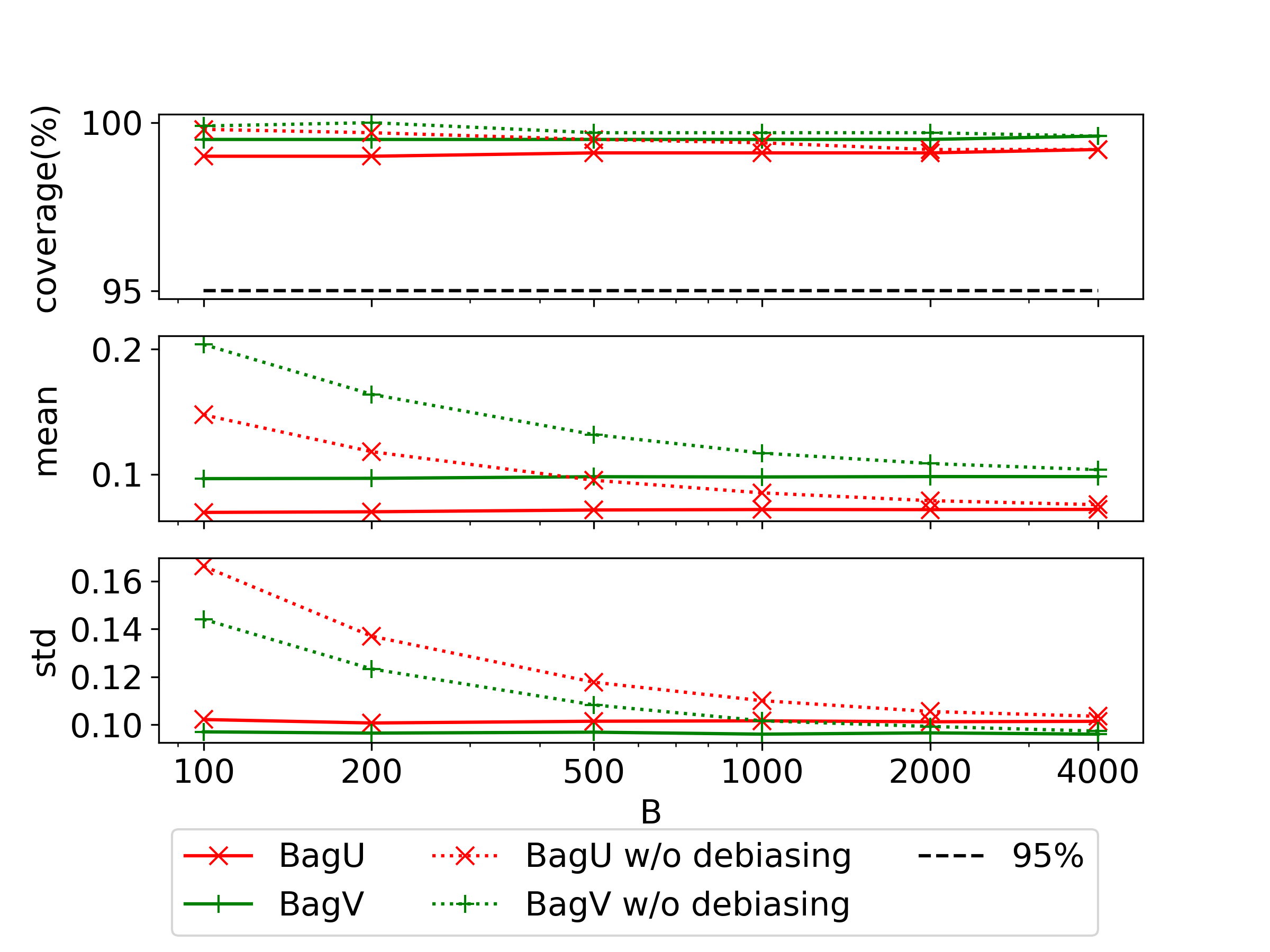}
         \caption{Integer program \eqref{IP}.}
         \label{fig:integer gap CRN varying B}
     \end{subfigure}
     \hfill
     \begin{subfigure}{0.49\textwidth}
         \centering
         \includegraphics[width=\textwidth]{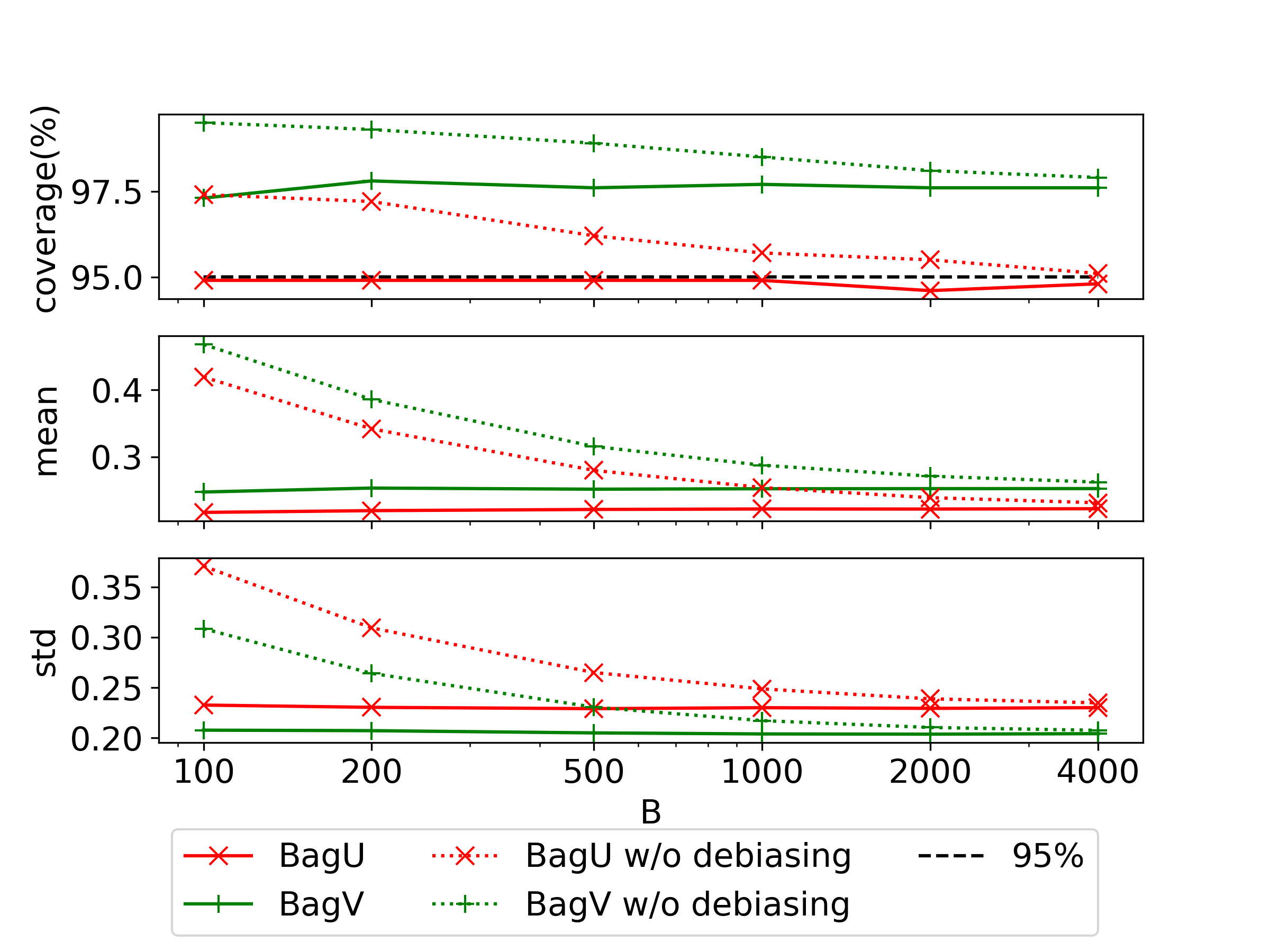}
         \caption{Simple linear program \eqref{binary}.}
         \label{fig:binary gap CRN varying B}
     \end{subfigure}
    \caption{Bounds of optimality gaps using CRN under fixed data size $n=1000,n_1=600,n_2=400$ and varying bootstrap sizes.}
    \label{fig:gap CRN varying B}
\end{figure}

\begin{figure}[h]
    \centering
         \begin{subfigure}{0.49\textwidth}
         \centering
         \includegraphics[width=\textwidth]{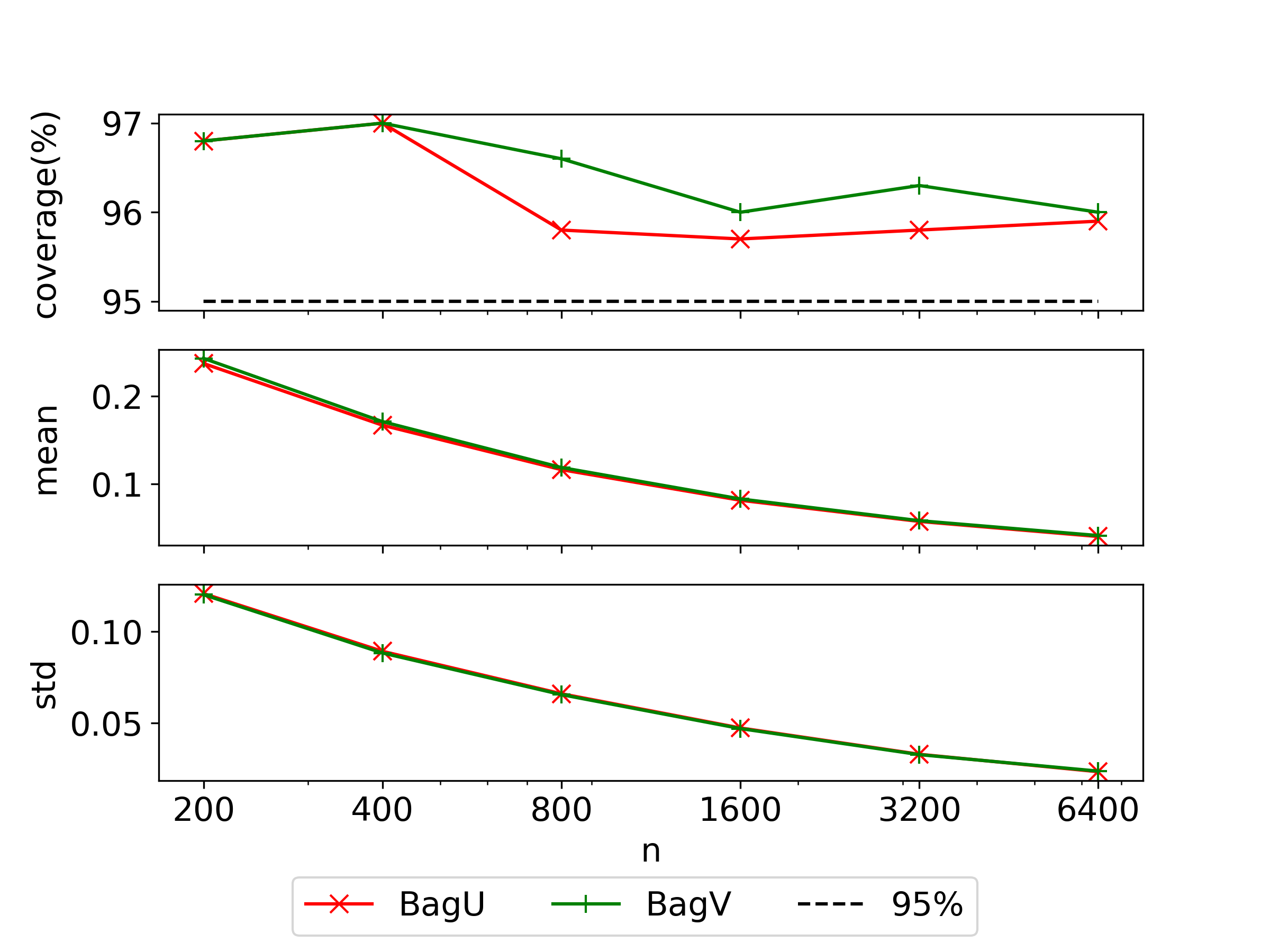}
         \caption{Bounds of the optimal value.}
         \label{fig:cvar optval varying N}
     \end{subfigure}
     \hfill
     \begin{subfigure}{0.49\textwidth}
         \centering
         \includegraphics[width=\textwidth]{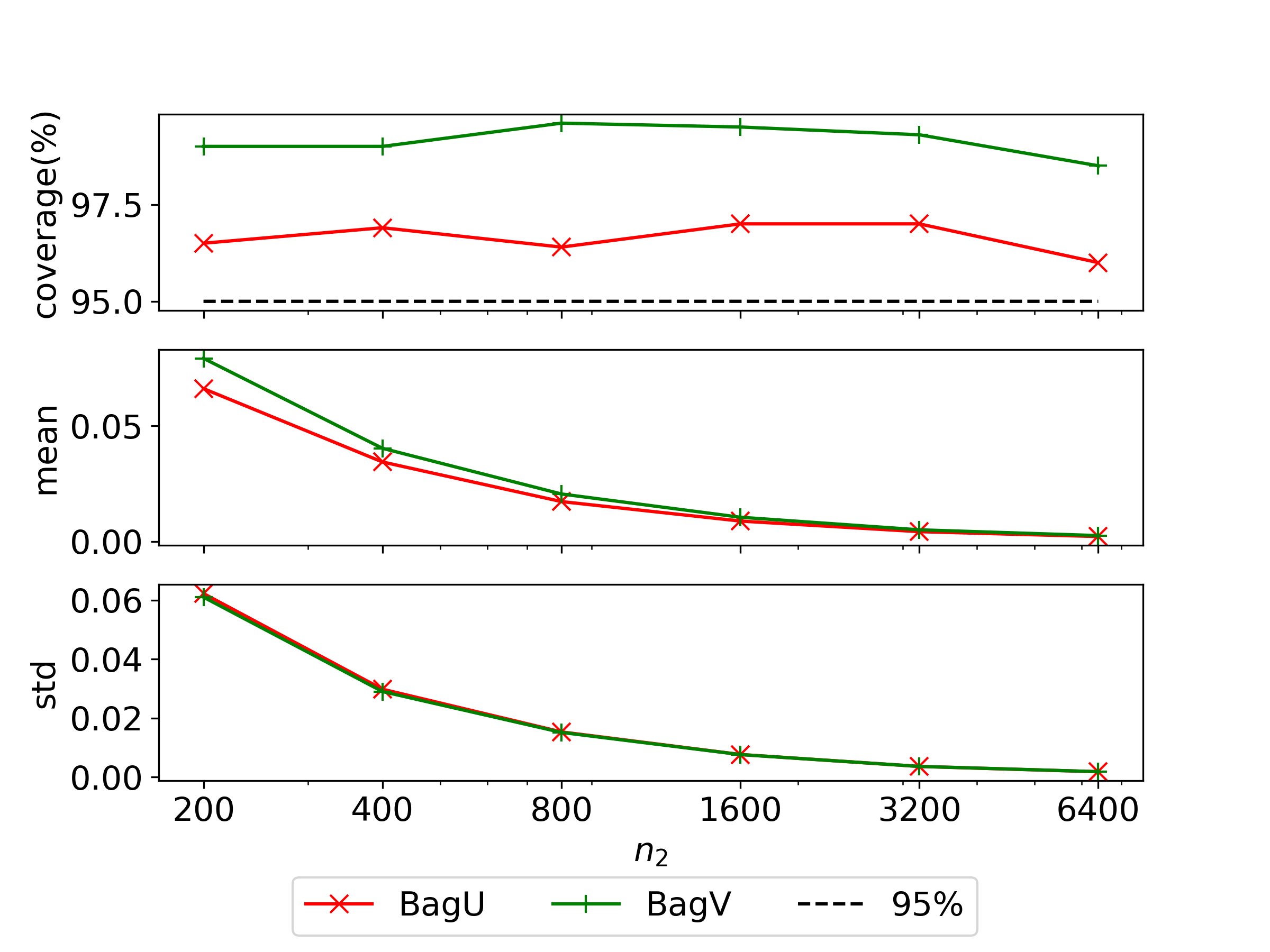}
         \caption{Bounds for optimality gaps.}
         \label{fig:cvar gap CRN varying N}
     \end{subfigure}
    \caption{Performance with fixed $B=500$ and growing data sizes for CVaR problem \eqref{cvar}.}
    \label{fig:cvar varying N}
\end{figure}

\begin{figure}[h]
    \centering
          \begin{subfigure}{0.49\textwidth}
         \centering
         \includegraphics[width=\textwidth]{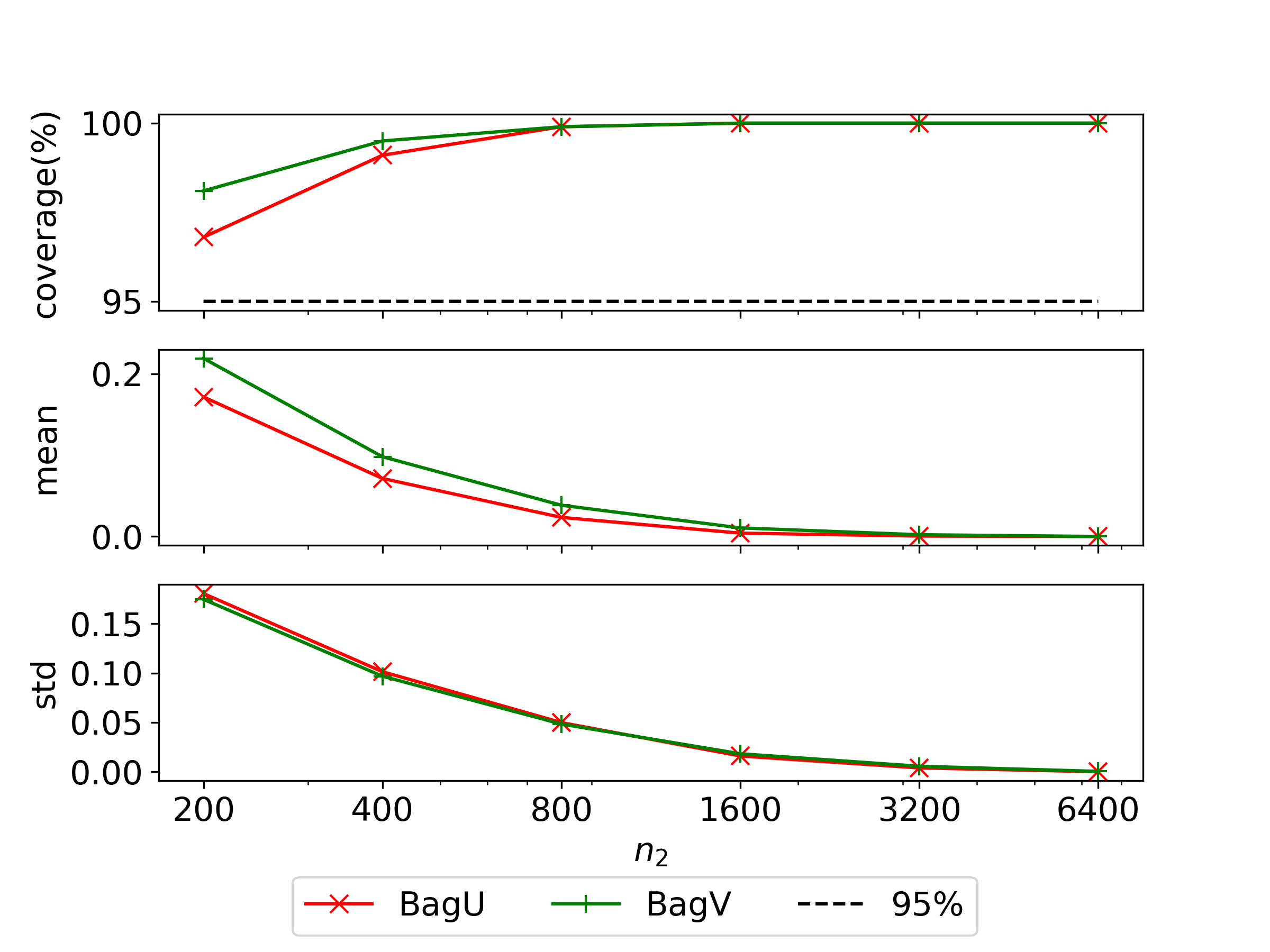}
         \caption{Integer program \eqref{IP}.}
         \label{fig:integer gap CRN varying N}
     \end{subfigure}
     \hfill
          \begin{subfigure}{0.49\textwidth}
         \centering
         \includegraphics[width=\textwidth]{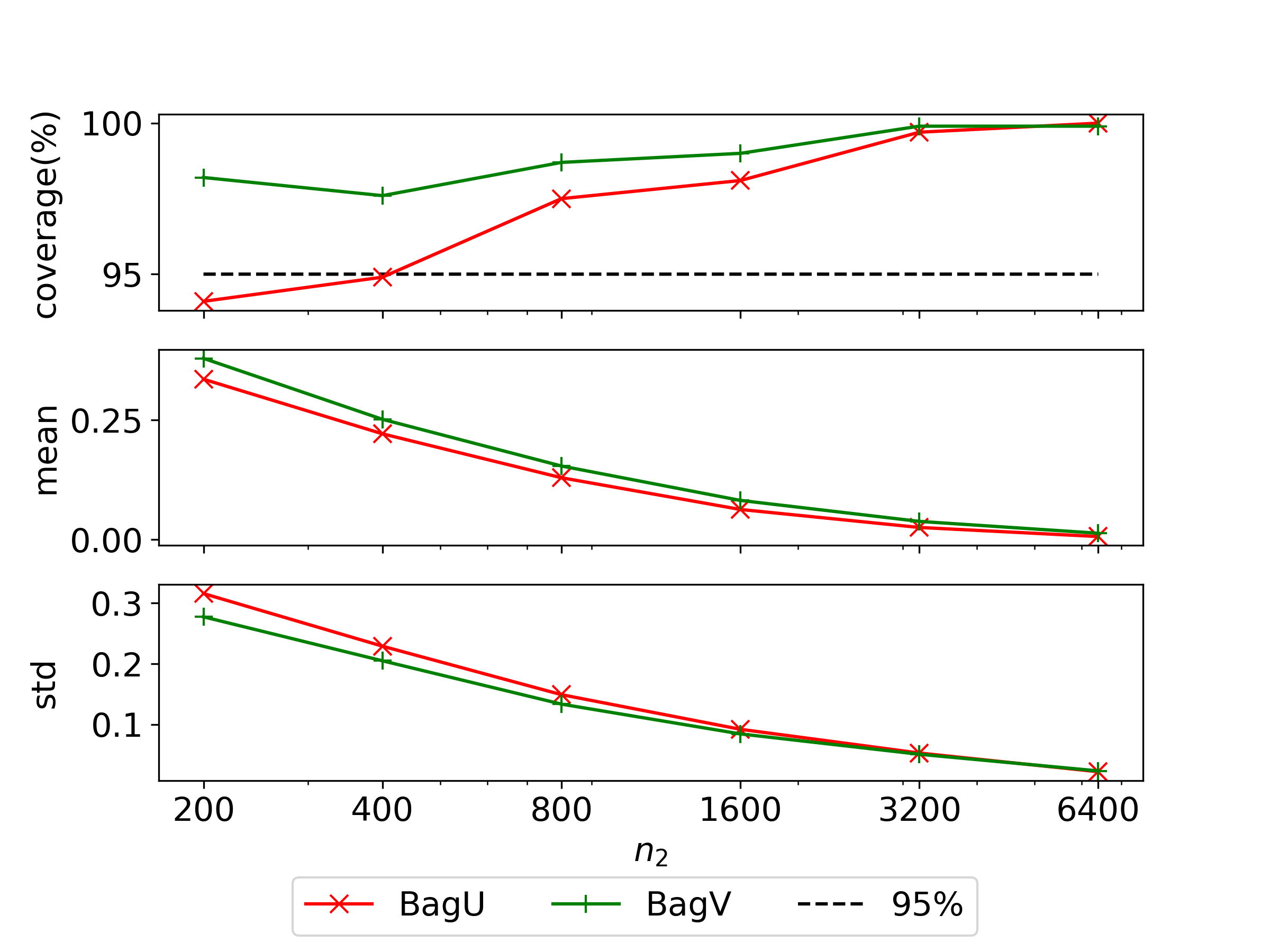}
         \caption{Simple linear program \eqref{binary}.}
         \label{fig:binary gap CRN varying N}
     \end{subfigure}
    \caption{Bounds of optimality gaps using CRN with fixed $B=500$ and growing data sizes.}
    \label{fig:gap CRN varying N}
\end{figure}

\begin{figure}[h]
    \centering
    \begin{subfigure}{0.49\textwidth}
         \centering
         \includegraphics[width=\textwidth]{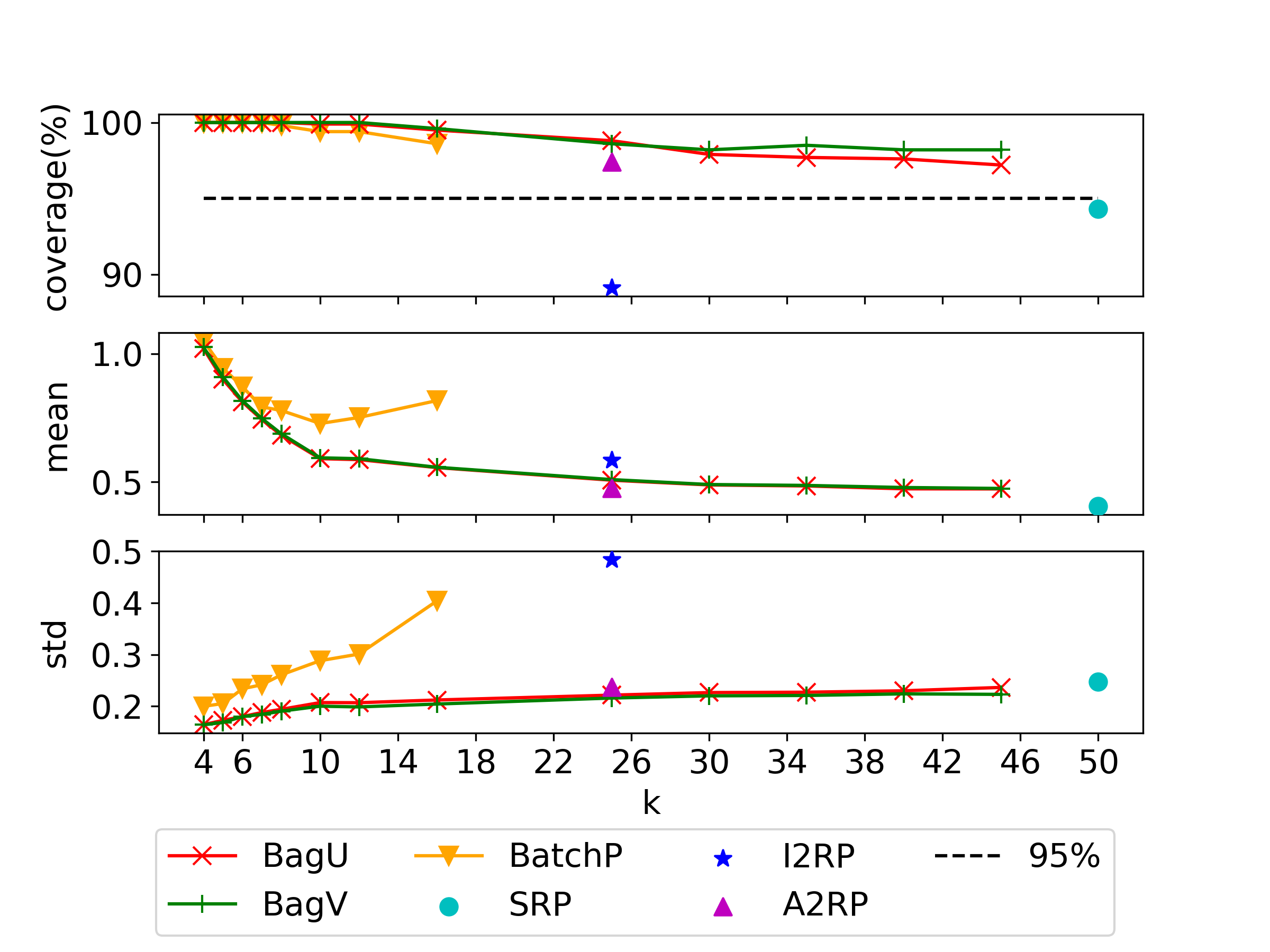}
         \caption{$n=50$}
         \label{fig:cvar optval 50}
     \end{subfigure}
     \hfill
     \begin{subfigure}{0.49\textwidth}
         \centering
         \includegraphics[width=\textwidth]{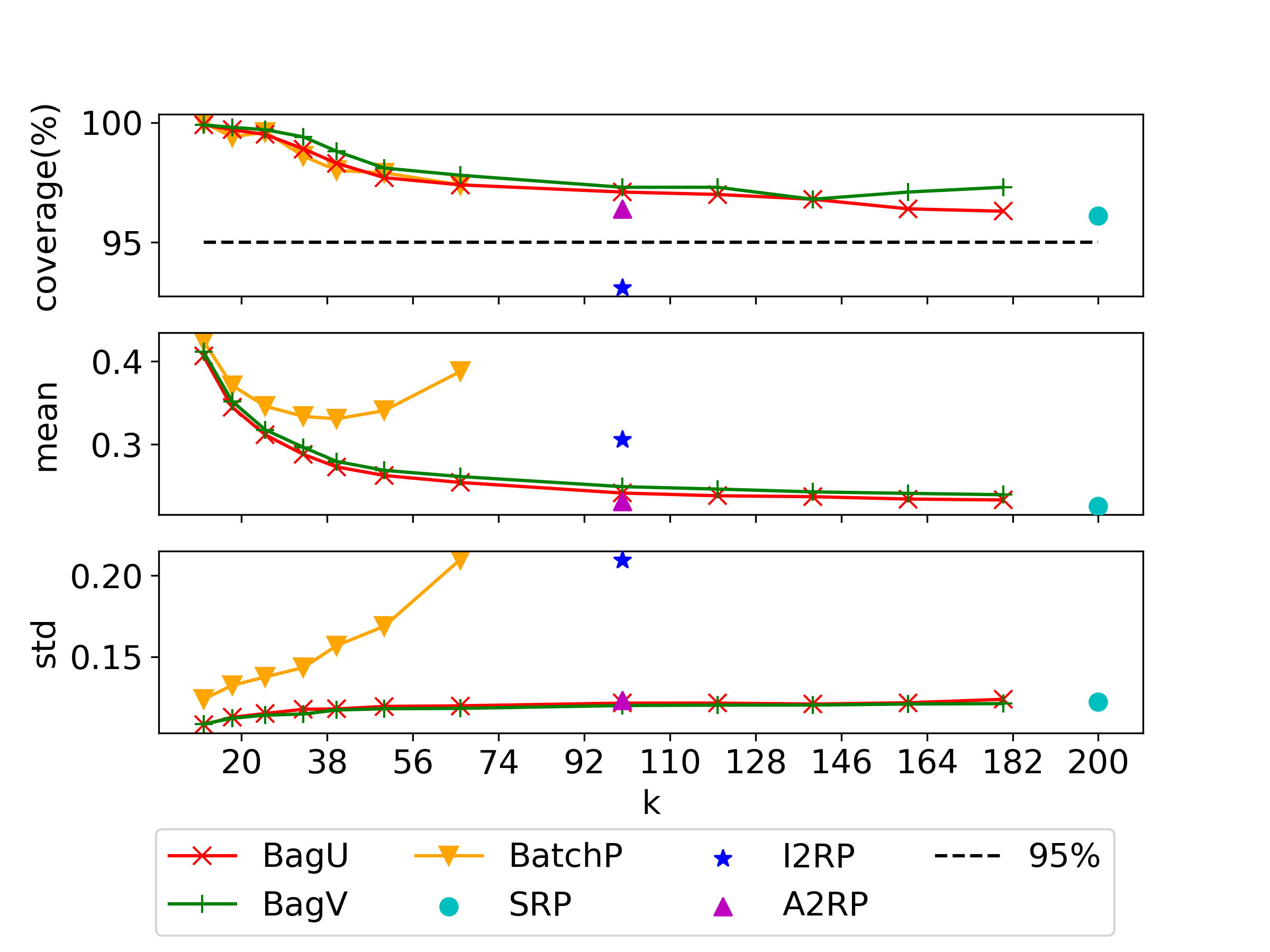}
         \caption{$n=200$}
         \label{fig:cvar optval 200}
     \end{subfigure}
    \caption{CVaR problem \eqref{cvar}. Lower bounds of optimal values.}
    \label{fig:cvar optval}
\end{figure}

\begin{figure}[h]
    \centering
    \begin{subfigure}{0.49\textwidth}
         \centering
         \includegraphics[width=\textwidth]{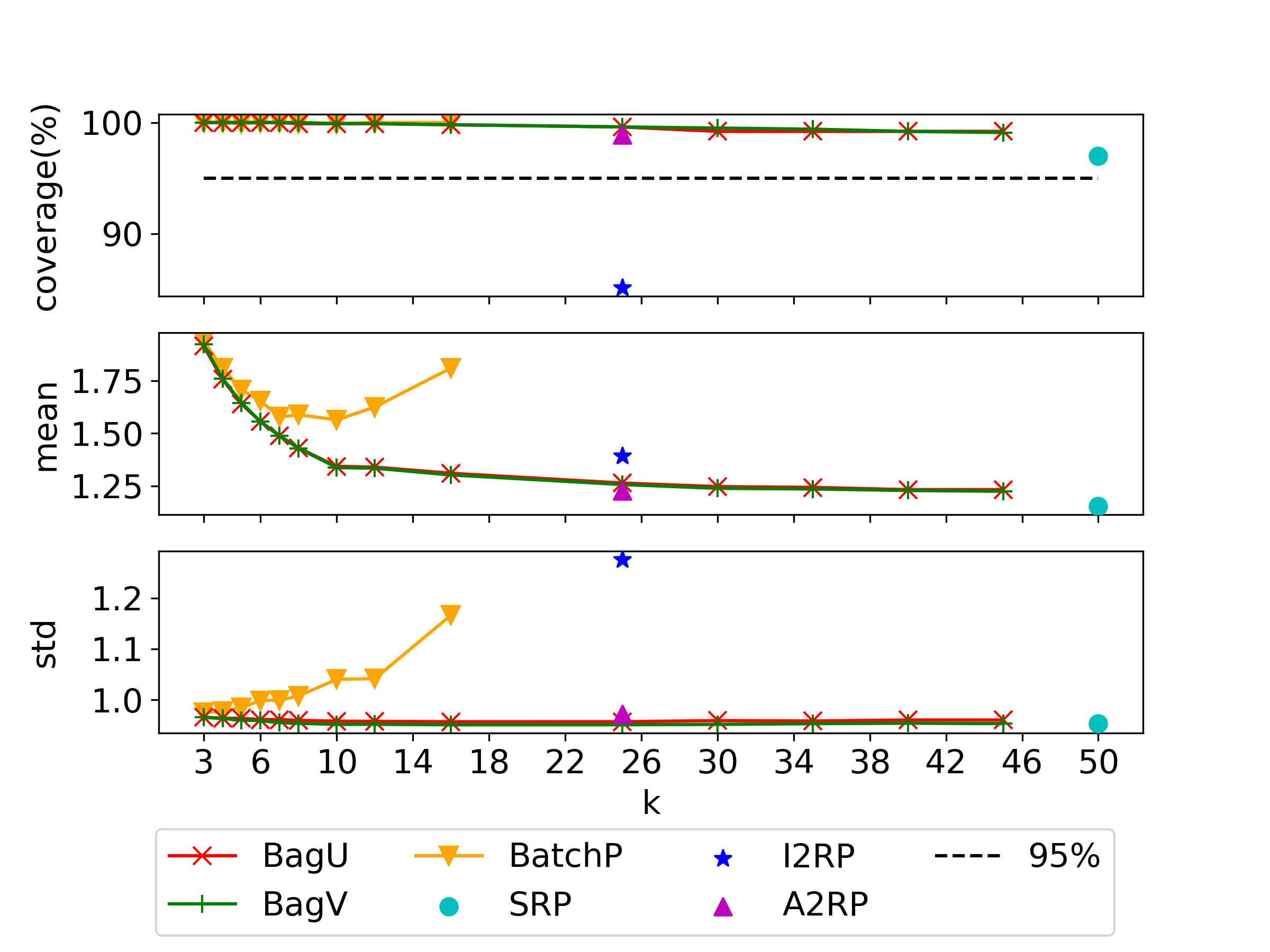}
         \caption{BC, $n=50, n_1=30, n_2=20$}
         \label{fig:cvar gap BC 50}
     \end{subfigure}
     \hfill
     \begin{subfigure}{0.49\textwidth}
         \centering
         \includegraphics[width=\textwidth]{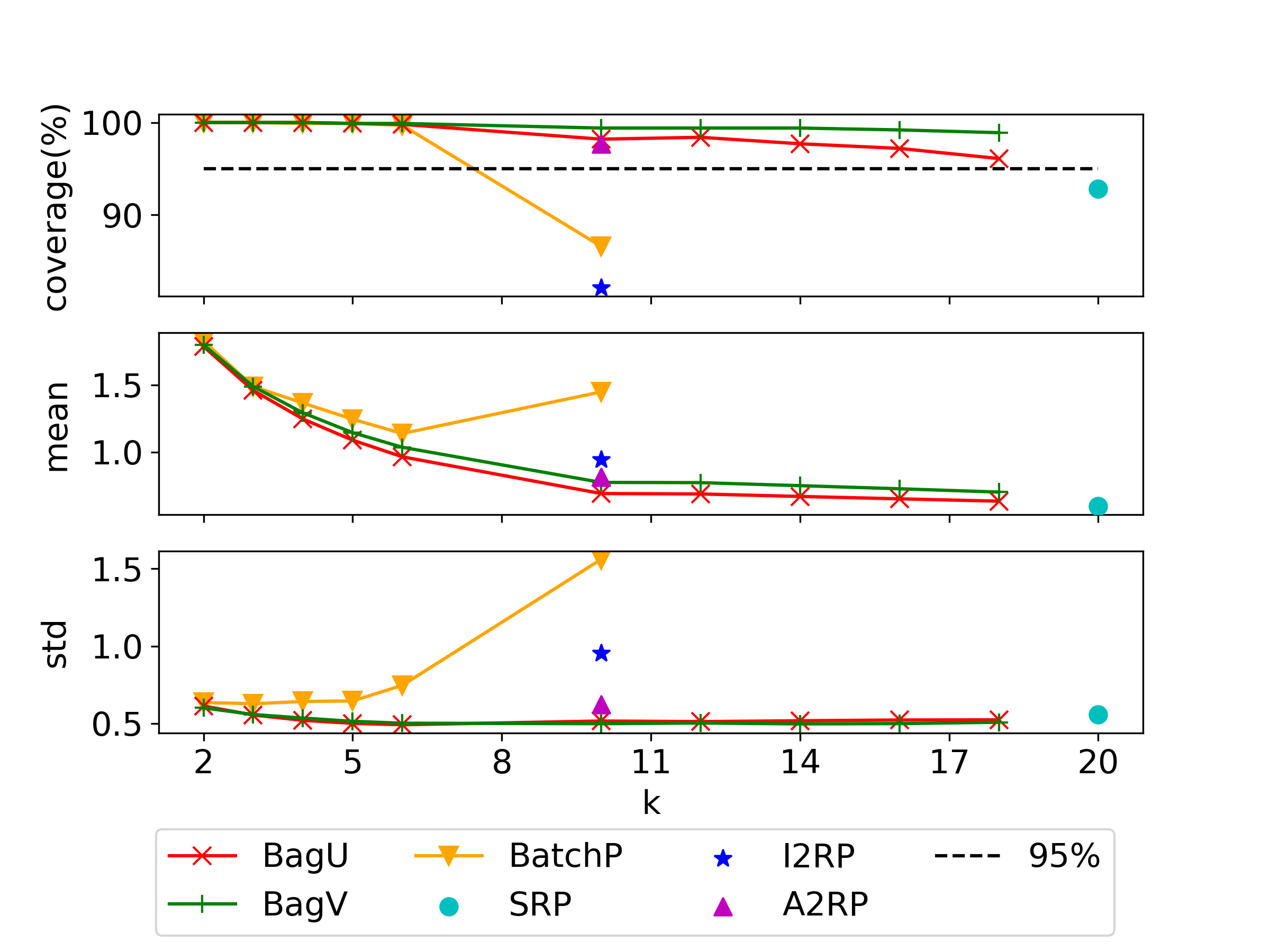}
         \caption{CRN, $n=50, n_1=30, n_2=20$}
         \label{fig:cvar gap CRN 50}
     \end{subfigure}\\
     \begin{subfigure}{0.49\textwidth}
         \centering
         \includegraphics[width=\textwidth]{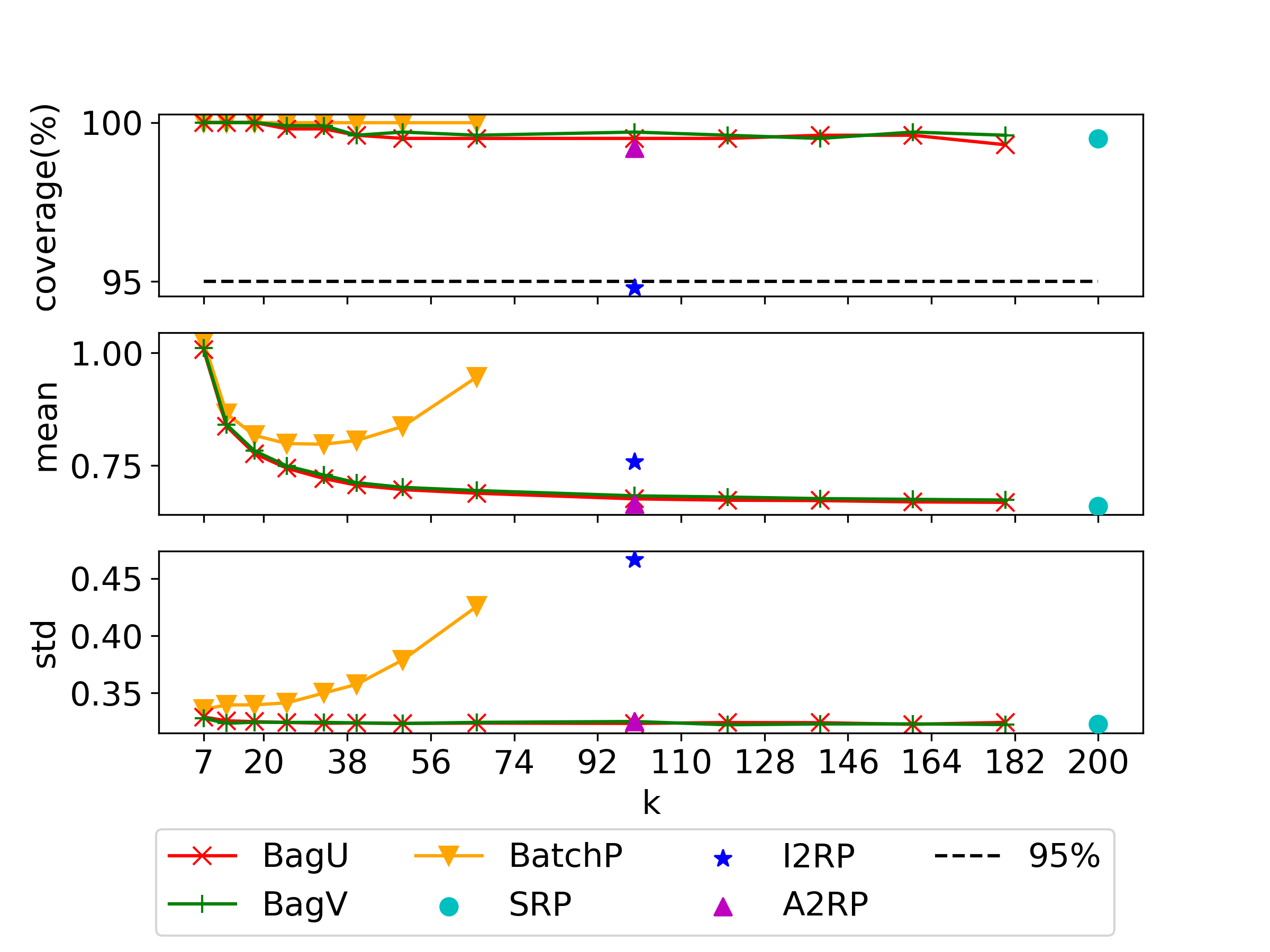}
         \caption{BC, $n=200, n_1=120, n_2=80$}
         \label{fig:cvar gap BC 200}
    \end{subfigure}
    \hfill
    \begin{subfigure}{0.49\textwidth}
         \centering
         \includegraphics[width=\textwidth]{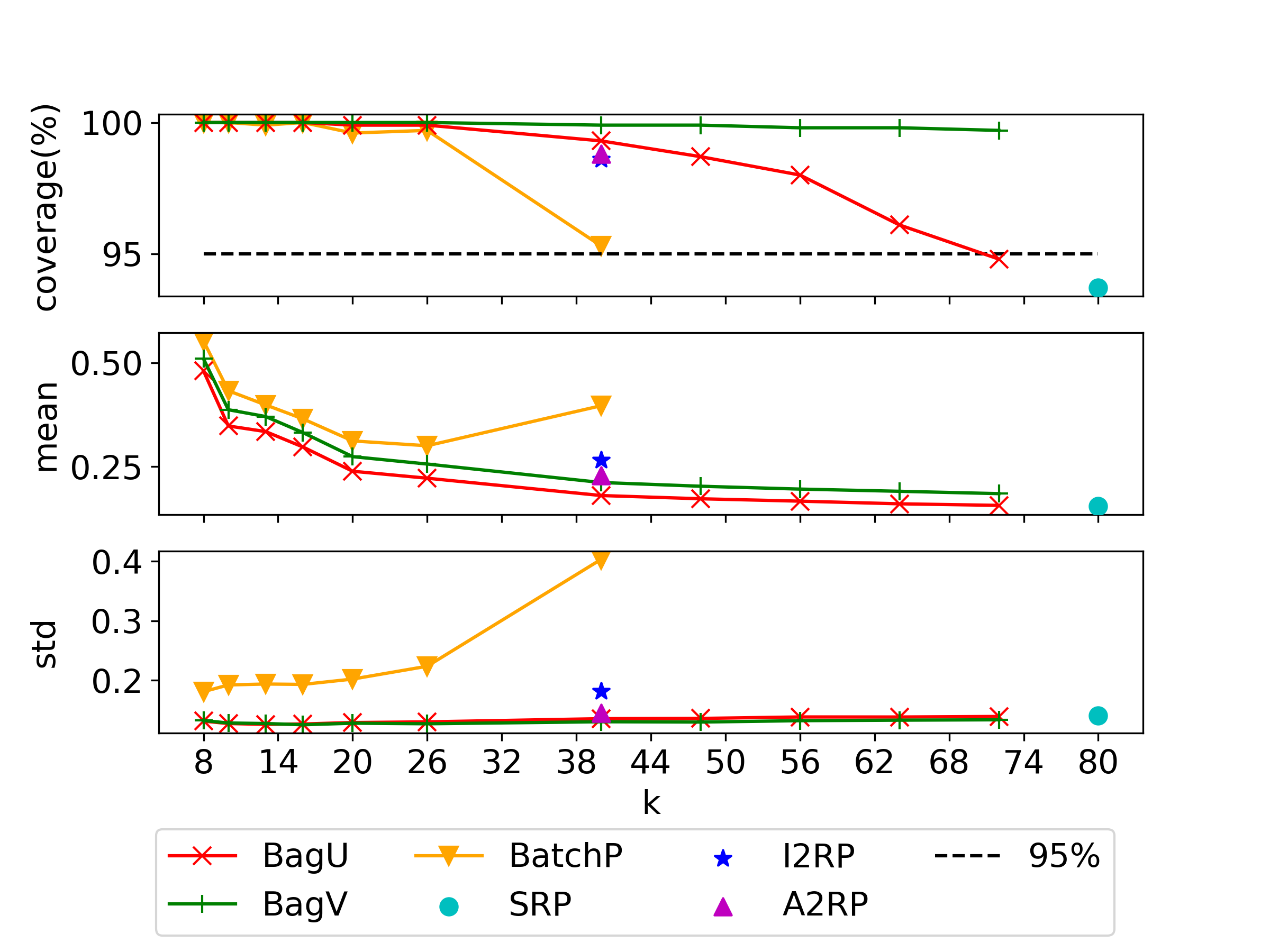}
         \caption{CRN, $n=200, n_1=120, n_2=80$}
         \label{fig:cvar gap CRN 200}
     \end{subfigure}
    \caption{CVaR problem \eqref{cvar}. Upper bounds of optimality gaps.}
    \label{fig:cvar gap}
\end{figure}

\begin{figure}[h]
    \centering
          \begin{subfigure}{0.49\textwidth}
         \centering
         \includegraphics[width=\textwidth]{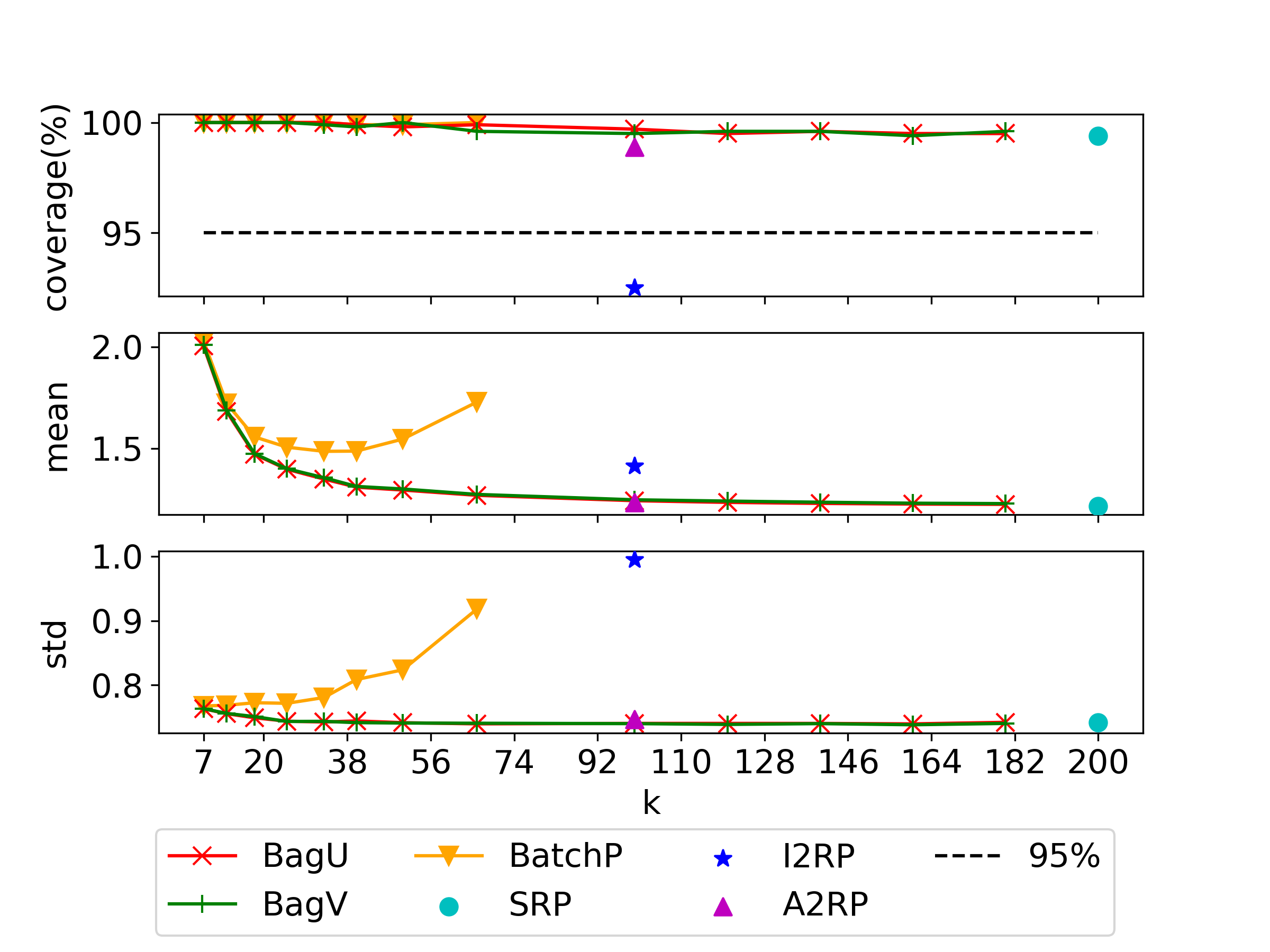}
         \caption{BC, $n=200, n_1=120, n_2=80$}
         \label{fig:portfolio gap BC 200}
    \end{subfigure}
    \hfill
    \begin{subfigure}{0.49\textwidth}
         \centering
         \includegraphics[width=\textwidth]{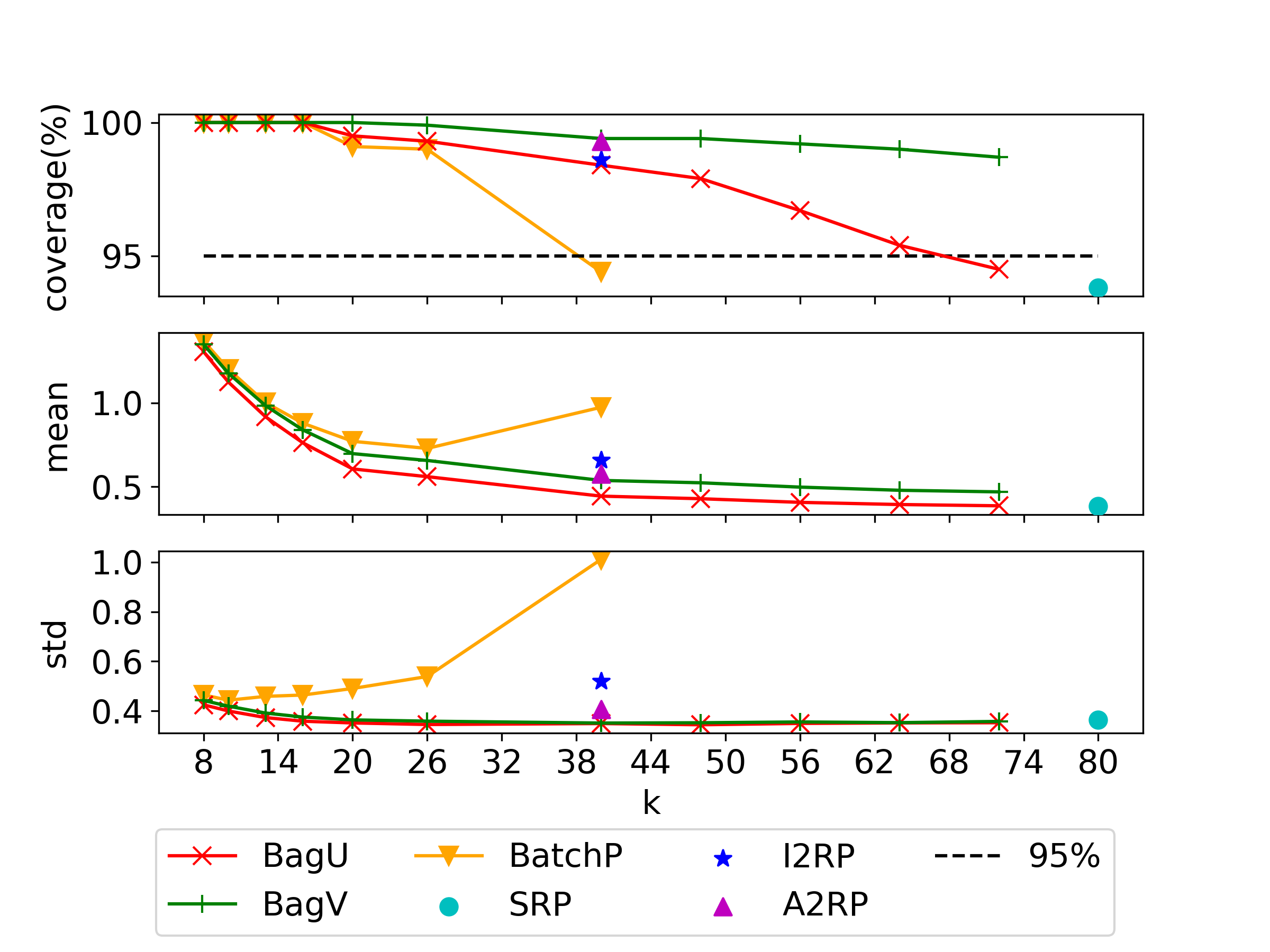}
         \caption{CRN, $n=200, n_1=120, n_2=80$}
         \label{fig:portfolio gap CRN 200}
     \end{subfigure}
    \caption{Portfolio problem \eqref{min_cvar}. Bounds of optimality gaps with $n=200$.}
    \label{fig:portfolio gap n=200}
\end{figure}

\begin{figure}[h]
    \centering
    \begin{subfigure}{0.49\textwidth}
         \centering
         \includegraphics[width=\textwidth]{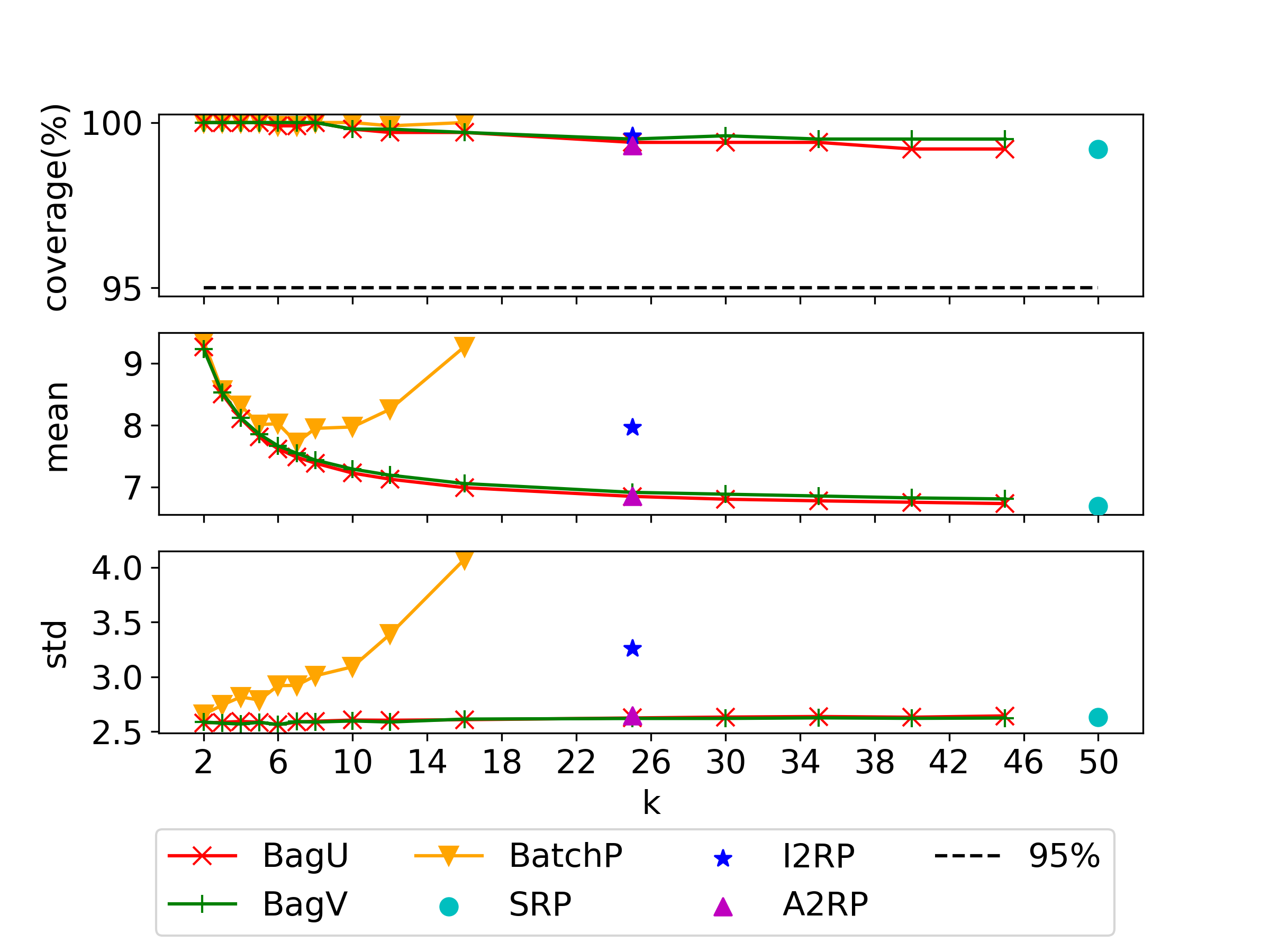}
         \caption{BC, $n=50, n_1=30, n_2=20$}
         \label{fig:integer gap BC 50}
     \end{subfigure}
     \hfill
     \begin{subfigure}{0.49\textwidth}
         \centering
         \includegraphics[width=\textwidth]{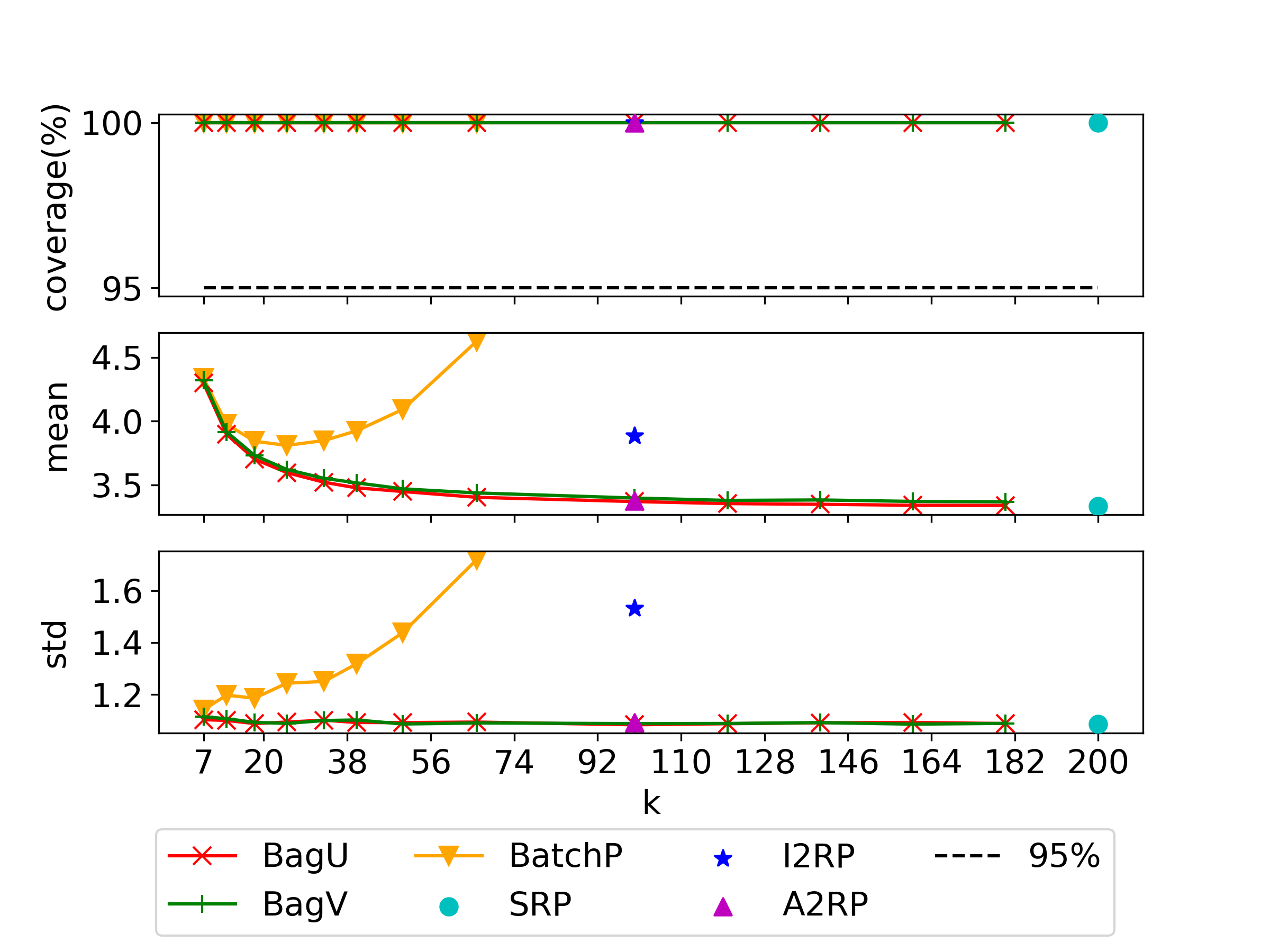}
         \caption{BC, $n=200, n_1=120, n_2=80$}
         \label{fig:integer gap BC 200}
    \end{subfigure}
    \caption{Integer problem \eqref{IP}. Bounds of optimality gaps via BC.}
    \label{fig:integer gap BC}
\end{figure}

\begin{figure}[h]
    \centering
     \begin{subfigure}{0.49\textwidth}
         \centering
         \includegraphics[width=\textwidth]{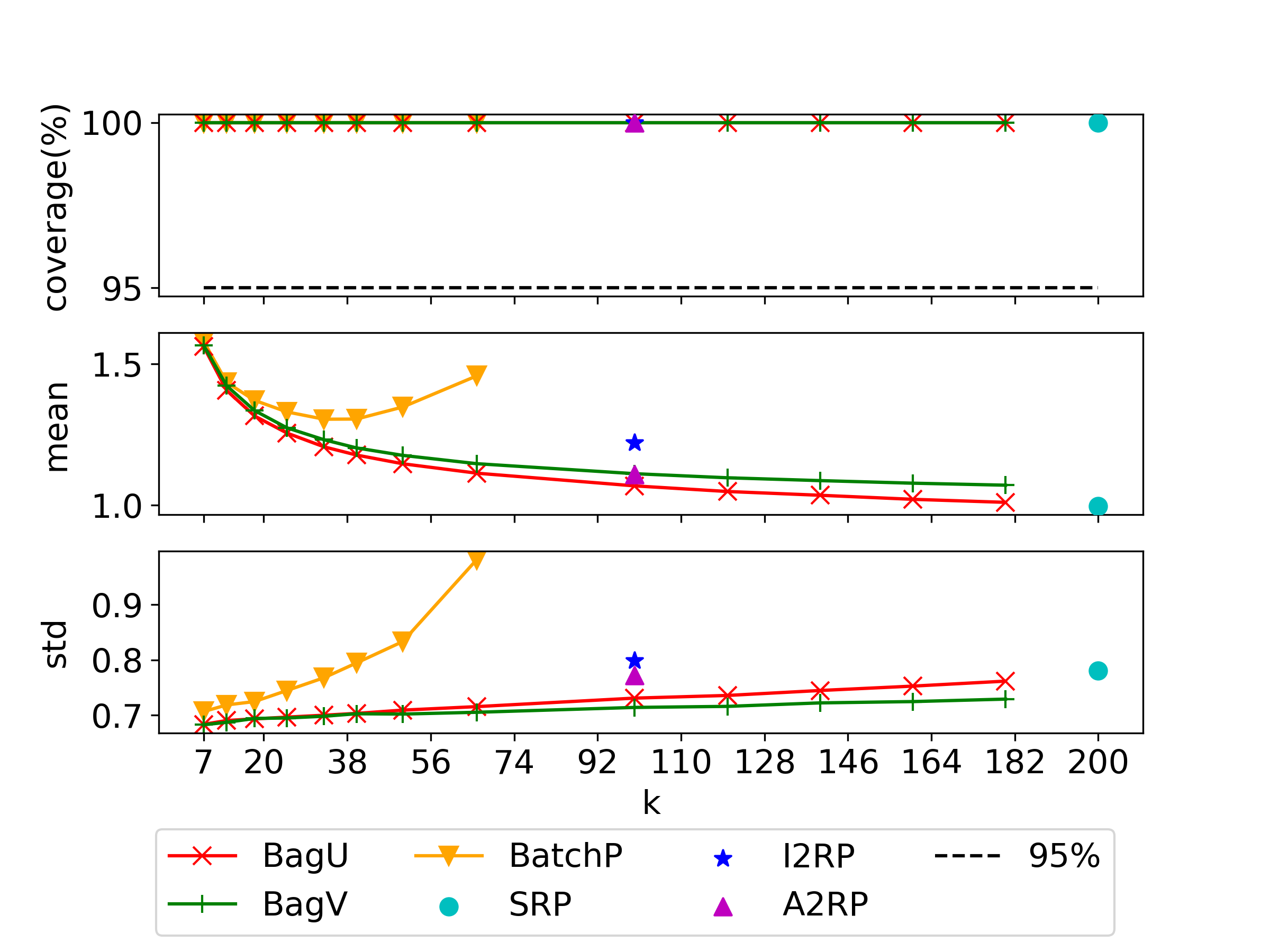}
         \caption{BC, $n=200, n_1=120, n_2=80$}
         \label{fig:binary gap BC 200}
    \end{subfigure}
    \hfill
    \begin{subfigure}{0.49\textwidth}
         \centering
         \includegraphics[width=\textwidth]{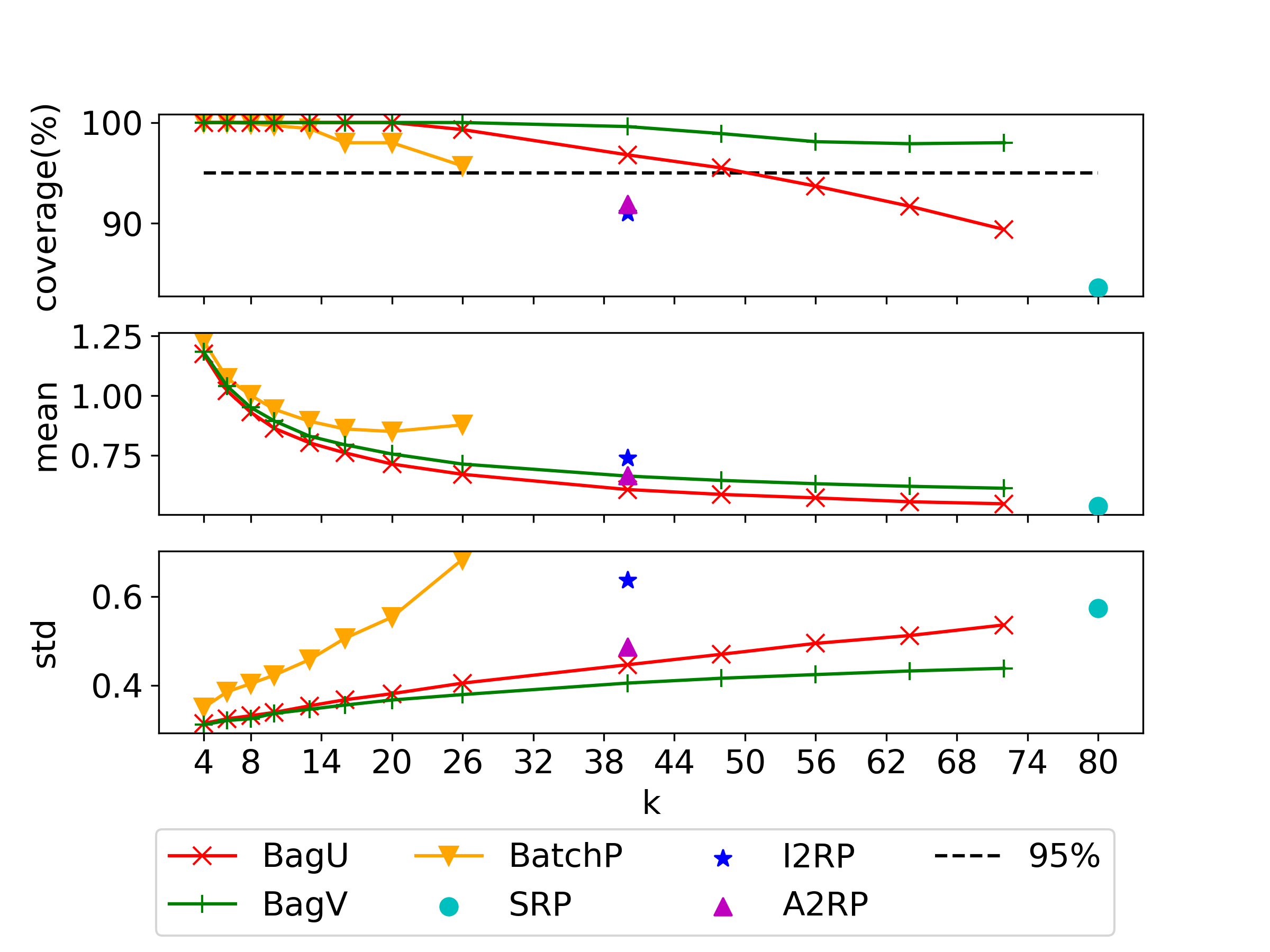}
         \caption{CRN, $n=200, n_1=120, n_2=80$}
         \label{fig:binary gap CRN 200}
     \end{subfigure}
    \caption{Simple linear problem \eqref{binary}. Bounds of optimality gaps with $n=200$.}
    \label{fig:binary gap n=200}
\end{figure}

\begin{figure}[h]
    \centering
    \begin{subfigure}{0.49\textwidth}
         \centering
         \includegraphics[width=\textwidth]{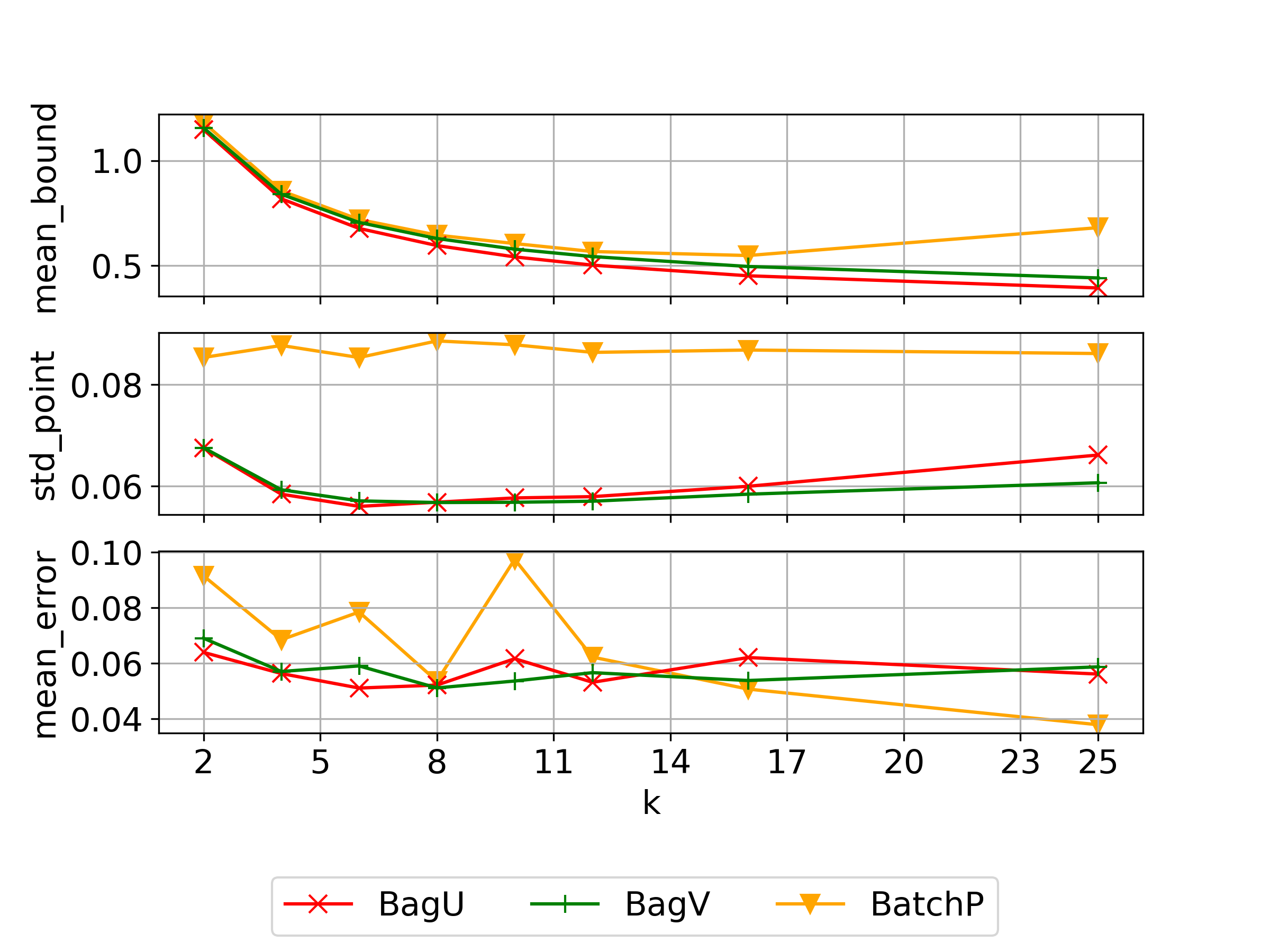}
         \caption{Bagging vs. BatchP, $n=50$}
         \label{fig:simplex with batching 50}
     \end{subfigure}
     \hfill
     \begin{subfigure}{0.49\textwidth}
         \centering
         \includegraphics[width=\textwidth]{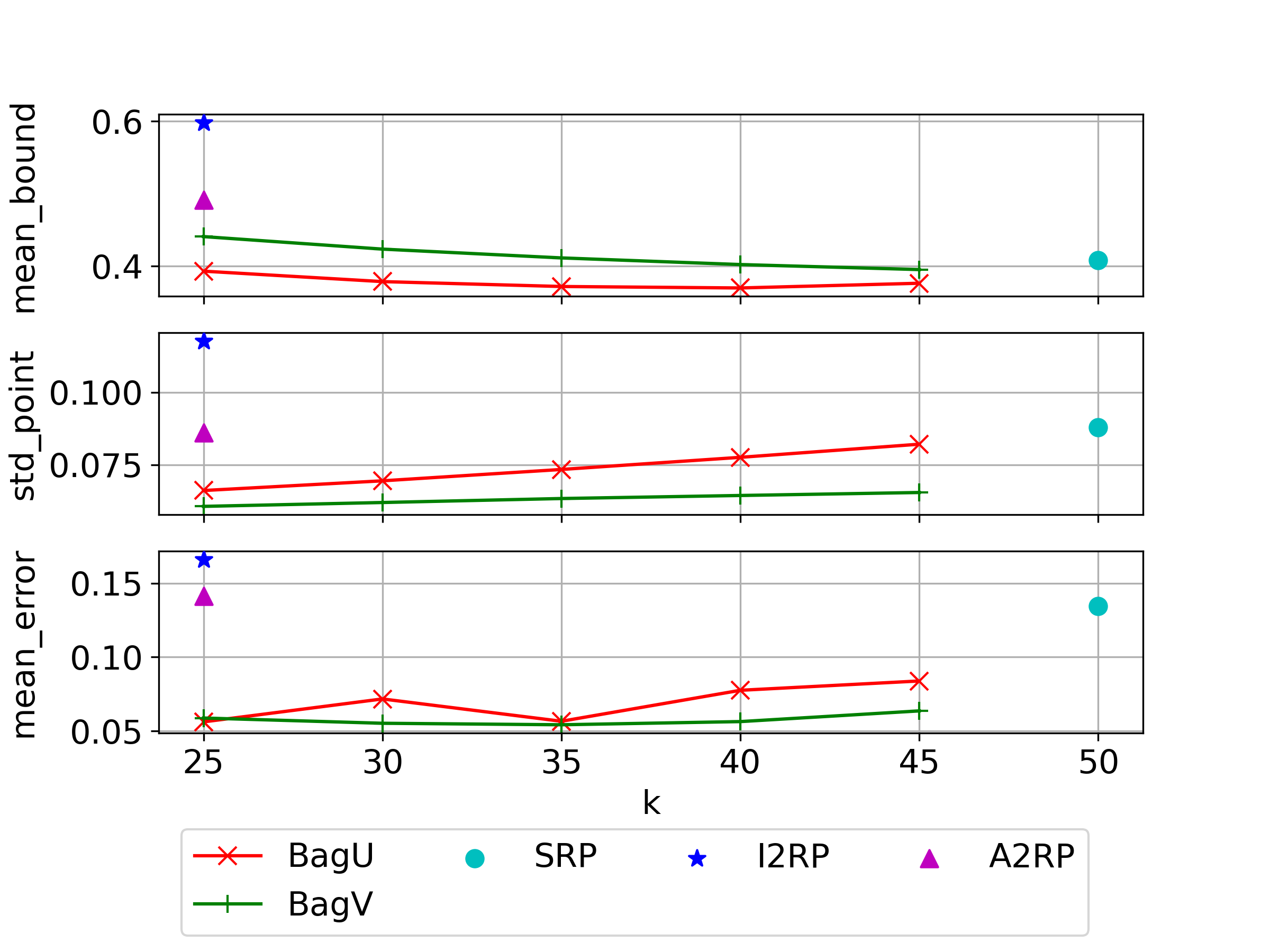}
         \caption{Bagging vs. SRP/I2RP/A2RP, $n=50$}
         \label{fig:simplex with SRP 50}
     \end{subfigure}
    \caption{Variance comparison on linear program \eqref{simplex} with $n=50$.}
    \label{fig:simplex opt val n=50}
\end{figure}

\bibliographystyleAPX{informs2014}
\bibliographyAPX{references}

\end{document}